\documentclass[12pt,reqno]{amsart}
\usepackage[margin=1in,marginparwidth=0.8in, marginparsep=0.1in]{geometry}
\pdfoutput=1

\usepackage[backend=biber, maxbibnames=99, bibencoding=utf8, doi=false, isbn=false, url=false]{biblatex}
\renewbibmacro{in:}{\ifentrytype{article}{}{\printtext{\bibstring{in}\intitlepunct}}}
\usepackage{etoolbox}

\makeatletter
\let\old@tocline\@tocline
\let\section@tocline\@tocline
\newcommand{\subsection@dotsep}{4.5}
\newcommand{\subsubsection@dotsep}{4.5}
\patchcmd{\@tocline}
  {\hfil}
  {\nobreak
     \leaders\hbox{$\m@th
        \mkern \subsection@dotsep mu\hbox{.}\mkern \subsection@dotsep mu$}\hfill
     \nobreak}{}{}
\let\subsection@tocline\@tocline
\let\@tocline\old@tocline

\patchcmd{\@tocline}
  {\hfil}
  {\nobreak
     \leaders\hbox{$\m@th
        \mkern \subsubsection@dotsep mu\hbox{.}\mkern \subsubsection@dotsep mu$}\hfill
     \nobreak}{}{}
\let\subsubsection@tocline\@tocline
\let\@tocline\old@tocline

\let\old@l@subsection\l@subsection
\let\old@l@subsubsection\l@subsubsection

\def\@tocwriteb#1#2#3{%
  \begingroup
    \@xp\def\csname #2@tocline\endcsname##1##2##3##4##5##6{%
      \ifnum##1>\c@tocdepth
      \else \sbox\z@{##5\let\indentlabel\@tochangmeasure##6}\fi}%
    \csname l@#2\endcsname{#1{\csname#2name\endcsname}{\@secnumber}{}}%
  \endgroup
  \addcontentsline{toc}{#2}%
    {\protect#1{\csname#2name\endcsname}{\@secnumber}{#3}}}%

\newlength{\@tocsectionindent}
\newlength{\@tocsubsectionindent}
\newlength{\@tocsubsubsectionindent}
\newlength{\@tocsectionnumwidth}
\newlength{\@tocsubsectionnumwidth}
\newlength{\@tocsubsubsectionnumwidth}
\newcommand{\settocsectionnumwidth}[1]{\setlength{\@tocsectionnumwidth}{#1}}
\newcommand{\settocsubsectionnumwidth}[1]{\setlength{\@tocsubsectionnumwidth}{#1}}
\newcommand{\settocsubsubsectionnumwidth}[1]{\setlength{\@tocsubsubsectionnumwidth}{#1}}
\newcommand{\settocsectionindent}[1]{\setlength{\@tocsectionindent}{#1}}
\newcommand{\settocsubsectionindent}[1]{\setlength{\@tocsubsectionindent}{#1}}
\newcommand{\settocsubsubsectionindent}[1]{\setlength{\@tocsubsubsectionindent}{#1}}

\renewcommand{\l@section}{\section@tocline{1}{\@tocsectionvskip}{\@tocsectionindent}{}{\@tocsectionformat}}%
\renewcommand{\l@subsection}{\subsection@tocline{1}{\@tocsubsectionvskip}{\@tocsubsectionindent}{}{\@tocsubsectionformat}}%
\renewcommand{\l@subsubsection}{\subsubsection@tocline{1}{\@tocsubsubsectionvskip}{\@tocsubsubsectionindent}{}{\@tocsubsubsectionformat}}%
\newcommand{\@tocsectionformat}{}
\newcommand{\@tocsubsectionformat}{}
\newcommand{\@tocsubsubsectionformat}{}
\expandafter\def\csname toc@1format\endcsname{\@tocsectionformat}
\expandafter\def\csname toc@2format\endcsname{\@tocsubsectionformat}
\expandafter\def\csname toc@3format\endcsname{\@tocsubsubsectionformat}
\newcommand{\settocsectionformat}[1]{\renewcommand{\@tocsectionformat}{#1}}
\newcommand{\settocsubsectionformat}[1]{\renewcommand{\@tocsubsectionformat}{#1}}
\newcommand{\settocsubsubsectionformat}[1]{\renewcommand{\@tocsubsubsectionformat}{#1}}
\newlength{\@tocsectionvskip}
\newcommand{\settocsectionvskip}[1]{\setlength{\@tocsectionvskip}{#1}}
\newlength{\@tocsubsectionvskip}
\newcommand{\settocsubsectionvskip}[1]{\setlength{\@tocsubsectionvskip}{#1}}
\newlength{\@tocsubsubsectionvskip}
\newcommand{\settocsubsubsectionvskip}[1]{\setlength{\@tocsubsubsectionvskip}{#1}}

\patchcmd{\tocsection}{\indentlabel}{\makebox[\@tocsectionnumwidth][l]}{}{}
\patchcmd{\tocsubsection}{\indentlabel}{\makebox[\@tocsubsectionnumwidth][l]}{}{}
\patchcmd{\tocsubsubsection}{\indentlabel}{\makebox[\@tocsubsubsectionnumwidth][l]}{}{}

\newcommand{\@sectypepnumformat}{}
\renewcommand{\contentsline}[1]{%
  \expandafter\let\expandafter\@sectypepnumformat\csname @toc#1pnumformat\endcsname%
  \csname l@#1\endcsname}
\newcommand{\@tocsectionpnumformat}{}
\newcommand{\@tocsubsectionpnumformat}{}
\newcommand{\@tocsubsubsectionpnumformat}{}
\newcommand{\setsectionpnumformat}[1]{\renewcommand{\@tocsectionpnumformat}{#1}}
\newcommand{\setsubsectionpnumformat}[1]{\renewcommand{\@tocsubsectionpnumformat}{#1}}
\newcommand{\setsubsubsectionpnumformat}[1]{\renewcommand{\@tocsubsubsectionpnumformat}{#1}}
\renewcommand{\@tocpagenum}[1]{%
  \hfill {\mdseries\@sectypepnumformat #1}}

\let\oldappendix\appendix
\renewcommand{\appendix}{%
  \leavevmode\oldappendix%
  \addtocontents{toc}{%
    \protect\settowidth{\protect\@tocsectionnumwidth}{\protect\@tocsectionformat\sectionname\space}%
    \protect\addtolength{\protect\@tocsectionnumwidth}{2em}}%
}
\makeatother

\makeatletter
\settocsectionnumwidth{2em}
\settocsubsectionnumwidth{2.5em}
\settocsubsubsectionnumwidth{3em}
\settocsectionindent{1pc}%
\settocsubsectionindent{\dimexpr\@tocsectionindent+\@tocsectionnumwidth}%
\settocsubsubsectionindent{\dimexpr\@tocsubsectionindent+\@tocsubsectionnumwidth}%
\makeatother

\settocsectionvskip{10pt}
\settocsubsectionvskip{0pt}
\settocsubsubsectionvskip{0pt}

\settocsectionformat{\scshape}%
\settocsubsectionformat{\mdseries}
\settocsubsubsectionformat{\mdseries}
\setsectionpnumformat{\scshape}%
\setsubsectionpnumformat{\mdseries}
\setsubsubsectionpnumformat{\mdseries}

\let\oldtableofcontents\tableofcontents
\renewcommand{\tableofcontents}{%
  \vspace*{-\linespacing}%
  \oldtableofcontents}

\usepackage{soul}
\usepackage{amsfonts,amssymb,latexsym,amsmath,amsthm,graphicx}
\usepackage{sidecap}
\usepackage{calligra}

\usepackage{esvect}

\usepackage[all,cmtip]{xy}

\usepackage[mathcal]{euscript}

\usepackage{tikz}
\usetikzlibrary{arrows.meta,decorations.pathmorphing,backgrounds,positioning,fit}

\usepackage[bookmarks=true, bookmarksopen=true,%
bookmarksdepth=3,bookmarksopenlevel=2,%
colorlinks=true,%
linkcolor=blue,%
citecolor=blue,%
filecolor=blue,%
menucolor=blue,%
urlcolor=blue]{hyperref}

\usepackage{times}
\usepackage{mathrsfs}

\newtheorem{lemma}{Lemma}[section]
\newtheorem{proposition}[lemma]{Proposition}
\newtheorem{corollary}[lemma]{Corollary}
\newtheorem{theorem}[lemma]{Theorem}

\theoremstyle{definition}
\newtheorem{definition}[lemma]{Definition}

\theoremstyle{remark} 
\newtheorem*{remark}{Remark}
\newtheorem*{example}{Example}

\newtheorem{conjecture}[lemma]{Conjecture}

\newcommand{\A}{\mathbb{A}}
\newcommand{\B}{\mathbb{B}}

\newcommand{\D}{\mathbb{D}}

\newcommand{\F}{\mathbb{F}}

\renewcommand{\L}{\mathbb{L}}

\newcommand{\N}{\mathbb{N}}

\renewcommand{\P}{\mathbb{P}}

\newcommand{\R}{\mathbb{R}}
\renewcommand{\S}{\mathbb{S}}
\newcommand{\T}{\mathbb{T}}

\newcommand{\X}{\mathbb{X}}

\newcommand{\Z}{\mathbb{Z}}

\newcommand{\cA}{\mathcal{A}}
\newcommand{\cB}{\mathcal{B}}
\newcommand{\cC}{\mathcal{C}}
\newcommand{\cD}{\mathcal{D}}
\newcommand{\cE}{\mathcal{E}}
\newcommand{\cF}{\mathcal{F}}
\newcommand{\cG}{\mathcal{G}}
\newcommand{\cH}{\mathcal{H}}

\newcommand{\cK}{\mathcal{K}}
\newcommand{\cL}{\mathcal{L}}
\newcommand{\cM}{\mathcal{M}}

\newcommand{\cP}{\mathcal{P}}

\newcommand{\cS}{\mathcal{S}}

\newcommand{\cU}{\mathcal{U}}
\newcommand{\cV}{\mathcal{V}}
\newcommand{\cW}{\mathcal{W}}
\newcommand{\cX}{\mathcal{X}}
\newcommand{\cY}{\mathcal{Y}}

\newcommand{\bC}{\mathbf{C}}

\newcommand{\mf}{\mathfrak}

\DeclareMathOperator{\ev}{ev}
\DeclareMathOperator{\coev}{coev}
\DeclareMathOperator{\id}{id}
\DeclareMathOperator{\Id}{Id}
\DeclareMathOperator{\Tr}{Tr}

\DeclareMathOperator{\Hom}{Hom}
\DeclareMathOperator{\Fun}{Fun}

\DeclareMathOperator{\End}{End}
\DeclareMathOperator{\eend}{end}

\DeclareMathOperator{\Push}{Push}
\DeclareMathOperator{\maps}{maps}
\DeclareMathOperator{\Map}{Map}
\DeclareMathOperator{\Lax}{Lax}
\DeclareMathOperator{\Cat}{Cat}
\DeclareMathOperator{\Ins}{Ins}

\DeclareMathOperator{\Cone}{Cone}

\DeclareMathOperator{\Mod}{Mod}
\DeclareMathOperator{\vvt}{Vect}
\DeclareMathOperator{\mmod}{mod}
\DeclareMathOperator{\Perf}{Perf}
\DeclareMathOperator{\Ch}{Ch}
\DeclareMathOperator{\Spec}{Spec}

\DeclareMathOperator{\Coh}{Coh}
\DeclareMathOperator{\Filt}{Filt}
\DeclareMathOperator{\Star}{Star}
\DeclareMathOperator{\Loc}{Loc}
\newcommand{\HH}{\mathrm{HH}}

\DeclareMathOperator{\HC}{HC}
\DeclareMathOperator{\Nrv}{N}
\DeclareMathOperator{\Sing}{Sing}
\DeclareMathOperator{\Alg}{Alg}
\DeclareMathOperator{\Sh}{Sh}
\DeclareMathOperator{\Cosh}{Cosh}
\DeclareMathOperator{\PreSh}{PreSh}
\DeclareMathOperator{\PreCosh}{PreCosh}

\newcommand{\pt}{\mathrm{pt}}
\newcommand{\uph}{\mathrm{h}}
\newcommand{\Compacts}{\mathcal{K}}
\newcommand{\Opens}{\mathcal{U}}

\DeclareMathOperator{\sheafHom}{\mathcal{H}\kern -1.2pt \mathit{om}}
\DeclareMathOperator*{\colim}{colim}

\newcommand{\del}{\partial}
\DeclareMathOperator{\LLoc}{\L\mathrm{oc}}
\DeclareMathOperator{\FFilt}{\F\mathrm{ilt}}
\DeclareMathOperator{\PPerf}{\P\mathrm{erf}}
\DeclareMathOperator{\SSh}{\S\mathrm{h}}
\DeclareMathOperator{\DGCat}{DGCat}
\DeclareMathOperator{\dgcat}{dgcat}

\DeclareMathOperator{\Arb}{Arb}
\DeclareMathOperator{\Ind}{Ind}

\DeclareMathOperator{\DHH}{DHH}
\DeclareMathOperator{\LD}{LD}

\newcommand{\Fr}{\mathrm{Fr}}
\newcommand{\rig}{\mathrm{rig}}
\newcommand{\RR}{\mathrm{RR}}
\renewcommand{\Lax}{\mathrm{Lax}}
\newcommand{\oplax}{\mathrm{oplax}}

\newcommand{\fib}{\mathrm{fib}}
\newcommand{\cofib}{\mathrm{cofib}}

\newcommand{\narrowdown}{\mkern-6mu \downarrow \mkern-3mu}

\let\emptyset\varnothing

\mathchardef\mh="2D
\newcommand{\sslash}{\mathbin{/\mkern-6mu/}}

\newcommand{\tightoverset}[2]{%
  \mathop{#2}\limits^{\vbox to -.5ex{\kern-0.75ex\hbox{$#1$}\vss}}}

\newsavebox{\smlmat}
\savebox{\smlmat}{$\left(\begin{smallmatrix}1&0\\0&1\end{smallmatrix}\right)$}

\newcommand{\arb}{\cA} 
\newcommand{\DeltaDelta}{\Delta\!\!\!\!\Delta} 
\newcommand{\DDelta}{\mathrm{D}\Delta}

\newcommand{\pullbackcorner}[1][dr]{\save*!/#1-1.2pc/#1:(-1,1)@^{|-}\restore}

\author{Vivek Shende and Alex Takeda}

\title{Calabi-Yau structures on topological Fukaya categories}

\begin{document}
\tikzset{darrow/.style={double distance = 4pt,>={Implies},->},
darrowthin/.style={double equal sign distance,>={Implies},->},
tarrow/.style={-,preaction={draw,darrow}},
qarrow/.style={preaction={draw,darrow,shorten >=0pt},shorten >=1pt,-,double,double
distance=0.2pt}} 

\thispagestyle{empty}

\begin{abstract}
We develop a local-to-global formalism for constructing  Calabi-Yau structures for 
global sections of constructible sheaves or cosheaves of dg categories.  The required data --- a morphism between the
sheafified Hochschild homology with the topological dualizing sheaf, satisfying a nondegeneracy condition --- specializes to the classical
notion of orientation when applied to the category of local systems on a manifold.  We apply this 
construction to the cosheaves on arboreal skeleta arising in the microlocal approach 
to the A-model. 
\end{abstract}

\maketitle

\tableofcontents

\section{Introduction}
A compact complex manifold with trivial canonical bundle is said to be Calabi-Yau.  On such a manifold, 
Serre duality takes the simple form $\Hom(E, F) \cong \Hom(F, E)^\vee[d]$.  This is a categorical
analogue of a nondegenerate trace on an algebra, and there is an associated notion of Calabi-Yau category,
generalizing the notion of Frobenius algebra.

Such categories are precisely the categories of boundary
conditions in two dimensional topological field theories, from which these theories may be reconstructed \cite{Cos1, Lur-tft}.  
In particular, much of the relationship between the topological A-model and B-model is captured
by an equivalence between the Fukaya category of Lagrangians on a symplectic manifold, and the derived category of coherent sheaves on the mirror algebraic variety, these being the categories of boundary conditions of 
open strings in the respective theories \cite{Kon}.   Moreover, other invariants associated to these theories
should in principle be computable from the Calabi-Yau categories, in particular
the partition function, i.e. Gromov-Witten type invariants \cite{Cos2}, and also BPS counts, i.e. 
Donaldson-Thomas type invariants \cite{KoSo, KoS}.   Related considerations show that moduli of objects in CY categories carry 
shifted symplectic structures and admit quantizations \cite{PTVV, CPTVV, BD1}.

\vspace{2mm} 

Our purpose here is to establish a local-to-global formalism for constructing certain Calabi-Yau categories which
arise in the topological A-model.   That such a local-to-global procedure should exist is strongly suggested by
the fact that, on the one hand, certain Fukaya categories should carry Calabi-Yau structures, and on the other, 
they are known to be equivalent to categories of microlocal sheaves in certain exact situations \cite{NZ, N1, GPS1, GPS2, Sh, NS, GPS3}.   
This equivalence underlies one approach to homological mirror symmetry    \cite{FLTZ, N5, N6, Ku, GaSh}.  
If one wants to e.g. use the ideas of \cite{Cos2} to extract enumerative mirror symmetry from this microlocal sheaf description, 
it will be necessary to construct Calabi-Yau structures in the microlocal sheaf setting, prove a comparison with 
\cite{Ganatra}, and and show that the aforementioned mirror symmetry results respect Calabi-Yau structures.  Here we 
take the first step. 

A second motivation is that moduli of microlocal sheaves
naturally describe many spaces of interest, for instance: moduli spaces of (possibly irregular) local systems, 
positroid varieties, multiplicative Nakajima varieties, the A-polynomial of a knot, etc.  
Many of these spaces are known or expected to carry symplectic structures and quantizations; in light of 
\cite{PTVV, CPTVV, BD1}, our work provides a uniform construction of these structures.

\vspace{2mm} 
The setting for our local-to-global formalism is: 

\begin{definition}\label{def: ccc and csc spaces}
A {\em constructibly cosheaved in categories (\textsc{ccc}) space} is a pair $(X, \cF)$ where $X$ is a topological space and $\cF$ is a cosheaf on $X$ valued in the $\infty$-category of small dg categories, constructible with respect to some finite stratification of $X$. We define analogously the notion of {\em constructibly sheaved in categories (\textsc{csc}) space} $(X, \cF)$. 
\end{definition}

Given a \textsc{ccc} or \textsc{csc} space as above, we say that a map $X \to Y$ or $Y \to X$ is {\em constructible} if it is stratifiable by a locally finite stratification with respect to which the (co)sheaf is constructible.  In particular 
we say $U \subset X$ is {\em constructible} if it is a union of strata for some finite stratification of $X$ with respect to
which $\cF$ is constructible.  

\begin{example} 
Let $M$ be a manifold and $k$ a field.  Let $Loc(M)$ be the category of (unbounded, possibly infinite rank) local systems of complexes of $k$-vector spaces on $M$.  Then $U \mapsto Loc(U)$, together the evident restriction maps, determines a sheaf $\cL oc$ on $M$ valued in the category of presentable dg categories with limit-preserving morphisms. To get a \textsc{csc} space, let $Loc^b(U) \subset Loc(U)$ be the full subcategory of local systems of perfect $k$-modules, and $\cL oc^b \subset \cL oc$ the corresponding full subsheaf. Then $(M, \cL oc^b)$ is a \textsc{csc} space. 

Taking adjoints gives a cosheaf, also denoted $\cL oc$, on $M$, valued in the 
category of presentable dg categories with colimit-preserving morphisms. Let $Loc^c(U) \subset Loc(U)$ be the full subcategory of compact objects, and $\cL oc^c \subset \cL oc$ the corresponding full subcosheaf. Then $(M, \cL oc^c)$ is a \textsc{ccc} space. The cospecialization maps of the cosheaf $\cL oc^c$ are nontrivial to describe, being the analogue in sheaf theory of wrapping in the Fukaya category \cite{N5, GPS3}. 
\end{example}

\begin{example}
	Let us consider the example above with $M=S^1$, which is Example 3.19 in \cite{N5}. Taking global sections, we get $Loc(S^1) = \Mod(k[t, t^{-1}])$, that is, the (large) category of unbounded complexes of modules over the associative algebra $k[t,t^{-1}]$, and its subcategory of compact objects is $Loc^c(S^1) = \Perf(k[t, t^{-1}])$, that is, complexes that are {\em perfect as} $k[t,t^{-1}]$-modules. On the other hand, the objects of $Loc^b(S^1)$ are the complexes of $k[t,t^{-1}]$-modules that have finite global rank, that is, are perfect as $k$-modules.
	
	We can describe these categories in another way, `on the other side of mirror symmetry'. There are canonical equivalences $Loc^c(S^1) = \Coh(\mathbb{G}_m)$, (a dg enhancement of) the derived category of $\Spec k[\Z]$, and $Loc^b(S^1) = \Coh_{prop}(\mathbb{G}_m)$, which can be identified with its subcategory of objects that have proper support; these objects are all direct sums of skyscraper sheaves, whose supports correspond to the eigenvalues of the monodromy around $S^1$.
	
	Note that, though the co/sheaf of large categories $\cL oc$ has the same information when seen either as a cosheaf or as a sheaf (by taking right adjoints), the global sections of the sheaf $\cL oc^b$ of small categories has strictly less information; in this example, it sees only finite collections of points of the space $\mathbb{G}_m$.
\end{example}

Microlocal sheaf theory provides the structure of a \textsc{csc} \cite{KS} and \textsc{ccc} \cite{N5} space 
to any (singular) subanalytic Legendrian in a cosphere bundle, $X \subset S^*M$.  More generally, 
one can do with only the germ of an embedding into some contact manifold, plus some topological `Maslov' data
\cite{Sh, NS}.  We briefly recall this in Section \ref{sec:microsheaves}.

For Legendrians $X$ with a certain prescribed set of `arboreal' singularities, Nadler showed in \cite{N3} that 
the local categories can be described combinatorially; in Section \ref{sec:arb} we turn this into a definition
to give a construction of some \textsc{csc} and \textsc{ccc} spaces which should be intelligible to a reader innocent of microlocal sheaf theory
and symplectic geometry.

\vspace{2mm}

\subsection{Orientations on \textsc{csc} spaces}

Let $(X,\cF)$ be a \textsc{csc} space.  Hochschild homology is functorial, so
$U \mapsto \HH(\cF(U))$ determines a presheaf.  
We understand this presheaf as taking values in the dg category 
of chain complexes localized along quasi-isomorphisms.\footnote{
For consistency with using \cite{LurHTT, LurHA} as our foundations, we officially view `dg' as meaning `$k$-linear stable $\infty$'.   
Note that in working with such dg categories we also must and do work with what would have historically been
termed (dg enhanced) derived functors, and we do not further indicate this in the notation.   
}
The presheaf $U \mapsto \HH(\cF(U))$ is not generally a sheaf.\footnote{ 
 Such a presheaf should be a sheaf if it ``carries covers to limits'', which in the homological sense at hand 
means that the natural map from a \u{C}ech complex valued in the 
presheaf  to its global sections should be an isomorphism.  For instance the 
constant sheaf with stalk $k$ on the topological space $X$ must have as its global sections
a complex computing $\mathrm{H}^*(X, k)$.} 
We form the sheafification $\cH\cH(\cF)$ of this presheaf; the resulting sheaf carries a $S^1$ action inherited from the local $S^1$ actions.

\begin{example}
$\HH(\cL oc^b(U))$ can
be identified with $k$-valued cochains on the free loop space of $U$.  Meanwhile, 
$\cH \cH(\cL oc^b)$ is the constant sheaf with stalk $k$, and the $S^1$ action is trivial. 
\end{example} 

Hochschild homology gives an universal pairing on Hom spaces: given any dg category $\cC$ and objects $x,y$, there is a map 
\[ \cC(y, x) \otimes \cC(x, y) \to \HH(\cC). \] 
It will be convenient for us to have an explicit description for this map. For that, we pick an appropriate representative for $\HH(\cC) = \cC_\Delta \otimes_{\cC^e} \cC_\Delta$, by using `bar complexes'. We choose a set of objects $S$ of $\cC$ meeting every isomorphism class, and define the following object of $\vvt$ for every $m,n$.
\[
	\HH(\cC)_{(m,n)} = \bigoplus_{x_i,y_i \in S} \cC(y_n,x_0) \otimes \cC(x_0,x_1)[1] \otimes \dots \otimes \cC(x_m,y_0) \otimes \dots \otimes \cC(y_n,x_0)[1]
\]
We can then put the usual differential on $\HH(\cC,\cM) = \bigoplus_{m,n} \HH(\cC,\cM)_{m,n}$, combining the internal differential of the Hom spaces and the composition maps of $\cC$; this complex then computes the Hochschild homology of $\cC$\footnote{Another way of phrasing this definition is as the realization of a simplicial object in $\vvt$, as in \cite[Sec.4.5]{HSS}}. For any pair of objects $x,y$ in $S$, the pairing map is given by mapping into the $(m,n)=(0,0)$ summand with $y_0=y,x_0=x$.

By definition, a proper Calabi-Yau structure is a map $\HH(\cC)_{S^1} \to k[-d]$ 
such that the composition 
\[ \cC(x, y) \otimes \cC(y, x) \to \HH(\cC) \to \HH(\cC)_{S^1} \to k[-d] \]
is a perfect pairing for every $x,y$. Note the existence of such a structure implies that 
$\cC(x, y)$ is perfect, i.e. $\cC$ is proper.

Now consider a \textsc{csc} space $(X, \cF)$.  We recall that Hom in a sheaf of categories
is itself computed by a sheaf.  That is, for any open $U\subset X$ and 
any $x, y \in \cF(U)$,  the functor
\[ V \mapsto \Hom_{\cF(V)}(x|_V, y|_V) \]
defines a sheaf on $U$, which we denote by $\cF|_U^{x,y}$. We define an analog of the expression for $\HH(\cC)$ by picking spanning object sets for each open and using the Hom sheaves instead of the Hom spaces (Definition \ref{def:HHpresheaf}). This gives an explicit representative for the Hochschild homology presheaf $\HH(\cF)$, and for any $U$ and objects $x,y$ in the chosen object set of $\cF(U)$, we have a pairing map
\[ \cF|_U^{y,x} \otimes \cF|_U^{x,y} \to \HH(\cF). \]
We introduce the following local notion of proper Calabi-Yau structure: 
\begin{definition} \label{def:cscCY} 
Let $(X, \cF)$ be a \textsc{csc} space. An orientation of dimension $d$ on $(X,\cF)$ is a map of sheaves of complexes $\Omega: \cH\cH (\cF)_{S^1} \to \omega_X[-d]$ such that for any pair $x,y \in \cF(X)$, the morphism
\[ \cF|_X^{x, y} \to \sheafHom(\cF|_X^{y, x}, \omega_X)[-d], \]
adjoint to the trace pairing induced by $\Omega: \cH\cH (\cF) \to \cH\cH (\cF)_{S^1} \to \omega_X[-d]$, is a quasi-isomorphism of sheaves.
\end{definition}
If such an $\Omega$ exists we will say that it orients $(X,\cF)$, and that this \textsc{csc} space is orientable. One important feature of this definition is that, since it amounts to checking a quasi-isomorphism of sheaves, it satisfies a locality property (Proposition \ref{prop:sheafLocality}). In other words, it is enough to check that the restrictions define an orientation on each open in a cover of $X$.

\begin{example}
We have $\cH\cH(\cL oc^b) = k_M$, with trivial $S^1$-action.  If $M$ is a manifold, an orientation of dimension $\dim(M)$ of $(M, \cL oc^b)$ is thus a choice of isomorphism $k_M \simeq \omega_M[-\dim M]$, i.e., an orientation of $M$ in the sense of classical topology.  
\end{example}

\begin{example}
(Prop \ref{prop:cscCYpoint}) A \textsc{csc} space $(\mathrm{point}, \cF)$ over the point amounts to just a single dg category $\cF(\mathrm{point})$. An orientation  
of $(\mathrm{point}, \cF)$ in the sense of Definition \ref{def:cscCY} is just a proper (aka right) Calabi-Yau structure on $\cF(\mathrm{point})$. 
\end{example}

\begin{example}
(Prop \ref{prop:cscCYinterval}) Consider a \textsc{csc} space 
$( [0,1) , \cF)$, where $\cF$ is constructible with respect to the stratification of the half-open interval $[0,1) = \{0\} \cup (0,1)$.  
Then an orientation on $( [0,1) , \cF)$ is equivalent to a proper (aka right) relative Calabi-Yau structure on the morphism $\cF([0,1)) \to \cF((0,1))$ in the sense of Brav and Dyckerhoff \cite{BD1}. 
\end{example}

Given a continuous map $f: X \to Y$ we may pushforward sheaves of dg categories by the usual formula $f_* \cF(U) = \cF (f^{-1}(U))$.  
There are natural morphisms of sheaves $\cH \cH(f_* \cF) \to f_* \cH \cH (\cF)$ and $\cH \cH(f_* \cF)_{S^1} \to f_* \cH \cH (\cF)_{S^1}$. Note that by the definition of Verdier dualizing complex, there is a map $f_! \omega_X = f_! f^! \omega_Y \to \omega_Y$. Thus if $f$ is proper (so $f_* = f_!$), an orientation $\Omega$ on $X$ determines the composition
\[ \cH \cH(f_* \cF)_{S^1} \to f_* \cH \cH (\cF)_{S^1} \xrightarrow{\Omega} f_* \omega_X \to \omega_Y, \]
which we denote as $f_* \Omega$. 

\begin{theorem}\label{thm:cscCYpushforward}
(Theorem \ref{prop:cscCYpushforwardText}) Let $(X, \cF)$ be a \textsc{csc} space, and $f: X \to Y$ be proper and constructible. If $\Omega$ orients $(X, \cF)$, then $f_* \Omega$ orients $(Y, f_* \cF)$. 
\end{theorem} 

\begin{corollary}
If $(X, \cF)$ is a \textsc{csc} space and $X$ is compact, then an orientation of $(X, \cF)$ induces a proper Calabi-Yau structure on $\cF(X)$. 
\end{corollary}

\begin{corollary}\label{cor:introRelProperCY}
Suppose $X = X_0 \cup_{\partial X_0} (\partial X_0 \times [0,1))$ where 
$X_0$ is compact.  Suppose $(X, \cF)$ is a  \textsc{csc} space and $\cF$ is stratifiable with respect to some stratification which, on 
$\partial X_0 \times [0,1)$, is the product of a stratification on $\partial X_0$ and the stratification $[0,1) = \{0\} \cup (0,1)$. Then an orientation of $(X, \cF)$ induces a proper (aka right) Calabi-Yau structure on the map $\cF(X) \to \cF(\partial X_0 \times (0,1))$. 
\end{corollary}
\begin{proof}
Push forward along the map $X \to [0,1)$ carrying $X_0 \to 0$ and $\partial X_0 \times [0,1) \to [0,1)$. 
\end{proof}

\subsection{Orientations on \textsc{ccc} spaces}
Let $(X,\cF)$ be a \textsc{ccc} space.  Dually to the notion of orientation on \textsc{csc} space, 
the data of an orientation will be a map of cosheaves $\Omega: \omega_X^\vee \to  \cH\cH(\cF)^{S^1}[-d]$ from the linear dual of the dualizing sheaf to the cosheafification of the negative cyclic homology precosheaf. However, it will be more difficult to phrase the nondegeneracy condition, because we do not know a useful local analogue of the evaluated-at-objects Hom sheaf $\cF|_U^{x,y}$. %

Our formulation will use Lurie's covariant Verdier duality, which makes sense for (co)sheaves valued in any stable $\infty$-category $\cA$. Given a locally compact Hausdorff space $X$, the covariant Verdier duality functor is a certain map $\D_X: \PreCosh(X,\cA) \to \PreSh(X,\cA)$, which restricts to an equivalence $\Cosh(X,\cA) \xrightarrow{\sim} \Sh(X,\cA)$. When $\cA = \Perf^{op}$, the classical (contravariant) Verdier duality functor is equivalent to $(\D_X)^{op}$ followed by linear duality $(-)^\vee$ over $k$. Therefore we can use the covariant Verdier duality functor to give a different definition of (the dual of) the Verdier dualizing complex. Defining $\varpi_X = \D_X(k_X)$, for $X$ locally finitely-stratifiable we have $\varpi_X = (\omega_X^\vee)$.

The covariant Verdier duality functor $\D_X$ is defined as the composition of two equivalences
\[ \Cosh(X,\cA) \xrightarrow[D_X]{\sim} \Sh_\Compacts(X,\cA) \xrightarrow[\theta]{\sim} \Sh(X,\cA) \]
where the intermediate term is the category of $\Compacts$-sheaves; these are presheaves on the poset $\Compacts(X)$ of compact subsets of $X$ satisfying certain conditions, which we recall in Definition \ref{def:Ksheaf}. 

We will construct, explicitly via bar complexes, two $\Compacts$-sheaves of $\cF(X)$-bimodules; we have the \emph{bimodule dual diagonal $\Compacts$-sheaf} $\varphi\theta\Delta_X^!\cF$ (Definition \ref{def:bimoduleDualDiagonal}) and the \emph{Verdier dual diagonal $\Compacts$-sheaf} $\DDelta_X\cF$ (Definition \ref{def:VerdierDualDiagonal}). These organize dualized diagonal bimodules in the following way. For any $K \in \Compacts(X)$, we have quasi-isomorphisms of bimodules
\[ \varphi\theta\Delta^!_X\cF(K) \simeq \colim_{K \subseteq U} (I_! \cF(U)_\Delta)^!, \quad \varphi\DDelta_X\cF(K) \simeq \cofib(I_!\cF(X \setminus K)_\Delta \to \cF(X)_\Delta) \]
where $I_!\cF(U)_\Delta$ is the diagonal bimodule of $\cF(U)$ pushed along the corestriction map, and $(-)^!$ is bimodule duality. 

We show that a map $\Theta: \varpi_X \to \cH\cH(\cF)^{S^1}$ determines a morphism of $\Compacts$-sheaves of bimodules $\varphi\theta\Delta^!_X\cF \to \varphi\DDelta_X\cF$, by using the map from Definition \ref{def:etaMap}.
\begin{definition} \label{def:cccCY}
Let $(X, \cF)$ be a  \textsc{ccc} space. An orientation of dimension $d$ on $(X,\cF)$ is a map of cosheaves of complexes $\Theta: \varpi_X \to  \cH\cH(\cF)^{S^1} [-d]$ if the induced map $\varphi\theta\Delta^!_X\cF \to \varphi\DDelta_X\cF[-d]$ is a quasi-isomorphism of sheaves of $\cF(X)$-bimodules.
\end{definition}

In fact, it is enough to check the nondegeneracy condition on restrictions $\Theta|_U$ for any collection of $U$ which cover $X$, e.g. the stars of all strata (Proposition \ref{prop:cosheafLocality}).

\begin{example}
We have $\cH\cH(\cL oc^c) = k_M$, the constant cosheaf with the trivial $S^1$-action.  
By contrast $\mathrm{HH}(\cL oc^c(U))$ is chains on the free loop space of $U$.  
If $M$ is a manifold, an orientation of $(M, \cL oc^c)$ 
is an orientation of $M$.   This is nontrivial to check directly, but we will shortly give an argument
reducing it to the case of $(M, \cL oc^b)$. 
\end{example}

\begin{example}
(Prop. \ref{prop:cccCYpoint}) A \textsc{ccc} space  $(\mathrm{point}, \cF)$ amounts to just a single category $\cF(\mathrm{point})$. An orientation 
on $(\mathrm{point}, \cF)$ is just a smooth (aka left) Calabi-Yau structure on $\cF(\mathrm{point})$ 
in the usual sense. 
\end{example}

\begin{example} \label{ex:bd1proper}
(Prop. \ref{prop:cccCYinterval}) Consider a \textsc{ccc} space 
$( [0,1) , \cF)$, where $\cF$ is constructible for  $[0,1) = \{0\} \cup (0,1)$.  
Then an orientation of  $( [0,1) , \cF)$ is equivalent to a smooth (aka left) Calabi-Yau structure
on the morphism $\cF((0,1)) \to \cF([0,1))$ in the sense of \cite{BD1}. 
\end{example}

Given a continuous map $f: X \to Y$ we may pushforward cosheaves of dg categories by the usual formula $f_* \cF(U) = \cF (f^{-1}(U))$. There are natural morphisms of precosheaves $f_* \cH \cH (\cF) \to \cH \cH(f_* \cF)$ and $f_* \cH \cH (\cF)^{S^1} \to \cH \cH(f_* \cF)^{S^1}.$ Thus if $f$ is proper (so $f_* = f_!$), an orientation of $(X, \cF)$ determines the composition
\[ \varpi_Y  \to  f_* \varpi_X \xrightarrow{\Theta}  f_* \cH \cH (\cF)^{S^1} \to \cH \cH(f_* \cF)^{S^1}, \]
which we denote as $f_* \Theta$. 

\begin{theorem} \label{thm:cccCYpushforward}
(Theorem \ref{prop:cccCYpushforwardText}) Let $(X, \cF)$ be a  \textsc{ccc} space, and $f: X \to Y$ be proper and constructible.    
If $\Theta$ orients $(X, \cF)$, then $f_* \Theta$ orients $(Y, f_* \cF)$. 
\end{theorem} 

\begin{corollary} \label{cor:cccCYglobal} 
If $(X, \cF)$ is a \textsc{ccc} space and $X$ is compact, then an orientation  of $(X, \cF)$ induces a smooth Calabi-Yau structure on the dg category $\cF(X)$. 
\end{corollary}

\begin{corollary}
Suppose $X = X_0 \cup_{\partial X_0} (\partial X_0 \times [0,1))$ where 
$X_0$ is compact.  Suppose $(X, \cF)$ is a \textsc{ccc} space and $\cF$ is stratifiable with respect to some stratification which, on 
$\partial X_0 \times [0,1)$, is the product of a stratification on $\partial X_0$ and the stratification $[0,1) = \{0\} \cup (0,1)$. 

Then an orientation of $(X, \cF)$ gives a smooth (aka left) Calabi-Yau structure on the map $\cF(X) \to \cF(\partial X_0 \times (0,1))$. 
\end{corollary}

\begin{example}
Suppose $([0,1], \cF)$ is a \textsc{ccc} space and $\cF$ is constructible for $[0,1] = \{0\} \cup (0,1) \cup \{1\}$.  This is equivalent 
to the data of  smooth Calabi-Yau structures on the morphisms 
$\cF([0,1)) \leftarrow \cF((0,1)) \rightarrow \cF((0,1])$ in the sense of \cite{BD1}.  Applying Corollary \ref{cor:cccCYglobal}, 
we find a smooth Calabi-Yau structure on the pushout, recovering one main result of \cite{BD1}. 
\end{example}

\subsection{Duality of orientations} 

For \textsc{csc} spaces, the nondegeneracy property of a putative orientation was formulated locally in terms
of Hom sheaves, and in practice is relatively easy to check (we will do so in examples).  However, it is not
so easy to verify, directly from the definition, the nondegeneracy property of a putative orientation on a \textsc{ccc} space, 
because of the appearance of duals of diagonal bimodules, etc.   
Here we give a method to bypass this issue in favorable situations. 

For a given dg category $\cC$, we write 
$\cC^{pp} = \Hom(\cC, \Perf(k))$ for its pseudoperfect modules. If $\cC$ is smooth, then $\cC^{pp}$ is proper, and it follows from Theorem 3.1 of \cite{BD1} that a smooth Calabi-Yau structure on $\cC$ determines a proper Calabi-Yau structure on $\cC^{pp}$.

If $(X, \cF)$ is a \textsc{ccc} space, then $(X, \cF^{pp})$ is a \textsc{csc} space. We have the following analogue of that result of \cite{BD1}: 
\begin{proposition} (Prop. \ref{prop:dualOne})
An orientation on a  \textsc{ccc} space $(X, \cF)$ induces an orientation on the  \textsc{csc} space $(X, \cF^{pp})$. 
\end{proposition} 

In general, one \emph{cannot} recover smooth Calabi-Yau structures on $\cC$ from proper Calabi-Yau structures on $\cC^{pp}$. However, it is possible to do so if $\cC$ is \emph{saturated}, that is, smooth and proper; then, these two sets of structures are in bijection.

Similarly, one can {\em not} generally recover $\cF(X)$ from $\cF(X)^{pp}$.
However, if $\cF$ is {\em locally saturated}, i.e. has smooth and proper stalks, then one can recover 
$\cF$ from $\cF^{pp}$, and we will show the following result. Let us pick a finite cover $X = \bigcup_i U_i$ such that $\cF(U_i)$ is saturated, in other words, by small enough $U_i$ (for example, the cover by stars of strata).
\begin{theorem} (Theorem \ref{thm:dualTwo}) \label{thm:checkOnProper} 
If $(X, \cF)$ is a locally saturated \textsc{ccc} space, then a morphism 
$\Theta:\varpi_X \to \cH\cH(\cF)^{S^1}[-d]$ orients each $(U_i, \cF|_{U_i})$ if and only if the dual morphism $\cH\cH(\cF^{pp})_{S^1} \to \omega_X[-d]$ orients $(X, \cF^{pp}|_{U_i})$. If that is the case, then $\Theta$ is an orientation on $(X,\cF)$.
\end{theorem}

\begin{example}
The stalk of $\cL oc^c$ at any point is $\mathrm{Perf}(k)$, so in particular it is smooth and proper.  
If the base space $M$ is a manifold, then a choice of orientation gives an orientation of the sheaf $\cL oc^b$, and by the result above, an orientation of the cosheaf $\cL oc^c$.
Applying Theorem \ref{thm:cccCYpushforward} gives a smooth Calabi-Yau structure on $\cL oc^c(M)$, recovering another main result of 
\cite{BD1}. 

Note however that one {\em cannot} obtain this result by instead applying the easier Theorem \ref{thm:cscCYpushforward} to $(M,  \cL oc^b)$ and then attempting to dualize, 
in particular because 
one {\em cannot} recover $\cL oc^c(M)$ from $\cL oc^b(M)$, already for $M = S^1$. 
\end{example}

\subsection{Applications}
We turn to constructions of \textsc{ccc} spaces and orientations. In Section \ref{sec:arb}, we elaborate a construction, originally due to David Nadler 
\cite{N3}, which takes as input a rooted tree $\vv{T}$ with some marked leaves, and outputs 
a \textsc{ccc} space $(\T, \arb)$.  The stalks of $\arb$ are themselves representation categories
of quivers which are `subquotients' of $T$; in particular, the stalks are smooth and proper.  
We show:

\begin{theorem} \label{thm:main} (Theorem \ref{thm: arborientation}) 
Let $\vv{T}$ be a rooted tree with marked leaves. The canonical $S^1$-action on $\cH\cH(\arb)$ is trivial, and 
there is an isomorphism, unique up to homotopy and multiplication by a scalar, $\cH\cH(\arb^{pp}) \cong \omega_\T[-\dim \T]$, giving an orientation on the \textsc{ccc}
space $(\T, \arb)$. 
\end{theorem} 

We define an {\em arboreal space} to be a \textsc{ccc} space locally modeled on $(\T, \cA)$ for varying $T$.  For such a space 
$(\X, \cF)$, the obstruction to its global orientability is 
the nontriviality of the rank one local system $\sheafHom(\cH\cH(\cF^{pp}), \omega_\X [-\dim \X])$; this is classified by the corresponding element of
$H^1(\X, k^\times)$.  We will show that this is in fact an element of $H^1(\X, \pm 1)$, which we term the 
first Stiefel-Whitney class of the locally arboreal space $(\X,  \cF)$.   When $\X$ is smooth, it is the usual first Stiefel-Whitney class.

More generally, microlocal sheaf theory
provides a general construction of \textsc{csc} and \textsc{ccc} spaces from subanalytic Legendrian varieties (carrying certain topological
`Maslov data') \cite{KS, N5, Sh, NS}. We recall the setup in Section \ref{sec:microsheaves}, and remind 
the reader that microsheaves can be used to compute Fukaya categories of Weinstein manifolds \cite{GPS3}.  
The arboreal 
spaces were originally introduced to give `resolutions' of these more general singular Legendrians
compatible with the \textsc{ccc} structure \cite{N4}.  Using this technology and our Theorems \ref{thm:cccCYpushforward} and \ref{thm:main}, 
we deduce: 

\begin{theorem} (Theorem \ref{thm:legendrian}) 
If $M$ is orientable, then for any subanalytic Legendrian $\Lambda \subset S^*M$ (or conic Lagrangian $\Lambda \subset T^*M$), the \textsc{ccc} space 
$(\Lambda, \mu \Sh_\Lambda^c)$ is orientable.

More generally, for any germ of Legendrian $\Lambda$ and any Maslov datum $\tau$, 
if $w_1(\Lambda, \tau)$ vanishes, then $(\Lambda, \mu \Sh_{\Lambda; \tau}^c)$ is orientable. 
\end{theorem} 

We also give a more hands-on construction of orientations in the case when $M$ is orientable, and $\Lambda \subset T^*M$ is the union of the zero section
and the cone over a Legendrian with immersed front projection.  Already this simple case is enough to give rise to many moduli spaces 
of interest in geometry, e.g. spaces appearing in group-valued Hamiltonian reduction, tame and wild character varieties and the augmentation variety of knot contact homology; 
we survey these in Section \ref{sec:apps}. 
Our results, together with the general work on Calabi-Yau categories and shifted symplectic structures \cite{PTVV, CPTVV, BD2}, give uniform constructions
of the symplectic structures and quantizations of these spaces.

\vspace{2mm} \noindent \textbf{Acknowledgements.}
We thank Ilyas Bayramov, Dori Bejleri, Philip Boalch, Christopher Brav, Tobias Dyckerhoff, Sheel Ganatra, Tatsuki Kuwagaki, Aaron Mazel-Gee, David Nadler, Tony Pantev, John Pardon, Theo Johnson-Freyd, Ryan Thorngren, Bertrand To\"en and Gjergji Zaimi for helpful conversations. We would also like to thank the anonymous reviewers of this paper, who pointed out many possible changes and improvements. This project was supported in part by NSF DMS-1406871. \\ 

\section{Review}
This section is a review of definitions and previously known results. 

\subsection{Categorical setting}\label{sec:dgcat}
We are primarily interested in $\Z$-graded dg (differential graded) categories over some fixed field $k$. One can find foundational accounts specific to dg categories in e.g. \cite{Kel1, Kel2, Drin, Toe1, TVa, BGT}. That being said, the starting point for our formalism are (co)sheaves of such dg categories, understood in an $\infty$-categorical sense. In order to organize all the coherence information, we need to work in an appropriate $\infty$-categorical setting; we chose to work in the $\infty$-category of $k$-linear stable $\infty$-categories.

In fact, the formalism we develop applies more generally to stable $\infty$-categories tensored over $\cE$, where $\cE$ is a rigid stable small symmetric monoidal $\infty$-category \cite{HSS}. We have the following $\infty$-categories.
\begin{itemize}
	\item $\mathrm{Cat}^\mathrm{perf}(\cE) = \Mod_\cE(\mathrm{Cat}^\mathrm{perf})$, an $\infty$-category whose objects are small stable idempotent-complete $\infty$-categories that are tensored over $\cE$, with $\cE$-linear exact functors as morphisms,
	\item $\mathrm{Pr}^\mathrm{L}(\cE) = \Mod_\cE(\mathrm{Pr}^\mathrm{L}_\mathrm{St})$, an $\infty$-category whose objects are stable presentable $\infty$-categories that are tensored over $\cE$, with $\cE$-linear exact functors as morphisms,
	\item $\mathrm{Cat}^\mathrm{Mor}(\cE)$, a full subcategory of $\mathrm{Pr}^\mathrm{L}(\cE)$ on the essential image of the embedding $\Ind: \mathrm{Cat}^\mathrm{perf}(\cE) \to \mathrm{Pr}^\mathrm{L}(\cE)$. This is an $\infty$-category whose objects can be identified with the objects as $\mathrm{Cat}^\mathrm{perf}(\cE)$, but whose morphisms from $\cA \to \cB$ are given by bimodules, that is, $\cE$-bilinear functors $\cA^{op} \times \cB \to \Ind(\cE)$.
\end{itemize}

Our formalism applies to any choice of $\cE$ but, since our main interest is in cases where $\cE = \Perf_k$, we will state the results for that choice. We will use the following shorthand notation:
\begin{itemize}
	\item $\dgcat$ for the $\infty$-category $\mathrm{Cat}^\mathrm{perf}(\Perf_k)$,
	\item $\DGCat$ for the $\infty$-category $\mathrm{Cat}^\mathrm{Mor}(\Perf_k)$.
\end{itemize}
These two categories have the same objects; we will refer to an object of either as a ``dg category''. For the places where we need to construct explicit objects of these categories, namely, the construction of the arboreal spaces in Section \ref{sec:arb} and the examples in Section \ref{sec:apps}, we are implicitly using the `Morita model structure' to present the $\infty$-category $\dgcat$. This is a model structure where one localizes with respect to Morita equivalences; see \cite{Cohn} for more details. Also, since all our functors are derived by default, from now on we will simply write `isomorphism' for what usually is elsewhere called `quasi-isomorphism'.

\subsubsection{Dualizability of dg categories}\label{sec:dualizability} 
Recall that an object $X \in \cC$ is dualizable in a symmetric monoidal category $\cC$ %
if there is another object $X^\vee$ and evaluation/coevaluation maps $\ev_X: X^\vee \otimes X \to 1_\cC$ and $\coev_X: 1_\cC \to X \otimes X^\vee$, where $1_\cC$ denotes the monoidal unit, such that the compositions
\begin{align*}
X &\xrightarrow{\coev_X \otimes \id_X} X \otimes X^\vee \otimes X \xrightarrow{\id_X \otimes \ev_X} X \\
X^\vee &\xrightarrow{\id_{X^\vee} \otimes \coev_X} X^\vee \otimes X \otimes X^\vee \xrightarrow{\ev_X \otimes \id_{X^\vee}} X
\end{align*}
are equivalent to the identity functors on $X$ and $X^\vee$. %
We now recall the classification of dualizable objects and morphisms in categories of dg categories, see e.g. \cite[Prop. 4.23]{HSS}. 
\begin{proposition} \label{dualizable DGCat}
Every object of $\cA \in \DGCat$ is dualizable, with the dual $\cA^\vee = \Ind((\cA^c)^{op})$.  The morphisms of $\DGCat$ that are left adjoints are exactly the morphisms in the image of ind-completion $\Ind:\dgcat \to \DGCat$.  
\end{proposition} 
The statement regarding morphisms is equivalent to the statement that admitting a continuous right adjoint is equivalent to preserving compact objects.

Dualizability is more complicated in the category $\dgcat$ of small dg categories. A small dg category $\cA$ is said to be \emph{smooth} if it is perfect as an $\cA$-bimodule. More precisely, the Yoneda embedding $\cA \hookrightarrow \Ind(\cA)$ corresponds to a bimodule $\cA_\Delta \in \Ind(\cA^{op} \otimes_k \cA)$, and $\cA$ is smooth if $\cA_\Delta$ is compact. The category $\cA$ is said to be \emph{proper} if for any two objects $a,a'$ the mapping object $\Hom_\cA(a,a') \in \vvt$ is compact, i.e., is in $\Perf$. We will also call the category \emph{saturated} if it is both smooth and proper.

Let us rephrase these definitions of smooth and proper. Consider the category $\Ind(\cA)$ as an object of $\DGCat$. There are co/evaluation morphisms
\begin{align*} 
\ev_{\Ind(\cA)}: & \Ind(\cA^{op}) \otimes \Ind(\cA) \to \vvt \\
\coev_{\Ind(\cA)}: & \vvt \to \Ind(\cA) \otimes \Ind(\cA^{op}),
\end{align*}
satisfying the usual relations. Then $\cA$ is smooth if and only $\coev_{\Ind(\cA)}$ preserves compact objects, and proper if and only if $\ev_{\Ind(\cA)}$ preserves compact objects. From this, one concludes: 

\begin{proposition}[\mbox{\cite[Lemma 4.14]{HSS}}]
A category $\cA \in \dgcat$ is dualizable if and only if it is saturated, in which case its dual is given by the opposite category $\cA^{op}$.
\end{proposition} 

\subsubsection{Pseudo-perfect modules}\label{sec:pseudoperfect}
It will important for us to relate smooth and proper dg categories using the following construction. Let $\cA$ be any dg category.
\begin{definition}\label{def:pp}
The category of pseudo-perfect modules over $\cA$ is the object of $\dgcat$ given by
\[ \cA^{pp} := \Fun^{ex}(\cA^{op}, \Perf), \]
i.e., it is the dg category whose objects are $k$-linear exact functors $\cA \to \Perf$.
\end{definition}

We have a natural embedding $\cA^{pp} \hookrightarrow \Ind(\cA)$. The category $\cA$ itself also embeds into $\Ind(\cA)$ by the $k$-linear Yoneda embedding; this allows one to compare these two categories as subcategories of $\Ind(\cA)$. The following result just follows directly from the definitions of smooth and proper.
\begin{lemma}\cite[Lemma 2.8]{TVa}\label{thm:sat}
If $\cA$ is smooth, then $\cA^{pp} \subset \cA$; if $\cA$ is proper, then $\cA \subset \cA^{pp}$. Therefore if $\cA$ is saturated, there is an equivalence $\cA \simeq \cA^{pp}$.
\end{lemma}

We have then the following easy corollary for smooth dg categories, which follows from restricting the canonical map $\cA^{op} \otimes \cA^{pp} \to \Perf$ along the inclusion $\cA^{pp} \subset \cA$.
\begin{corollary}
If $\cA$ is smooth, then $\cA^{pp}$ is proper.
\end{corollary}

\subsection{Dualizability of bimodules, Hochschild homology, and Calabi-Yau structures}
In this section we largely follow the exposition in Section 2 of \cite{BD1}. We will study the properties of bimodules over dg categories; to recall, an $(\cA,\cB)$-bimodule is a dg functor $\cA\otimes \cB^{op} \to \vvt$.

\subsubsection{Dualizability of bimodules} 
Consider the `linear dual' functor $\Hom_k(-,k)$ on the dg category $\vvt$. Note that this is an anti-automorphism when restricted to the subcategory $\Perf$. For any $\cM$ in $\cA\mh\Mod\mh\cB$, we define its \emph{linear dual}
\[ \cM^\vee = \Hom_k(\cM, k). \]
which is an object of $\cB\mh\Mod\mh\cA$. Explicitly, as a functor $\cM^\vee: \cB\otimes\cA^{op} \to \vvt$, it is given by
\[ (b,a) \mapsto \Hom_k(\cM(a,b), k). \]

For any bimodule $\cM$ there is an  adjunction
\[ - \otimes_{\cA\otimes\cB^{op}} \cM: \cB\mh\Mod\mh\cA \rightleftarrows \vvt : \Hom_k(\cM,-). \]
$\cM$ is called linear-dualizable (or right-dualizable) if the natural transformation $- \otimes_k \cM^\vee \to \Hom_k(\cM,-)$ is an equivalence of functors. Equivalently, $\cM$ is linear-dualizable if it always evaluates to a perfect $k$-complex, i.e. $\cM(a,b) \in \Perf$ for every $(a,b) \in \cA\otimes\cB^{op}$. In that case, there is an isomorphism of bimodules $\cM \overset{\simeq}\to (\cM^\vee)^\vee$, so we also get another adjunction
\[ \cM^\vee \otimes_{\cA\otimes\cB^{op}} -: \cA\mh\Mod\mh\cB \rightleftarrows \vvt: \cM \otimes_k - .\]

There is also the \emph{bimodule dual} of $\cM$ defined using the internal Hom of bimodules
\[ \cM^! = \Hom_{\cA\otimes\cB^{op}}(\cM, \cA \otimes_k \cB^{op}), \]
where we regard $\cA \otimes_k \cB^{op}$ as a $\cA^e \otimes \cB^e$-module (i.e., with two actions of $\cA$ and two of $\cB$). This gives an object of $\cB\mh\Mod\mh\cA$, which maps a pair of objects $(b,a) \in \cB \times \cA^{op}$ to
\[ \cM^!(b,a) = \Hom_{\cA\mh\Mod\mh\cB}(\cM(-,-'), \cA(a,-)\otimes_k \cB(-',b)). \]

For any bimodule $\cM$ there is an adjunction
\[ \cM \otimes_k -: \vvt \rightleftarrows \cA\mh\Mod\mh\cB: \Hom_{\cA\mh\Mod\mh\cB}(\cM, -). \]
$\cM$ is called bimodule-dualizable (or left-dualizable) if the natural transformation $\cM^! \otimes_{\cA\otimes\cB^{op}} - \to \Hom_{\cA\mh\Mod\mh\cB}(\cM, -)$ is an equivalence. Equivalently $\cM$ is bimodule-dualizable if it is perfect as a bimodule (i.e., is a  compact object in the category of bimodules). In that case we get an isomorphism of bimodules $\cM \overset{\simeq}\to (\cM^!)^!$, and we also have another adjunction
\[ - \otimes_k \cM^!: \vvt \rightleftarrows \cA\mh\Mod\mh\cB: - \otimes_{\cA\otimes\cB^{op}} \cM. \]

\begin{remark}
These duals are also called in the literature respectively \emph{right dual} and \emph{left dual}, or respectively \emph{proper dual} and \emph{smooth dual}; the `handedness' comes from the asymmetrical choice of looking at $\cM$ as an object of $\cA\otimes\cB^{op}\mh\Mod\mh k$. 
\end{remark}

\subsubsection{Hochschild homology}
Given a pair $(\cA,\cM)$ of a dg category $\cA$ and an $\cA$-bimodule $\cM$, the Hochschild complex $\HH(\cA,\cM)$ of $\cA$ with coefficients in $\cM$ is defined as a (derived) tensor product
\[ \HH(\cA,\cM) = \cA_\Delta \otimes_{\cA^e} \cM \]
of right and left modules over $\cA^e = \cA^{op} \otimes \cA$. When $\cM = \cA_\Delta$ is the identity (or diagonal) bimodule itself, we will denote
\[ \HH(\cA) := \HH(\cA,\cA_\Delta) = \cA_\Delta \otimes_{\cA^e} \cA_\Delta. \]
We will also use the notation $\HH_n(\cA) = H^n(\HH(\cA))$ to denote the Hochschild homology groups.

The tensor product above can be computed by taking any appropriate resolution; in particular one can take the bar resolution \cite{Lod}. Tensoring over $\cA^e$ with the diagonal bimodule gives then the \emph{reduced bar complex} $\overline{C}^\mathrm{bar}(\cA)$ \cite{Lod,Yeu} representing $\HH(\cA)$. This complex carries a canonical homotopy $S^1$ action, whose homotopy orbits $\HH(\cA)_{S^1}$ and fixed points $\HH(\cA)^{S^1}$ were classically termed the cyclic and negative cyclic complexes, \cite{Hoy, Kas} with homology groups denoted by
\[ \HC_n(\cA) = H^*(\HH(\cA)_{S^1}), \quad \HC^-_n(\cA) = H^*(\HH(\cA)^{S^1}). \]
Explicit representatives for these complexes can be given by using the formalism of mixed complexes \cite{Lod}; we refer the reader to \cite[Sec. 2.4]{Yeu} for an explicit application of this formalism to the context of dg categories.

This construction is functorial in $\cA$: using the formalism of trace functors \cite[Sec.4]{BD2}, one sees that taking Hochschild complexes gives a functor $\HH(-): \dgcat \to \vvt$. Moreover, the $S^1$-actions are also compatible, so we also have functors $\HH(-)^{S^1}$ and $\HH(-)_{S^1}$ computing the negative cyclic and cyclic complexes.

\subsubsection{Calabi-Yau structures}\cite{KoSo-alg, Ginz} 
Recall that a category $\cA$ is said to be {\em proper} if all the Hom spaces are perfect as $k$-complexes, which is equivalent to the identity bimodule $\cA_\Delta$ being linear-dualizable. In this case the dual of Hochschild homology can be computed in terms of the linear dual $\cA^\vee$ (by adjunction):
\[ \Hom_k(\HH(\cA),k) = \Hom_k(\cA_\Delta \otimes_{\cA^e} \cA_\Delta, k) \simeq \Hom_{\cA^e}(\cA_\Delta,\cA^\vee). \]

Recall also that a category $\cA$ is said to be {\em smooth} if the diagonal bimodule is compact as a module over $\cA^e$, or equivalently bimodule-dualizable. In this case the Hochschild homology of $\cA$ can be calculated in terms of $\cA^!$:
\[ \HH(\cA) = \Hom_k(k, \cA_\Delta \otimes^L_{\cA^e} \cA_\Delta) \simeq \Hom_{\cA^e}(\cA^!,\cA_\Delta). \]

In the case where $\cA$ is both smooth and proper (i.e., dualizable as an object of $\dgcat$), tensoring over $\cA$ with the bimodules $\cA^\vee$ and $\cA^!$ give inverse autoequivalences of the category $\Perf\mh\cA = (\Mod\mh\cA)^c$ of compact $\cA$-modules; these are respectively a Serre functor and its inverse for this category \cite{BD1}.

More generally, fix an arbitrary $\cA$-bimodule $\cM$. If $\cA$ is proper there is an isomorphism
\[ \Hom_k(\HH(\cA,\cM),k) \simeq \Hom_{\cA^e}(\cM,\cA^\vee), \]
and if $\cA$ is smooth there is an isomorphism
\[ \HH(\cA,\cM) \simeq \Hom_{\cA^e}(\cA^!, \cM). \]
allowing one to describe Hochschild homology with coefficients in an arbitrary bimodule.

\begin{definition}\label{def:absoluteCYstructures}
A $d$-dimensional proper (or right) Calabi-Yau structure on a proper dg category $\cA$ is a morphism in $\vvt$
\[ \HH(\cA)_{S^1} \to k[-d] \]
so that the induced map $\cA_\Delta \to \cA^\vee[-d]$ 
is an isomorphism of $\cA$-bimodules .

A $d$-dimensional smooth (or left) Calabi-Yau structure on a smooth dg category $\cA$ is a morphism in $\vvt$
\[ k[d] \to \HH(\cA)^{S^1} \]
so that the induced map $\cA^![d] \to \cA_\Delta$ is an isomorphism of $\cA$-bimodules .  
\end{definition}

\begin{remark}
We can rephrase the definitions above by saying that a right/left Calabi-Yau structure is a map out of/into Hochschild homology that is compatible with the $S^1$ action; the data of a proper CY structure is a dual cyclic homology class $[\omega] \in \Hom_k(\HC_d(\cA),k)$, and the data of a smooth CY structure is a negative cyclic homology class $[\lambda] \in \HC^-_d(\cA)$. 
\end{remark}

There are also relative versions of the definitions above \cite{Toe2,BD1}. A functor $f:\cA\to\cB$ induces a map of Hochschild complexes $\HH(f): \HH(\cA) \to \HH(\cB)$, compatible with the $S^1$ action.  We define the relative Hochschild complex 
\[ \HH(\cB|\cA) := \Cone(\HH(\cA) \to \HH(\cB)) \]
and denote similarly by $\HH(\cB|\cA)^{S^1}$ and $\HH(\cB|\cA)_{S^1}$ the relative negative cyclic and cyclic complexes.

\begin{definition}\label{def:relativeCYstructures}
	\cite{BD1} A $d$-dimensional proper (or right) relative Calabi-Yau structure on a functor $f:\cA\to\cB$ between proper dg categories 
	is a dual class in cyclic homology
	\[ \phi \in \Hom_k(\HH(\cB|\cA)_{S^1}, k[-d+1]) \]
	satisfying a nondegeneracy condition.
	
	A $d$-dimensional proper (or right) relative Calabi-Yau structure on a functor $f:\cA\to\cB$ between smooth dg categories 
	is a negative cyclic class
	\[ \theta \in \Hom_k(k[d], \HH(\cB|\cA)^{S^1}) \]
	satisfying a nondegeneracy condition.
\end{definition}
We refer to \cite[Sec.4]{BD1} for the more detailed description of these definitions.

\subsection{Constructible sheaves and cosheaves of $\infty$-categories}\label{sec:constsh}
Let us denote by $X$ a topological space, and by $\Opens(X)$ the category of open sets of $X$; from this one can produce the nerve of this category \cite[Ch.1]{LurHTT}; this is an $\infty$-category which we will equally denote by $\Opens(X)$. Given some $\infty$-category $\cC$, a $\cC$-valued presheaf is an $\infty$-functor $\cF: \Opens(X)^{op} \to \cC$. A presheaf is a sheaf if it carries covers to limits.  A precise definition in rather greater generality than we need here is
given in \cite[Def. 6.2.2.6]{LurHTT}.

Now let us consider a finite partially ordered set $A$, regarded as a topological space where a subset is open if it is closed upwards. Following \cite[Def.A.5.1]{LurHA}, we define an $A$-stratification of $X$ to be a continuous map $X \to A$; we will moreover assume that this stratification satisfies the three conditions of \cite[Prop.A.5.9]{LurHA} (note that condition (iii) is automatically satisfied since we assumed $A$ is finite). We follow \cite[Def.A.5.2]{LurHA} and define a full subcategory $\Sh^A(X,\cC)$ of $A$-constructible sheaves, spanned by sheaves that are locally constant when restricted to the stratum $X_a$ for all $a \in A$.

Cosheaves are the opposite category of sheaves valued in the opposite category: 
\[ \PreCosh(X,\cC) := \PreSh(X,\cC^{op})^{op}, \qquad \Cosh(X,\cC) := \Sh(X,\cC^{op})^{op} \]
Unraveling this definition, a precosheaf is a functor $\cF: \Opens(X) \to \cC$. This is a cosheaf if it carries covers to colimits. We will likewise denote the $\infty$-categories of $A$-constructible cosheaves and precosheaves by $\Cosh^A(X,\cC)$ and $\PreCosh^A(X,\cC)$, respectively. 

We now restrict our attention to sheaves and cosheaves valued in $\cC = \dgcat$. Recall Definition \ref{def: ccc and csc spaces} from the introduction: a \textsc{ccc} space is $(X, \cF)$ where $X$ is an $A$-stratified space and $\cF \in \PreCosh^A(X,\dgcat)$, and similarly a \textsc{csc} space is $(X,\cF)$ where now $\cF \in \Sh^A(X,\dgcat)$.

\subsubsection{The sheaf of morphisms}\label{sec:sheaf of Homs}
We will use the following well-known fact about sheaves of categories: given a sheaf of categories $\cF$ on some space $X$ and two objects $x,y \in \cF(X)$, the morphism spaces between the images of $x$ and $y$ form a \emph{sheaf of morphisms}, which associates $U \mapsto \cF(U)(\rho x, \rho y)$, where $\rho:\cF(X) \to \cF(U)$ is the restriction map of $\cF$.

In the main text of this article, we will assume that this fact holds equally for sheaves $\cF$ valued in the $\infty$-category $\dgcat$, giving sheaves of morphisms valued in $\vvt$, and that compositions in $\vvt$ between the morphism spaces give compositions maps at the level of sheaves of morphisms. More precisely, we assume:
\begin{proposition}\label{prop:morphismFunctor}
	Let $\cF:I^\triangleleft \to \dgcat$ be a diagram, and $x,y$ two objects in $\cF(*)$. Then there is a diagram $\cF|^{x,y}:I \to \vvt$, such that for every $i$ there is an equivalence $\cF|^{x,y}(i) \simeq \cF(i)(\rho x,\rho y)$; if $\cF$ is a limit diagram then so is $\cF|^{x,y}$. Moreover, for any triple of objects $x,y,z$ of $\cF(*)$, there is a canonical map $\cF|^{x,y} \otimes \cF|^{y,z} \to \cF|^{x,z}$ in $\Fun(I^\triangleleft, \vvt)$.
\end{proposition}

We could not find in the literature precise statements of the proposition above for the specific setting that we are working in, namely, for diagrams valued in the $\infty$-category of small stable $\infty$-categories tensored over some rigid symmetric monoidal $\cE$ (in our case, $\cE = \Perf_k$). For completeness, we include in Appendix \ref{sec:appendixSheafOfHoms} a proof of this fact using a different, albeit equivalent, model for these categories.

\begin{definition}\label{def:sheafOfHoms}
	For $U \in \Opens(X)$ and $x,y\in \cF(U)$, the \emph{sheaf of morphisms} $\cF|_{U}^{x,y}$ is the $\vvt$-valued sheaf on $X$ is given by applying Proposition \ref{prop:morphismFunctor} to $J = \Opens(X)^{op}_{\subseteq U}$, with left cone point $U$; and extending the resulting sheaf to all of $\Opens(X)$ by zero. For any triple $x,y,z \in \cF(U)$, we will denote by $c_U^{x,y,z}$ the corresponding morphism $\cF|_{U}^{x,y} \otimes \cF|_{U}^{y,z} \to \cF|_{U}^{x,z}$ in $\Sh(X,\vvt)$.
\end{definition}

We extend the notation $\cF|^{x,y}_U$ to allow for objects of sections on other opens. Given $x \in \cF(U'), y \in \cF(U'')$, we write: 
\[ \cF|^{x,y}_U = \begin{cases}
	\cF|^{\rho x, \rho y}_U \quad \text{if} \, U \subset U' \cap U'' \\
	0 \quad \text{otherwise}
\end{cases}\]

\section{Orientations on \textsc{csc} spaces}\label{sec:orientationCSC}
In this section we will study a version of orientation that applies to spaces endowed with sheaves of dg categories, generalizing the notion of a proper (relative) Calabi-Yau structure.

\subsection{Sheafified Hochschild homology} \label{sec:sheafHH}
Let $(X,\cF)$ be a \textsc{csc} space. We would like to apply the construction that produces the Hochschild complex of each dg category $\cF(U)$, in order to obtain a presheaf calculating $\HH(\cF(U))$. We could pick any model for the functor $\HH$ and apply this to $\cF$, but it will be convenient to work with a more explicit construction , which gives an complex calculating $\HH(\cF(U))$ for every $U \in \Opens(X)$.

We pick, for each $U \in \Opens(X)$, a spanning set of objects $S=\{S_U\} \subseteq \mathrm{Ob}(\cF(U))$, that is, a set of objects meeting all equivalence classes, such that for any $V \subset U$, $S_V$ contains the image of $S_U$ under the restriction map $\rho$ of $\cF$. This can always be obtained; first we pick any collection of spanning sets of objects and then adjoin all the restrictions from bigger open sets.

For each pair of nonnegative integers $(m,n)$, we define the following object of the category $\Sh(X,\vvt)$
\[ \HH(\cF)_{(m,n)} = \bigoplus \cF|_{V_n}^{y_n,x_0} \otimes \cF|_{U_0}^{x_0,x_1}[1] \otimes \dots \otimes \cF|_{U_m}^{x_m,y_0} \otimes \cF|_{V_0}^{y_0,y_1}[1] \otimes \dots \otimes \cF|_{V_{n-1}}^{y_{n-1},y_n}[1], \]
where the sum is taken over all multisets of elements $U_i$ and $V_i$ in $\Opens(X)$ and sequences of objects in the appropriate object sets $S_U$.

\begin{definition}\label{def:HHpresheaf}
The Hochschild homology presheaf $\HH(\cF)$ associated to $\cF$ and $S=\{S_U\}$ is the object of $\PreSh(X,\vvt)$ given by $\HH(\cF) = \bigoplus_{m,n} \HH(\cF)_{(m,n)}$ endowed with the usual Hochschild chain differential using the composition maps $c^{x,y,z}_{U,V}$.
\end{definition}
Note that $\HH(\cF)$ is generally not a sheaf, though each individual morphism sheaf $\cF|_U^{x,y}$.This is expected, since Hochschild homology does not preserve arbitrary limits.

\begin{proposition}\label{prop:HHsheafValues}
For any $U \in \Opens(X)$, the value of $\HH(\cF)$ on $U$ is equivalent to the Hochschild homology of $\cF(U)$, and for any pair $U \subseteq V$, the restriction map of $\HH(\cF)$ is equivalent to the map induced on Hochschild homologies by the restriction map $\rho: \cF(V) \to \cF(U)$.
\end{proposition}
\begin{proof}
By definition, the value of $\HH(\cF)$ on $U$ is given by the two-pointed Hochschild chain complex $\cF(U)\otimes_{\cF(U)^e} \cF(U)$ computed using the set of objects given by the union of all the images of all the $S_V$ for all $U \subseteq V$, including $U$. By assumption, this meets every isomorphism class of objects in $\cF(U)$, so this complex does compute Hochschild homology.

The statement about the morphisms follows from the fact that the map on Hochschild homology can be described on the two-pointed complex by applying $\rho$ to each factor, together with the fact that the restriction maps of each $\cF|_U^{x,y}$ are given by the maps induced on hom spaces by $\rho$ (Proposition \ref{prop:morphismFunctor}).
\end{proof}

We can perform the exact same construction as above, but with the negative cyclic or cyclic complexes instead of Hochschild complexes, and obtain presheaves $\HH(\cF)^{S^1}, \HH(\cF)_{S^1}$, respectively the negative cyclic presheaf, with its canonical map $\HH(\cF)^{S^1} \to \HH(\cF)$ and the cyclic sheaf, with its canonical map $\HH(\cF) \to \HH(\cF)_{S^1}$.
\begin{definition} \label{def:HHsheaves} If $(X,\cF)$ is a \textsc{csc} space, we denote by $\cH\cH(\cF), \cH\cH(\cF)^{S^1}$ and $\cH\cH(\cF)_{S^1}$ the objects of $\Sh(X,\vvt)$ given respectively by the sheafifications of the presheaves $\HH(\cF), \HH(\cF)^{S^1}$ and $\HH(\cF)_{S^1}$.
\end{definition}

There are canonical morphisms $\cH\cH(\cF)^{S^1} \to \cH\cH(\cF) \to \cH\cH(\cF)_{S^1}$ and
\[ \HH(\cF) \to \cH\cH(\cF), \quad \HH(\cF)^{S^1} \to \cH\cH(\cF)^{S^1}, \quad \HH_{S^1}(\cF)_{S^1} \to \cH\cH(\cF)_{S^1} \]
in $\PreSh(X,\vvt)$. We note also that, for any $U \in \Opens(X)$, there is a canonical equivalence of presheaves $\HH(\cF|_U) \overset{\sim}{\to} \HH(\cF)|_U$.

Let us recall the example given in the introduction. If $M$ is a stratifiable topological space (for example, a manifold), there is a cosheaf $\cL oc^c$ (`compact local systems') and a sheaf $\cL oc^b$ (`bounded local systems'), giving both a \textsc{ccc} space and a \textsc{csc} space with underlying topological space $M$. This sheaf and this cosheaf both have the category $\Perf_k$ as stalks everywhere; the difference is that over larger open sets, the sections of $\cL oc^b$ are still perfect as $k$-modules, whereas the cosections of $\cL oc^c$ may not be. The Hochschild complex presheaf is then isomorphic to an presheaf given by \emph{cochains on the free loop space}; this is obtained by dualizing the statement of \cite[Thm 7.3.14]{Lod}. That is, this presheaf assigns, for every $U \in \Opens(M)$,
\[ \HH(\cL oc^b(U)) \simeq C^*(\cL U;k), \]
where $\cL U$ denotes the free loop space of $U$, i.e., the space of continuous maps $S^1 \to U$, with the $S^1$-action given by loop rotation. Note that the presheaf is locally constant, since around any point, one can find a contractible neighborhood $U \in \Opens(X)$, for which we will have $C^*(\cL U;k) \simeq k$. Therefore, its sheafification $\cH\cH(\cL oc^b)$ is a rank one local system on $X$, and the local identification above gives a local trivialization of the $S^1$-action.

\begin{proposition}
The sheaf $\cH\cH(\cL oc^b) \in \Sh(X,\vvt)$ is constant and the $S^1$-action on it is trivial.
\end{proposition}
\begin{proof}
It is enough to give a nonzero morphism of presheaves $\HH(\cL oc^b) \to k_X$. For any open set $U$, consider the inclusion $U \hookrightarrow \cL U$ as the constant maps from the circle, which gives a map of $\infty$-nerves $\Opens(X) \to \Opens(\cL X)$. Pullback along this map induces the desired morphism of presheaves. The second part of the claim follows from the fact that the $S^1$-action is locally trivial.
\end{proof}
Note from the example above that the presheaf $\HH(\cL oc^b)$ and the sheaf $\cH\cH(\cL oc^b)$ are generally quite different; for instance, $\cH\cH(\cL oc^b)$ 
forgets all the information about the large loops in $\cL X$ and the $S^1$-action on them.

\subsubsection{Sheafified trace map} 
Hochschild homology gives a sort of universal trace on morphism spaces of a category, and the same is true for the sheafified version we defined above. For any $U \in \Opens(X)$ and $x,y$ in the object set $S_U$, by including into the $(m,n)=(0,0)$ summand of $\HH(\cF)_{(m,n)}$, we have a morphism of $\vvt$-valued presheaves $\cF|_U^{x,y} \otimes \cF|_U^{y,x} \to \HH(\cF)$. 
\begin{definition}\label{def:sheafified trace map}
	The \emph{sheafified trace map} $\Tr^{x,y}_U: \cF|_U^{x,y} \otimes \cF|_U^{y,x} \to \cH\cH(\cF)$ is the morphism of sheaves given by the lift of the map above.
\end{definition}

\subsection{Orientations} \label{sec:orientationsProper}
Here we recall the definitions and statements presented in the introduction. If $X$ is a stratifiable space, we write $\omega_X$ for the Verdier dualizing sheaf; this is a $\Perf$-valued sheaf given by $\omega_X = \pt^! k$, where $\pt$ is the canonical map from $X$ to the point and $\pt^!$ is the right adjoint to the pushforward with compact supports. For any point $x \in X$ with open neighborhood $U$, an explicit representative of the object $\omega_X$ can be described on sufficiently small open sets $U$ by the relative chains $\omega_X(U) = C_d(X, X \setminus U; k)$, as explained in \cite[p.91]{GM}. Let us denote $\D': \Sh(X,\vvt) \to \Sh(X,\vvt)^{op}$ the classical Verdier duality functor $\D' = \sheafHom(-,\omega_X)$.

 Let $(X,\cF)$ be a \textsc{csc} space. Recall that the data of an orientation of dimension $d$ (for some integer $d$) is given by a morphism of sheaves
\[ \Omega: \cH\cH(\cF)_{S^1} \to \omega_X[-d], \]
or equivalently a global section of $\D'(\cH\cH(\cF)_{S^1})$. For any $U \in \Opens(X)$ and objects $x,y \in S_U$, we compose the restricted map $\Omega|_U: \cH\cH(\cF|_U)_{S^1} \to \omega_X[-d]|_U$ this with the sheafified trace of Definition \ref{def:sheafified trace map} and the canonical map $\cH\cH(\cF|_U) \to \cH\cH(\cF|_U)_{S^1}$ to get a morphism of sheaves $\cF|_U^{x,y} \otimes \cF|_U^{y,x} \to \omega_X[-d]$ which by adjunction gives a morphism
\[ \cF|_U^{x,y} \to \sheafHom(\cF|_U^{y,x},\omega_X[-d]). \]
Recalling Definition \ref{def:cscCY}, we say that $\Omega$ is an \emph{orientation} on $(U,\cF|_U)$ if for any $(U,x,y)$, this latter is an isomorphism in $\Sh(U,\vvt)$. Very importantly, this condition requires checking isomorphisms of sheaves, which can be done locally. In other words, we have:
\begin{proposition}\label{prop:sheafLocality}
	If $X = \bigcup_i U_i$ is a cover and a map $\Omega: \cH\cH(\cF)_{S^1} \to \omega_X[-d]$ is such that, for any $U_i$, $\Omega|_{U_i}$ is an orientation on $(U_i,\cF|_{U_i})$, then $\Omega$ is an orientation on $(X,\cF)$.
\end{proposition}

\subsubsection{Relation to proper CY structures}
We now discuss the relation between our definition of orientations and previously existing notions of Calabi-Yau structures on proper dg categories.
\begin{proposition}\label{prop:cscCYpoint}
	If $X$ is a point, then an orientation on the \textsc{csc} space $(X,\cF)$ is a proper Calabi-Yau structure on the category $\cF$, in the sense of Definition \ref{def:absoluteCYstructures}.
\end{proposition}
\begin{proof}
	Follows from the fact that for any $x,y \in \cF$, the induced morphism of hom spaces $\cF(x,y) \to \Hom(\cF(y,x), k)[-d]$ is naturally equivalent to the morphism obtained from the induced map of bimodules in $\cF \to \cF^\vee$ at the pair $(x,y)$.
\end{proof}

\begin{proposition}\label{prop:cscCYinterval}
	If $X = [0,1)$ with stratification $\{0\} \cup (0,1)$, then an orientation on $(X,\cF)$ is a relative proper Calabi-Yau structure on the restriction functor $f: \cF([0,1)) \to \cF((0,1))$, in the sense of Definition \ref{def:relativeCYstructures}.
\end{proposition}
\begin{proof}
	The data of $\cF$ is given just by the morphism $f$ between two proper dg categories $\cF([0,1))$ and $\cF((0,1))$. Note that in this case, exactly as it happened for the point, there is no difference between $\HH(\cF)$ and $\cH\cH(\cF)$, since $\HH(\cF)$ is already a sheaf. We now calculate the sections of Verdier dualizing sheaf $\omega_X$ to be
	\[ \omega_X([0,1)) = 0, \quad \omega_X((0,1)) = k[1] \]
	Therefore, the data of a morphism $\HH(\cF)_{S^1} \to \omega_X[-d]$ is equivalent to the data of a map $\HH(\fib) \to k[-d]$, where $\HH(\fib)$ denotes the fiber of the induced morphism $\HH(\cF([0,1))) \to \HH(\cF((0,1)))$. It remains to check that the nondegeneracy condition coincides with the one given in \cite[Def.4.7]{BD1}; this is the condition that a certain map of distinguished triangles of $\cF(X)$-bimodules be an isomorphism. But isomorphisms of bimodules can be checked by evaluation at each pair $(x,y)$ of representatives for isomorphism classes of objects of $\cF(X)$; we can pick these objects to be in our chosen object set $S_U$. The restriction maps from $X=[0,1)$ to $(0,1)$ combined with the isomorphism of sheaves $\cF|_X^{x,y} \to \sheafHom(\cF|_X^{y,x},\omega_X[-d])$ gives the desired isomorphism of distinguished triangles.
\end{proof}

\subsubsection{Pushforward}
We now recall and prove the pushforward result stated in Theorem \ref{thm:cscCYpushforward}. Consider $f: X \to Y$ a proper and constructible map from a \textsc{csc} space $(X,\cF)$. There is a sheaf of proper dg categories $f_*\cF$ making $Y$ a \textsc{csc} space. Choosing as object sets for the sheaf $f_*\cF$ spanning sets containing the images of the spanning sets $S_{f^{-1}(U)}$, we get a map of presheaves $\HH(f_*\cF) \to f_*\HH(\cF)$, and by the universal property of sheafification there is a canonical morphism of sheaves $\cH \cH(f_* \cF) \to f_* \cH \cH (\cF)$ lifting it, and similarly for the corresponding cyclic homology sheaves. Since $f$ is proper (so $f_* = f_!$), we consider the composition 
\[ \cH \cH(f_* \cF)_{S^1} \to f_* \cH \cH (\cF)_{S^1} \xrightarrow{\Omega} f_* \omega_X \to \omega_Y, \]
which we denote as $f_*\Omega$.

\begin{proposition}\label{prop:cscCYpushforwardText}
	If $\Omega$ is an orientation on $(X,\cF)$ then $f_*\Omega$ is an orientation on $(Y,f_*\cF)$.
\end{proposition}
\begin{proof}
	For any open $U \in \Opens(Y)$ and objects $x,y \in S_U$, consider the morphism
	\[ \cF|_{f^{-1}(U)}^{x,y} \to \D'(\cF|_{f^{-1}(U)}^{x,y}) [-d] \]
	determined by the orientation $\Omega$; this is an isomorphism by assumption. We apply the pushforward $f_*$ to obtain a morphism
	\[ \cF|_{f^{-1}(U)}^{x,y} \to f_* \D'(\cF|_{f^{-1}(U)}^{x,y}) [-d] \simeq \D'(f_*\cF|_{f^{-1}(U)}^{x,y}) [-d], \] 
	where we used the fact that there is a canonical equivalence $f_! \simeq f_*$  since $f$ was proper. By construction, this is equivalent to the morphism determined by the pushforward orientation $f_*\Omega$.
\end{proof}

As stated in the introduction, we can combine Propositions \ref{prop:cscCYinterval} and \ref{prop:cscCYpushforwardText} in the case where the space is of the form $X = X_0 \cup_{\del X_0} (\del X_0 \times [0,1))$, with $X_0$ compact, with $\cF$ stratified with respect to some stratification restricting to a product of a stratification on $\del X_0$ and the stratification we considered above on $[0,1)$. Then an orientation on $(X,\cF)$ gives a relative CY structure on the restriction map $\cF(X) \to \cF(\del X_0 \times (0,1))$, as stated in Corollary \ref{cor:introRelProperCY} in the introduction.

We can refine that result and rephrase the nondegeneracy condition entirely in terms of relative CY structures on a cover. Consider the cover $X = \bigcup_\sigma U_\sigma$, indexed over the strata of $X$, with $U_\sigma = \Star(\sigma)$. For each $\sigma$, we then choose a small compact retract $K_\sigma$ of $U_\sigma$, that is, a compact subset $K$ of $U$ to which $U$ has a stratification-respecting retraction.
\begin{proposition}\label{prop:cscLocalCheck}
	A map $\Omega: \cH\cH(\cF)_{S^1} \to \omega_X[-d]$ is an orientation on $(X,\cF)$ if for each $\sigma$, $\Omega|_{U_\sigma}$ induces a relative proper Calabi-Yau structure on the restriction $\cF(U_\sigma) \to \cF(U_\sigma \setminus K_\sigma)$.
\end{proposition}
\begin{proof}
	By Proposition \ref{prop:sheafLocality}, it is sufficient to check the nondegeneracy condition on sections over each $U_\sigma$. For each such $U_\sigma$, we refine the stratification such that there is a stratified map $U_\sigma \to [0,1)$ sending $K_\sigma \to \{0\}$. By the proof of \ref{prop:cscCYpushforwardText}, the nondegeneracy condition on $U_\sigma$ is equivalent to the condition that the restriction $\Omega|_{U_\sigma}$ give a relative Calabi-Yau structure on the restriction map $\cF(U_\sigma) \to \cF(U_\sigma \setminus K_\sigma)$.
\end{proof}
In other words, combining the two propositions above gives a way to globalize relative Calabi-Yau structure, as we explained in the introduction.

\section{Orientations on \textsc{ccc} spaces}\label{sec:orientationsCCC}
Let $(X,\cF)$  be a \textsc{ccc} space. In this section we will describe a dual version of the definition of orientation given in the previous section. This definition will be a generalization of the notion of a smooth (relative) Calabi-Yau structure.

\subsection{Covariant Verdier duality}\label{sec:lurieVerdier}
We will need Lurie's covariant Verdier duality (\cite[Sec.7.3.4]{LurHTT} and \cite[Sec.5.5.5]{LurHA}) in order to phrase our definition of orientation on a \textsc{ccc} space. Here we recall the relevant assertions, mostly following the notation from the exposition in \cite{Vol}. 
  
\subsubsection{K-sheaves}
Let $X$ be a locally compact Hausdorff space. 
We write $\Compacts(X)$ for the set of all compact subsets of $X$; we view it as a partially ordered set, with respect to inclusion. Let $\cA$ be a stable $\infty$-category which admits small limits and colimits.   

\begin{definition}\cite[Def.7.3.4.1]{LurHTT}\label{def:Ksheaf}
	An $\cA$-valued presheaf $\cG$ on $\Compacts(X)$ is a $\Compacts$-sheaf if 
	\begin{enumerate}
		\item $\cG(\emptyset)$ is final.
		\item For any $K,K' \in \Compacts(X)$,
		\[\xymatrix{
			\cG(K\cup K') \ar[r] \ar[d] & \cG(K) \ar[d]\\
			\cG(K') \ar[r] & \cG(K \cap K')
		}\]
		is a pullback square in $\cA$.
		\item The canonical map $\cG(K) \to \colim_{K \Subset K'} \cG(K')$ is an isomorphism in $\cA$, where the notation $K \Subset K'$ means that $K \subseteq U \subseteq K'$ for some open $U$.
	\end{enumerate}
	We denote by $\Sh_\cK(X,\cA)$ the full subcategory of $\Fun(\Compacts(X)^{op},\cA)$ spanned by the $\Compacts$-sheaves.
\end{definition}
Analogously to usual sheaves, we write $\Cosh_\cK(X,\cA) = \Sh_\cK(X,\cA^{op})^{op}$. One can go back and forth between ordinary (co)sheaves and $\Compacts$-(co)sheaves:
\begin{proposition}\cite[Lem.5.5.5.3]{LurHA}\label{prop:functorsUsheafKsheaf}
	There are $\infty$-functors $\theta:\Fun(\Opens(X)^{op},\cA) \to \Fun(\Compacts(X)^{op},\cA)$ and $\psi: \Fun(\Compacts(X)^{op},\cA)  \to \Fun(\Opens(X)^{op},\cA)$ which restrict to equivalences between $\Sh(X,\cA)$ and $\Sh_{\Compacts}(X,\cA)$.
	
	The functor $\theta$ is defined so that for any compact set $K$, we have $\theta(\cG)(K) \cong \colim_{K \subseteq U} \cG(U)$. The functor $\psi$ is defined so that for any open $U$, we have $\psi(\cG)(\cU) \cong \lim_{K \subseteq U} \cG(K)$. 
\end{proposition}

Following \cite{Vol}, we will also make use of an intermediate category between the $\Compacts$-sheaves and $\Compacts$-presheaves. We denote by $\Sh(\Compacts(X),\cA)$ the full subcategory of $\Fun(\Compacts(X)^{op},\cA)$ spanned by the objects satisfying conditions (1) and (2) of Definition \ref{def:Ksheaf}.
\begin{proposition}\cite[Lem.5.13]{Vol}\label{prop:functorPhi}
	The inclusion $\Sh_\cK(X,\cA) \hookrightarrow \Sh(\Compacts(X),\cA)$ has a right-inverse
	\[ \varphi: \Sh(\Compacts(X),\cA) \to \Sh_\cK(X,\cA), \] 
	which preserves filtered colimits and gives an equivalence $\varphi\cF(K) \simeq \colim_{K \Subset K'} \cF(K')$ for every $\cF \in \Sh(\Compacts(X),\cA)$ and $K \in \Compacts(X)$.
\end{proposition}

\subsubsection{Verdier duality functor}\label{sec:VerdierDualityFunctor}
Let us now describe the construction of the Verdier duality functor for cosheaves valued in $\cA$. Let $\cG$ be an $\cA$-valued precosheaf on $X$, that is, an object $\cG \in \Fun(\Opens(X)^{op},\cA)$. Given a closed subset $K \subset X$, the {\em cosections cosupported in $K$} are by definition $\cG_K(X) := \cofib(\cG(X \setminus K) \to \cG(X))$. Taking $\cG_K$ is functorial in $K$ and $\cG$, as it is obtained from the following construction. Consider the map $\Compacts(X)^{op} \to \Fun(\Delta^1,\Opens(X))$ given by $K \mapsto (X \setminus K \to X)$, and take the composition
\[ \Compacts(X)^{op} \times \Fun(\Opens(X),\cA) \to \Fun(\Delta^1,\Opens(X)) \times \Fun(\Opens(X),\cA) \to \Fun(\Delta^1,\cA) \overset{\fib}{\to} \cA \]
We will denote by $D_{X,\cA}: \Fun(\Opens(X),\cA) \to \Fun(\Compacts(X)^{op},\cA)$ the adjoint $\infty$-functor; by construction we have an equivalence $D_{X,\cA}(\cG)(K) \simeq \cG_K(X)$ for each $K$.
\begin{theorem}\cite[Thm.5.5.5.1 and Prop.5.5.5.10]{LurHA}
	The functor $D_{X, \cA}$ restricts to an equivalence $\Cosh(X,\cA) \xrightarrow{\sim} \Sh_\cK(X,\cA)$.
\end{theorem}
The composition $\D_{X, \cA} = \psi \circ D_{X, \cA}$ is called the \emph{covariant Verdier duality functor}; it gives an equivalence of $\infty$-categories $\Cosh(X,\cA) \simeq \Sh(X,\cA)$. It can be proven that its inverse is given, under the identification $\Cosh(X,\cA) \simeq \Sh(X,\cA^{op})^{op}$, by the same procedure with $\cA^{op}$ substituted for $\cA$, that is, by the functor $(\D_{X, \cA^{op}})^{op}$. When $X$ or $\cA$ are clear from context, we omit them from the notation. 

\subsubsection{Relation to classical Verdier duality}
We cannot simply take $\cA$ to be $\Perf$ in the description above, since this category is not closed under all small limits or colimits. However, upon restricting to sheaves that are constructible with respect to a finite stratification $A$, every (co)limit we take when defining the Verdier dual is equivalent to a finite (co)limit, so the equivalence $\D_{X, \vvt}$ restricts to an analogous equivalence\footnote{See \cite{Vol2} for a more detailed proof of this fact, using Lurie's $\infty$-categories of exit paths}
\[ \D_{X,\Perf}: \Sh^A(X,\Perf) \xrightarrow{\sim} \Cosh^A(X,\Perf). \] 
One can further compose this functor with the linear duality functor $\Hom_k(-,k)$; since linear duality exchanges limits and colimits in $\Perf$ there is an equivalence
\[ \D'_X: \Sh^A(X,\Perf) \to \Sh^A(X,\Perf)^{op} \]
corresponding to the classical contravariant notion of Verdier duality for sheaves of (complexes) of vector spaces. Namely, there is an object $\omega_X \in \Sh(X,\Perf)$, called the \emph{dualizing complex} of the space $X$, and a canonical equivalence of functors $\D'_X(-) \simeq \sheafHom_X(-, \omega_X)$.
\begin{definition}
	Let us denote $\varpi_X = D_X(\underline{k}_X)\in \Cosh(X,\Perf)$ the covariant Verdier dual of the constant sheaf $\underline{k}_X$.
\end{definition}
Since $\D'_X$ is an equivalence, we have an equivalence $\varpi_X \simeq \omega_X^\vee$, that is, identifying the object defined above with the $k$-linear dual of the dualizing complex.

Classical Verdier duality does not extend all the way to $\vvt$-coefficients, but it does extend beyond $\Perf$ enough for our purposes. Namely, we consider constructible sheaves valued in the following categories.
\begin{definition}
	The subcategories $\vvt^b_-$ (resp. $\vvt^b_+$) is the full subcategory of $\vvt$ spanned by the objects that are isomorphic to bounded above (resp. below) complexes 
	whose cohomologies are finite-dimensional $k$-vector spaces.
\end{definition}

Note that the $\pm$ subscript refers to which direction the complexes \emph{are allowed} to be supported in infinitely many degrees. From the fact that the compact objects in $\vvt$ are exactly the ones that are isomorphic to complexes of finite-dimensional vector spaces bounded in both directions, it follows that the intersection of $\vvt^b_-$ and $\vvt^b_+$ is identified with $\Perf \subseteq \vvt$. Each of these three subcategories is closed under finite (but not arbitrary) limits and colimits, %
so $\D$ restricts to equivalences $\Sh^A(X,\cC) \xrightarrow{\sim} \Cosh^A(X,\cA)$ for $\cA = \vvt^b_-, \vvt^b_+$ or $\Perf$. In terms of classical Verdier duality, we recover the usual description:
\begin{corollary}
	Contravariant Verdier duality gives an equivalence
	\[ \D': \Sh^A(X, \vvt^b_\pm) \xrightarrow{\sim} (\Sh^A(X, \vvt^b_\mp))^{op}. \]
\end{corollary}

\subsubsection{The bivariant Verdier dual}
We now describe a variant of the construction above which we will later need to define the notion of an orientation on a cosheaf of categories. Recall that the Verdier dual $\Compacts$-presheaf of some $\cA$-valued precosheaf on $X$ is given by applying the composition
\[ \Fun(\Opens(X),\cA) \to \Fun(\Compacts(X)^{op},\Fun(\Delta^1,\cA)) \to \Fun(\Compacts(X)^{op},\cA) \]
where the first functor is pullback under $K \mapsto (X \setminus K \to X)$ and the last functor is induced by the functor $\Fun(\Delta^1,\cA) \to \cA$ taking the cofiber. Let us now make a variant of this definition by considering instead the functor
\[ \Opens(X) \times \Compacts(X)^{op} \to \Fun(\Delta^1,\Opens(X)) \]
given by $(U,K) \mapsto (U \setminus K \to U)$. We then get a composition
\[ (\Opens(X)\times \Compacts(X)^{op}) \times \Fun(\Opens(X),\cA) \to \Fun(\Delta^1,\Opens(X)) \times \Fun(\Opens(X),\cA) \to \Fun(\Delta^1,\cA) \]
We compose the adjoint to the map above with the three maps $\Fun(\Delta^1,\cA) \to \cA$ given by the source, target and cofiber, to get three maps $\Fun(\Opens(X),\cA) \to \Fun(\Opens(X)\times \Compacts(X)^{op},\cA)$, which we call $(-)_{\Opens-\Compacts},(-)_{\Opens}$ and $D_{\Opens-\Compacts}(-)$.
\begin{definition}\label{def:bivariantVerdier}
	For any $\cG \in \Fun(\Opens(X),\cA)$, the \emph{bivariant Verdier dual of $\cG$} is the object of $\Fun(\Opens(X)\times \Compacts(X)^{op},\cA)$ given by $D_{\Opens-\Compacts}(\cG)$.
\end{definition}

We have canonical natural transformations among the three functors out of $\cA$, inducing, for any $\cG \in \Fun(\Opens(X),\cA)$, an equivalence
\[ \cofib((\cG)_{\Opens-\Compacts} \to (\cG)_{\Opens}) \xrightarrow{\sim} D_{\Opens-\Compacts}(\cG) \]
Let us describe these objects: for any $(U,K) \in \Opens(X)\times \Compacts(X)^{op}$, we have equivalences
\[ (\cG)_{\Opens-\Compacts}(U,K) \simeq \cG(U \setminus K), \quad (\cG)_{\Opens}(U,K) \simeq \cG(U), \quad D_{\Opens-\Compacts}(\cG)(U,K) \simeq \cofib(\cG(U \setminus K) \to \cG(U)). \]
We think of the object $D_{\Opens-\Compacts}(\cG)$ as a more local variant of the Verdier dual $\Compacts$-presheaf to the precosheaf $\cG$; the restriction of $D_{\Opens,\Compacts}(\cG)$ to $\{X\} \times \Compacts(X)$ is equivalent to $D(\cG) = D_{X,\cA}(\cG)$. 

We are interested in the values of $D_{\Opens,\Compacts}(\cG)$ on pairs where $K \subseteq U$; therefore, let us denote by $L$ the sub-poset of $\Opens(X) \times \Compacts(X)^{op}$ where that holds.

\begin{remark}
	For a \emph{cosheaf} $\cG$, the object $D_{\Opens-\Compacts}(\cG)|_L \in \Fun(L,\cA)$ contains just as much information as the usual Verdier duality $\Compacts$-sheaf $D(\cG)$; for any $(U,K) \in L$, the map 
	\[ D_{\Opens-\Compacts}(\cG)(U,K) \to D_{\Opens-\Compacts}(\cG)(X,K) = D(\cG)(K)\]
	is an equivalence since the square obtained by applying $\cG(-)$ to the cover of $X$ given by $\{U,X \setminus K\}$ is a pushout square, by the cosheaf property. On the other hand, for an arbitrary \emph{pre}cosheaf $\cG$, the object $D_{\Opens-\Compacts}(\cG)|_L$ cannot be recovered from $D(\cG)$, since its `local Verdier duals' may have more information.
\end{remark}

\subsection{Cosheafified Hochschild homology} 
We have an analogue of the Hom sheaves used throughout the previous section. Let $\cF \in \Cosh(X,\dgcat)$.
\begin{definition}
	For any $U \in \Opens(X)$ and $x,y \in \cF(U)$, the \emph{precosheaf} of morphisms $\cF|_U^{x,y}$ is the object of $\Fun(\Opens(X),\vvt)$ obtained by applying Proposition \ref{prop:morphismFunctor} to $I = \Opens(X)_{U\subseteq}$, with left cone point $U$, and extending by zero to all of $\Opens(X)$. For any triple $x,y,z \in \cF(U)$, we denote by $c^{x,y,z}_U$ the corresponding composition morphism of precosheaves.
\end{definition}
In other words, for any opens $U, V \subseteq X$ and objects $x,y \in \cF(U)$, we have equivalences
\[
\cF|^{x,y}_{U}(V) \simeq \begin{cases} \cF|^{x,y}_{U}(V) \simeq \cF(V)(\iota x,\iota y) & \mbox{if} \,\, U \subseteq V \\ 0 & \mbox{otherwise} \end{cases}
\]
where we generically use $\iota$ to denote the appropriate corestriction map of $\cF$. This is a $\vvt$-valued precosheaf on $X$. In the same fashion as before, we extend the notation to allow for $x,y \in U'$ when $U' \subseteq U$. As we did in the sheaf case, we generalize the notation to define composition maps
\[ c^{x,y,z}_{U,V}: \cF|_U^{x,y} \otimes \cF|_V^{y,z} \to \cF|_{U \cup V}^{x,z} \]
of precosheaves.

\subsubsection{$\Compacts$-presheaf version}
We give a $\cK$-presheaf version of these homs.  Consider the inclusion of posets $\Compacts(X)^{op} \to \Opens(X)$ given by $K \mapsto X \setminus K$; 
considering compacts contained in a fixed $K$ this restricts to an inclusion of posets 
\[ \Compacts(X)^{op}_{\subseteq K} \to \Opens(X)_{X\setminus K \subseteq} \]
into the poset of opens containing $X \setminus K$. 

\begin{definition}
For a compact $K \subset X$ and $x,y \in \cF(X\setminus K)$, we write $\cF|^{x,y}_K$ for the $\Compacts$-precosheaf obtained by 
pulling back $\cF|^{x,y}_{X\setminus K}$ along the inclusion of posets.  That is, for any $K' \subseteq K$, we take 
$\cF|^{x,y}_K(K') = \cF|^{x,y}(X \setminus K')$.  
\end{definition} 

Just as we did for the hom sheaves in Section \ref{sec:sheafHH}, we extend this notation. Given any $x \in \cF(X \setminus K')$ and $y \in \cF(X \setminus K'')$, we define
\[ \cF|^{x,y}_K= \begin{cases}
	\cF|^{\iota x, \iota y}_K \quad \text{if\ } K \subseteq K', K \subseteq K'' \\
	0 \quad \text{otherwise}.
\end{cases}\]

\subsubsection{Hochschild homology of a cosheaf}\label{sec:cosheafHH}
Let us pick spanning sets of objects $S_U \subseteq \mathrm{Ob}(\cF(U))$ for every $U \in \Opens(X)$, such that for any $U \subseteq V$, $S_V$ contains the image of $S_U$ under the corestriction map of $\cF$. This is always possible; one can start by independently picking spanning sets of objects independently for every open, and then adjoining all the images of spanning sets of smaller opens.

We define an object $\HH(\cF)_{(m,n)}$ for each $(m,n)$ by the same formula we used for the sheaf case, namely:
\[ \HH(\cF)_{(m,n)} = \bigoplus \cF|_{V_n}^{y_n,x_0} \otimes \cF|_{U_0}^{x_0,x_1}[1] \otimes \dots \otimes \cF|_{U_m}^{x_m,y_0} \otimes \cF|_{V_0}^{y_0,y_1}[1] \otimes \dots \otimes \cF|_{V_{n-1}}^{y_{n-1},y_n}[1] \]
where the sum is taken over all collections of elements of $\Opens(X)$ and all objects in $S$.
\begin{definition}
	The Hochschild homology precosheaf associated to $\cF$ and $S=\{S_U\}$ is the object of $\PreCosh(X,\vvt)$ given by $\HH(\cF) = \bigoplus_{m,n} \HH(\cF)_{(m,n)}$, endowed with the usual two-pointed Hochschild chain differential using the composition maps $c^{x,y,z}_{U,V}$.
\end{definition}
By the assumption on the object sets $S_U$, we have the exact analogue of Proposition \ref{prop:HHsheafValues}:
\begin{proposition}\label{prop:HHcosheafValues}
	For any $U \in \Opens(X)$, the value of $\HH(\cF)$ on $U$ is equivalent to the Hochschild homology of $\cF(U)$, and for any pair $U \subseteq V$ in $\Opens(X)$, the corresponding corestriction map is equivalent to the map induced on Hochschild homologies by $\rho: \cF(U) \to \cF(V)$.
\end{proposition}
$\HH(\cF)$ is generally not a cosheaf, and we denote by $\cH\cH(\cF)$ its cosheafification. Analogously to the sheaf case, we denote by $\cH\cH(\cF)^{S^1}, \cH\cH(\cF)_{S^1}$ the negative cyclic and cyclic homology cosheaves, given by cosheafifying the analogous respective precosheaves.

\subsection{Orientations}

\subsubsection{The Verdier dual diagonal $\Compacts$-sheaf}
In order to state the definition of nondegeneracy condition on $\Omega$, we will need to check a certain map of sheaves of bimodules; to construct that, we use a similar method to the one we used in Section \ref{sec:cosheafHH}, but using sheaves and $\Compacts$-sheaves valued in bimodules instead. We define the following $\cF(X)^e\mh\Mod$-valued $\cK$-presheaves on $X$:
\begin{align*}
	\DDelta_X(\cF)_{m,n} = \bigoplus \cofib(&\cF(X)(-,\iota x_0) \otimes \dots \otimes \cF|_{K_m}^{x_m,y_0} \otimes \dots \otimes \cF(X)(\iota y_n,-) \\
	&\to \cF(X)(-,\iota x_1) \otimes \dots \otimes \cF(X)(\iota x_m,\iota y_0) \otimes \dots \otimes\cF(X)(\iota y_n,-)) 
\end{align*}
where the sum is over all compacts and objects in $S$, and the terms $\cF(X)(\dots)$ denote constant presheaves on $X$, whose sections over any open are given by the denoted object of $\vvt$. 

Defining $\DDelta_X(\cF) = \bigoplus_{m,n}(\DDelta_X(\cF)_{m,n})$, together with the differential given by the composition maps, we get a $\cF(X)^e\mh\Mod$-valued $\Compacts$-presheaf on $X$ whose value on $K$ is equivalent to $\Cone(I_! \cF(X\setminus K)_\Delta \to \cF(X)_\Delta)$, that is, the cofiber of the map of bimodules from the pushforward of the diagonal bimodule of $\cF(X \setminus K)$ to the diagonal bimodule of the global sections category.
\begin{proposition}
	If $\cF$ is a cosheaf, the $\Compacts$-presheaf $\DDelta_X(\cF)$ is an object of the subcategory $\Sh(\Compacts(X),\cF(X)^e\mh\Mod)$.
\end{proposition}
\begin{proof}
	We need to verify that this object satisfies conditions (1) and (2) of Definition \ref{def:Ksheaf}; condition (1) holds trivially since the sections over $K=\emptyset$ are given by the cofiber of an isomorphism. Using the description of sections of $\DDelta_X(\cF)$ as pushforwards of diagonal bimodules, condition (2) says that for every pair $(K,K')$ of compacts
	\[\xymatrix{
		I_!(\cF(X \setminus (K \cup K'))_\Delta) \ar[r] \ar[d] & I_!(\cF(X \setminus K)_\Delta) \ar[d] \\
		I_!(\cF(X \setminus K')_\Delta) \ar[r] & I_!(\cF(X \setminus (K \cap K'))_\Delta)
	}\]
	is a pullback square, or equivalently a pushout square, since the category of $\cF(X)$-bimodules is stable. But since $\cF$ is a cosheaf, this is exactly the statement of Theorem 6.1 in \cite{BD1}.
\end{proof}

A similar argument also proves the following result, concerning functoriality with respect to $X$.
\begin{proposition}\label{prop:VerdierDiagonalFunctoriality1}
	If $\cF$ is a cosheaf, there is a canonical isomorphism $I_! \DDelta_U(\cF|_U) \to \DDelta_X(\cF)|_U$.
\end{proposition}
\begin{proof}
	We can construct this map using the explicit description of $\DDelta_U(\cF)$ and using the fact that $I_!$ can be expressed tensoring on both sides over $\cF_U$ with $\cF(X)_\Delta$. It remains to check that for any $K \in \Compacts(U)$ this map induces an isomorphism, which follows from the fact that the square
	\[\xymatrix{
		I_!(\cF(U \setminus K)_\Delta) \ar[r] \ar[d] & I_!(\cF(U)_\Delta) \ar[d] \\
		I_!(\cF(X \setminus K)_\Delta) \ar[r] & I_!(\cF(X)_\Delta)
	}\]
	is a pushout square because $\cF$ is a cosheaf.
\end{proof}

We can apply the functor $\varphi$ of Proposition \ref{prop:functorPhi} to get a $\Compacts$-sheaf of bimodules.
\begin{definition}\label{def:VerdierDualDiagonal}
	The \emph{Verdier dual diagonal $\Compacts$-sheaf} is the object of $\Sh_\cK(X,\cF(X)^e\mh\Mod)$ given by $\varphi\DDelta_X(\cF)$.
\end{definition}
The same functoriality property of Proposition \ref{prop:VerdierDiagonalFunctoriality1} holds for this $\Compacts$-sheaf:
\begin{proposition}\label{prop:VerdierDiagonalFunctoriality2}
	If $\cF$ is a cosheaf, there is a canonical isomorphism $I_! \varphi\DDelta_U(\cF|_U) \to \varphi\DDelta_X(\cF)|_U$.
\end{proposition}
\begin{proof}
	Follows from applying $\varphi$ and using the fact that the functor $\varphi$ is calculated on each $K$ by a colimit, and the functor $I_!:\cF(U)^e\mh\Mod \to \cF(X)^e\mh\Mod$ is a left adjoint thus commutes with colimits.
\end{proof}

\subsubsection{The bimodule dual diagonal $\Compacts$-sheaf}
We now define another $\cK$-sheaf, which will be our representative for the `bimodule dual' of the diagonal cosheaf. Recall that the bimodule dual of a perfect $\cA$-bimodule $\cM$ is given by $\cM^!=\Hom_{\cA^e}(\cM,\cA^e)$; to calculate this derived hom one can take e.g. the bar complex resolution. We now use this strategy, but with sheaves of morphisms instead.

Let us define the following object of $\Fun(\Opens(X),\cF(X)^e\mh\Mod)$
\[ \Delta_X(\cF)_{p,q} :=  \bigoplus \cF(X)(-,\iota x_0) \otimes \dots \otimes \cF|_{U_p}^{x_p,y_0} \otimes \dots \otimes \cF(X)(\iota y_q,-)\]
and its left dual as a $\cF(X)$-bimodule, that is, the object of $\Fun(\Opens(X)^{op},\cF(X)^e\mh\Mod)$ given by
\[ \Delta^!_X(\cF)_{p,q} :=  \prod \Hom_{\cF(X)^e}\left( \cF(X)(-,\iota x_0) \otimes \dots \otimes \cF|_{U_p}^{x_p,y_0} \otimes \dots \otimes \cF(X)(\iota y_q,-), \cF(X)^e\right) \]

We denote by $\Delta_X(\cF)$ and $\Delta^!_X(\cF)$ the objects given by the direct sum/product over all $p,q$ and endowed with the differential coming from the composition maps. We can apply the functor $\theta$ mentioned in the statement of Proposition \ref{prop:functorsUsheafKsheaf} to pass to a $\Compacts$-presheaf; for any $K$ there is an equivalence $\theta(\Delta^!_X(\cF))(K) \cong \lim_{K \subseteq U} \Delta^!_X(\cF)(U)$.
\begin{proposition}
	The object $\theta(\Delta^!_X(\cF))$ belongs to the category $\Sh(\Compacts(X),\cF(X)^e\mh\Mod)$.
\end{proposition}
\begin{proof}
	We must verify that this object satisfies conditions (1) and (2) of Definition \ref{def:Ksheaf}. Condition (1) is automatically satisfied, and to prove condition (2), we choose small enough open neighborhoods $U, U'$ of $K,K'$; it is then sufficient to prove that the square
	\[\xymatrix{
		I_!(\cF(U \cup U')^!) \ar[r] \ar[d] & I_!(\cF(U)^!) \ar[d] \\
		I_!(\cF(U')^!) \ar[r] & I_!(\cF(U \cap U')^!)
	}\]
	is a pullback square for all small enough open neighborhoods $U, U'$ of $K,K'$. This holds since the square above is obtained from a pushout square of compact bimodules by applying the functor $(-)^!$, and this latter is an anti-equivalence when restricted to compact bimodules.
\end{proof}

We can apply the functor of Proposition \ref{prop:functorPhi} to get a $\Compacts$-sheaf.
\begin{definition}\label{def:bimoduleDualDiagonal}
	The \emph{bimodule dual diagonal $\Compacts$-sheaf} is the $\cF(X)^e\mh\Mod$-valued $\Compacts$-sheaf on $X$ given by $\varphi\theta(\Delta^!_X(\cF))$.
\end{definition}
We have functoriality of $\theta(\Delta^!_X(\cF))$ and $\varphi\theta(\Delta^!_X(\cF))$ with respect to $X$:
\begin{proposition}\label{prop:bimoduleDualFunctoriality}
	If $\cF$ is a cosheaf, there are canonical isomorphisms $I_! \theta\Delta^!_U(\cF|_U) \simeq \theta\Delta^!_X(\cF)|_U$ and $I_! \varphi \theta\Delta^!_U(\cF|_U) \simeq \varphi \theta\Delta^!_X(\cF)|_U$
\end{proposition}
\begin{proof}
	Follows from similar arguments used to prove Propositions \ref{prop:VerdierDiagonalFunctoriality1} and \ref{prop:VerdierDiagonalFunctoriality2}. We can construct the map $I_!\Delta^!_U(\cF|_U) \to \Delta^!_X(\cF)|_U$ using the explicit representatives and the corestriction map $\cF(U) \to \cF(X)$, and then check that it gives an isomorphism on every $V\subseteq U$ as a consequence of the fact that all of its sections are perfect bimodules so by Lemma 4.2 of \cite{BD1}, pushforwards commute with taking bimodule dual. The statement about the $\Compacts$-sheaves (that is, after applying $\varphi$) follows since $\varphi$ commutes with $I_!$.
\end{proof}

\subsubsection{Nondegeneracy condition}
Let us recall the definition from the introduction: the data of an orientation of dimension $n$ on the \textsc{ccc} space $(X,\cF)$ is a morphism of cosheaves
\[ \Omega: \varpi_X[d] \to \cH\cH(\cF)^{S^1} \]
where $\omega_X^\vee$ is the covariant Verdier dual of the constant sheaf $\underline{k}$ on $X$, and $\cH\cH(\cF)^{S^1}$ is the cosheafification of the negative cyclic homology cosheaf. In order to give an orientation, the corresponding morphism of cosheaves $\Omega:\varpi_X[d] \to \cH\cH(\cF)$ must satisfy a certain nondegeneracy condition, namely, we require that under a certain universal map
\[ \eta_X: \Hom_{\Cosh(X,\vvt)}(\varpi_X,\cH\cH(\cF)) \to \Hom_{\Sh_\cK(X,\cF(X)^e\mh\Mod)}(\varphi\theta\Delta^!_X(\cF),\varphi\DDelta_X(\cF)) \]
the morphism of cosheaves $\Omega$ maps to an isomorphism of $\Compacts$-sheaves $\varphi\theta\Delta^!_X(\cF)[d] \simeq \varphi\DDelta_X(\cF)$. The remainder of this subsection is devoted to the construction of the map $\eta_X$.

We can apply the `bivariant Verdier dual' functor of Definition \ref{def:bivariantVerdier} to the Hochschild homology precosheaf to get an object $D_{\Opens-\Compacts}(\HH(\cF))$ of the category $\Fun(\Opens(X)\times \Compacts(X)^{op},\vvt)$. By definition, for any $(U,K) \in \Opens(X)\times\Compacts(X)^{op}$, there is an isomorphism
\[ D_{\Opens-\Compacts}(\HH(\cF))(U,K)\simeq \cofib(\HH(\cF(U \setminus K)) \to \HH(\cF(U))). \]
We note that, using bar complexes, one can construct maps
\[ \HH(\cF(U\setminus K)) \simeq \cF(U\setminus K)_\Delta \otimes_{\cF(U\setminus K)} \cF(U \setminus K)_\Delta \to I_!\cF(U)_\Delta \otimes_{\cF(X)^e} I_!(\cF(X\setminus K)_\Delta) \]
giving a map $D_{\Opens-\Compacts}(\HH(\cF))(U,K) \to I_!\cF(U)_\Delta \otimes_{\cF(X)^e} \cofib(I_!(\cF(X\setminus K)_\Delta) \to \cF(X)_\Delta)$. Since $\cF$ is a cosheaf of smooth dg categories by assumption, we have an isomorphism between this latter complex and the bimodule hom $\Hom_{\cF(X)^e}(I_!(\cF(U)_\Delta)^!,\cofib(I_!(\cF(X\setminus K)_\Delta) \to \cF(X)_\Delta))$.

We now consider the objects
\begin{align*}
	\Delta^!_X(\cF) &\in \Fun(\Opens(X)^{op},\cF(X)^e\mh\Mod) \simeq \Fun(\Opens(X),(\cF(X)^e\mh\Mod)^{op}) \\
	\DDelta_X(\cF) &\in \Fun(\Compacts(X)^{op},\cF(X)^e\mh\Mod)
\end{align*}
constructed in the preceding subsections. We denote by
\[ \Hom_{\cF(X)^e\mh\Mod}(\Delta^!_X(\cF),\DDelta_X(\cF)) \in \Fun(\Opens(X)\times \Compacts(X)^{op},\vvt) \]
the object obtained by combining these two objects with the $\vvt$-enriched internal hom
\[ \Hom_{\cF(X)^e\mh\Mod}: (\cF(X)^e\mh\Mod)^{op} \times \cF(X)^e\mh\Mod \to \vvt. \]
Explicitly, the object $\Hom(\Delta^!_X(\cF),\DDelta_X(\cF))$ associates to a pair $(U,K)$ the object of $\vvt$ given by
\[ \Hom_{\cF(X)^e\mh\Mod}(\Delta^!_X(\cF)(U),\DDelta_X(\cF)(K) \simeq \Hom_{\cF(X)^e}(I_!(\cF(U)_\Delta)^!,\cofib(I_!(\cF(X\setminus K)_\Delta) \to \cF(X)_\Delta)).  \]\
Therefore, for each fixed $(U,K)$, combining the identifications above, we have a map of complexes
\[ \xi_{X,(U,K)}: D_{\Opens-\Compacts}(\HH(\cF))(U,K) \to \Hom(\Delta^!_X(\cF),\DDelta_X(\cF))(U,K). \]

The following Lemma shows that these maps can be all assembled coherently over $\Opens(X)\times \Compacts(X)^{op}$.
\begin{lemma}
	There is a natural morphism in the category $\Fun(\Opens(X)\times \Compacts(X)^{op},\vvt)$
	\[ \xi_X: D_{\Opens-\Compacts}(\HH(\cF)) \to \Hom_{\cF(X)^e\mh\Mod}(\Delta^!_X(\cF),\DDelta_X(\cF)), \]
	whose induced map on any $(U,K)$ is equivalent to the map $\xi_{X,(U,K)}$.
\end{lemma}
\begin{proof}
	By definition, the bivariant Verdier dual $D_{\Opens-\Compacts}(\HH(\cF))$ is equivalent to a cofiber $(\HH(\cF))_{\Opens-\Compacts} \to (\HH(\cF))_{\Opens}$, so we have an equivalence
	\begin{align*}
		&D_{\Opens-\Compacts}(\HH(\cF)) \simeq \\
		\hspace{-1cm}\bigoplus \cofib&\left(\cF|_{U\setminus K}^{w_q,x_0} \otimes \dots \otimes \cF|_{U \setminus K}^{x_m,y_0} \otimes \dots \otimes \cF|_{U\setminus K}^{y_n,z_0} \otimes \dots \otimes \cF|_{U \setminus K}^{z_p,w_0} \otimes \dots\right. \\
		&\left. \to \cF|_U^{w_q,x_0} \otimes \dots \otimes \cF|_U^{x_m,y_0} \otimes \dots \otimes \cF|_U^{y_n,z_0} \otimes \dots \otimes \cF|_U^{z_p,w_0} \otimes \dots \right)
	\end{align*}
	where again the sum is taken over all $(m,n,p,q),(U,K)$ and objects in our chosen object sets, and the tensor factors are the appropriate functors out of $\Opens(X)\times\Compacts(X)^{op}$, and the whole complex is endowed with the differential coming from composition maps.

	We map the object above to
	\begin{align*}
		\hspace{-1cm}\bigoplus \cofib&\left(\cF(X)(w_q,x_0) \otimes \dots \otimes \cF|_U^{x_m,y_0} \otimes \dots \otimes \cF(X)(y_n,z_0) \otimes \dots \otimes \cF|_{X \setminus K}^{z_p,w_0} \otimes \dots\right. \\
		&\left. \to \cF(X)(w_q,x_0) \otimes \dots \otimes \cF|_U^{x_m,y_0} \otimes \dots \otimes \cF(X)(y_n,z_0) \otimes \dots \otimes \cF(X)(z_p,w_0) \otimes \dots \right)
	\end{align*}
	where the factors in the dots are of the forms $\cF|_U^{-,-}[1]$, $\cF|_{X\setminus K}^{-,-}$ or $\cF(X)(-,-)$. Since tensor commutes with cofiber, the object above is isomorphic to
	\begin{align*}
		\hspace{-1cm}\bigoplus &\left(\cF(X)(-,x_0) \otimes \dots \otimes \cF|_U^{x_m,y_0} \otimes \dots \otimes \cF(X)(y_n,-)\right) \otimes_{\cF(X)^e} \\
		 &\cofib\left(\cF(X)(-,z_0) \otimes \dots \otimes \cF|_{X \setminus K}^{z_p,w_0} \otimes \dots \cF(X)(w_q,-) \to \right.\\
		 &\left.\cF(X)(-,z_0) \otimes \dots \otimes \cF(X)(z_p,w_0) \otimes \dots \cF(X)(w_q,-)\right)
	\end{align*}
	that is, the two-sided tensor $\Delta_X\cF \otimes_{\cF(X)^e} \DDelta_X(\cF)$. Evaluating the image of this map against $\Delta^!_X\cF$ gives the desired map.
\end{proof}

Recall that if $C$ is an ordinary category, morphisms in presheaf categories can be described by the formalism of `ends': the morphism space between two $C$-valued sheaves on $X$ is given by 
\[ \Hom_{\Sh(X,C)}(\cG_1,\cG_2) = \int_{\Opens(X)^{op}} \Hom_C \circ (\cG_1 \times \cG_2) \]
the end of the bifunctor given by composing with $\cG_1 \times \cG_2$ with
\[ \Hom_C: C^{op} \times C \to \mathrm{Set}. \]
In the setting of $\infty$-categories, the appropriate notion of end can be given as the limit over a certain category of `twisted arrows' \cite{Hau2}. For any $\infty$-category $I$, there is an $\infty$-category $\mathrm{Tw}^\ell(I)$ with a canonical morphism
\[ p: \mathrm{Tw}^\ell(I) \to I^{op} \times I \]
such that the end of an $\infty$-functor $\cH:I^{op} \times I \to \cA$ is given by $\int_I \cH = \lim(F \circ p)$. If we have a $\vvt$-enriched $\infty$-category $\cC$, and $\cG_1,\cG_2 \in \Sh(X,\cC)$, we can then express their hom space as an object of $\vvt$ as $\lim(p \circ (\cG_1 \otimes \cG_2))$, a limit over $\mathrm{Tw}^\ell(\Opens(X)^{op})$.

In order to construct the desired morphism $\eta_X$, we must relate these ends to our previous constructions. Recall the following two functors: Lurie's covariant Verdier functor
\[ \D: \Fun(\Opens(X)^{op},\cA) \to \Fun(\Opens(X),\cA) \]
and bivariant Verdier duality
\[ D_{\Opens-\Compacts}: \Fun(\Opens(X),\cA) \to \Fun(\Opens(X)\times\Compacts(X)^{op},\cA) \]
Recall that the functor $\D$, when restricted to sheaves, gives an equivalence of categories, with inverse given again by the functor $\D$ on cosheaves. We will now prove a generalization of this statement involving the functor $D_{\Opens-\Compacts}$. 

Note that if $\cG \in \Sh(X,\cA)$, upon fixing $U$, $D_{\Opens-\Compacts} \circ \D_X(\cG)$ restricted to $U \times \Compacts(U)$ is isomorphic to $D_U \circ \D_U(\cG)\simeq \theta\cG$. The following proposition is just about making this statement coherent over all opens. Consider the projection map
\[ p: \Opens(X) \times \Compacts(X)^{op} \to \Compacts(X)^{op} \]
Composing the corresponding pullback with the functor $\theta$ gives a functor
\[ p^* \circ \theta: \Fun(\Opens(X)^{op},\cA) \to \Fun(\Opens(X)\times\Compacts(X)^{op},\cA) \]
We now have two functors with the same source and target, $D_{\Opens-\Compacts} \circ \D$ and $p^*\circ \theta$. We can further restrict the source to sheaves and compose the target with restriction to the locus $L \subset \Opens(X)\times\Compacts(X)^{op}$ where $K \subseteq U$.
\begin{proposition}\label{prop:doubleVerdier}
	There is a natural equivalence
	\[  (-)|_L \circ D_{\Opens-\Compacts} \circ \D \simeq (-)|_L \circ p^*\circ \theta \]
	of functors $\Sh(X,\cA) \to \Fun(L,\cA)$.
\end{proposition}
\begin{proof}
	We already know that, for any $\cG \in \Sh(X,\cA)$, the sections of $D_{\Opens-\Compacts} \circ \D(\cG)$ and $p^*\circ \theta(\cG)$ on $(U,K)$ with $K \subseteq U$ are abstractly isomorphic, so it just remains to construct the correct natural transformation of functors. For conciseness let us write $\Opens = \Opens(X)$ and $\Compacts = \Compacts(X)$. We define the following maps of simplicial sets
	\[ a: \Delta^1 \times \Delta^1 \times \Opens\times \Compacts^{op} \to \Delta^1\times \Opens, \qquad b: \Delta^1 \times \Compacts \to \Opens^{op} \]
	given by
	\[ a(r,s,U,K) = \begin{cases}
		(r,U \setminus K), & s=0\\
		(r,U), & s=1
	\end{cases}  \qquad b(r,K') = \begin{cases}
		X, &r=0\\
		X\setminus K', &r=1
	\end{cases}\]
	Denoting $\cM = \Opens \cup \Compacts$ as a subset of the poset of subsets of $X$, we consider the diagram
	\[ \Delta^1 \times \Delta^1 \times \Opens \times \Compacts \xrightarrow{a} \Delta^1 \times \Opens \overset{i}{\hookrightarrow} \Delta^1 \times \cM \overset{j}{\hookleftarrow} \Delta^1 \times \Compacts \xrightarrow{b} \Opens^{op} \]
	We then have a functor $a^* i^* j_! b^*: \Fun(\Opens^{op},\cA) \to \Fun(\Delta^1\times\Delta^1\times\Opens\times\Compacts^{op},\cA)$, which gives an adjoint
	\[ (a^* i^* j_! b^*)^\sharp: \Fun(\Opens^{op},\cA) \to \Fun(\Opens\times\Compacts^{op},\Fun(\Delta^1\times\Delta^1,\cA)\]
	It follows from the definitions of the functors $D_{\Opens-\Compacts}$ and $\D$ that the functor above, composed with the functor induced by $\fib^2:\Fun(\Delta^1\times\Delta^1,\cA))\to \cA$, gives $D_{\Opens-\Compacts} \circ \D$.

	Let us define a simplicial set $Z \subset \Delta^1\times\Delta^1\times \Opens \times \Compacts^{op} \times \Compacts$ given by
	\[ Z = \left\{ (r,s,u,K,K')\ \middle|\ \begin{array}{ll} K' \subseteq U \setminus K, &s=0 \\ K' \subseteq U, &s=1 \end{array} \right\} \]
	We form the square
	\[\xymatrix{
		Z \ar[r]^-{\pi_2} \ar[d]_{\pi_1} & \Delta^1 \times \Compacts \ar[d]^j \\
		\Delta^1 \times \Delta^1 \times \Opens \times \Compacts \ar[r]_-{i\circ a} & \Delta^1 \times \cM
	}\]
	where $\pi_1(r,s,U,K,K') = (r,s,U,K)$ and $\pi_2(r,s,U,K,K') = (r,K')$. By the definition of $Z$, the image of $\pi_2$ on some fixed input is always contained in the image of $a \circ \pi_1$ on the same input,so we have a natural transformation of functors $j\circ \pi_2 \Rightarrow i \circ a \circ \pi_1$; using (co)unit maps of the relevant adjunction gives then a natural transformation $\pi_{1!}\pi_2^* a^* \Rightarrow a^* i^* j_! b^*$.

	We then form another square
	\[\xymatrix{
		Z \ar[r]^-{b \circ \pi_2} \ar[d]_{\pi_1} & \Opens^{op} \ar[d]^{\lambda_2} \\
		\Delta^1 \times \Delta^1 \times \Opens \times \Compacts^{op} \ar[r]_-{\lambda_1} & \Delta^1 \times \Delta^1 \times \cM^{op}
	}\]
	where $\lambda_1(r,s,U,K)=(r,s,K)$ and $\lambda_2(V)=(0,0,V)$. We now have a natural transformation $\lambda_2\circ b \circ \pi_2 \Rightarrow \lambda_1\circ \pi_1$; using adjunctions we get a natural transformation $\pi_{1!}\pi_2^* a^* \Rightarrow \lambda_1^*\lambda_{2!}$. This latter functor, upon taking adjoints and composing with $\fib^2$, gives the functor $p^*\circ \theta$ of the statement we are trying to prove. 
	
	In summary, we get natural transformations
	\[ D_{\Opens-\Compacts} \circ \D \Leftarrow \fib^2\circ (\pi_{1!} \circ \pi_2^*)^\sharp \Rightarrow p^* \circ \theta \]
	of functors $\Fun(\Opens^{op},\cA) \to \Fun(\Opens\times \Compacts^{op},\cA)$, so to prove the statement we want it is enough to show that these induce equivalences when applied to a sheaf $\cG$ and then restricting the result to the locus $L$. We can see this by calculating these maps explicitly; for example, picking some pair $K\subseteq U$, we have that $(\fib^2\circ (\pi_{1!} \circ \pi_2^*)^\sharp \cG)(U,K)$ is given by taking the double fiber of the square
	\[\xymatrix{
		\cF(X) \ar[r] \ar[d]	& \colim_{K' \subseteq U\setminus K} \cF(X \setminus K') \ar[d] \\
		\cF(X) \ar[r] 			& \colim_{K' \subseteq U} \cF(X\setminus K')
	}\]
	which, when $\cF$ is a sheaf, is equivalent to $\colim_{K\subseteq V} \cF(V)$. Similar calculations of the other two functors show that on each $(U,K)$ the maps are isomorphisms.
\end{proof}

We now use the maps we described to produce our desired map. We have a composition
\begin{align*}
	\Hom_{\Cosh(X,\vvt)}&(\varpi_X,\cH\cH(\cF)) \to \Hom_{\Cosh(X,\vvt)}(\varpi_X,\HH(\cF)) \\
	&\xrightarrow{D_{\Opens-\Compacts}} \Hom_{\Fun(\Opens(X)\times \Compacts(X)^{op},\vvt)}(D_{\Opens-\Compacts}\D k_X,D_{\Opens-\Compacts}(\HH(\cF))) \\
	& \xrightarrow{\xi_X} \Hom_{\Fun(\Opens(X)\times \Compacts(X)^{op},\vvt)}(D_{\Opens-\Compacts}\D k_X,\Hom_{\cF(X)^e}(\Delta^!_X\cF,\DDelta_X\cF)) \\
	& \xrightarrow{(-)|_L} \Hom_{\Fun(L,\vvt)}((D_{\Opens-\Compacts}\D k_X)|_L,\Hom_{\cF(X)^e}(\Delta^!_X\cF,\DDelta_X\cF)|_L) \\
	&\to \Hom_{\Fun(L,\vvt)}((D_{\Opens-\Compacts}\D k)|_L,\Hom_{\cF(X)^e}(\Delta^!_X\cF,\DDelta_X\cF)|_L)
\end{align*}
where $k$ denotes the constant functor with value $k$ and the last map is induced by the canonical map from the constant presheaf to the constant sheaf. By Proposition \ref{prop:doubleVerdier}, we have that $(D_{\Opens-\Compacts}\D k)|_L$ is the constant functor on $L$ with value $k$. Applying $\theta$ along the first coordinate, to get functors out of $L_\Compacts = \{(J,K)| K \subset J\} \subset \Compacts \times \Compacts^{op}$, we get a further map from the last complex above:
\begin{align*}
	&\to \Hom_{\Fun(L_\Compacts,\vvt)}(k,\Hom_{\cF(X)^e}(\theta\Delta^!_X\cF,\DDelta_X\cF)|_L)
\end{align*}
That is, the complex above is the data of a coherent choice of map $k \to \Hom_{\cF(X)^e}(\theta\Delta^!_X\cF(J),\DDelta_X\cF(K))$ for any pair of compacts with $K \subseteq J$. We now take ends, noting that, by definition, the image of the canonical map from the category of twisted arrows $\mathrm{Tw}^\ell(\Compacts(X)^{op}) \to \Compacts(X)\times\Compacts(X)^{op}$ lands inside of the locus $L_\Compacts$. Therefore we can pullback along this map and take limits, giving a map of complexes
\begin{align*} 
	\Hom_{\Fun(L_\Compacts,\vvt)}(k,\Hom_{\cF(X)^e}(\theta\Delta^!_X\cF,\DDelta_X\cF)|_L) &\to \int^{\Compacts(X)^{op}} \Hom_{\cF(X)^e}(\theta\Delta^!_X\cF,\DDelta_X\cF) \\ 
	&= \Hom_{\Sh(X,\cF(X)^e\mh\Mod)}(\theta\Delta^!_X\cF,\DDelta_X\cF)
\end{align*}
\begin{definition}\label{def:etaMap}
	The map $\eta_X$ is given by composing the map
	\[ \Hom_{\Cosh(X,\vvt)}(\varpi_X,\cH\cH(\cF)) \to \Hom_{\Fun(\Compacts(X)^{op},\cF(X)^e\mh\Mod)}(\theta\Delta^!_X\cF,\DDelta_X\cF) \]
	constructed above with the map
	\[ \Hom_{\Fun(\Compacts(X)^{op},\cF(X)^e\mh\Mod)}(\theta\Delta^!_X\cF,\DDelta_X\cF) \to \Hom_{\Sh_\Compacts(X,\vvt)}(\varphi\theta\Delta^!_X\cF,\varphi\DDelta_X\cF) \] 
	induced by the functor $\varphi$ of Proposition \ref{prop:functorPhi}.
\end{definition}

This map is universal in $X$, in the following sense. Recall that for any sheaf $\cF$ and $U \in \Opens(X)$, we have a natural isomorphism $(\D_X\cF)_U \simeq \D_U(\cF_U)$, so we have a natural isomorphism $\varpi_U \simeq \varpi_X|_U$. Therefore, restricting on both sides of the map $\eta_X$, we get a square of complexes
\[ \xymatrix{
	\Hom_{\Cosh(X,\vvt)}(\varpi_X,\cH\cH(\cF)) \ar[r] \ar[d]	& \Hom_{\Sh_\Compacts(X,\cF(X)^e\mh\Mod)}(\varphi \theta \Delta^!_X(\cF),\varphi\DDelta_X(\cF)) \ar[d] \\
	\Hom_{\Cosh(U,\vvt)}(\varpi_U,\cH\cH(\cF|_U)) \ar[r] 		& \Hom_{\Sh_\Compacts(U,\cF(X)^e\mh\Mod)}(\varphi \theta \Delta^!_U(\cF|_U),I_!\varphi\DDelta_U(\cF|_U))
}\]
\begin{proposition}\label{prop:etaFunctoriality}
	The square above commutes up to homotopy.
\end{proposition}
\begin{proof}
	This follows from the functoriality results we gave in Propositions \ref{prop:VerdierDiagonalFunctoriality2} and \ref{prop:bimoduleDualFunctoriality}, together with the fact that the map $\xi_X$ is also functorial in $X$, in the sense that $\xi_X$ and $\xi_U$ intertwine the identifications given by the two Propositions cited above; this can be seen by a calculation using the explicit description we gave for these $\xi$ maps.
\end{proof}

\begin{example} \emph{The half-open interval} \label{ex:openIntervalCosheaf}
Let $X = [0,1)$ with stratification $\{0\} \cup (0,1)$. Let us denote $U = (0,1)$; the data of a constructible cosheaf $\cF$ on $X$ is then given by a corestriction functor $\iota: \cF(U) \to \cF(X)$.

	To describe the relevant $\Compacts$-sheaves, we note that there are only two types of connected compact subsets; let us pick one representative for each type, for example, compact subsets $K_1 = [0, 3/4]$ and $K_2=[1/4,1/2]$. Denoting $\sim$ for the equivalence relation generated by stratified retracts of open subsets, we have
	\[ X \setminus K_1 \sim U, \quad X \setminus K_2 \sim X \sqcup U \]
	with the inclusion $X \setminus K_1 \hookrightarrow X \setminus K_2$ equivalent to the inclusion $U \hookrightarrow X \sqcup U$.
	
	Let us describe the Verdier dual Hochschild $\Compacts$-presheaf. We have equivalences
	\begin{align*}
		\DHH_X(\cF)(K_1) &\cong \cofib\left(\HH(\cF(U)) \xrightarrow{\iota} \HH(\cF(X))\right) \\
		\DHH_X(\cF)(K_2) &\cong \cofib\left(\HH(\cF(U)) \oplus \HH(\cF(X)) \xrightarrow{(\iota,\id)} \HH(\cF(X))\right) \cong \HH(\cF(U))[1]
	\end{align*}
	with restriction map given by the connecting map. We can also describe the version over $\Opens(X)\times \Compacts(X)^{op}$. By definition we have $\DHH(\cF)(X,K_i) = \DHH_X(\cF)(K_i)$, and that the map $\DHH(\cF)(U,K) \to \DHH(\cF)(X,K)$ is the equivalence
	\begin{align*}
		\cofib&\left(\HH(\cF(U))\oplus \HH(\cF(U))\to \HH(\cF(U))\right) \xrightarrow{\sim} \\
		&\cofib\left(\HH(\cF(X))\oplus \HH(\cF(U))\to \HH(\cF(X))\right) \cong \HH(\cF(U))[1]
	\end{align*}
	
	As for the $\Compacts$-sheaves of bimodules, we note that every sufficiently small open set containing $K_1$ or $K_2$, is equivalent, by retracts respecting the stratification, to $X$ or $U$, respectively. Therefore, we have equivalences
	\[ \varphi\theta\Delta^!_X(\cF)(K_1) \cong \cF(X)^!, \quad \varphi\theta\Delta^!_X(\cF)(K_2) \cong I_!\cF(U)^!. \]
	Calculating the Verdier dual diagonal bimodule $\Compacts$-sheaf gives
	\begin{align*}
		\varphi\DDelta_X(\cF)(K_1) &\cong \cofib\left(I_!\cF(U)_\Delta \to \cF(X)_\Delta\right) \\ \varphi\DDelta_X(\cF)(K_2) &\cong \cofib\left(\cF(X)_\Delta \oplus I_!\cF(U)_\Delta \to \cF(X)_\Delta \right) \cong I_!\cF(U)_\Delta[1]
	\end{align*}

	Suppose now that we are given a section $s:\underline{k}_L \to \DHH(\cF)|_L$. This is just given by an element $s_{(X,K_1)} \in \cofib\left(\HH(\cF(U)) \xrightarrow{\iota} \HH(\cF(X))\right)$. %

	We can describe explicitly the map of bi-simplicial objects. For each compact $K$, we choose small enough open $V$ such that $K \subseteq V$. We then have a map
	\begin{align*}
		\phi:&\DHH(\cF)(V,K)_{m,n,p,q} = \\
		\hspace{-1cm}\bigvee \cofib&\left(\cF(V\setminus K)(w_q,x_0) \otimes \dots \otimes \cF(V \setminus K)(x_m,y_0) \otimes \dots \otimes \cF(V\setminus K)(y_n,z_0) \otimes \dots \otimes \cF(V \setminus K)(z_p,w_0) \otimes \dots\right. \\
		&\left. \to \cF(V)(w_q,x_0) \otimes \dots \otimes \cF(V)(x_m,y_0) \otimes \dots \otimes \cF(V)(y_n,z_0) \otimes \dots \otimes \cF(V)(z_p,w_0) \otimes \dots \right)
		&\to \cofib\left(\cF(X)(w_q,x_0) \otimes \dots \otimes \cF(X \setminus K)(x_m,y_0) \otimes \dots \otimes \cF(X)(y_n,z_0) \otimes \dots \otimes \cF(V)(z_p,w_0) \otimes \dots \right.\\
		&\left. \to \cF(X)(w_q,x_0) \otimes \dots \otimes \cF(X)(x_m,y_0) \otimes \dots \otimes \cF(X)(y_n,z_0) \otimes \dots \otimes \cF(V)(z_p,w_0) \otimes \dots \right)
	\end{align*}
	We can evaluate the image $\phi(s_{(V,K)})$ against any
	\[ \xi \in \Delta^!_X(\cF)(V)^{p,q} = \bigvee \Hom_{\cF(X)^e}(\cF(X)(-,z_0) \otimes \dots \otimes \cF(V)(z_p,w_0) \otimes \dots \otimes \cF(X)(w_q,-), \cF(X)^e) \]
	to get an element in
	\begin{align*}
		\bigvee \cofib&\left(\cF(X)(-,x_0) \otimes \dots \otimes \cF(X \setminus K)(x_m,y_0) \otimes \dots \cF(X)(y_n,-) \right. \\
		&\left.\to \cF(X)(-,x_0) \otimes \dots \otimes \cF(X)(x_m,y_0) \otimes \dots \cF(X)(y_n,-)\right)  \cong \DDelta_X(\cF)_{m,n}
	\end{align*}
	giving the desired map of bar complexes.

	We can take $V=X$ for $K_1$ and $V=U$ for $K_2$; given an section $s_{(X,K_1)}$ as above, the map induced on sections by the corresponding map of $\Compacts$-sheaves $\theta\Delta^!_X(\cF) \to \DDelta_X(\cF)$ on the pair $(K_1,K_2)$ is the data of a homotopy-commuting square
	\[\xymatrix{
		\cF(X)^! \ar[d] \ar[r] & I_!\cF(U)^! \ar[d] \\
		\cofib\left(I_!\cF(U)_\Delta \to \cF(X)_\Delta\right) \ar[r] & I_!\cF(U)_\Delta[1]
	}\] \\
\end{example}

We can use the maps described above to define the nondegeneracy condition on orientations on \textsc{ccc} spaces that we mentioned in the Introduction.
\begin{definition}
	A map $\Theta:\varpi_X[d] \to \cH\cH(\cF)^{S^1}$ is an orientation on $(X,\cF)$ if its image $\eta_X(\Theta) \in \Hom_{\Sh_\cK(X,\cF(X)^e-\Mod)}(\varphi\theta\Delta^!_X(\cF), \varphi\DDelta_X(\cF)[-d])$ is an isomorphism of $\cF(X)$-bimodules valued $\Compacts$-sheaves.
\end{definition}

Analogously to the sheaf case, a morphism of cosheaves can be checked to be an isomorphism locally, so we have:
\begin{proposition}\label{prop:cosheafLocality}
	If $X = \bigcup_i U_i$ and $\Theta:\varpi_X[d] \to \cH\cH(\cF)^{S^1}$ is such that $\Theta|_{U_i}$ is an orientation on $(U_i,\cF|_{U_i})$, then $\Theta$ is  an orientation on $(X,\cF)$. 
\end{proposition}
\begin{proof}
	Follows from the fact that a morphism of $\Compacts$-sheaves can be checked to be an isomorphism locally, together with the functoriality of the map $\eta_X$ taking the data of $\Omega:\varpi_X[d] \to \cH\cH(\cF)$ to a morphism of $\Compacts$-sheaves (Proposition \ref{prop:etaFunctoriality}).
\end{proof}

\subsubsection{Relation to smooth CY structures}
In a dual manner to the case of sheaves of proper dg categories, an orientation on a cosheaf of smooth dg categories induces smooth Calabi-Yau structures. The following proposition follows directly from the definition.
\begin{proposition}\label{prop:cccCYpoint}
	If $X$ is a point, then an orientation on the \textsc{ccc} space $(X,\cF)$ is a smooth Calabi-Yau structure on the category $\cF$, in the sense of Definition \ref{def:absoluteCYstructures}.
\end{proposition}

\begin{proposition}\label{prop:cccCYinterval}
	If $X = [0,1)$ with stratification $\{0\} \cup (0,1)$, then an orientation on $(X,\cF)$ is a relative smooth Calabi-Yau structure on the corestriction functor $\iota: \cF((0,1)) \to \cF([0,1))$, in the sense of Definition \ref{def:relativeCYstructures}.
\end{proposition}
\begin{proof}
	Follows from comparing the description in Example \ref{ex:openIntervalCosheaf}, of the induced map of distinguished triangles in the category of $\cF(X)$-valued $\Compacts$-sheaves, to the definition of relative CY structure.
\end{proof}

\subsubsection{Pushforward}
We now prove the pushforward result stated in Theorem \ref{thm:cccCYpushforward}. Consider $f: X \to Y$ a proper and constructible map from a \textsc{ccc} space $(X,\cF)$. We have a cosheaf $f_*\cF$ giving $Y$ the structure of a \textsc{ccc} space; moreover, by the universal property of cosheafification, there is a canonical map $ f_* \cH \cH (\cF) \to \cH \cH(f_* \cF)$. We consider then the composition
\[ \varpi_Y  \to  f_* \varpi_X \xrightarrow{\Theta}   f_* \cH \cH (\cF)^{S^1} \to \cH \cH(f_* \cF)^{S^1}, \]
which we call $f_*\Omega$. 

\begin{proposition}\label{prop:cccCYpushforwardText}
	If $\Theta$ is an orientation on $(X,\cF)$ then $f_*\Theta$ is an orientation on $(Y,f_*\cF)$.
\end{proposition}
\begin{proof}
	We note that since $f$ is proper, the Verdier duality functors we use intertwine the pushforwards $f_*$ (on usual and $\Compacts$-sheaves). The result then follows from the locality statement Proposition \ref{prop:cosheafLocality}; we can check nondegeneracy on each open of a cover $\bigcup_\alpha U_\alpha$ of $Y$,  which follows from nondegeneracy on the open cover $\bigcup_\alpha f^{-1}(U_\alpha)$ of $X$.
\end{proof}

Finally, analogously to Proposition \ref{prop:cscLocalCheck} for \textsc{csc} spaces, we can combine the two previous results to give a way to check nondegeneracy in terms of relative CY structures. Given a \textsc{ccc} space $(X,\cF)$, we take again the cover indexed over all strata $X= \bigcup_\sigma U_\sigma, U_\sigma = \Star(\sigma)$, and a small retract $K_\sigma$ of $U_\sigma$.
\begin{proposition}\label{prop:cccLocalCheck}
	A map $\Theta: \varpi_X[d] \to \cH\cH(\cF)^{S^1}$ is an orientation on $(X,\cF)$ if, for each $\sigma$, $\Theta|_{U_\sigma}$ induces a relative smooth Calabi-Yau structure on the corestriction $\cF(U_\sigma \setminus K_\sigma) \to \cF(U_\sigma)$.
\end{proposition}
\begin{proof}
	Entirely analogous to the proof of \ref{prop:cscLocalCheck}, using a stratified map to $[0,1)$ and applying Propositions \ref{prop:cccCYinterval} and \ref{prop:cccCYpushforwardText}.
\end{proof}

\subsection{Duality}
We now address the relation between orientations on a \textsc{ccc} space $(X,\cF)$, as defined above in this Section, and orientations on some associated \textsc{csc} space $(X,\cP)$, as defined in Section \ref{sec:orientationsProper}. This relation will be a sheafified version of the description in Section 3.2 of \cite{BD1}, which we paraphrase now.

Let $\cA$ is a smooth dg category, and let $\cP \subseteq \cA$ be any full dg subcategory spanned by locally perfect objects; this is a proper dg category. One can obtain such a category by taking any subcategory of $\cA^{pp}$, identified with a full subcategory of $\cA$. Then, $\cA_\Delta$ is a proper $(\cP,\cA)$-bimodule, and induces a pairing
\[ \langle -,- \rangle_{\cP,\cA}: \HH(\cP) \otimes \HH(\cA) \to k \]
as described in \cite[Sec.5.1]{Sher}, following \cite{Shk}. This is described by a composition
\[ \HH(\cP) \otimes \HH(\cA) \xrightarrow{\sim} \HH(\cP \otimes \cA^{op}) \xrightarrow{\wedge} \HH(\Perf) \xrightarrow{\int} k \]
The map $\wedge$ can be explicitly given on bar complexes by a shuffle product. The third is a trace map, providing an inverse to the isomorphism $k \overset{\sim}{\hookrightarrow} \HH(k)$. 

It is known that the pairing $\langle -,- \rangle_{\cP,\cA}$ can be lifted to a pairing $\langle \widetilde{-,-} \rangle_{\cP,\cA}$ on the negative cyclic chain complex \cite{Shk}, referred to by the name of `higher residue pairing', which by adjunction gives a map
\[ \langle \widetilde{-,-} \rangle_{\cP,\cA}^\sharp: \HH(\cA)^{S^1} \to \Hom_{k[[u]]}(\HH(\cP)^{S^1},k[[u]]) \cong \Hom(\HH(\cP)_{S^1},k), \]
On the other hand, the inclusion $\cP \subset \cA$ also gives a $k$-linear functor
\[ \mathrm{LD}_\cA: \cA\mh\Mod\mh\cA \to (\cP\mh\Mod\mh\cP)^{op}, \]
such that, for any $\cA$-bimodule $\cM$, there is an equivalence in $\vvt$
\[ \mathrm{LD}_\cA(\cM)(p,p') \cong \Hom_{\cA}(\cM(-,p),\cA_\Delta(-,p')) \]
Again, resolving $\cM$ as a left $\cA$-module using bar complexes, one can prove that up to equivalence this functor preserves diagonal bimodules and sends bimodule duals to linear duals. That is, there are equivalences of $\cP$-bimodules
\[ \mathrm{LD}_\cA(\cA_\Delta) \simeq \cP_\Delta, \qquad \mathrm{LD}(\cA^!) \simeq \cP^\vee. \]
Therefore, $\mathrm{LD}$ gives a morphism in $\vvt$
\[ \mathrm{LD}_\cA(\cA^!,\cA_\Delta): \HH(\cA) \simeq \Hom_{\cA^e}(\cA^!,\cA_\Delta) \to \Hom_{\cP^e}(\cP_\Delta,\cP^\vee) \simeq \Hom_k(\HH(\cP),k). \]
Since $\mathrm{LD}_\cA(\cA^!,\cA_\Delta)$ is the map on morphism spaces of bimodules induced by a functor, it preserves isomorphisms.
\begin{theorem}\cite[Thm.3.1]{BD1}
The adjoint map to the pairing $\langle-,-\rangle_{\cP,\cA}$ is equivalent to $\mathrm{LD}_\cA(\cA^!,\cA_\Delta)$, and therefore
\[ \langle \widetilde{-,-} \rangle_{\cP,\cA}^\sharp: \HH(\cA)^{S^1} \to \Hom(\HH(\cP)_{S^1},k) \]
maps smooth CY structures on $\cA$ to proper CY structures on $\cP$.
\end{theorem}

When $\cA$ is saturated, we can take $\cP = \cA^{pp} \simeq \cA$, and then $\LD$ is an anti-autoequivalence of $\cA\mh\Mod\mh\cA$, reducing to the result in \cite[Prop.6.10]{GPSher}.
\begin{corollary}
	If $\cA$ is saturated, the map $\langle \widetilde{-,-} \rangle_{\cA,\cA}^\sharp$ gives an bijection between (equivalence classes of) smooth and proper CY structures on $\cA$.
\end{corollary}

\begin{remark}
On the other hand, when $\cA$ is not saturated, in general there is no way of producing a smooth CY structure on it from a proper CY structure on its category of pseudo-perfect modules. We will see later in Theorem \ref{thm:dualTwo} that there is a way around this problem in the case where $\cA$ is the global sections of a locally saturated sheaf.
\end{remark}

\subsubsection{Duality for relative CY structures}
We will need a relative version of the duality statement above, which appears as a claim in \cite[Sec.4.2]{BD1}. Let $f:\cA \to \cB$ be a functor between smooth dg categories, with pullback map $f^{pp}:\cB^{pp} \to \cA^{pp}$ on the corresponding categories of pseudo-perfect modules. We pick spanning sets of objects for all these categories in order to write bar complexes; the pair $f,f^*$ then induces a morphism of distinguished triangles of complexes
\[\xymatrix{
	\HH(\cA) \ar[r]^{\HH(f)} \ar[d] & \HH(\cB) \ar[r] \ar[d] & \HH(\cB|\cA) \ar[d] \\
	\Hom(\HH(\cA^{pp}),k) \ar[r]^{(f^{pp})^\vee} & \Hom(\HH(\cB^{pp}),k) \ar[r] & \Hom(\HH(\cB^{pp}|\cA^{pp}),k)[1]
}\]
Analogously to the absolute case, one can lift these maps to maps between the negative cyclic and the dual cyclic complexes. The same argument works in the relative case, by applying the functor $\LD_\cB$; using the same bar complexes one can give isomorphisms of bimodules
\[ \mathrm{LD}_\cB(F_! \cA_\Delta) \simeq (F^{pp})^* (\cA^{pp})_\Delta, \quad \mathrm{LD}_\cB(F_! \cA^!) \simeq (F^{pp})^* (\cA^{pp})^\vee. \]
and then compare the map induced by $\LD_\cB$ on morphism spaces with the map induced on Hochschild homologies, which gives
\begin{proposition}\label{prop:relativeProperToSmooth}
	The map $\HH(\cB|\cA)^{S^1}[-d] \to \Hom(\HH(\cA^{pp}|\cB^{pp})_{S^1},k)[-d+1]$ maps $d$-dimensional relative smooth Calabi-Yau structures on $f$ to $d$-dimensional relative proper Calabi-Yau structures on $f^{pp}$.
\end{proposition}

We note now that $\LD_\cB$ is an anti-autoequivalence of $\cB^e\mh\Mod$ when $\cB$ is saturated. Therefore, regardless if $\cA$ is proper or not, we have
\begin{corollary}\label{cor:relativeSmoothToProper}
	If $f:\cA\to\cB$ with $\cA$ smooth and $\cB$ saturated, if the image of $\Omega$ under the map $\HH(\cB|\cA)^{S^1}[-d] \to \Hom(\HH(\cA^{pp}|\cB^{pp})_{S^1},k)[-d+1]$ is a $d$-dimensional relative proper CY structure on $f^{pp}$, then $\Omega$ is a $d$-dimensional relative smooth CY structure on $f$.
\end{corollary}

\begin{remark}
	Note that nothing forces the map above to give a bijection between classes of relative smooth CY structures on $f$ and \emph{all classes of} relative proper CY structures on $f^{pp}$, only to the ones in the image of the map.
\end{remark}

\subsubsection{Orientations on sheaves of pseudo-perfect modules}\label{sec:smoothtoproper}
Now, let us take a \textsc{ccc} space $(X,\cF)$, such that $\cF$ is a cosheaf of smooth dg categories. Taking pseudo-perfect modules gives a sheaf $\cF^{pp}$ of proper dg categories. Recall that the data of an orientation on $\cF$ is a map of cosheaves $\varpi_X \to \cH\cH(\cF)^{S^1}$, and the data of an orientation on $\cP$ is a map of sheaves $\cH\cH(\cP)_{S^1} \to \omega_X$. Recall also that since the classical Verdier duality functor $\D'_X = (-)^\vee \circ \D_X$ is an anti-autoequivalence on sheaves, $\varpi_X$ and $\omega_X$ are linear duals of each other.

We will now repeat the exact same strategy used in the duality between smooth and proper CY structures, but with orientations on (co)sheaves instead. Let us pick spanning sets of objects for the cosheaf $\cF$ and the sheaf $\cP$; this allows us to write representatives for the Hochschild homology pre(co)sheaves using bar complexes of the precosheaves $\cF|_U^{x,y}$ on one side, and of the sheaves $\cF^{pp}|_U^{x,y}$ on the other.

Writing the same formula for the adjoint of the higher residue pairing map as in the case of $\cP \subset \cA^{pp} \subset \cA$, but replacing the factors of $\cA(x,y)$ and $\cP(x,y)$ by the objects above gives us a map in $\PreSh(X,\vvt)$
\[ \mathrm{DL}:\HH(\cP)_{S^1} \to \Hom(\HH(\cA)^{S^1},k) \]
Note that this is the adjoint map in the opposite direction from a sheafy version of the $\LD$ map we wrote before. Since linear dual sends colimits in $\vvt$ to limits, by universal property of the co/sheafification functors we get a morphism
\[ \cD\cL: \cH\cH(\cP)_{S^1} \to \Hom(\cH\cH(\cA)^{S^1},k) \]
in $\Sh(X,\vvt)$. Thus, given a map of cosheaves $\Theta: \varpi_X[d] \to \cH\cH(\cA)^{S^1}$, we can take its $k[[u]]$-linear dual to get a map of sheaves $\Omega = \Theta^\vee: \Hom(\cH\cH(\cA)^{S^1},k) \to \omega_X[-d]$.

We again take the cover $X = \bigcup_\sigma U_\sigma, U_\sigma = \Star(\sigma)$ indexed by the set of strata.
\begin{proposition}\label{prop:dualOne}
	If $\Theta|_{U_\sigma}$ is an orientation on each \textsc{ccc} space $(U_\sigma,\cF|_{U_\sigma})$, then $\cD\cL \circ \Theta^\vee$ is an orientation on the \textsc{csc} space $(X,\cP)$.
\end{proposition}
\begin{proof}
	By the locality result (Proposition \ref{prop:sheafLocality}) it is enough to check on any open cover. By Proposition \ref{prop:cccLocalCheck}, on each $U_\sigma$, $\Omega$ gives a relative CY structure on the corestriction $f:\cF(U_\sigma \setminus K_\sigma) \to \cF(U_\sigma)$ where $K_\sigma$ is a small compact retract of the star $U_\sigma$. The desired result then follows from applying Propositions \ref{prop:relativeProperToSmooth} and \ref{prop:cscLocalCheck}.
\end{proof}

\begin{remark}
	The proof above shows that the data of an orientation on a smooth cosheaf always gives rise to an orientation on its sheaf of pseudo-perfect modules. 
	However, this is of limited utility in practice. Despite their apparent symmetry, the nondegeneracy condition for orientations on smooth cosheaves is hard to check.
\end{remark}

\subsubsection{The case of locally saturated cosheaves} \label{sec:propertosmooth}
Let us now add the assumption that the cosheaf $\cF$ is locally saturated. With this extra assumption, we will prove a converse statement to Proposition \ref{prop:dualOne}. For this we will need some facts about the Hochschild/cyclic homology (co)sheaves associated to $\cF$. Consider the Hochschild homology cosheaf $  \cH\cH(\cF)$, the negative cyclic homology cosheaf $  \cH\cH(\cF)^{S^1}$ and the cyclic homology cosheaf $  \cH\cH(\cF)_{S^1}$. A priori these are all cosheaves valued in $\vvt$. Recall now the subcategories $\vvt^b_+,\vvt^b_-$ of $\vvt$, to which classical Verdier duality extends.
\begin{lemma}\label{lemma:bounded}
If the cosheaf $\cF$ is locally saturated, then the cosheaf $  \cH\cH(\cF)$ (resp. $  \cH\cH(\cF)^{S^1}$, $\cH\cH(\cF)_{S^1}$) is valued in $\Perf$ (resp. $\vvt^b_+, \vvt^b_-$). Moreover, the map $\cD\cL$ is an equivalence, exhibiting the sheaf $\cH\cH(\cF^{pp})$ (resp. $\cH\cH(\cF^{pp})_{S^1}$) as the linear dual of the cosheaf $\cH\cH(\cF)$ (resp. $\cH\cH(\cF)^{S^1}$).
\end{lemma}
\begin{proof}
Let $\cA$ be a saturated dg category. Then we have $\HH(\cA) \in \Perf$, and $\HH(\cA)^{S^1} \in \vvt^b_+$ and $\HH(\cA)_{S^1} \in \vvt^b_-$; this follows from the fact that they can be computed respectively by the negative cyclic and cyclic bicomplexes (for an exposition in the context of $\infty$-categories see \cite{Hoy}). Note that the directions are opposite from the usual complexes in the literature since we use cohomological grading. The desired statements then follow from finiteness of the stratification and from the closedness of the relevant subcategories of $\vvt$ under finite colimits.
\end{proof}

We now pick an inverse $\cD\cL^{-1}: \Hom(\cH\cH(\cF)^{S^1},k) \xrightarrow{\sim} \cH\cH(\cF^{pp})$. The same argument of the proof of Proposition \ref{prop:dualOne}, together with Corollary \ref{cor:relativeSmoothToProper}, give the following theorem.
\begin{theorem}\label{thm:dualTwo}
	If $\Omega: \cH\cH(\cP)_{S^1}\to \omega_X[-d]$ is an orientation on each \textsc{csc} space $(U_i,\cF^{pp}|_{U_i})$, then
	\[ (\Omega \circ \cD\cL^{-1})^\vee: \varpi_X \to \cH\cH(\cF)^{S^1}[-d] \]
	is an orientation on the \textsc{ccc} space $(X,\cF)$.
\end{theorem}

\begin{remark}
	The global sections $\cF(X)$ will in general be a smooth but not proper dg category, and the proposition above gives a way of checking the condition of being smooth CY structures purely in terms of usual Verdier duality, and on the sheaf of proper dg categories $\cF^{pp}$.
\end{remark}

\subsection{Quotients of cosheaves}\label{sec:quotients}
In this section we describe a construction whose purpose is to construct Calabi-Yau structures on the sort of quotients which arise from the sheaf theoretic version of 
``stop removal'' \cite{GPS2, GPS3}. 

Let $(X,\cF)$ be a \textsc{ccc} space and let $Y \subseteq X$ be an \emph{open} subset given by a union of strata. Let us denote by $j:Y \hookrightarrow X$ the inclusion map. Using the equivalence $\Cosh(X,\dgcat)^{op} \cong \Sh(X,\dgcat^{op})$, we regard $\cF$ as a $\dgcat^{op}$-valued sheaf on $X$. There are pullback and pushforward functors $j^* \dashv j_*$ for sheaves valued in any $\infty$-category \cite[Sec.7.3]{LurHTT}, in particular for the $\infty$-category $\dgcat^{op}$.

Now let $\cG \to \cF \to j_* j^* \cF$ be a \emph{fiber} sequence in $\Sh(X,\dgcat^{op})$, where $\cF \to j_* j^* \cF$ is the unit map of the adjunction. Using the anti-equivalence above this is equivalent to a \emph{cofiber} sequence $j_* j^* \cF \to \cF \to \cG$ in $\Cosh(X,\dgcat)$. Let us denote $\cF_{X/Y} = \cG$ to be the cofiber object; this is a $\dgcat^{op}$-valued constructible sheaf, that is, a constructible cosheaf of dg categories.

\begin{remark}
	This construction is best explained as an example of \emph{recollement} of $\infty$-categories of constructible sheaves, as described in \cite[Sec.A.8 and Sec.A.9]{LurHA}. If we set $Z = X \setminus Y$ to be the complement with a closed inclusion $i: Z \hookrightarrow X$, then the situation above describes a recollement of $\Sh(X,\cC)$ from the pair $\Sh(Z,\cC), \Sh(Y,\cC)$. Assuming that the pushforward $i_*$ has a further \emph{right adjoint} $i^!$, then from the recollement we get a fiber sequence
	\[ i_* i^! \cF \to \cF \to j_* j^* \cF \]
	and therefore can construct the fiber $\cF_{X/Y}$ as $i_* i^! \cF$. This is the case when, for instance, the coefficient category $\cC$ is a stable $\infty$-category \cite[Rem.A.8.19]{LurHA}. However, note that in our case $\cC = \dgcat^{op}$ is not a stable category. We do not know if the adjoint $i^!$ exists in this context; if it does then we can use it to set $\cG = i_* i^! \cF$. But the existence of this adjoint does not ultimately matter for us; we just set $\cG$ to be the fiber of the unit map, and will not need anything more from the general theory of recollements.
\end{remark}

Let us now discuss the cofiber of a fully faithful functor between saturated categories.
\begin{lemma}
	Let $\cA \xrightarrow{f} \cB \xrightarrow{g} \cC$ be a cofiber sequence in $\dgcat$, where $\cA$ and $\cB$ are saturated and $f$ is fully faithful. Then $\cC$ is also saturated.
\end{lemma}
\begin{proof}
	This seems to be well-known, but let us include the proof for completeness. The conditions of this lemma mean that we have an \emph{exact sequence} as defined in \cite[Def.5.1]{HSS}, following \cite{BGT}. By \cite[Prop.5.15]{BGT}, this implies that, at the level of homotopy categories, the cofiber $\cC$ is (the idempotent completion of) the Verdier quotient $\cB/\cA$. As smoothness is inherited by quotients \cite{Efi}, we know that $\cC$ is automatically smooth. Consider now the ind-completion of this sequence
	\[ \Ind(\cA) \xrightarrow{F} \Ind(\cB) \xrightarrow{G} \Ind(\cC), \]
	which is an exact sequence in $\DGCat$. Since $F$ and $G$ preserve compact objects, they have continuous right-adjoints $F^r,G^r$. By the argument in the proof of \cite[Prop.5.4]{HSS}, we know that $G^r$ is fully faithful and the null sequence
	\[ F F^r \to \id \to G^r G \]
	is a cofiber sequence. Now, since $\cA$ and $\cB$ are saturated, $F^r$ restricted to $\cB \cong \cB^{pp}$ lands in $\cA \cong \cA^{pp}$ and therefore preserves compact objects. Since compact objects are preserved by colimits, the have that $G^r G$ preserves compact objects as well, and since $G$ is essentially surjective the same goes for $G^r$. Therefore $g$ has a fully faithful right adjoint $g^r$ exhibiting $\cC$ as a full subcategory of the proper dg category $\cB$, proving properness. 
\end{proof}

Now as above take the cofiber sequence $j_* j^* \cF \to \cF \to \cF_{X/Y}$ in $\Cosh(X,\dgcat)$.

\begin{lemma} \label{lem:quotcosheaf}
	Suppose that $\cF$ is locally saturated, and that for any small enough open $U$ such that $\cF(U)$ is saturated, $\cF(U\ \cap Y)$ is also locally saturated, and the corestriction map
	\[ \cF(U\ \cap Y) \to \cF(U) \]
	is fully faithful. Then $\cF_{X/Y}$ is locally saturated.
\end{lemma}
\begin{proof}
	Follows from the lemma above, together with the observation that there is an equivalence $j_* j^* \cF(U) \cong \cF(U \cap Y)$.
\end{proof}

Let us now pass to the \textsc{csc} spaces obtained by taking pseudo-perfect modules. By the lemma above we know that $(X,\cF_{X/Y}^{pp})$ is also a \textsc{csc} space. Moreover, we have a morphism of sheaves
\[ \cH\cH(\cF_{X/Y}^{pp})_{S^1} \to \cH\cH(\cF^{pp})_{S^1}. \]
Given a map $\Omega: \cH\cH(\cF^{pp})_{S^1} \to \omega_X[-d]$ we can pull it back along the morphism above to a map $\Omega'$ from $\cH\cH(\cF_{X/Y}^{pp})_{S^1}$.

\begin{lemma}\label{lem: orientation on quotient}
	If $\Omega$ is an orientation on the \text{csc} space $(X,\cF^{pp})$, then $\Omega'$ is an orientation on the \text{csc} space $(X, \cF_{X/Y}^{pp})$
\end{lemma}
\begin{proof}
	Let $U \subset X$ be any open, and $a,b$ a pair of objects in $\cF_{X/Y}^{pp}(U)$. Consider the maps of sheaves
	\begin{align*}
		(\cF^{pp})|^{a,b}_U &\to \sheafHom((\cF^{pp})|^{b,a}_U, \omega_X[-d]) \\
		(\cF^{pp}_{X/Y})|^{a,b}_U &\to \sheafHom((\cF^{pp}_{X/Y})|^{b,a}_U, \omega_X[-d])
	\end{align*}
	induced by $\Omega$ and $\Omega'$, respectively. The fully faithful morphism $f: \cF_{X/Y}^{pp} \to \cF^{pp}$ in $\Sh(X,\dgcat)$ induces isomorphisms $(\cF_{X/Y}^{pp})|^{a,b}_U \xrightarrow{\simeq} (\cF^{pp})|^{a,b}_U$ in $\Sh(X, \Perf)$, intertwining the maps above, so nondegeneracy of $\Omega$ implies nondegeneracy of $\Omega'$.
\end{proof}

\section{Arboreal spaces} \label{sec:arb}
In \cite{N3, N4}, Nadler introduced a class of Legendrian singularities, termed `arboreal' because they were indexed by trees.  Microlocal sheaf
theory equips these with the structure of \textsc{ccc} spaces, and Nadler showed that 
(1) the cosheaf of dg categories is explicitly computable \cite{N3} and (2) every Legendrian singularity admits a deformation to a space
with only arboreal singularities, `noncharacteristic' in the sense of preserving global sections of the Kashiwara-Schapira sheaf.  

In this section we recall Nadler's `combinatorial' construction of these spaces (we rewrite \cite[Sec. 2.2]{N3} with more pictures), and 
give a corresponding combinatorial / representation-theoretic construction of their (co)sheaves of categories (roughly speaking, turning 
the right hand side of \cite[Prop. 4.6]{N3} into a definition). 

The main purpose is to fix notation and formulate the Definitions \ref{def:arborealSpace} and \ref{def:genArborealSpace}: an {\em arboreal space} is 
a pair $(\mathbb{X}, W)$ of a space $\mathbb{X}$ and a cosheaf of dg categories $\mathcal{W}$, 
locally modeled on Nadler's arboreal singularities. 

\subsection{Trees and arboreal singularities}
\subsubsection{Trees and correspondences} 
A tree is a connected acyclic graph.  A correspondence of trees $\mf p$ is a diagram 
\[ R \overset{q}\twoheadleftarrow S \overset{i}\hookrightarrow T \]
where $R, P, T$ are trees, and the maps $q$, $i$ are respectively surjective and injective maps of graphs.

For us, an injective map of graphs is given by an injection between the sets of vertices and an injection between sets of edges, respecting incidence; note this is more restrictive than an injective map of topological spaces. A surjective map of graphs is defined by the contraction of some subset of edges. An isomorphism of correspondences 
$ R \overset{q}\twoheadleftarrow S \overset{i}\hookrightarrow T $
and 
$R' \overset{q'}\twoheadleftarrow S' \overset{i'}\hookrightarrow T'$ 
is the data of isomorphisms $R \simeq R', S \simeq S', T \simeq T'$ which intertwine the maps. 

Given trees and maps $Q \hookrightarrow R \twoheadleftarrow S$, the
fiber product graph $Q \times_R S$ is the subtree of $S$ whose vertices
map to the image of $Q$.  
Thus correspondences can be composed: 
\[ (P \twoheadleftarrow Q \hookrightarrow R)  \circ 
 (R \twoheadleftarrow S \hookrightarrow T) = 
  (P \twoheadleftarrow Q\times_R S \hookrightarrow T). \]

\begin{definition}\label{def:arb}
The category $\Arb$ has objects given by trees, and morphisms given by correspondences of trees:
\[ \Hom_{\Arb}(T, R) := \{ \mf{p}\ |\ \mf{p} = (R \overset{q}\twoheadleftarrow S \overset{i}\hookrightarrow T) \} \]
with composition of morphisms given by composition of correspondences.
\end{definition}

\begin{definition} \label{def:arbT} 
Let $T$ be a tree. We write $\Arb_T$ for the category whose objects are given by correspondences $\mf{p} = (R \overset{q}\twoheadleftarrow P \overset{i}\hookrightarrow T)$ (where the rightmost tree is fixed to be $T$) and whose morphisms are given by
\[ \Hom_{\Arb_T}(\mf{p},\mf{p}') = \{ \mf{q}\ |\ \mf{p}' = \mf{q} \circ \mf{p} \}, \]
with composition induced by composition in $\Arb$.
\end{definition}
Note that the category $\Arb_T$ is equivalent to the undercategory of $\Arb$ on the object $T$.

\begin{lemma}
For any tree $T$, $\Arb_T$ is a poset.
\end{lemma}
\begin{proof} 
The reflexivity axiom is satisfied by the trivial correspondence; we also have transitivity by composition of correspondences. As for the anti-symmetry axiom, suppose that we have $\mf{p}' = \mf{q} \circ \mf{p}$ and $\mf{p}' = \mf{q}' \circ \mf{p}$. Then $\mf{q}\circ \mf{p}$ is a correspondence of the form $(T \twoheadleftarrow - \hookrightarrow T)$ which is necessarily the trivial correspondence, implying that $\mf{q}$ and $\mf{q}'$ are also trivial. It remains to check that for any $\mf{p}, \mf{p}'$, there is at most one element in $\Hom(\mf{p}, \mf{p}')$.
Suppose
$ (P \twoheadleftarrow Q \hookrightarrow R)  \circ 
 (R \twoheadleftarrow S \hookrightarrow T) = 
  (P \twoheadleftarrow N \hookrightarrow T) $.
We want to reconstruct $(P \twoheadleftarrow Q \hookrightarrow R)$
from just $(R \twoheadleftarrow S \hookrightarrow T)$ and $(P \twoheadleftarrow N \hookrightarrow T)$.    

The map  $N = Q\times_R S \to S$ determined by taking the (necessarily unique)
factorization of $N \hookrightarrow T$ as 
$N \hookrightarrow S \hookrightarrow T$.  From this we can characterize $Q$ 
as the image of $N$ under the map $S \twoheadrightarrow R$.  The map $Q \twoheadrightarrow
P$ is determined by the (necessarily unique) factorization of $N \twoheadrightarrow P$
into $N \twoheadrightarrow Q \twoheadrightarrow P$. 
\end{proof} 

We will also be interested in rooted trees, i.e. trees with a distinguished vertex.  
We regard a rooted tree as a directed graph by directing all edges \emph{towards} the root, and denote it by $\vv T$. 
A correspondence $\vv R \twoheadleftarrow \vv S \hookrightarrow \vv T$ of rooted trees is a correspondence 
of the underlying trees such that the image of the root of $S$ is the root of $R$, and the image of the root of $S$ is 
the closest vertex in (the image of) $S$ to the root of $T$.  We write $\vv{\Arb}$ for the category of rooted trees with morphisms given by correspondences of rooted trees, and $\vv{\Arb}_{\vv T}$ for the undercategory of $\vv T$.
 
The following lemma follows from an elementary consideration of compatible rootings.
\begin{lemma}\label{lem: forgetful map}
The forgetful map ${\vv{\Arb}}_{\vv T} \to \Arb_T$ is an isomorphism.
\end{lemma}
\begin{proof}
Given a correspondence $R \twoheadleftarrow S \hookrightarrow T$ and a rooting $\vv T$ of $T$, there
is a unique choice of rootings $\vv{R}, \vv{S}$ so that the same maps 
define a rooted correspondence $\vv R \twoheadleftarrow \vv S \hookrightarrow \vv T$.
\end{proof}

\subsubsection{Arboreal singularities}\label{subsec:arbsing}
Recall that the nerve $\Nrv(\cC)$ of a category $\cC$ is the simplicial complex whose vertices 
are the objects of $C$, morphisms are the edges, triangles are commuting triangles
giving compositions, etc. We will denote $|\Nrv(\cC)|$ for its geometric realization. When $\cC$ is a poset, this is also called the order complex; if moreover $\cC$ is finite this is a complex with cells of bounded dimension.

\begin{definition} Let $T$ be a tree. We write $\overline{\T} = |\Nrv(\Arb_T)|$ for the order complex of the arboreal category. 
We write $\T$ for the union of simplices containing the initial correspondence $\mf{p}_T = (T \twoheadleftarrow T \hookleftarrow T)$, 
and $\T^\text{link} = \overline{\T} \setminus \T$ for the complement of this union.  
\end{definition}

The space $\T$ is conical; the initial object $\mf{p}_T \in \Arb_T$ gives the 
cone point and $\T^\mathrm{link}$ gives the link.  

\begin{example}
We write $A_2$ for the tree $\bullet - \bullet$.  We label the vertices $\alpha$ and $\beta$. 
There are four correspondences: the trivial correspondence
$\mf{p}_0 = (\alpha - \beta) \twoheadleftarrow (\alpha - \beta) \hookrightarrow (\alpha - \beta)$,
the correspondence $\bullet \twoheadleftarrow (\alpha - \beta) \hookrightarrow (\alpha - \beta)$, and two correspondences 
$\bullet \twoheadleftarrow \bullet \hookrightarrow (\alpha - \beta)$
for inclusions of $\alpha$ or $\beta$. 

We abbreviate these by enclosing in parenthesis those vertices of 
$A_2$ which get identified in the quotient $R \twoheadleftarrow S$.  
So, for example, we will denote
the three nontrivial correspondences by $\alpha, \beta, (\alpha\beta)$ 
and the trivial correspondence simply by $\alpha\beta$.

In the poset structure, the three nontrivial correspondences are incomparable, and the 
correspondence $\mf{p}_0$ is smaller than all of them.  Thus there are 7 strata in the order
complex $\overline{\A_2}$: the four 0-simplices $[\alpha\beta],[\alpha],[\beta],[(\alpha\beta)]$,
and three 1-simplices $[\alpha\beta \to \alpha],[\alpha\beta \to \beta],[\alpha\beta \to (\alpha\beta)]$. 
This can be realized as the following stratified space 
of dimension one (Figure \ref{fig:a2sing}). Note that $\A_2^{link}$ 
is the disjoint union of 3 points,
each labeled by a 0-simplex $[\mf q] \neq [\mf p_0]$.

\begin{figure}[h]
    \centering
    \includegraphics[width=0.40\textwidth]{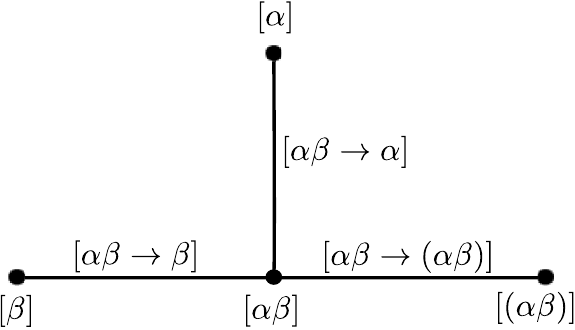}
         \caption{Arboreal singularity $\overline{\A_2}$. For simplicity we use the
     notation described above for each correspondence.}
    \label{fig:a2sing}
\end{figure}
\end{example}

\begin{example}
We write $A_3$ for the tree $\bullet - \bullet - \bullet$, whose vertices we label
$\alpha - \beta - \gamma$.  There are eleven correspondences: the trivial one,
four correspondences of the form $[\bullet \to \bullet] \twoheadleftarrow \dots$ and six of the form 
$[\bullet] \twoheadleftarrow \dots$. There are 45 strata in the order complex $\overline{\A_3}$:
11 zero-dimensional, 22 one-dimensional and 12 two-dimensional strata, assembled as in Figure \ref{fig:a3sing}. The link $\A_3^{link}$ can be glued out of copies of $\overline{\A_2}$. 

\begin{figure}[h]
    \centering
    \includegraphics[width=0.45\textwidth]{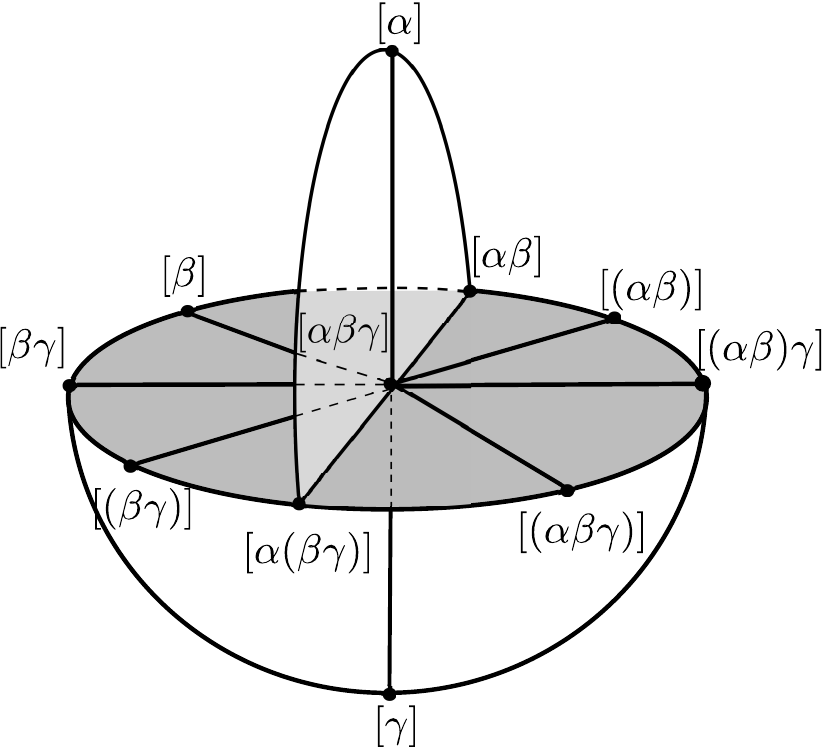}    
    \caption{Arboreal singularity $\overline{\A_3}$. This singularity is homeomorphic to a union of three 2-discs, along half-discs; in the figure above the gray disc is horizontal and the two white half-discs discs are glued to it along two perpendicular diameters, one to the top and another to the bottom. We only labeled the 0-simplices; the labels on all other simplices can be deduced from their vertices. The link $\A_3^\mathrm{{link}}$ can be seen to be homeomorphic to the 1-skeleton of a tetrahedron.}
    \label{fig:a3sing}
\end{figure}
\end{example}

\subsubsection{Tree quiver representations} 
We recall some standard facts about representations of tree quivers.  
These facts can be found in many sources, e.g. \cite{Hap, Par}. Let $\vv T$ be a finite quiver, i.e., a finite directed graph. We write $k[\vv T]$ for the path algebra of the quiver, whose generators are the vertices and arrows, subject to the relations that the vertex generators
are idempotent, and $f g = 0$ unless the head of $f$ is the tail of $g$ (and the head and tail of a vertex are itself). 
That is, we read paths {\em from left to right}.

Quiver representations correspond to 
{\em right modules} over this algebra; we write $\mmod(\vv T)$ for the (underived!) abelian category of such modules, 
$\Mod(\vv T)$ for the derived dg category of dg modules, and $\Perf(\vv T)$ for category of perfect dg modules, that is, the subcategory of perfect objects in $\Mod(\vv T)$. These are dg enhancements of the usual derived categories of algebra representations, and we can identify them as objects $\Mod(\vv T) \in \DGCat$, $\Perf(\vv T) \in \dgcat$, and indeed $\Mod(\vv T) = \Ind \Perf(\vv T)$. Henceforth we restrict attention to the case when the underlying graph of $\vv T$ is a tree. 
In this case the path algebra $k[\vv T]$ is smooth and proper, 
so the perfect and pseudoperfect modules coincide. 

For vertices $\alpha, \beta \in \vv{T}$, we write $\alpha \ge \beta$ when
there is a path \emph{from} $\alpha$ \emph{to} $\beta$, and we denote this unique path by $|\alpha\rangle \langle \beta|$. 
These compose in the usual way, $|\alpha \rangle \langle \beta | |\beta \rangle \langle \gamma| = |\alpha \rangle \langle \gamma|$, and all other compositions vanish. We are particularly interested in the case when the edge directions arise 
from the choice of a fixed root vertex of $T$, by directing all the edges 
\emph{toward} the root.  We pronounce $\alpha \ge \beta$ as ``alpha is above beta'' or ``beta is below alpha", 
so that everything is above the root. For any vertex $\alpha$, we denote $\hat\alpha$ to be the unique vertex such that there is an arrow of length one from $\alpha \to \hat\alpha$, i.e., the unique vertex directly below $\alpha$.

We write $P_\alpha :=|\alpha \rangle \langle \alpha| k[\vv T]$ for the right
module of ``paths from $\alpha$''.  All paths must have a beginning, and so: 

\begin{lemma}
There is an isomorphism of right $\vv T$ modules 
$k[\vv T] = \bigoplus_\alpha P_\alpha$, and 
$\Hom_{\mmod(\vv T)}(P_\alpha, P_\beta)$ vanishes
unless $\alpha \le \beta$, in which case it is one dimensional, generated
by composition on the left with the unique path  $ |\beta \rangle \langle \alpha|$.  

In particular, the $P_\alpha$ are the indecomposable projectives of $\mmod(\vv T)$, 
and in particular their Homs are calculated by the same formula in $\Mod(\vv T)$. 
\end{lemma}

\begin{example}
For the quiver $1 \to 2 \to 3 \to \cdots \to n$, the path algebra can be identified with an algebra of 
 triangular $n \times n$ matrices.  The matrix $| i \rangle \langle j |$ corresponds to the unique path
 from the $i$th vertex to the $j$th vertex.  The composition 
 $|i \rangle \langle j | |j \rangle \langle k | = |i \rangle \langle k|$ corresponds to the left-to-right composition rule
 $(i \to j)(j \to k) = (i \to k)$. 
\end{example}

We recall that quivers corresponding to the same underlying tree but with different arrow orientations have representation categories related by reflection
functors, defined in \cite{BGP}.  A source (sink) is a vertex that only has outgoing (incoming) 
arrows. Given a source $\alpha$, let $s_\alpha \vv T$ be the quiver obtained by reversing all the arrows at $\alpha$. There is 
a reflection functor $R^+_\alpha:\mmod(\vv T) \to \mmod(s_\alpha \vv T)$, which in fact preserves compact objects and so
induces a dg derived equivalence $\Perf(\vv T) \to \Perf(s_\alpha \vv T)$.  Likewise, at sinks there are similar reflection functors $R^-_\alpha$. The quiver structure for a rooted tree has all arrows pointing to the root. 
Because the underlying graph is acyclic,  two such structures corresponding to different roots $\rho_1,\rho_2$ 
can be related by a sequence of moves $s_\alpha$.  Thus the derived categories $\Mod(\vv T)$ and $\Perf(\vv T)$
depends only on the underlying tree (up to non-canonical equivalence).  Choosing a root determines a $t$-structure,
and thus the distinguished set of projective generators $\{P_\alpha \}$.

\subsubsection{A representation of Arb in categories of quiver representations}
We now construct functors from our arboreal categories into the category of dg categories, which map each rooted tree $\vec T$ to the representation category of its path algebra. Functors between such representation categories are defined by (derived) tensor product with bimodules; in order to explicitly define those we use a model structure induced from the model structure on chain complexes.

Let us denote by $\Ch$ the category of (unbounded) chain complexes of $k$-vector spaces. This category has a combinatorial monoidal model structure (see e.g. \cite[Sec.1.3.5]{LurHA}) for which every object is cofibrant; the $\infty$-category $\vvt$ is a localization. Let $A,B$ be two dg algebras; and let us denote by $A\mh\mmod\mh B(\Ch)$ the (ordinary) category of $(A,B)$-bimodule objects of $\Ch$. We now paraphrase \cite[Prop.4.3.3.15 and Thm.4.3.3.17]{LurHA} in our context:
\begin{proposition}
	The category $A\mh\mmod\mh B(\Ch)$ inherits a natural combinatorial model structure from $\Ch$, for which there is a natural equivalence of $\infty$-categories
	\[ \Nrv(A\mh\mmod\mh B(\Ch))[W^{-1}] \xrightarrow{\sim} A\mh\Mod\mh B \]
	to the $\infty$-category $A\mh\Mod\mh B$ of $A,B$-bimodules valued in $\vvt$.
\end{proposition}
Therefore, we can describe objects of $A\mh\Mod\mh B$ and equivalences between them by dg bimodules and bimodule maps which are weak equivalences of chain complexes.

Let $\mf{p} = (\vv R \overset{q}\twoheadleftarrow \vv S \overset{i}\hookrightarrow \vv T)$ be a correspondence of rooted trees. We now associate some bimodules to the maps in this correspondence; as vector spaces, each of them will be identified with a subspace of the path algebra $k[\vv T]$ (seen simply as a vector space). We will denote them by $k[\dots]$ for some collection of paths in $\vv T$. For ease of notation we also identify $\vv S$ with its image in $\vv T$.
\begin{definition}
	To the inclusion $i: \vv S \hookrightarrow \vv T$ we associate $(k[\vv T], k[\vv S])$-bimodules $M_i^0,M_i^{-1}$ given by
	\begin{align*}
	M_i^{-1} &= k\left[ |\alpha\rangle\langle\beta|\ \middle|\ \alpha > \vec S, \beta \in \vv S \right], \\
	M_i^{0} &= k\left[ |\alpha\rangle\langle\beta|\ \middle|\ \beta \in \vv S \right],
	\end{align*}
	where the left $k[\vv T]$ and right $k[\vv S]$-actions are given by composition of paths. The natural inclusion $M_i^{-1} \subset M_i^{0}$ commutes with these actions and defines a dg bimodule $M_i = [M_i^{-1} \to M_i^0]$.
	
	To the quotient $q: \vv S \twoheadrightarrow \vv R$ let us denote by $\vv Q \subset \vv S$ the (possibly not connected) subquiver that gets contracted by $q$, and by $\tilde Q$ the subset of its vertices that are not roots of connected components of $\vv Q$. That is, $\tilde Q$ contains exactly the vertices $\beta \in \vv Q$ such that $\hat\beta \in \vv Q$. We then associate the space
	\[ M_q = k\left[ |\alpha\rangle\langle\beta|\ \middle|\ \beta \notin \tilde Q \right] \]
	and define a $(k[\vv S], k[\vv R])$-bimodule structure on it, where $k[\vv S]$ acts by composition on the left, and a path in $k[\vv R]$ acts by composing with its unique lift to $k[\vv S]$ that has the lowest starting point (this canonical lift always exists).
\end{definition}

Recall that we are working in the categories of algebras and bi(modules) over the category $\Ch$ of chain complexes with its model structure. We have the following statement (about ordinary categories).
\begin{lemma}\label{lem:bimodules}
	The bimodules $M^{0,-1}_i$ are projective over $k[\vv T]$ and the bimodule $M_q$ is projective over $k[\vv S]$. Moreover, the functors
	\begin{align*}
	- \otimes_{k[\vv T]} M_i &: \mmod\mh k[\vv T](\Ch) \to \mmod\mh k[\vv S](\Ch) \\
	- \otimes_{k[\vv S]} M_q &: \mmod\mh k[\vv S](\Ch) \to \mmod\mh k[\vv R](\Ch)
	\end{align*}
	have left adjoints.
\end{lemma}
\begin{proof}
	The first statement follows from the fact that each of the bimodules above can be identified with a direct sum whose components are projective (left) modules $P^\mathrm{left}_\alpha$ for some $\alpha$ (the left module of all paths ending at $\alpha$). For the second statement, we now define the left adjoint by giving explicit complexes of bimodules, together with unit and counit maps.
	
	Let us first deal with the inclusion map $i$. We define a dg $(k[\vv S], k[\vv T])$-bimodule $L_i = [L^0_i \to L^1_i]$, whose degree zero term is simply
	\[ L^0_i = k\left[ |\alpha\rangle\langle\beta|\ \middle|\ \alpha \in \vv S \right], \]
	with left and right actions given by composition. 
	
	The degree 1 term is more complicated; it is a vector space spanned by pairs
	\[ \left( |\hat\gamma\rangle\langle\alpha|, |\gamma\rangle\langle\beta| \right) \]
	of paths in (the image of) $\vv S$ and $\vv T$, respectively, where $\alpha \in \vv S$, $\gamma \in \vv T \setminus \vv S$ but its successor $\hat\gamma \in \vv S$ (that is, $\gamma$ is one vertex above $\vv S$). The right $k[\vv T]$ action is just given by composition in the second factor (appending on $\beta$) but the left $k[\vv S]$-action is given by the formula
	\[ |\alpha'\rangle\langle\alpha| \cdot \left( |\hat\gamma\rangle\langle\alpha|, |\gamma\rangle\langle\beta| \right)= \begin{cases} \left( |\hat\gamma\rangle\langle\alpha'|, |\gamma\rangle\langle\beta| \right),\ \text{if}\ \alpha' < \gamma \\ 0,\ \text{otherwise.}
	\end{cases} \]
	We define the map $L^0_i \to L^1_i$ by sending
	\[ |\alpha\rangle\langle\beta| \mapsto \sum_{\gamma \in \vv T \setminus \vv S\ \text{such that}\ \alpha \le \hat\gamma \in \vv S}  \left( |\hat\gamma\rangle\langle\alpha|, |\gamma\rangle\langle\beta| \right) \]
	and checking that this intertwines the left and right actions. 
	
	We now establish that $- \otimes_{k[\vv S]} L_i$ is left-adjoint to $- \otimes_{k[\vv T]} M_i$, by exhibiting the unit and counit maps, given by morphisms of bimodules $\eta: k[\vv S] \to L_i \otimes_{k[\vv T]} M_i$ and $\epsilon: M_i \otimes_{k[\vv S]} L_i \to k[\vv T]$. From the description above we see that $L^0_i \otimes_{k[\vv T]} M^0_i$ can be naturally identified with $k[\vv S]$, and that $L^1_i \otimes_{k[\vv T]} M^{-1}_i$ can be identified with the space spanned by pairs of paths
	\[ \left( |\hat\gamma\rangle\langle\alpha|, |\gamma\rangle\langle\beta| \right) \]
	as in $L^1_i$, but with both $\alpha, \beta \in \vv S$. With these identifications, we define the unit map
	\[ \eta: k[\vv S] \to \left( L^0_i \otimes_{k[\vv T]} M^0_i \right) \oplus \left( L^1_i \otimes_{k[\vv T]} M^{-1}_i \right) \]
	by sending
	\[ |\alpha\rangle\langle\beta| \mapsto |\alpha\rangle\langle\beta| + \sum_{\gamma \in \vv T \setminus \vv S\ \text{such that}\ \alpha \le \hat\gamma \in \vv S} \left(|\hat\gamma\rangle\langle\alpha|, |\gamma\rangle\langle\beta|\right). \]
	
	As for the counit map, we identify $M^0_i \otimes_{k[\vv S]} L^0_i$ with the space
	\[ k\left[ |\alpha\rangle\langle\beta|\ \middle|\ \alpha > \vv S\ \text{or}\ \alpha \in \vv S, \beta \in \vv T \right] \]
	and the space $M^{-1}_i \otimes_{k[\vv S]} L^1_i$ with the subspace of $k[\vv T]$ of paths with nontrivial intersection with $\vv S$; both of these spaces have natural maps to $k[\vv T]$ which moreover give maps of $k[\vv T]$-bimodules. It is then a straightforward computation to check that $\eta,\epsilon$ satisfy the triangle relations.
	
	We now do the same for the quotient map $q$. We define a $(k[\vv R],k[\vv S])$-bimodule $L_q$ by
	\[ L_q = k\left[ |\alpha\rangle\langle\beta|\ \middle|\ \alpha \notin \tilde Q \right], \]
	with actions defined in a similar way to the actions on $M_q$. We then have identifications
	\[ L_q \otimes_{k[\vv T]} M_q \simeq k\left[ |\alpha\rangle\langle\beta|\ \middle|\ \alpha, \beta \notin \tilde Q \right], \]
	that is, paths that can cross $\tilde Q$ but cannot end or start there, and
	\[ M_q \otimes_{k[\vv T]} L_q \simeq k\left[ |\alpha\rangle\langle\beta|\ \middle|\ |\alpha\rangle\langle\beta| \notin \mathrm{Image}(\tilde Q) \right], \]
	that is, paths that may start or end in $\tilde Q$, but cannot be contained in a single connected component of the subforest on the $\tilde Q$ vertices. The unit map is given by lifting a path in $k[\vv R]$ to its preimage with minimal start (which by construction cannot be in $\tilde Q$) and the counit map is inclusion into $k[\vv T]$. Again we can directly check the triangle relations.
\end{proof}

We now put these functors together; for a correspondence $\mf{p} = (\vv R \overset{q}\twoheadleftarrow \vv S \overset{i}\hookrightarrow \vv T)$ we define the $(k[\vv T],k[\vv R])$-bimodule $M_{\mf p} = M_i \otimes_{k[\vv S]} M_q$. By definition, it is a two-term complex of bimodules, where each term can be identified with a subspace of $k[\vv T]$. These identifications provides enough functoriality at the level of 2-morphisms, in the following sense.
\begin{lemma}\label{lem:compositions}
	Given correspondences $\mf{q} = \mf{p}' \circ \mf{p}$, there is an isomorphism of bimodules
	\[ \phi_{\mf{p}, \mf{p}'}: M_{\mf{p}' \circ \mf{p}} \to M_{\mf p} \otimes_{k[\vv R]} M_{\mf p'}. \]
	These maps satisfy a natural compatibility relation for compositions of three correspondences.
\end{lemma}
\begin{proof}
	Let $\mf{p} = (\vv R \overset{q}\twoheadleftarrow \vv S \overset{i}\hookrightarrow \vv T)$ and $\mf{p}' = (\vv R' \overset{q'}\twoheadleftarrow \vv S' \overset{i'}\hookrightarrow \vv R)$ be the two compositions, with subforests $Q$ and $Q'$ corresponding to the quotient maps $q,q'$, respectively. The composition $\mf{p}'' = \mf{p}' \circ \mf{p}$ is then given by $(\vv R' \overset{q''}\twoheadleftarrow \vv S' \times_{\vv R} \vv S \overset{i''}\hookrightarrow \vv T)$. We can identify $\vv S' \times_{\vv R} \vv S$ with a rooted subtree of $\vv T$ on the vertices of $S$ which map to the image of $i$; one can also see that the rooted subforest corresponding to $q''$ is given by
	\[ \vv Q'' = \left( \vv Q \cap (\vv S' \times_{\vv R} \vv S) \right) \cup \vv Q', \]
	where we identify everything with sub-trees or forests inside $\vv T$.
	
	Now for some path $|\alpha\rangle\langle\beta|$ in $\vv T$ we define conditions $C_1,C_2,C_3$ as follows: $|\alpha\rangle\langle\beta|$ satisfies $C_1$ if $\alpha > \vv S$ in $\vv T$, $C_2$ if there is $\gamma \in S \setminus (q')^{-1}(i'(S'))$ such that $\alpha > \gamma > \beta$ and $C_3$ if $\beta \notin \tilde Q''$ (as before, $\tilde Q''$ is the non-root vertices of $\vv Q''$). We can then calculate each term in
	\[  M_{\mf p} \otimes_{k[\vv R]} M_{\mf p'} = \left[ M^{-1}_{\mf p} \otimes_{k[\vv R]} M^{-1}_{\mf p'} \to M^{0}_{\mf p} \otimes_{k[\vv R]} M^{-1}_{\mf p'} \oplus M^{-1}_{\mf p} \otimes_{k[\vv R]} M^{0}_{\mf p'} \to M^{0}_{\mf p} \otimes_{k[\vv R]} M^{0}_{\mf p'}\right] \]
	as vector spaces spanned by paths in $\vv T$ following some subset of these conditions.
	\begin{align*}
		M^{-1}_{\mf p} \otimes_{k[\vv R]} M^{-1}_{\mf p'} &\simeq k \left[ |\alpha\rangle\langle\beta| \middle|\ C_1,C_2\ \text{and}\ C_3 \right] \\
		M^{-1}_{\mf p} \otimes_{k[\vv R]} M^{0}_{\mf p'} &\simeq k \left[ |\alpha\rangle\langle\beta| \middle|\ C_1\ \text{and}\ C_3 \right] \\
		M^{0}_{\mf p} \otimes_{k[\vv R]} M^{-1}_{\mf p'} &\simeq k \left[ |\alpha\rangle\langle\beta| \middle|\ C_2\ \text{and}\ C_3 \right] \\
		M^{0}_{\mf p} \otimes_{k[\vv R]} M^{0}_{\mf p'} &\simeq \left[ |\alpha\rangle\langle\beta| \middle|\ C_3 \right], \\
	\end{align*}
	with maps given by the inclusions, with a minus sign on the map $M^{-1}_{\mf p} \otimes_{k[\vv R]} M^{-1}_{\mf p'} \to M^{-1}_{\mf p} \otimes_{k[\vv R]} M^{0}_{\mf p'}$, due to the $-1$ shift.
	
	Comparing the conditions above with the rooted subtree $\vv S' \times_{\vv R} \vv S$ we see that there are identifications
	\begin{align*} M^{-1}_{\mf{p}' \circ \mf{p}} &\simeq k \left[ |\alpha\rangle\langle\beta| \middle|\ \text{not}\ C_1,\ \text{not}\ C_2, C_3 \right] \\
	M^{0}_{\mf{p}' \circ \mf{p}} &\simeq k \left[ |\alpha\rangle\langle\beta| \middle|\ C_3 \right].
	\end{align*}
	Using these identifications we define a map $M_{\mf{p}' \circ \mf{p}} \to M_{\mf p} \otimes_{k[\vv R]} M_{\mf p'}$ as identity in degree zero and diagonal embedding into $M^{0}_{\mf p} \otimes_{k[\vv R]} M^{-1}_{\mf p'} \oplus M^{-1}_{\mf p} \otimes_{k[\vv R]} M^{0}_{\mf p'}$ in degree -1. We check this gives an isomorphism of bimodules, and functoriality for three-term compositions follows from an explicit calculation, which basically holds due to the fact that we can identify all the corresponding vector spaces as subsets of the same path algebra $k[\vv T]$.
\end{proof}

We now put together Lemmas \ref{lem:bimodules} and \ref{lem:compositions} in the formalism of $\infty$-categories.
\begin{proposition}\label{prop:arborealfunctors}
	There is an $\infty$-functor
	\[ \cA: \Nrv(\vv{\Arb}) \to \dgcat \]
	that on objects and morphisms satisfies $\cA(\vv T) \simeq \Perf(\vv T)$ and $\cA(\mf p) \simeq - \otimes_{k[\vv T]} M_{\mf p}$. Moreover, it lands in the subcategory of $\dgcat$ where all morphisms are left-dualizable, that is, dg functors with (derived) left adjoints. Taking left adjoints gives another $\infty$-functor
	\[ \cA^\ell: \Nrv(\vv{\Arb}) \to \dgcat^{op}. \]
	We will call these functors $\cA$ and $\cA^\ell$ \emph{arboreal functors}.
\end{proposition}
\begin{proof}
	We will begin by constructing the desired map at the level of strict categories: by Lemma \ref{lem:compositions}, the bimodules we associated to correspondences give a pseudofunctor from the 1-category $\vv \Arb$ to the strict 2-category $\Alg^2(\Ch)$ whose objects are dg algebras, 1-morphisms are dg bimodules and 2-morphisms are morphisms of dg bimodules. Using the results from \cite[Sec.4.3.3]{LurHA} and \cite[Ch.11.B]{GR}, by localization along the morphisms of bimodules which are weak equivalences, we get a functor into an $(\infty,2)$-category $\Alg^2(\vvt)$ whose objects are algebra objects in $\vvt$, and 1-morphisms are bimodules; this category appears in \cite{Hau} and its maximal $(\infty,1)$-subcategory is $\mathrm{Morita}(\vvt)$ in the notation of \cite[Sec.4.8]{LurHA}.
	
	From the fact that the morphisms of bimodules $\phi_{\mf{p}, \mf{p}'}$ of Lemma \ref{lem:compositions} is a weak equivalence we get that our map restricts to an $\infty$-functor $\Nrv(\vv{\Arb}) \to \mathrm{Morita}(\vvt)$. The latter can be identified with the full subcategory of $\DGCat$ on the objects of the form $\Mod\mh\cA$ for some $\cA \in \Alg(\vvt)$. Recall that every bimodule in the complex $M_{\mf p}$ is projective as a $k[\vv T]$ module; the corresponding dg functor preserves perfectness, so our arboreal functor $\cA$ factors through $\cA: \Nrv(\vv{\Arb}) \to \dgcat$; moreover by Lemma \ref{lem:bimodules} it furthermore factors through the subcategory on the left-dualizable morphisms.
\end{proof}

Let us characterize the image of morphisms under the functor $\cA$, by looking at its action on the generating projectives.
\begin{lemma} \label{lem:projectivecalc}	
Consider a correspondence $\mf p: \vv R \overset{q} \twoheadleftarrow \vv S  \overset{i} \hookrightarrow \vv T$, 
and any vertex $\alpha \in T$.  Then there is an equivalence
\[ \cA(\mf p)(P_\alpha) \cong \begin{cases} P_{q(i^{-1}(\alpha))} & \qquad \alpha \in i(S)  \\ 0 & \qquad \text{otherwise.} \end{cases} \]
The morphism $\Hom_{\vv T}(P_\alpha, P_\beta) \to \Hom_{\vv R}(\cA(\mf p)(P_\alpha), \cA(\mf p)(P_\beta))$ sends 
$|\beta \rangle \langle \alpha| \to |q(i^{-1}(\beta)) \rangle \langle q(i^{-1}(\alpha))|$ (and hence is an isomorphism) when
these are defined; otherwise it is zero. 
\end{lemma}
\begin{proof}
	By definition, the dg functor $\cA(\mf p)$ is calculated by $- \otimes_{k[\vv T]} M_{\mf p}$, that is, tensoring with a two-term complex of bimodules that are projective as left $k[\vv T]$-modules; the statement above then follows from an explicit computation.
\end{proof}

\begin{remark}
Comparing the functor in Lemma \ref{lem:projectivecalc} to the description in \cite[Sec.4.2]{N3} shows that the (co)sheaf we have constructed here agrees
with the Nadler's calculation of the microlocal sheaf categories. 
\end{remark}

\subsection{Sheaves and cosheaves on an arboreal singularity}
\subsubsection{Constructible sheaves on simplicial complexes}
We briefly recall how to describe constructible sheaves on a simplicial complex.  
For a simplex $\sigma$ in a simplicial complex $X$, we write $\Star(\sigma)$ for the union of
open simplices whose closure contains $\sigma$.  To give a sheaf $\cF$ on $X$, constructible 
with respect to the stratification by simplices, it suffices to give the values of $\cF$ on 
the open sets $\Star(\sigma)$, and the corresponding restriction maps
$\cF(\Star(\sigma)) \to \cF(\Star(\tau))$ when $\Star(\tau) \subset \Star(\sigma)$, i.e., 
when $\sigma$ lies in the closure of $\tau$.  The appropriate diagrams should commute. 
Our definition of simplicial complex demands that the closure of an open simplex is a closed simplex, so 
there are no non-trivial overlaps, hence no descent conditions. 

The restriction
$\cF(\Star(\sigma)) \to \cF_\sigma$ is then necessarily an isomorphism, so one can instead
discuss `generization maps' $\cF_\sigma \to \cF_\tau$ when $\sigma$ lies in the closure of
$\tau$; this is the so-called `exit path' description of a constructible sheaf. 
A similar description works for any sufficiently fine stratification satisfying appropriate local contractibility
conditions. 

From a functor $F: \cX \to \cY$ we may extract the generization maps for a $\cY$-valued 
constructible sheaf $\Nrv(F)$ on the geometric realization of the nerve $|\Nrv(\cX)|$, as follows.  On objects, we set
\[ \Nrv(F) ([x_1 \to x_2 \to \cdots \to x_n])  = F(x_n) \]
and the generization maps are given by 
\[ \Nrv(F) ([x_{m} \to  \cdots \to x_{m'}] \to [x_1 \to x_2 \to \ldots \to x_n]) 
= F(x_m' \to x_{n}), \]
where the map $x_m' \to x_n$ comes from the fact that  $x_{m} \to  \cdots \to x_{m'}$
was a subsequence of  $x_1 \to x_1 \to \ldots \to x_n$.  
The fact that $F$ was a functor translates to the compatibility of the generization maps. 

Note however that in order for this data to define a sheaf valued in $\cY$, we must be able 
to say what values it takes on unions of strata.  These will be given by a limit.  In the setting
at hand, we will work with finite stratifications, so insist that $\cY$ be closed under finite limits. 

\begin{remark}
Though it is not necessary, one can rephrase the above in the language of exit-path categories, as developed in \cite[Sec.A.6]{LurHA}. For any $A$-stratified space $X$, one can associate an $\infty$-category $\Sing^A(X)$, its exit-path category, such that there is an equivalence
\[ \Fun(\Sing^A(X),\cC) \simeq \Sh^A(X,\cC) \]
for any $\infty$-category $\cC$ with small limits. In the case where $X = |\Nrv(\cX)|$ for some category $\cX$, with its natural stratification $A$, there is an equivalence $\Sing^A(|\Nrv(\cX)|) \simeq \Nrv(\cX)$ so any functor out of $\cX$  determines a constructible sheaf on the geometric realization of the nerve.  If the nerve is finite, then we do not need to require $\cC$ to have all small limits; it is enough to require existence of finite limits.
\end{remark}

\subsubsection{The arboreal sheaf and cosheaf}\label{subsection:nsheaf}
Let $T$ be a tree. Recall that we have a category $\Arb_T$ and a space $\overline{\T} := \Nrv(\Arb_T)$. Choosing a root determines a quiver $\vv T$, giving a category $\vv\Arb_{\vv T}$ which is equivalent to $\Arb_T$. Composing the functors of proposition \ref{prop:arborealfunctors} with the forgetful functor $\vv\Arb_{\vv T} \to \vv\Arb$, we have functors $\arb: \Arb_T \to \dgcat^{op}$ and $\arb^\ell: \Arb_T \to  \dgcat$.   

\begin{definition} \label{def:nsheaf}
The \emph{arboreal sheaf}  $\cA_{\vv T} \in \Sh(\overline{\T}, \dgcat)$ is the sheaf associated to the functor $\arb$, and the \emph{arboreal cosheaf} $\cA^{co}_{\vv T} \in \Cosh(\overline{\T}, \dgcat)$ 
is the cosheaf associated to the functor $\arb^\ell$. 
\end{definition}

Note that by definition of the functors $\arb$ and $\arb^\ell$, the arboreal cosheaf and sheaf are both locally saturated; we then have a \textsc{ccc} space $(\overline{\T},\cA_{\vv T}^{co})$ and a \textsc{csc} space $(\overline{\T},\cA_{\vv T})$, with the arboreal singularity as underlying topological space; also we have $\cA_{\vv T} = (\cA_{\vv T}^{pp})$.

\begin{remark}
Here we understand the topology on $\overline{\T}$ as being the stratification topology: stars of strata are the basis of opens.  The relevance of this
is that $\dgcat$  has (only) homotopy finite limits and colimits. 

Writing this topological space as $\overline{\T}^{str}$ 
and the usual topology as $\overline{\T}^{us}$, there is a continuous map $t: \overline{\T}^{us} \to \overline{\T}^{str}$.  To pull back the sheaf or cosheaf, we would however
have to change the target category from $\dgcat$ to $\DGCat$ so as to have all limits and colimits.  Having done so we may
consider the sheaf or cosheaf $t^* \Ind \cA_{\vv T}$ in the usual topology.  
\end{remark} 

For brevity we will usually drop the subscript $\vv T$ when it is unambiguous. Let us describe the arboreal sheaf explicitly. Let $\mf{p}$ be a correspondence $R \twoheadleftarrow S \hookrightarrow T$.  
We write  $\T(\mf{p})$ for the union of all simplices $[\mf{p}_1 \to \mf{p}_2 \cdots \to \mf{p}]$.  
Then $\overline\T = \coprod \T(\mf{p})$,  and, by definition, any sheaf associated to a functor from $\Arb_T$ is constant
on the  $\T(\mf{p})$.  
In fact, each $\T(\mf{p})$ with $\mf{p} \neq \mf{p}_T$ is topologically an open cell of dimension $|T| - |R|$ \cite[Prop 2.14]{N3}.\footnote{Nadler 
writes $\mathrm{L}_T$ for our $\T$ and  $\mathrm{L}_T(\mf{p})$ for our  $\T(\mf{p})$ in \cite{N3}.  Our notation
is chosen to emphasize that no symplectic geometry or microlocal sheaf theory is directly needed to 
understand the essentially combinatorial definitions and proofs.} 
The cone point $\T(\mf{p}_T)$ is the unique zero dimensional cell.  
Moreover, 
$\overline{\T(\mf{p})} = 
\coprod_{\mf{p}' \leq \mf{p}} \T(\mf{p}')$ 
\cite[Prop. 2.18]{N3}.

\begin{figure}[h]
    \centering
    \includegraphics[width=0.45\textwidth]{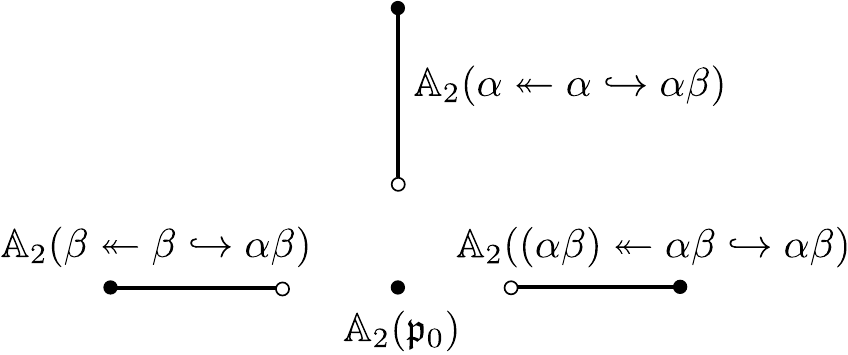}
         \caption{Stratification of the arboreal singularity $\overline\A_2$, by
         the strata $\A_2(\mf p)$.} 
    \label{fig:a2strata}
\end{figure}

Fix $\vv T$ and denote $\arb = \arb_{\vv T}$.  Because $\arb$ is constructible with respect to a stratification by a 
union of cells which all adjoin $\mf{p}_T$, 
the restriction map $\Gamma(\T, \arb) \to \arb_{\mf{p}_T}$ 
is an isomorphism. In particular, $\Gamma(\T, \arb) \simeq \Perf(\vv{T})$.  
Given an object $X \in \Gamma(\T, \arb) = \Perf(\vv{T})$, 
the germ at a point in 
$\T(R \twoheadleftarrow S \hookrightarrow T)$ is an element of 
the category $\Perf(\vv{R})$.  The desired object 
is produced by applying the correspondence functor $c_{\mf p}$ obtained from
$(R \twoheadleftarrow S \hookrightarrow T)$ to $X$. 

\subsubsection{Hom sheaves} Recall that given a sheaf of dg categories $\cF$ and objects $x,y \in \cF(X)$, there is a sheaf $\cF|_X^{x,y} \in \Sh(X,\vvt)$ which on $U$ evaluates to $\Hom_{\cF(U)}(\rho x,\rho y)$. For ease of notation, we will denote
\[ \arb_{\alpha, \beta} := \arb|_\T^{P_\alpha,P_\beta} \]
for the Hom sheaf between the generating projectives.

We will need explicit descriptions of the Hom sheaves between the generating projectives $P_\alpha$. Since we know
what the functors $c_{\mf p}$ do to the projective objects from Lemma \ref{lem:projectivecalc}, it is just a matter of assembling
the sheaf of the morphisms between Hom spaces.

\begin{definition}
For $\alpha$ a vertex of $T$, we write
\[ \T(\alpha):= \coprod_{\alpha \in S} \T(R \twoheadleftarrow S \hookrightarrow T). \]
\end{definition}

\begin{remark}\label{rmk:discs}
In \cite{N3}, Nadler gives an explicit construction of the arboreal singularities: for each
vertex $\alpha$ of $T$, take a copy of $\R^{|T|-1}$
with coordinates $x^\gamma(\alpha), \gamma \neq \alpha$. The topological space $\T$ is recovered by gluing these spaces: 
for each edge $\{\alpha,\beta\}\in E(T)$, identify points 
with coordinates $x_\gamma(\alpha)$ and $x_\gamma(\beta)$ whenever
$
x_\beta(\alpha) = x_\alpha(\beta) \geq 0$ and 
$ x_\gamma(\alpha) = x_\gamma(\beta)$ for  $\gamma \neq \alpha, \beta$.
Comparing this construction with the combinatorial definition \cite[Sec. 2]{N3} it is proven that the strata 
$\T(R \twoheadleftarrow S \hookrightarrow T)$ sit in the closure of the Euclidean space corresponding to $\alpha$
exactly when $\alpha \in S$. Thus $\T(\alpha)$ is homeomorphic to a closed ball of dimension $|T|-1$.
\end{remark}

\begin{figure}[h]
    \centering
    \includegraphics[width=0.8\textwidth]{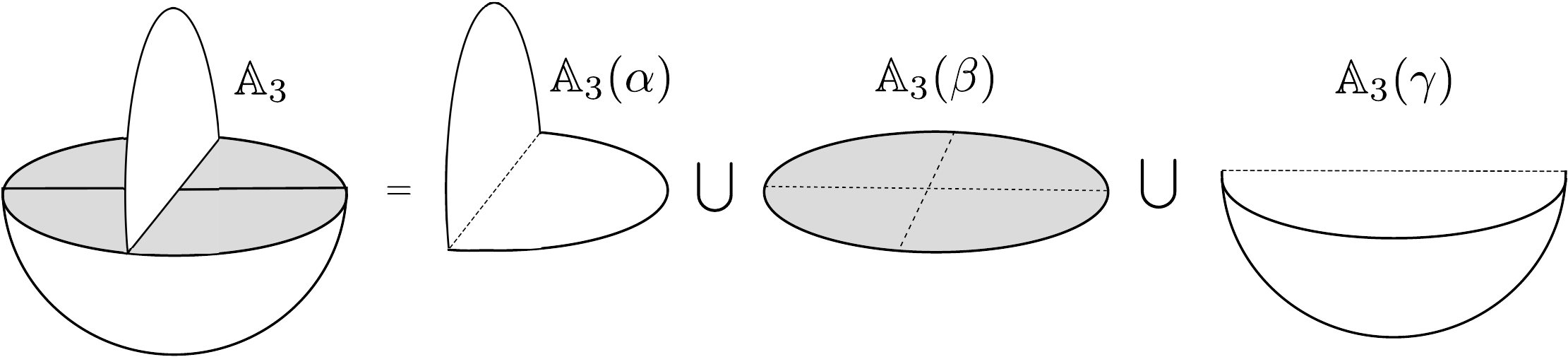}
         \caption{Gluing of $\overline\A_3$ from the discs $\A_3(\bullet)$. $\A_3(\alpha)$ and $\A_3(\gamma)$ are shown creased, with half-disc flaps pointing up and down, respectively.}
    \label{fig:a3decomposition}
\end{figure}

The following calculations are new.
\begin{proposition}
The sheaf $\arb_{\alpha,\alpha} \in \Sh(\T,\vvt)$ is the constant rank one (in degree zero) sheaf on $\T(\alpha)$.  
\end{proposition}
\begin{proof}
Let us describe the  functor on $\Arb_T$ whose nerve is the sheaf $\arb_{\alpha,\alpha}$.
By Lemma \ref{lem:projectivecalc}, on objects this functor is:
\[ (R \twoheadleftarrow S \hookrightarrow T) \mapsto 
\Hom_{\vv R}(P_{q(i^{-1}(\alpha))}, P_{q(i^{-1}(\alpha))}) = \begin{cases} k, \qquad \alpha \in i(S) \\ 0, \qquad 
\mathrm{otherwise.} \end{cases} \]
These Hom spaces have the identity as a basis element, which must be preserved by 
the functorial structure, hence gives a global section trivializing the sheaf hom. 
\end{proof}

To describe other Hom sheaves we have to worry about the orientation of the arrows in the quiver. For any pair of vertices $\alpha,\beta$, consider the following subset of $\T$:
\[ \T(\alpha, \beta)  := \coprod_{\substack{\alpha, \beta \in S \\ q(i^{-1}(\alpha)) \le q(i^{-1}(\beta))}}
\T(R \twoheadleftarrow S \hookrightarrow T). \]

\begin{figure}[h]
    \centering
    \includegraphics[width=0.6\textwidth]{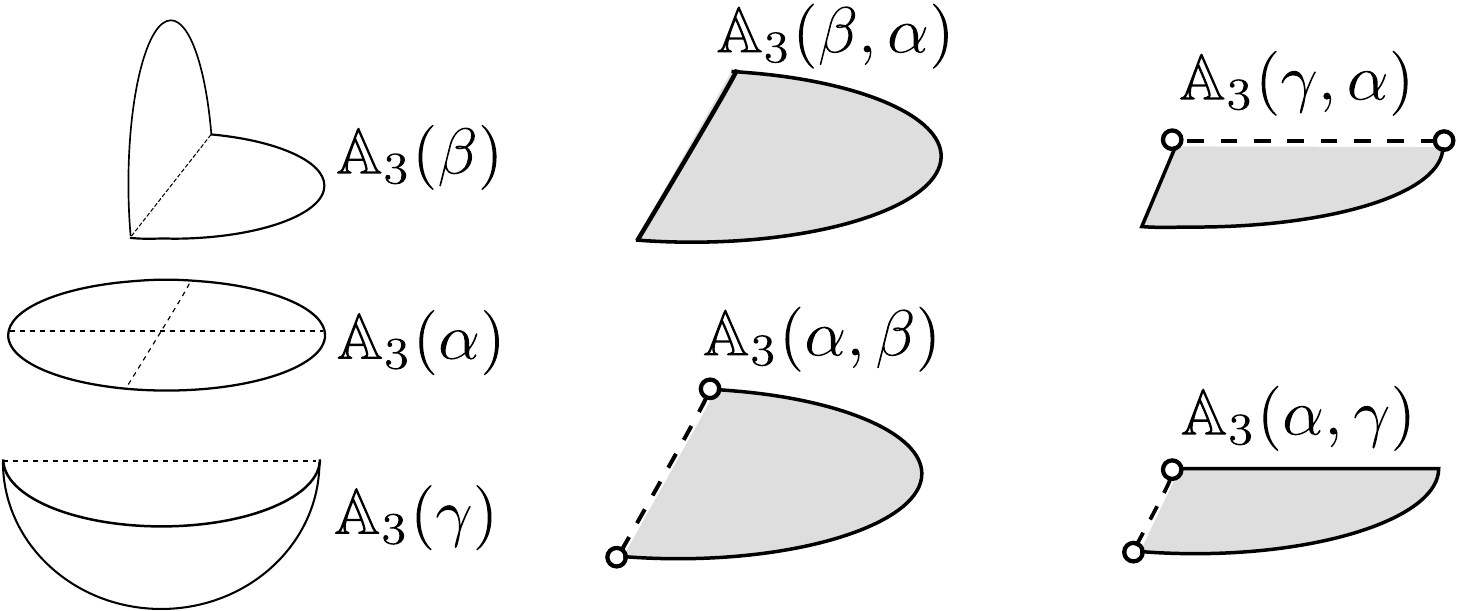}
         \caption{The subsets $\A_3(\bullet,\bullet)$ where $\A_3$ has the quiver structure $\alpha \rightarrow \beta \leftarrow \gamma$. The subsets $\A_3(\alpha),\A_3(\beta),\A_3(\gamma)$ are homeomorphic to closed discs, and the arboreal singularity is obtained by gluing them appropriately; on the left $\A_3(\alpha),\A_3(\gamma)$ are shown with folded flaps up and down, and are glued along the horizontal parts to $\A_3(\beta)$.
         Note that the subsets $\T(\lambda_1,\lambda_2)$ depend on the directions of the arrows in $T$, and moreover as in the proof above
         for any vertices $\lambda_1,\lambda_2$,
         the difference between $\T(\lambda_1,\lambda_2)$ and $\T(\lambda_1) \cap \T(\lambda_2)$ is
         at most deletion of some boundary strata.}
    \label{fig:a3supportHom}
\end{figure}

\begin{proposition} \label{prop:sheafhom}
The sheaf $\arb_{\alpha, \beta}$ is the constant rank one sheaf on $\T(\alpha, \beta)$. 
\end{proposition}
\begin{proof}
Again, let us write the functor on $\Arb$ giving rise to this sheaf.  By Lemma \ref{lem:projectivecalc}, on objects it is 
\[ (R \twoheadleftarrow S \hookrightarrow T) \mapsto 
\Hom_{\vv R}(P_{q(i^{-1}(\alpha))}, P_{q(i^{-1}(\beta))}) = \begin{cases} k, \qquad \alpha, \beta  \in S\,\, \& 
\,\, q(\alpha) \le q(\beta) \\ 0, \qquad 
\mathrm{otherwise.} \end{cases} \]
Here, when $\alpha \notin S$, we interpret $P_{q(i^{-1}(\alpha))}$ as $0$; and similarly for $\beta$.  
This shows that the sheaf has the correct stalks.  Lemma 
\ref{lem:projectivecalc} also described how the functor acts on the natural basis for these spaces, showing
that the sheaf is locally constant and in fact giving a global section, showing it is constant. 

Another way to see that the sheaf is constant is to explicitly describe the locus 
$\T(\alpha, \beta)$ and show it is contractible.
First, note that $\T(\alpha, \beta) \subseteq
\T(\alpha)\cap \T(\beta)$, with equality if $\alpha \leq \beta$. Since $\T(\alpha)$
and $\T(\beta)$ are both homeomorphic to $\R^{|T|-1}$ and are glued together along a half-space, their
intersection is homeomorphic to a closed ball.
For the case where $\alpha \nleq \beta$, the inclusion is strict, so suppose that we have a simplex 
\[ [\mf p_1 \to \dots \to \mf p_m] \subset (\T(\alpha)\cap \T(\beta))-\T(\alpha, \beta). \]
If we denote $\mf p_m = (R_m \twoheadleftarrow S_m \hookrightarrow T)$, this is the same as having 
$\alpha,\beta \in S_m$ but $q(\alpha) \neq q(\alpha \vee \beta)$.
Consider now the correspondence $\mf q = (R \twoheadleftarrow Q \hookrightarrow R_m)$, where $Q$ is some subtree
of $R_m$ containing $q(\alpha)$ but not containing $q(\beta)$. Then $\beta \notin Q \times_{R_m} S_m$,
which means the simplex $[\mf p_1 \to \dots \mf p_m \to \mf p_m \to \mf q \circ \mf p_m]$ is in $\T(\alpha)$ but
not in $\T(\beta)$, so this simplex is contained in the boundary of $\T(\alpha)\cap\T(\beta)$. Thus $\Lambda(\alpha,\beta)$ is obtained by deleting parts of the boundary of the closed ball $\T(\alpha)\cap\T(\beta)$,so it is contractible. 
\end{proof}

\subsubsection{Generalized arboreal singularities}\label{subsec:genarb}
In the expansion of Legendrian singularities, certain additional `generalized arboreal singularities' appear \cite{N4}.  Such a space is obtained from an ordinary arboreal singularity by deleting some of the strata; its data is encoded by a rooted tree $\vv T$ together with a subset of marked leaves $\ell \subseteq V(T)$, such that every element of $\ell$ is maximal for the partial order induced by the rooting. Let us define $\vv T^+$ to be the rooted tree obtained from $\vv T$ by adjoining a new vertex $\alpha^+$ for each $\alpha \in \ell$, with a new edge $\alpha^+ \to \alpha$.   

\begin{definition} \label{def: generalized arboreal singularity} (following Prop. 4.29 in \cite{N4})
The \emph{generalized arboreal singularity} $\overline\T^*$ corresponding to $(\vv T, \ell)$ is the stratified space obtained from the arboreal singularity $\overline\T^+$ by deleting the union $Y$ of subsets $\T^+(\mf p_\alpha) \cup \T^+(\mf p_{(\alpha^+ \alpha)})$ for every $\alpha \in \ell$. We denote $\T^* = \T^+ \cap \overline\T^*$.

From the construction of Section \ref{sec:quotients}, there is a locally saturated cosheaf on $\T^*$ given by
\[ \arb^{co}_{\vv T,\ell} := (\arb^{co}_{\vv T})_{\T^+/Y}|_{\T^*} \]
giving $\T^*$ the structure of a \textsc{ccc} space; and also of a \textsc{csc} space with the sheaf $\arb_{\vv T,\ell}$ of pseudo-perfect modules.
\end{definition}

\begin{figure}[h]
    \centering
    \includegraphics[width=0.8\textwidth]{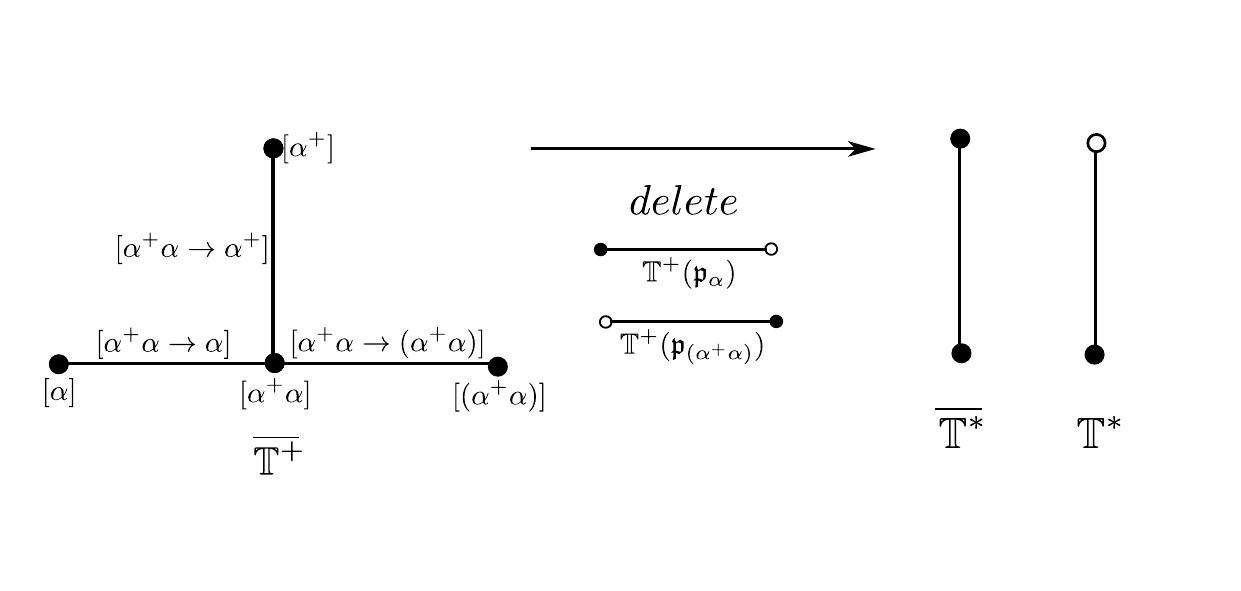}
         \caption{The generalized arboreal singularity corresponding to the singleton quiver $\vec{T} = \alpha$ with the vertex also a marked leaf $\ell = \{\alpha\}$. We delete two open subsets to get $\overline\T^*$; also shown is $\T^* = \overline\T^* \cap \T^+$.}
    \label{fig:genarb1}
\end{figure}

\begin{figure}[h!]
    \centering
    \includegraphics[width=0.8\textwidth]{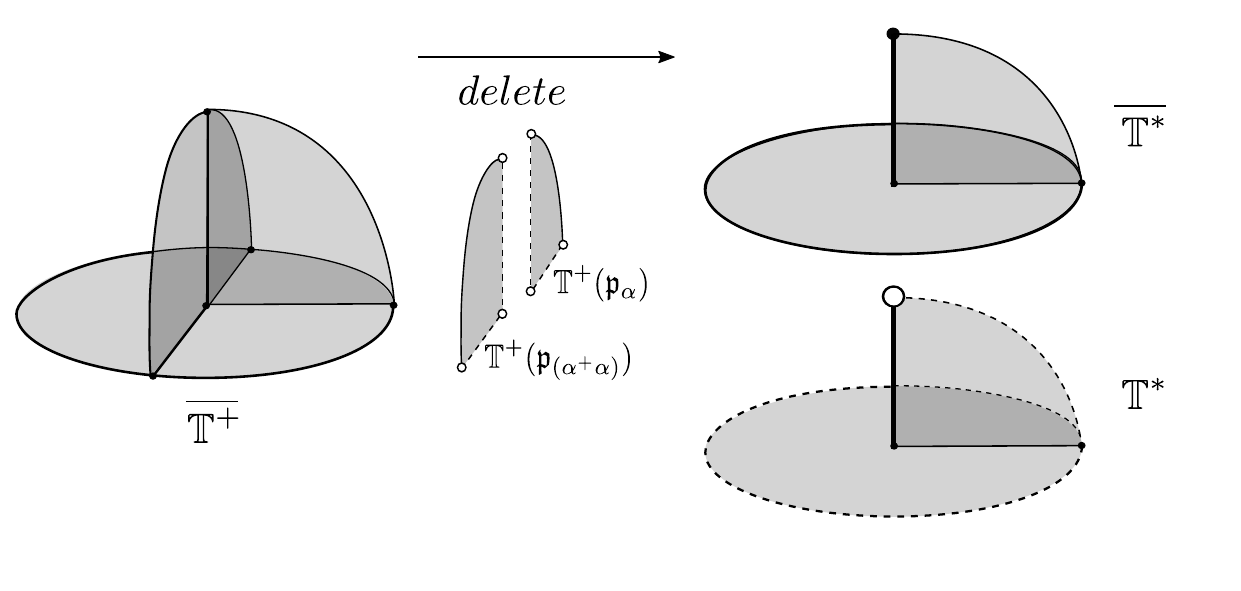}
         \caption{The generalized arboreal singularity corresponding to the $A_2$ quiver $\vec T = \alpha\to \beta$ with marked leaf $\ell = \{\alpha\}$. Here we obtain $\overline\T^*$ from deleting some of the strata in the arboreal singularity $\overline\T^+$ corresponding to the augmented quiver $\alpha^+ \to \alpha \to \beta$. Bottom right corner shows $\T^*$; note that it contains part of its boundary.}
    \label{fig:genarb2}
\end{figure}

\section{Orientations on arboreal spaces}
\subsection{Arboreal spaces}
\begin{definition}\label{def:arborealSpace}
A \textsc{ccc} space $(\X, \cW)$ is {\em special arboreal}  
if it is locally modeled on $(\T, \cA^{co}_{\vv T})$ for varying $\vv T$.  

More precisely, for every $x \in \X$, there is some $\vv T = \vv T(x)$, some refinements of stratifications on $\X$ and $\T$,
after which there is some stratified neighborhood $\mathbb{U} \ni x$ and a stratified embedding $m: \mathbb{U} \hookrightarrow \T$ 
and an isomorphism $\cW|_{\mathbb{U}} \cong m^* \cA^{co}_{\vv T}|_{\mathbb U}$. 
\end{definition}

Note that we ask $\X$ is locally modeled on the {\em interiors} $\T$ of the arboreal singularities $\overline{\T}$.  

\begin{remark} Because the spaces $(\T, \cA_{\vv T})$ are themselves, away from the central point, locally modeled on other
such spaces for smaller trees, in the above definition we could equivalently 
ask that $x$ maps to the central point of $\T$ if we allow $m$ to be a stratified submersion.  This would have 
the virtue of making $T(x)$ a well defined function of $x$. 
\end{remark}

More generally, 
\begin{definition}\label{def:genArborealSpace}
A \textsc{ccc} space $(\X, \cW)$ is  {\em arboreal} if it is locally
modeled on the $(\T^*, \arb^{co}_{\vv T,\ell})$ of Def. \ref{def: generalized arboreal singularity}, for varying choices of $ (\vv T,\ell)$. 
\end{definition}

Note that a special arboreal space is an arboreal space in which we choose the set of marked leaves to be empty.

\begin{lemma}
An arboreal space $(\X, \cW)$ is smooth (i.e. $\cW(U)$ is smooth for all $U$) and stalkwise proper.  The dual \textsc{csc} space 
$(\X, \cW^{pp})$ is proper and stalkwise smooth. 
\end{lemma}
\begin{proof}
Follows from the fact that the cosheaves $\arb^{co}_{\vv T,\ell}$ are locally saturated, together with the fact that finite colimits of smooth dg categories are smooth.
\end{proof}

\subsection{An orientation on an arboreal space}\label{sec:locarb}
Consider an arboreal space $(\X,\cW)$ as above. By Theorem \ref{thm:checkOnProper}, which will be proved later in Section \ref{sec:propertosmooth}, orientations on the \textsc{ccc} space $(\X, \cW)$ are equivalent to orientations on the \textsc{csc} space $(\X,\cW^{pp})$. Moreover, by Lemma \ref{lem: orientation on quotient}, we only have to consider orientations on special arboreal spaces.

We will first show the local models $(\T, \arb)$ admit orientations. More precisely, 
we will show that a rooting of $T$ induces a canonical isomorphism $\cH\cH(\arb) \simeq \omega_{\T}[1 -|T|]$,
which moreover gives an orientation on the \textsc{csc} space $(\T, \arb)$.

\subsubsection{The dualizing complex of an arboreal singularity}
Recall that an explicit representative for the Verdier dualizing sheaf $\omega_X$ is given by the ``sheaf of local singular chains'' \cite[p.91]{GM}. That is, let ${C}^{-d}$ be the sheaf which on
sufficiently small open sets is given by $\bC^{-d}(U) = C_d(X, X \setminus U; k)$, where $C_d$ 
is the singular $d$-chains, and the sheaf structure is defined by the evident restriction maps; the singular chain differential 
collects these into a complex of sheaves.

\begin{proposition}\label{prop:omega}
With notation as above, the stalk of $\omega_\T$ at a stratum labeled by a correspondence $\mf p = (R \twoheadleftarrow S \hookrightarrow T)$
is concentrated in degree $-(n-1)$, where it is given by a direct sum decomposition
\[ \omega_\T^{-(n-1)}(\T(\mf p)) \simeq \bigoplus_{\alpha \in R} k_\alpha \simeq k^{|R|}, \]
where each $k_\alpha \simeq k$.
Now suppose we have correspondences $\mf{p}' = \mf{q} \circ \mf{p}$, where $\mf{p} = (R' \twoheadleftarrow S' \hookrightarrow T)$ and 
$\mf{q} = (R' \overset{q}\twoheadleftarrow Q \overset{i}\hookrightarrow R)$. Then the simplex $[\mf{p}]$ is in the closure of $[\mf{p}\to \mf{p}']$
and the generization map $\omega_\T(\T(\mf{p})) \to \omega_\T(\mf{p}')$ is given, in the decomposition above, by
\[ \bigoplus_{\alpha \in R} k_\alpha \to \bigoplus_{\beta \in R'} k_\beta, \]
where $1 \in k_\alpha$ gets sent to $1 \in k_{q(\alpha)}$ if $\alpha \in S$ and $0$ otherwise. In other words, the map adds all
the factors corresponding to vertices that get identified by the quotient $q$. Moreover, if one picks an orientation of a top stratum of $\T$, there is a canonical choice of isomorphisms above.
\end{proposition}

\begin{proof}
As mentioned above, on any space $X$, for sufficiently small $U$, there is a natural isomorphism $\omega_X(U) = C_*(X, X\setminus U)$, with the chain complex $C_*$ interpreted as a cochain complex just by negating all the degrees.

Let us first calculate the stalk at the center of $\T$, that is at the simplex $[\mf p_0]$ given by the trivial
correspondence. Consider the decomposition of $\T$ into the discs $\T(\alpha)$, and for each vertex $\alpha$ let
$U_\alpha = U \cap \T(\alpha)$ and $Z_\alpha = \T(\alpha) \setminus U_\alpha$
 where $U$ is a small neighborhood of $[\mf p_0]$. Each $U_\alpha$ is an open disc inside
of the disc $\T(\alpha)$, so the relative homology $H_*(\T(\alpha),Z_\alpha)$ is $k$ in degree $n-1$ and 
zero in other degrees.

For an edge $\alpha-\beta$, the Mayer-Vietoris sequence for relative homology gives a distinguished triangle
\[ C_*(\T(\alpha)\cap \T(\beta), Z_\alpha \cap Z_\beta) \to
C_*(\T(\alpha), Z_\alpha)\oplus C_*(\T(\beta), Z_\beta) \to
C_*(\T(\alpha)\cup \T(\beta), Z_\alpha \cup Z_\beta)
\]

Note that since the discs $\T(\alpha), \T(\beta)$ are glued by their halves, the pair 
$(\T(\alpha)\cap \T(\beta), Z_\alpha \cap Z_\beta)$ has zero relative homology
\[ C_*(\T(\alpha)\cap \T(\beta), Z_\alpha \cap Z_\beta) \simeq C_*(\text{half-ball},\text{half-sphere}) \simeq 0 \] 
so we can identify 
\[ C_*(\T(\alpha)\cup \T(\beta), Z_\alpha \cup Z_\beta) \simeq C_*(\T(\alpha), Z_\alpha)\oplus C_*(\T(\beta), Z_\beta)
\simeq k^2[n-1]. \]

We can iterate and glue the pairs $\T(\alpha), \T(\beta)$ according to the tree $T$. More specifically, take 
connected subtrees $Q \subset Q' \subseteq T$ such that $|Q'| = |Q|+1$, that is the subtree $Q'$ is obtained from $Q$ by 
adding a new vertex $\gamma$ of $T$ with an edge $\beta-\gamma$ for some $\beta \in Q$. Now suppose by induction that we
know a direct sum decomposition of the relative cohomology of the pair obtained by taking all $\T(\alpha)$ and $Z_\alpha$
with $\alpha \in Q$:
\[ C_*(\bigcup_{\alpha \in Q} \T(\alpha), \bigcup_{\alpha \in Q} Z_\alpha) \simeq 
\bigoplus_{\alpha \in Q} C_*(\T(\alpha), Z_\alpha) \simeq k^{|Q|}[n-1]. \]

Note that since we glued $\gamma$ only to one vertex $\beta$ in $Q$,
the intersection $(\cup_{\alpha \in Q} \T(\alpha)) \cap \T(\gamma)$ coincides with $\T(\beta) \cap \T(\gamma)$, and also
$(\cup_{\alpha \in Q} Z_\alpha) \cap Z_\gamma = Z_\beta \cap Z_\gamma$. So as noted above this pair has vanishing cohomology,
and the Mayer Vietoris distinguished triangle
\[ C_*((\bigcup_{\alpha \in Q} \T(\alpha)) \cap \T(\gamma), (\bigcup_{\alpha \in Q} Z_\alpha) \cap Z_\gamma) \to
C_*(\bigcup_{\alpha \in Q} \T(\alpha), \bigcup_{\alpha \in Q} Z_\alpha) \oplus C_*(\T(\gamma), Z_\gamma) \to
C_*(\bigcup_{\alpha \in Q'} \T(\alpha), \bigcup_{\alpha \in Q'} Z_\alpha)\]
gives us a direct sum decomposition 
$C_*(\bigcup_{\alpha \in Q'} \T(\alpha), \bigcup_{\alpha \in Q'} Z_\alpha) \simeq k^{|Q'|}[n-1]$ 
Iterating this until we reach all of $T$ gives the desired result.

\begin{figure}[h]
    \centering
    \includegraphics[width=0.6\textwidth]{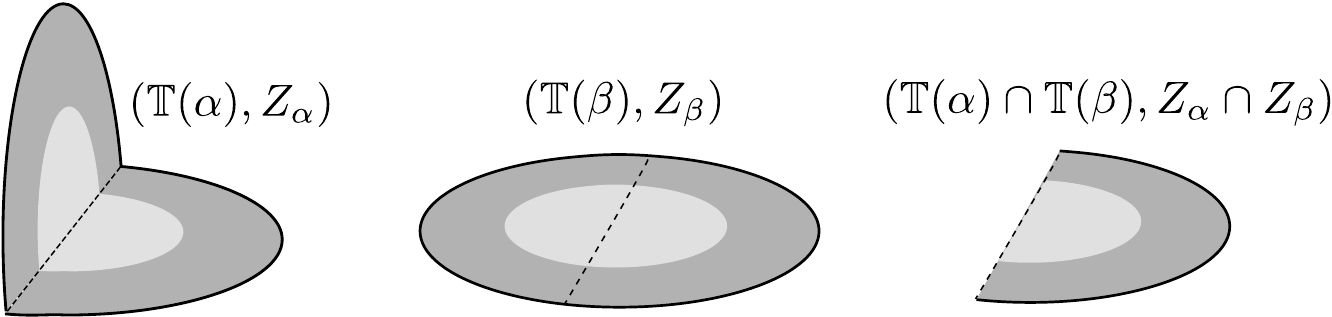}
         \caption{The pairs $(\T(\alpha),Z_\alpha)$ appearing in the Mayer-Vietoris decomposition. The complements 
         $Z_\alpha = \T(\alpha)\setminus U_\alpha$ are in dark gray. The intersection pairs 
         $(\T(\alpha)\cap \T(\beta),Z_\alpha \cap Z_\beta)$ have vanishing relative homology.}
    \label{fig:omegadiscs}
\end{figure}

To describe all the other stalks and generization maps, let's first describe the map 
\[ (\omega_\T)_{[\mf p_0]} \to (\omega_\T)_{[\mf p_0 \to \mf p]}\]
from the central stalk calculated above to the stalk over a neighboring 1-simplex $[\mf p_0 \to \mf p]$ where
\[ \mf p = (R \overset{q}\twoheadleftarrow S \hookrightarrow T), \]
with $q$ the contraction of a single edge $\mu-\lambda$. As above, let $U$ be a neighborhood of the origin, and take a 
point in the simplex $[\mf p_0 \to \mf p]$ inside of $U_0$. Taking a neighborhood $V \subset U_0$ of this point, 
we see that the map between the two stalks of $\omega_\T$ is given by restriction of relative chains.

We use the same decomposition of $\T$ into discs that we have above. Let $V_\alpha = V \cap \T(\alpha)$ and $Y_\alpha = \T(\alpha) \setminus V_\alpha$. We see that $V_\alpha = \emptyset$ if $\alpha \notin S$, and if $\alpha \in S$ then the pair $(\T(\alpha),V_\alpha)$ can be homotoped to $(\T(\alpha), U_\alpha)$. Moreover, if we consider an edge
that's not being contracted, say $\alpha-\beta \neq \mu-\lambda$, $V_\alpha$ and $V_\beta$ intersect in an open half-ball inside 
the closed half-ball $\T(\alpha) \cap \T(\beta)$, so the pair $(\T(\alpha) \cap \T(\beta), Y_\alpha \cap Y_\beta)$ is homotopic to
$(\T(\alpha) \cap \T(\beta), Z_\alpha \cap Z_\beta)$ and has vanishing relative homology. 

A new feature only occurs for the edge 
$\mu-\lambda$ being contracted: every stratum adjoining the 1-simplex $[\mf p \to \mf p_0]$ must be labeled by correspondences 
where this edge is also contracted. So the neighborhood $V$ is entirely contained in the intersection $\T(\mu)\cap \T(\lambda)$, and
so the pair $(\T(\mu) \cap \T(\lambda), Y_\mu \cap Y_\lambda)$ is homotopic to $(\T(\mu), Y_\mu)$ and has relative homology 
$H_*(\T(\mu) \cap \T(\lambda), Y_\mu \cap Y_\lambda) \simeq k[n-1]$.

Now we can apply the same process as above and iterate over all edges of the original tree $T$. Suppose we start
with a subtree $Q$ of $T$ such that $Q \subset R, Q \subset S$, meaning no edge of $Q$ gets contracted by $q$, and suppose by induction
that we already know a direct sum decomposition
\[ C_*(\bigcup_Q \T(\alpha), \bigcup_Q Y_\alpha) \simeq \bigoplus_Q C_*(\T(\alpha),Y_\alpha) \simeq k^{|Q|}[n-1]. \]
Consider the Mayer-Vietoris triangle for adjoining an extra vertex $\gamma$ via an edge $\beta-\gamma$
as above. The relative chain restriction maps give a map between the distinguished triangles
\[
\xymatrix{
C_*(\T(\beta) \cap \T(\gamma), Z_\beta \cap Z_\gamma) \ar[r] \ar[d]
& \bigoplus_{\alpha\in Q} C_*(\T(\alpha), Z_\alpha) \oplus C_*(\T(\gamma), Z_\gamma) \ar[r] \ar[d]
& \bigoplus_{\alpha\in Q'} C_*(\T(\alpha), Z_\alpha) \ar[d] \\
C_*(\T(\beta) \cap \T(\gamma), Y_\beta \cap Y_\gamma) \ar[r]
& \bigoplus_{\alpha\in Q} C_*(\T(\alpha), Y_\alpha) \oplus C_*(\T(\gamma), Y_\gamma) \ar[r]
& C_*(\bigcup_{Q'} \T(\alpha), \bigcup_{Q'} Y_\alpha).
}
\]
Since this is a map of distinguished triangles it suffices to know two of the vertical maps. There are three cases to be checked.
Suppose $\gamma \in S$ but $\beta-\gamma \neq \mu-\lambda$: the two terms on the left vanish, and
the remaining maps are all isomorphisms. Now suppose $\gamma \notin S$:  the two terms on the left vanish, but then since 
$V_\gamma = \emptyset$ the vertical map $C_*(\T(\gamma), Z_\gamma) \to C_*(\T(\gamma), Y_\gamma)$ is zero so the direct sum factor
$C_*(\T(\gamma), Z_\gamma)$ of $C_*(\cup_{Q'} \T(\alpha), \cup_{Q'} Z_\alpha)$ gets killed by the restriction map.
The last case to consider is when the vertex being added is $\lambda$ via the contracted edge $\mu-\lambda$. Then the
intersection $(\T(\mu) \cap \T(\lambda), Y_\mu \cap Y_\lambda)$ has relative homology $k$ in degree $n-1$ so the map above in 
degree $n-1$ becomes 
\[
\xymatrix{
0 \ar[r] \ar[d] 
& \bigoplus_{\alpha\in Q} C_{n-1}(\T(\alpha), Z_\alpha) \oplus C_{n-1}(\T(\lambda), Z_\lambda) \ar[r] \ar[d] 
& \bigoplus_{\alpha\in Q'} C_{n-1}(\T(\alpha), Z_\alpha) \ar[d] \\
k \ar[r] & \bigoplus_{\alpha\in Q} C_{n-1}(\T(\alpha), Y_\alpha) \oplus C_{n-1}(\T(\lambda), Y_\lambda) \ar[r]
& C_{n-1}(\bigcup_{Q'} \T(\alpha), \bigcup_{Q'} Y_\alpha).
}
\]

Here the middle vertical map $C_{n-1}(\T(\lambda), Z_\lambda) \to C_{n-1}(\T(\lambda), Y_\lambda)$ is restriction of relative chains, so
it is an isomorphism $k \overset{\sim}\to k$. Thus $C_{n-1}(\bigcup_{Q'} \T(\alpha), \bigcup_{Q'} Y_\alpha) \simeq k^{|Q|}$, and
does \emph{not} gain an extra direct sum factor $k$ from $C_{n-1}(\T(\lambda), Y_\lambda)$. A better way of phrasing this is that
we have a decomposition
\[ C_*(\bigcup_{Q'} \T(\alpha), \bigcup_{Q'} Y_\alpha) \simeq \bigoplus_{\alpha \in q(Q')} k_\alpha, \]
with one factor $k_\alpha \simeq k$ for each vertex $\alpha$ in the image $q(Q')\subset R$, since $\mu$ and $\lambda$ 
contribute only one factor $k_{(\mu\lambda)}$ of $k$. Using this decomposition, we can write down the vertical restriction map on the right as:
\[ \bigoplus_{\alpha \in Q'} k_\alpha \to \bigoplus_{\alpha \in q(Q')} k_\alpha, \]
which is an isomorphism on all factors except for $k_\mu \oplus k_\lambda \to k_{\mu\lambda}$ where it is \emph{addition} of cochains.

We can calculate all the other stalks and maps by iterating this procedure, since every correspondence of trees can be decomposed into successive contraction of a single edge; this gives the calculation in the theorem. Note that in order to describe the maps as addition, we implicitly used 
the fact that we have a distinguished basis element $1 \in H_{n-1}(\T(\alpha),Z_\alpha)$. Picking this element uniquely requires picking a
orientation of each disc $\T(\alpha)$; in fact since all discs are glued it is only necessary to pick an orientation of one of the discs, or equivalently an orientation of a top-dimensional stratum of $\T$.
\end{proof}

\subsubsection{The Hochschild homology and the cyclic homology sheaf}
Recall that the Hochschild homology of an algebra $A$ is calculated by the Hochschild chain complex
\[ \HH(A): \dots \to A \otimes A \otimes A \to A \otimes A \to A, \]
where $A^{\otimes n}$ is placed in degree $-n$, and the differential $d_{-n}: C_{-n}(A) \to C_{-n+1}(A)$ given by 
\[ d(x_n \otimes\dots\otimes x_0) = x_{n-1}\otimes\dots\otimes x_0 x_n + 
\sum_{i=0}^{n-1} (-1)^i x_n\otimes\dots\otimes x_{i+1}x_i \otimes\dots \otimes x_0. \]

We are interested in the stalks of the Hochschild homology sheaf $\cH\cH(\arb)$ and of the cyclic homology sheaf $\cH\cH(\arb)_{S^1}$, which 
are the Hochschild/cyclic homology of the stalks of $\arb$, i.e., 
of the categories $\Perf(\vv T)$.  To calculate these, recall that
Hochschild/cyclic homology is invariant under dg Morita equivalences, and 
for any dg-algebra $\cA$, there is an isomorphism of Hochschild complexes
$ \HH(\cA) \simeq \HH(\Perf\mh\cA) $
with the Hochschild homology of the category of perfect $\cA$-modules \cite{Kel1}, with compatible $S^1$-actions. We recall: 

\begin{proposition}  \label{prop:Cib} \cite{Cib} Let $\vv Q$ be an acyclic quiver.  Then $\HH_0(k[\vv Q]) = k^{|Q|}$ 
and all higher Hochschild homologies vanish. Moreover the $S^1$-action on the Hochschild complex is trivial, i.e. the cyclic complex is given by
\[ \HH(k[\vv Q])_{S^1} = H^*(BS^1)\otimes \HH(k[\vv Q]) = k[u] \otimes \HH(k[\vv Q]), \]
with $u$ denoting canonical generator of $H^*(BS^1)$ in degree $2$, such that the natural map $\HH(k[\vv Q]) \to \HH(k[\vv Q])_{S^1}$ sends $x \mapsto 1\otimes x$.
\end{proposition}
\begin{proof}
	This follows from the fact that additive invariants map semi-orthogonal decompositions of categories to direct sums, so for our tree quivers we reduce to the calculation of Hochschild/cyclic homology for the ground field $k$. For an explicit proof using the path algebra structure one can consult \cite{Cib}.
\end{proof}

Since all the $S^1$ actions we consider will be trivial, we will ignore it from now on; every map out of $\HH(k[\vv T])$ can be factored through the map $\HH(k[\vv T]) \to \HH(k[\vv T])_{S^1}$ by sending $u\mapsto 0$, so for all our applications we can just construct maps out of/into the Hochschild homology $\HH$ itself.

There is a natural basis on $\HH_0(\Perf(\vv Q)) = \HH_0(k[\vv Q])\simeq k^{|Q|}$, given by the images of the idempotents 
$| \alpha \rangle \langle \alpha| \in k[\vv Q]$, or equivalently, of the modules $P_\alpha \in \Perf(\vv Q)$. 
Note  that this basis depends on $\vv{Q}$ and not just the underlying graph. 
In terms of these bases, Lemma \ref{lem:projectivecalc} gives the generization functors of the Hochschild homology
sheaf. 

\begin{proposition} \label{prop:corrhoch}
Let $\mf p = (R \twoheadleftarrow S \hookrightarrow T)$ be a correspondence inducing a functor $\Perf(\vv T) \to \Perf(\vv R)$. The
induced map between Hochschild homologies is given by
\[ |\alpha \rangle \langle \alpha| \mapsto \begin{cases}  |q(i^{-1}(\alpha)) \rangle \langle q(i^{-1}(\alpha))| & \text{\ if\ } \alpha \in i(S) \\ 0 & \text{\ otherwise.} \end{cases} \]
\end{proposition}

\subsubsection{Comparison}\label{subsec:comparison}
Comparing Proposition \ref{prop:omega} with Propositions \ref{prop:Cib} and  \ref{prop:corrhoch} gives 
an abstract isomorphism $\cH\cH(\arb) \simeq \omega_{\T}[1-n]$.

The choice of this isomorphism is not unique, but as we saw above, upon fixing the decomposition of $\T$ as the union of discs $\T(\alpha)$ and an orientation of one of these discs, we get distinguished bases for the stalks of $\omega_{\T}[1-n]$. 
In addition if we pick a root in $T$ this induces choices of roots in all $R$, and we get sets of distinguished elements $|\alpha \rangle \langle \alpha|$ in all the stalks of $\cH\cH(\arb)$. We can then make a \emph{canonical} choice of isomorphism 
$\cH\cH(\arb)\xrightarrow{\sim} \omega_{\T}[1-n]$, which on a stalk over the stratum $\T(R \twoheadleftarrow S \hookrightarrow T)$ gives the isomorphism
\[ \HH_0(k[\vv R]) \xrightarrow{\sim} \bigoplus_{\alpha \in V(R)} k_\alpha, \]
sending $|\alpha \rangle \langle \alpha|$ to $1 \in k_\alpha \simeq k$ in the direct sum decomposition of proposition \ref{prop:omega}. 

Recall we have constructed an isomorphism $ \cH\cH(\arb) \xrightarrow{\sim} \omega_{\T}[-d]$, with $d = |T|-1$; moreover since the $S^1$ action on $\cH\cH(\arb)$ is trivial, this naturally descends to an isomorphism $\cH\cH(\arb)_{S^1} \xrightarrow{\sim} \omega_{\T}[-d]$

\begin{theorem} \label{thm: arborientation}
The map $\cH\cH(\arb)_{S^1} \xrightarrow{\sim} \omega_{\T}[-d]$ constructed above is an orientation on the \textsc{csc} space $(\T,\arb)$.
\end{theorem}

\begin{proof}
We are trying to show that, for any objects $x,y$ of $\arb(\T)$, the map 
\[ \arb_\T^{x,y} \to \sheafHom_{\T}(\arb_\T^{y,x}, \omega_X)[-d]\]
induced by the orientation is an isomorphism. It is enough to check the assertion on
generators of the category, so we use the projective objects $P_\alpha$. 

As calculated in Proposition \ref{prop:sheafhom},
the sheaf $\arb_{\alpha, \beta}$ is the constant sheaf on $\T(\alpha,\beta)$. 
Let us check that the sheaves $\arb_{\alpha, \beta}$ and 
$\sheafHom_\T(\arb_{\beta, \alpha},\omega_{\T})[-d]$ are isomorphic 
by analyzing the topology of $\T(\alpha,\beta)$ and $\T(\beta,\alpha)$.

Let $\gamma = \alpha \vee \beta$ be the minimum point in the geodesic between $\alpha$ and $\beta$, so 
 $\T(\beta,\alpha) = \T(\beta,\gamma) \cap \T(\alpha)$.  We will 
describe the topology of the subset $\T(\beta,\gamma)$ for $\gamma \leq \beta$. 
The space $\T(\beta,\gamma)$ is a subset of the closed discs $\T(\beta)$ and $\T(\gamma)$, 
and only differs from the intersection on its boundary. In fact we have
\[ (\T(\beta) \cap \T(\gamma))\setminus \T(\beta,\gamma) = 
\coprod_{\substack{\beta,\gamma \in S \\ q(\beta) \neq q(\gamma)}} \T(\mf p) \]

If $q(\beta)\neq q(\gamma)$ then there is a correspondence $\mf q$ that deletes $q(\beta)$ and
keeps $q(\gamma)$, i.e. there's a map of the form $[\mf p_1 \to\dots\to \mf p] \to [\mf p_1 \to\dots\to \mf p \to \mf q \circ \mf p]$.
This last simplex is not in $\T(\beta)$, so $[\mf p_1 \to\dots\to \mf p]$ must be on the boundary of $\T(\beta)\cap\T(\gamma)$. Conversely, this map only can exist if $q(\beta) \neq q(\gamma)$ in the correspondence $\mf p$, so $(\T(\beta) \cap \T(\gamma))\setminus \T(\beta,\gamma)$ is exactly the boundary of $\T(\beta) \cap \T(\gamma)$ inside of $\T(\beta) \cup \T(\gamma)$. Thus
\[ \T(\beta,\gamma) \hookrightarrow \T(\beta) \cup \T(\gamma) \]
is an open inclusion, with closure $\T(\gamma,\beta) = \T(\beta)\cap\T(\gamma)$.
Obviously when $\beta = \gamma$ all these are the same.

For any vertices $\mu,\lambda$, let's denote the open inclusion
\[ i_{\mu\lambda}: \T(\mu,\lambda) \hookrightarrow \T(\mu) \cup \T(\lambda) \]
and the closed inclusions
\[ j_{\mu\lambda}: \T(\mu) \cup \T(\lambda) \hookrightarrow \T \]
\[ j_\mu: \T(\mu) \hookrightarrow \T \]

The sheaf $\arb_{\alpha, \beta}$ is supported on the subset $\T(\alpha,\gamma) \cap \T(\beta)$, so it can be expressed in terms of
the constant sheaf on $\T(\alpha,\gamma)$ by
\[ \arb_{\alpha, \beta} = j_{\beta,*} j^*_\beta j_{\alpha\gamma,!} i_{\alpha\gamma,!} 
\underline k_{\T(\alpha,\gamma)}. \]
Similarly we have $\arb_{\beta, \alpha} = j_{\alpha,*} j^*_\alpha j_{\beta\gamma,!} i_{\beta\gamma,!} 
\underline k_{\T(\beta,\gamma)}.$

Let us denote by $\D':\cF \mapsto \sheafHom_X(\cF,\omega_X)$ the contravariant Verdier duality operation. As preparation
for the proof of nondegeneracy, let's first check if there is at least an isomorphism between 
\[ \arb_{\alpha, \beta} \qquad \mbox{and} \qquad 
\sheafHom_{\T}(\arb_{\beta, \alpha}, \omega_{\T})[-d] = \D'(\arb_{\beta, \alpha})[-d]. \]

Verdier duality intertwines the pairs of functors $f_*,f_!$ and $f^*,f^!$, so there are equivalences
\[ \D'(\arb_{\beta, \alpha})[-d] \simeq 
\D' j_{\alpha,*} j^*_\alpha j_{\beta\gamma,!} i_{\beta\gamma,!} \underline k_{\T(\beta,\gamma)}[-d] \simeq
j_{\alpha,!} j^!_\alpha j_{\beta\gamma,*} i_{\beta\gamma,*} \D' \underline k_{\T(\beta,\gamma)}[-d]. 
\]
Since $\T(\beta,\gamma)$ is homeomorphic to an open ball (it is the intersection of two discs with the boundary removed), $D_{\T(\beta,\gamma)} \underline k_{\T(\beta,\gamma)} \simeq k_{\T(\beta,\gamma)}[d]$. Also, by proper base change, $j^!_\alpha j_{\beta\gamma,*} = j_{1,*}j^!_2$, where
$j_1, j_2$ are the closed inclusions
\[ j_1: \T(\alpha) \cap \T(\beta) \cap \T(\gamma) \hookrightarrow \T(\alpha), \qquad 
j_2: \T(\alpha) \cap \T(\beta) \cap \T(\gamma) \hookrightarrow \T(\beta) \cap \T(\gamma) \]
Since $i_{\beta\gamma}$ is the inclusion of $\T(\beta,\gamma)$ into its closure $\T(\beta)\cap \T(\gamma)$, the sheaf $j^!_2 i_{\beta\gamma,*}k_{\T(\beta,\gamma)}$ is isomorphic to $j^*_3 i_{\alpha\gamma,!} \underline k_{\T(\alpha,\gamma)}$, for 
$j_3:\T(\alpha) \cap \T(\beta) \cap \T(\gamma) \hookrightarrow \T(\alpha) \cap \T(\gamma)$ so we have
\[ \D'(\arb_{\beta, \alpha})[-d] \simeq j_{\alpha,!}j_{1,*}j^*_3 i_{\alpha\gamma,!} \underline k_{\T(\alpha,\gamma)}. \]
But denoting $j_4: \T(\alpha) \cap \T(\beta) \cap \T(\gamma) \hookrightarrow \T(\beta)$ we have $j_{\alpha,!}j_{1,*}\simeq j_{\beta,*} j_{4,!}$ and then performing another base change we get
\[ \D'(\arb_{\beta, \alpha})[-d] \simeq j_{\beta,*} j^*_\beta j_{\alpha\gamma,!} i_{\alpha\gamma,!} \underline k_{\T(\alpha,\gamma)}[-d] 
= \arb_{\alpha, \beta}. \]
Verdier duality thus induces an isomorphism between these sheaves. It still remains to check that the orientation we defined indeed
induces this isomorphism.

Since we have an isomorphism $\cH\cH(\arb) \simeq \omega_{\T}[-d]$, this gives an identification 
\[ \sheafHom_\T(\arb_{\beta, \alpha},\omega_{\T})[-d] \simeq \sheafHom_\T(\arb_{\beta, \alpha}, \cH\cH(\arb)). \]
We have to prove that the sheaf morphism given by the trace
\[ \arb_{\alpha, \beta} \to \sheafHom_\T(\arb_{\beta, \alpha},\cH\cH(\arb)) \]
is an isomorphism, which can be checked on stalks. We will perform the calculation at the origin, so we will need a description of the sections of the sheaf Hom on a neighborhood around the origin:
\[ \sheafHom_\T(\arb_{\beta, \alpha}, \cH\cH(\arb))(U) = \mathrm{Hom}((\arb_{\beta, \alpha})|_U, \cH\cH(\arb)|_U), \]
consisting of compatible collections of morphisms between $\arb_{\beta, \alpha}$ and $\cH\cH(\arb)$ on all the strata. 
For every stratum $\T(\mf p)$, we have a map 
\[ \Hom((\arb_{\beta, \alpha})|_U, \cH\cH(\arb)|_U) \to \Hom((\arb_{\beta, \alpha})_{\T(\mf p)}, \cH\cH(\arb(\T(\mf p)))) \]
and a collection of such morphisms is a morphism of sheaves if and only if the diagram
\[\xymatrix{
(\arb_{\beta, \alpha})_{\T(\mf p)} \ar[r] \ar[d] & \cH\cH(\arb(\T(\mf p))) \ar[d] \\
(\arb_{\beta, \alpha})_{\T(\mf{p}')} \ar[r] & \cH\cH(\arb(\T(\mf p')))
}\]
commutes whenever $\T(\mf p)$ is in the closure of $\T(\mf p')$, i.e. whenever there's a map of correspondences $\mf p \to \mf p'$.

Let $\mf p_\text{max} = (R_\text{max} \overset{q}\twoheadleftarrow T \xrightarrow{\sim} T)$ be the ``maximal" correspondence defined by contracting the geodesic from $\beta$ to $\gamma=\alpha\vee\beta$. Then consider the composition of the trace morphism 
at the origin with taking sections on the stratum $\T(\mf p_\text{max})$:
\[ 
\xymatrix{
\Hom_{\arb(U)}(P_\alpha,P_\beta) \ar[r]^-{\Tr} \ar[rd]^{f} &
\mathrm{Hom}(\arb_{\beta, \alpha}|_U, \cH\cH(\arb)|_U) \ar[d] \\ &
\Hom(\Hom_{\vv R_\text{max}}(P_{q(\beta)},P_{q(\alpha)}), \HH(\Perf(\vv R_\text{max}))).
}\]

We know from the Verdier duality argument above that the top two are abstractly isomorphic and are either both zero or $k$.
The only nontrivial case to check is when $\alpha \leq \beta$, where they are both isomorphic
to $k$. So in that case it is enough to check the injectivity of
\[ f: \Hom(P_\alpha,P_\beta) \to \Hom(\Hom(P_{q(\alpha)},P_{q(\beta)}), \HH(\Perf(\vv R_\text{max}))). \]
But by definition this map factors through the isomorphism $\Hom_{\vv T}(P_\alpha,P_\beta) 
\to \Hom_{\vv R_\text{max}}(P_{q(\alpha)}, P_{q(\beta)})$.  Moreover, the map 
\[ \Hom_{\vv R_\text{max}}(P_{q(\alpha)}, P_{q(\beta)}) \to 
\Hom(\Hom_{\vv R_\text{max}}(P_{q(\alpha)},P_{q(\beta)}),  \HH(\Perf(\vv R_\text{max}))) \]
is the adjoint map to the trace pairing, so the morphisms of sheaves induces an isomorphism at the origin.
The calculation of the other stalks are analogous 
and can be obtained by a substitution of $\vv R$ for $\vv T$ and $\omega_{\R}[-d]$ for $\omega_{\T}[-d]$.
\end{proof}

\begin{example}
Consider the arboreal space $\T$ for the quiver $T = \alpha\rightarrow\beta\leftarrow\gamma$. Here
$\beta = \alpha\vee\gamma$, and the sheaves $\arb_{\alpha,\gamma}$ and 
$\arb_{\gamma, \alpha}$ are the constant sheaves respectively supported on the subsets
$\T(\alpha,\gamma)$ and $\T(\gamma,\alpha)$ of Figure \ref{fig:a3supportHom}, 
which are switched by Verdier duality (up to a shift).
\end{example}

\subsection{Global orientations} \label{sec:global}
Consider an arboreal space $(\X,\cW)$; by definition, around each point in $\X$ there is some local model for its neighborhood given by a generalized arboreal singularity.

\begin{proposition} \label{prop: w1}
There is a class  $w_1(\X, \cW) \in H^1(\X, \pm 1)$, which we will call the \emph{(first) Stiefel-Whitney class} of the arboreal space $(\X, \cW)$.
If the image of $w_1$ in $H^1(\X, k^\times)$ vanishes, then $(\X, \cW)$ is orientable. 
\end{proposition}
\begin{proof}
First, let us argue that it is sufficient to prove this result in the case where $(\X,\cW)$ is special arboreal, i.e., when every local model is a (non-generalized) arboreal singularity. By definition, every generalized arboreal singularity $\T^*$ comes with a given embedding into some arboreal singularity $\T^+$, which can be seen from its explicit description to be a homotopy equivalence. Moreover, by Lemma \ref{lem: orientation on quotient}, orientations on $\T^*$ are induced from orientations on $\T^+$.

Let us now assume that all our local models are special arboreal. From our calculation for the local models, we have learned that
for any arboreal space $(\X, \cW)$, the Hochschild homology sheaf $\cH\cH(\cW^{pp})$ is locally isomorphic to the (shifted) dualizing complex, and has trivial $S^1$-action.  

Note that the space of isomorphisms to the dualizing sheaf is a torsor for automorphisms of the dualizing sheaf; by Verdier duality
$\Hom(\omega_\X, \omega_\X)=k =H^0(\X, k)$; so the only automorphism is scalar multiplication.  It follows that any 
local isomorphism $\cH\cH(\cW^{pp}) \cong \omega_\X[-d]$ is a locally a scalar multiple of the isomorphism we have already constructed.
Nondegeneracy can be checked locally, so any such map is nondegenerate and gives an orientation. 

Thus the question of global orientability is the question of constructing a global isomorphism $\cH\cH(\cW^{pp}) \xrightarrow{\sim} \omega_\X[-d]$. 
It follows from the existence of local (degree zero) isomorphisms that for some locally constant rank one sheaf $\cL$, there is an isomorphism 
$\cH\cH(\cW^{pp}) \simeq \omega_\X[-d] \otimes \cL$.  In particular, sheaves locally isomorphic to $\omega_X[-d]$ are classified by $H^1(\X, k^\times)$.  

It remains to explain why the obstruction is the image of a class in $H^1(\X, \pm 1)$.  It will suffice to show that $\HH(\Perf(\vv T))$ is in fact the image
of a sheaf with $\Z$ coefficients. 

We write $K_0$ for the Grothendieck group of a category; 
given a sheaf of dg categories we may form the sheafified Grothendieck group similarly as we formed the sheafified Hochschild homology; we will
be interested in $\cK_0(\cW^{pp})$.  Recall that there is a natural transformation of functors $K_0 \to \HH$ given by the Dennis trace map from $K$-theory to 
Hochschild homology. By naturality this 
induces a morphism $\cK_0(\cW^{pp}) \to  \cH\cH(\cW^{pp})$. 

Consider now this map on a local model given by an arboreal singularity. We have $K_0(\Perf(\vv T)) = \Z^{|T|}$, and one can check that the induced map
\[ K_0(\Perf(\vv T)) \otimes_\Z k \to \HH(\Perf(\vv T)) \]
is an isomorphism. 
\end{proof}

Given the underlying topological space $\X$, let us consider the space of possible choices of different cosheaves $\cW'$ on $\X$ that endow it with the structure of a locally arboreal space. The space of such choices is a torsor over 
\[ H^1(\X, \cA ut(\cW)) \times H^2(\X, \cA ut(1_{\cW})) = H^1(\X, \Z) \times H^2(\X, k^\times). \]
There are no other terms because there is no higher local automorphisms of the cosheaf $\cW$, since
Hochschild cohomology of the tree quivers is just $k$ in degree zero and nothing else  \cite{Toe1}.  
Here, the fact that the connected components of the local automorphisms of $\cW$ are just the 
shift functor can be seen by observing that, for an arboreal singularity $\T$,  
the restriction from $\T$ to the smooth locus of $\T$ remembers the subcategories
generated by every indecomposable. 

Our $w_1(\X, \cW)$ is the reduction mod 2 of the above $H^1$ information; which takes values in a vector space
rather than a torsor because we have now the basepoint given by comparison with the dualizing sheaf.

\begin{remark}
	In the discussion above, we just abstractly described the space of choices for a global cosheaf given a fixed local structure. It turns out that these possibilities are realized in the context of Fukaya categories of Weinstein manifolds. For a choice of Lagrangian skeleton $L$, there is a certain space of data on which the Fukaya category depends, obtained by descent from classes in the Lagrangian Grassmannian; roughly, the first term $H^1(L,\Z)$ comes from the choice of grading, and the second term $H^2(L,k^\times)$ (or $\Z/2$ for integer coefficients) comes from the choice of orientation. A careful discussion of these facts is carried out in \cite[Sec.5.3]{GPS3} and \cite[Sec.10]{NS}.
\end{remark}

\begin{example}
Let us return to the example given in the introduction, for local systems on a manifold $M$ of dimension $d$. The space $(M, \cL oc^c)$ is an arboreal space, whose dual \textsc{csc} space is $(M,\cL oc^b)$. The sheaf $\cH\cH(\cL oc^b)$ is evidently a locally constant rank one sheaf; one may see it is constant e.g. because the constant local system gives a distinguished object on each open, the endomorphisms of which give a section of $\cH\cH(\cL oc^b)$. 

Thus an orientation in the categorical sense coincides with a nonzero section $k_M \to \omega_M[-d]$, i.e., an orientation in the classical sense. 
One can see similarly that the class $w_1(M, \cL oc^c)$ coincides with the usual Stiefel-Whitney class $w_1(M) \in H^1(M,\Z/2\Z)$. 
\end{example}

\section{Arboreal spaces and microlocal sheaf theory} \label{sec:microsheaves}
While we have given an entirely combinatorial study of arboreal spaces, their origins and applications lie in microlocal sheaf theory \cite{N3, N4}.  
Here we recall how microlocal sheaf theory gives rise to \textsc{ccc} spaces, and the role of arboreal singularities in the subject. 
We sketch a general theorem regarding the local existence of orientations, and give a detailed ad hoc construction in a special
case sufficient to treat the explicit examples of the following section.

\subsection{Microlocalization of sheaves}
Let $M$ be a manifold and $T^*M$ its cotangent bundle.  We write $\Sh(M, k)$ for the dg category of sheaves on $M$ valued 
in dg $k$-modules.  To any $\cF \in Sh(M, k)$ can be associated its {\em microsupport} $ss(\cF) \subset T^*M$, a conic 
coisotropic subset \cite{KS}.  Roughly speaking, $ss(\cF)$ measures the directions along which $\cF$ fails to be locally constant. 
In particular, local constancy of $\cF$ is equivalent to the condition that $ss(\cF)$ is contained in the zero section.  More generally, 
the condition that $\cF$ is constructible with respect to a subanalytic stratification is equivalent to the condition that $ss(\cF)$ is a 
subanalytic (singular) Lagrangian. 

Given conical $\Lambda \subset T^*M$, it is natural to study the category $\Sh_\Lambda(M, k)$ of sheaves microsupported
in $\Lambda$.  Here we will be exclusively interested in $\Lambda$ subanalytic Lagrangian, hence in certain full subcategories
of the category of constructible sheaves.

A fundamental result is that $\Sh_\Lambda(M, k)$ localizes, not only over $M$, but in fact over $\Lambda$.  One constructs
a sheaf of dg categories $\mu \Sh_\Lambda$ over $T^*M$ as follows:  sheafify\footnote{The $\mu \Sh_\Lambda(U)$ are presentable 
and the restriction maps  presheaf are continuous and cocontinuous, 
hence the presheaf can be viewed as valued
in $Pres^L$, $Pres^R$, or $Cat$.  In general one must discuss where to sheafify it, but for $\Lambda$ subanalytic Lagrangian (the only case of relevance to us) 
this question evaporates because the directed system whose colimit computes the stalks of $\mu Sh_{\Lambda}^{pre}$ has a cofinal and constant subsequence.} 
the presheaf
\[ \mu \Sh^{pre}_\Lambda(U) = Sh_{\Lambda \cup (T^*M \setminus U)}(M, k) / Sh_{T^*M \setminus U}(M, k). \]

In \cite{KS}, the authors study only $\mu \Sh^{pre}$, and especially its stalks.  As
stalks largely determine the behavior of the sheafification, the results there control $\mu \Sh$.  
In particular, the stalk of $\mu \Sh$ at a smooth Legendrian point of $\Lambda$ is noncanonically
isomorphic to $\vvt$; a choice of isomorphism is termed a microstalk functor. For $\Lambda$ subanalytic Lagrangian, it follows from the involutivity theorem \cite[Thm. 6.5.4]{KS} 
that $\mu \Sh_\Lambda(U)$ is in fact compactly generated by the microstalks at smooth Lagrangian
points of $\Lambda \cap U$.  That is, the sheaf $\mu \Sh_\Lambda$ is in fact valued in $\DGCat$. A key observation of \cite{N5} is that, because the restriction maps of $\mu \Sh_\Lambda$ are continuous
and cocontinuous, upon passing to left adjoints we obtain a cosheaf $\mu \Sh_\Lambda$ 
valued not just in $\DGCat$ but in fact in the image of $\Ind: \dgcat \to \DGCat$.  Thus we may pass to 
the sub-cosheaf of compact objects, 
$\mu \Sh_\Lambda^c$.  

In short, we have explained that microlocal sheaf theory associates a \textsc{ccc} space $(\Lambda, \mu \Sh_\Lambda^c)$ 
to a conical Lagrangian $\Lambda \subset T^*M$.  Restricting to the cosphere bundle, we see similarly that 
Legendrians in $S^*M$ acquire naturally the structure of \textsc{ccc} space; henceforth we work with Legendrians.  (Any exact Lagrangian
$\Lambda \subset T^*M$, in particular conical Lagrangians, are naturally identified with Legendrians in $S^* (M \times \R)$, 
compatibly with $\mu \Sh$, so this is no loss.) 

The theory of contact transformations from \cite[Sec. 7]{KS} implies that $\mu \Sh_\Lambda$ is locally independent of
the embedding $\Lambda \subset S^*M$.  More is true \cite{Sh}, \cite[Sec. 10]{NS}: in fact a subanalytic space $\Lambda$ with the germ of a Legendrian
embedding carries a stable symplectic normal bundle, giving rise to a `Maslov obstruction' in $\Hom(\Lambda, B^2 Pic( \Z)) = H^2(\Lambda, \Z) \times H^3(\Lambda, \Z/2\Z)$.
Trivializations are a torsor for $H^1(\Lambda, \Z) \times H^2(\Lambda, \Z/2\Z)$, and a choice of trivialization $\tau$ determines
a \textsc{ccc} space $(\Lambda, \mu \Sh_{\Lambda; \tau})$.  

Such a trivialization $\tau$, though an element of a torsor,  has however a well defined image  $w_1(\Lambda, \tau) \in H^1(\Lambda, \Z/2\Z)$ 
having to do with whether the structure in question, a priori a trivialization of an obstruction factoring through the 
Lagrangian Grassmannian, in fact trivializes a more refined obstruction which factors through the oriented Lagrangian Grassmannian. 

In case of a Legendrian embedding $\Lambda \subset S^*M$, there is a canonical trivialization $\tau_0$ coming
from the cosphere foliation. For this trivialization $\tau_0$ the $(\Lambda, \mu \Sh_{\Lambda; \tau_0})$ agrees (essentially by definition) 
with the original $(\Lambda, \mu \Sh_\Lambda)$.  If $M$ is oriented, then the cospheres give an oriented Lagrangian
foliation, so $w_1(\Lambda, \tau_0) = 0$. 

\subsection{Arboreal singularities, resolution, and the existence of orientations} \label{sec: general case}

While our previous discussion of arboreal spaces was purely topological, there is also a geometric
notion \cite{N3} of a Legendrian arboreal singularity.  An elegant formulation from \cite{Sta} is that
special arboreal singularities (corresponding to trees with no marked leaves) 
are the smallest class of Legendrian singularities containing smooth spaces, and 
such that if $\Lambda \subset S^*M$ is arboreal and transverse to the cosphere foliation, then 
$M \cup Cone(\Lambda^\epsilon)$ is arboreal where $\Lambda^\epsilon$ is a generic contact perturbation
of $\Lambda$.  

The main results of \cite{N3, N4} are: 

\begin{theorem} \cite{N3} Let $\vv T$ be a rooted tree.  Then there is a map  
$\iota_{\vv T} : \T \hookrightarrow T^\infty \R^T$ with subanalytic legendrian image and an isomorphism of \textsc{ccc} spaces 
$(\iota_{\vv T}(\T), \mu \Sh_{\iota_{\vv T}(\T)}^c) \cong (\T, \cA)$.  The same holds in the setting of rooted trees with marked leaves \cite{N4}. 

Conversely, if $\Lambda \subset S^*M$ is geometrically arboreal, then $(\Lambda, \mu \Sh_\Lambda)$ is an arboreal space. 
\end{theorem} 

\begin{theorem} \cite{N4} 
For any $\Lambda \subset S^*M$, there is an arboreal space $(\X, \cW)$, a proper stratified map $\pi: \X \to \Lambda$, which is 
moreover a homotopy equivalence, 
and an isomorphism $ \mu \Sh_\Lambda^c \cong \pi_* \cW$. 
\end{theorem}

\begin{remark}
In \cite{N4}, the theorem is stated locally at one point in $\Lambda$, but the construction (iteratively replacing strata of the front projection of $\Lambda$ by the boundaries
of their tubular neighborhoods) is easily seen to make sense for any Legendrian which is finite over its (assumed Whitney stratifiable) front projection. 
It is standard that subanalyticity of $\Lambda$ implies Whitney stratifiability both of the Legendrian and the front projection, which therefore
becomes finite after a generic perturbation.  Then the considerations of \cite{N4} serve to construct the morphism of sheaves
noted above 
(alternatively the general microlocal nearby cycle construction of \cite{NS} could be used) and the result stated in \cite{N4} checks that this morphism is an isomorphism
at stalks, hence is an isomorphism. 
\end{remark}

We use these results to deduce the existence of orientations on general microlocal
sheaf categories.  Because Nadler's construction of arborealization takes place in the cotangent bundle 
(where there is a canonical trivialization of the Maslov obstruction), it follows that 
the twisting of the arboreal space $(\X, \cW)$ is classified (in the sense 
described in the discussion after Proposition \ref{prop: w1}) by the pullback to $\X$ of the 
element $\tau$ in the torsor of Maslov data $H^1(\Lambda, \Z) \times H^2(\Lambda, \Z/2\Z)$; similarly
$w_1(\X, \cW)$ and $w_1(\Lambda, \tau)$ are identified under $ H^1(\X, \Z/2\Z) \cong H^1(\Lambda, \Z/2\Z)$. 

\begin{theorem} \label{thm: mush cy}
If $M$ is orientable, then for any subanalytic Legendrian $\Lambda \subset S^*M$ (or conic Lagrangian $\Lambda \subset T^*M$), the \textsc{ccc} space 
$(\Lambda, \mu \Sh_\Lambda^c)$ is orientable.

More generally, for any germ of Legendrian $\Lambda$ and any Maslov datum $\tau$, 
If $w_1(\Lambda, \tau)$ vanishes, then $(\Lambda, \mu \Sh_{\Lambda; \tau}^c)$ is orientable. 
\end{theorem} 
\begin{proof}
Take an arboreal resolution; some arboreal space $(\X, \cW)$ and $\pi: \X \to \Lambda$ with $\pi_* \cW \cong \mu \Sh_{\Lambda; \tau}^c$.  
As noted above,  $w_1(\Lambda, \tau) = w_1(\X, \cW)$.  If this vanishes (which it always does if 
$\tau=\tau_0$ is induced from the cosphere polarization of an orientable manifold) then $(\X, \cW)$ has some orientation, 
$\Omega$.  By Theorem \ref{thm:cccCYpushforward}, we have that $\pi_* \Omega$ orients $(\Lambda, \mu \Sh_{\Lambda; \tau}^c)$. 
\end{proof} 

\begin{remark}
It is presumably possible to give a proof of Theorem \ref{thm: mush cy} without appeal to arborealization.  A hint
in this direction is the result \cite[Prop. 8.4.14]{KS} that 
$\D'_{T^*M} \mu hom(F, G) \cong  \mu hom(G, F) \otimes \pi_M^* \omega_M$: if $M$ is orientable, this looks like a 
verification that some candidate orientation on a \textsc{csc} space is nondegenerate.  Likely the orientation in question 
can be extracted from \cite{KS-hochschild}: one would have to understand how the notion of microlocal Hochschild
homology introduced there is related to the sheafified Hochschild homology of $\mu \Sh_\Lambda^{pp}$ (and how the cyclic
structures behave).  

Possibly similar ideas could also be used to construct the corresponding orientation on the \textsc{ccc} space, 
though note $\mu \Sh_\Lambda^c$ is {\em not} generally locally saturated, 
so one would have directly check the conditions of Definition \ref{def:cccCY}; the approach via arborealization avoids this
(ultimately by appeal to Theorem \ref{thm:checkOnProper}). 
\end{remark}

\subsection{Immersed front projections} 
Here we give a more hands-on construction of orientations in a very special case, which nevertheless suffices to treat 
the examples we present in the following section. 
We say that a Legendrian $\Lambda \subset T^\infty M$ has an immersed front projection when the projection $\Lambda \to M$ is an immersion. 
In this case, there is a natural identification of the Kashiwara-Schapira sheaf with the category of local systems on $\Lambda$, i.e., 
$\mu\Sh_\Lambda = \cL oc_\Lambda$.  For this result in the embedded case, see \cite[Chap. 4]{KS}; the immersed case is no different, 
an explicit discussion can be found in \cite{STW}.   
 
We say that the front projection has {\em normal crossings} when it is locally diffeomorphic to 
a union of coordinate hyperplanes. 

\begin{lemma} Let $M$ be a manifold, $\Lambda \subset T^\infty M$ a smooth Legendrian with normal crossings projection. Let $m \in M$ be a point where 
the front projection of $\Lambda$ is immersed with normal crossings image.  Let $U$ be a conical neighborhood of $m$ in 
$\X = M \cup \R_+ \Lambda$.  Then there is a canonical isomorphism of $(U, \mu\Sh^c)$ with (a trivial factor times) an arboreal singularity 
corresponding to a star quiver, given by a single root vertex with as many leaves as there are points of $\Lambda$ over $m$. 
\end{lemma}
\begin{proof}
Let $\pi:T^\infty M \to M$, and $d=\dim M$. Around a smooth point of $\pi(\Lambda)$ \emph{on} the base $M$, $U$ is homeomorphic to $\R^d$ with a $d$-dimensional half-space glued to a coordinate hyperplane, which is homeomorphic to the local model $\A_2 \times \R^{d-1}$.

\begin{figure}[h]
    \centering
    \includegraphics[width=0.7\textwidth]{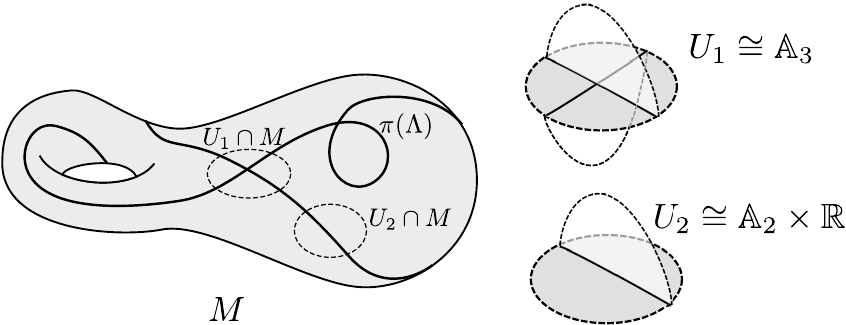}
         \caption{The case $d=2$. The neighborhood of a point in the front projection of the Legendrian $\Lambda$ is homeomorphic to either
         $\A_2 \times \R = \S tar_1 \times \R$ or $\A_3 = \S tar_2$. Note that in general it is the star quivers $Star_k$ that appear, and not the 
         $A_k$ series.}
    \label{fig:normalcrossings}
\end{figure}

Similarly, around a singular point of $\pi(\Lambda)$ by the normal crossings condition, $U$ is homeomorphic to $\R^d$ with $k$ half-spaces glued along $k$ coordinate hyperplanes. Consider now the arboreal singularity $\S tar_k$ corresponding to the star quiver with $k$ leaves; this is homeomorphic to $\R^k$ with $k$ half-spaces glued along coordinate hyperplanes. By comparison we have a homeomorphism $U \simeq \S tar_k \times \R^{d-k}$.
\end{proof} 

Thus $(\X, \mu\Sh^c_\X)$ is a locally arboreal space.  We check that a sufficient criterion for its orientability is the orientability of the base manifold $M$.

\begin{theorem}\label{thm:legendrian} 
The arboreal space $(\X, \mu\Sh^c_\X)$ admits an orientation if and only if $M$ is orientable; in that case it extends the orientation on $M$, i.e. which agrees with the given orientation on $M$ when restricted to the smooth locus of $\X$. 
\end{theorem}

\begin{proof}
Theorem  \ref{thm:checkOnProper} reduces us to consider $\mu \Sh = (\mu\Sh^c)^{pp}$. There is an embedding $\cL oc^b_M \hookrightarrow \mu\Sh$; in terms of the local models $U$, this identifies the locus $U \cap M$ with the disc corresponding to the root vertex $\rho$ of the quiver, and the embedding functor is the embedding of $\Perf$ as the subcategory spanned by the corresponding projective (and simple) object $P_\rho$.

This embedding is an isomorphism of sheaves of dg categories on the top-dimensional locus (away from the projection of $\Lambda$), and moreover the embedding of $M$ is a homotopy equivalence. Under the isomorphism induced in homology we have $w_1(\X,\mu\Sh^c_\X) = w_1(M,\cL oc^c)$, which vanishes if $M$ is orientable.
\end{proof}

\begin{figure}[h]
    \centering
    \includegraphics[width=0.7\textwidth]{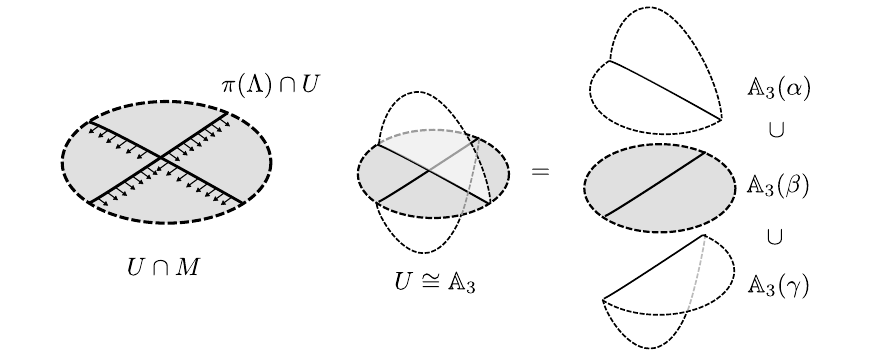}
         \caption{The coorientation of the front projection $\pi(\Lambda)$ defines a decomposition of $U \simeq \A_3$ into discs $\A_3(\alpha)$
         labeled by each vertex of $A_3$. The choice of rooting is given by picking the center vertex of the quiver (which labels the disc
         along $M$) to be the root.} 
    \label{fig:normalcrossings2}
\end{figure}

\section{Examples} \label{sec:apps} 
In this section we recall various categories which arise as the global sections of a sheaf or cosheaf of dg categories on an oriented locally arboreal space. As a consequence of our results, we conclude the existence of certain absolute and relative Calabi-Yau structures.

To relate this discussion to the subject of shifted symplectic geometry, we recall the following theorem.
\begin{theorem}\cite[Thm.5.5]{BD2}
Let $C$ be a smooth dg category with a $d$-dimensional smooth Calabi-Yau structure. Then there is an induced $(2-d)$-shifted symplectic structure on its moduli space of objects $\cM_C$. Moreover, if $f:C \leftrightarrows D: f^r$ is a continuous adjunction of smooth dg categories with a relative smooth Calabi-Yau structure on $f$ which is compatible with the (absolute) structure on $C$, the induced map $\cM_D \to \cM_C$ has a Lagrangian structure.
\end{theorem}
In the theorem above, $\cM_C$ denotes the `moduli space of objects' of $C$ as described in \cite{TVa}.

\begin{remark}
There were similar results previously proven in the literature; in particular we have \cite[Theorem 1.7]{Toe3} which proves the existence of the derived Lagrangian structure in the case that $\cA$ and $\cB$ are smooth and proper, which is not directly applicable to this paper since the categories of interest are smooth but not proper in general. Another result of relevance is \cite[Theorem 1.19]{Yeu} which applies to dg categories ``concentrated in non-negative degrees'' and constructs shifted symplectic structures on their derived moduli of representations, which is the substack of the moduli $\cM_\cA$ representing objects of zero Tor amplitude.
\end{remark}

Using the theorem above together with our result about orientations gives the following corollary.
\begin{corollary}
Let $X$ be a stratified space with compact boundary $\del X$, equipped with a locally saturated cosheaf $\cW$. Then an orientation of dimension $d$ on $\cW^{pp}$ determines a $(3-d)$-shifted symplectic structure on $\cM_{\cW(\del X)}$ and a Lagrangian structure on the morphism $\cM_{\cW(X)} \to \cM_{\cW(\del X)}$.
\end{corollary}

Therefore constructing orientations will then give rise to shifted symplectic structures, Poisson structures, quantizations, etc. in the appropriate circumstances \cite{PTVV, CPTVV}.  Direct application of our methods yield constructions in the ``type A'' cases;
e.g. the moduli space for a point is $\cM_{\Perf}$ inside which one can find the various $BGL_n$.  We expect that the desired symplectic
structures for the analogous moduli spaces for other groups can be constructed via Tannakian considerations as in \cite[Sec. 6]{Sim}, but
do not develop this in detail here.

\begin{remark}
In this section to ease notation we will denote the moduli stack of objects using a blackboard bold capital to distinguish it from the category, e.g.
the category of perfect complexes is denoted $\Perf$ and the moduli stack of perfect complexes is $\PPerf = \cM_{\Perf}$, same for local systems $\Loc$ and $\LLoc$.
\end{remark}

\subsection{The associated graded of a filtration} \label{sec:filt}

Let $\X$ be a comb: a one-dimensional space formed as the union of $\R$ and the positive cone 
on some $n$ points $\{p_i\}$ at positive contact infinity.  The category $\Sh_{\{ p_i \}} (\R)$ is equivalent
to the category $\FFilt_n$ of $n$-step filtered perfect complexes, which just means sequences of perfect complexes
\[ F_0 \to F_1 \to \ldots \to F_n. \]

\begin{corollary}
The functor
\begin{eqnarray*}
\Filt_n & \to & \Perf^{\otimes(n+1)}  \times \Perf \\
F_0 \to \ldots \to F_n & \mapsto & (F_0, Cone(F_0 \to F_1), Cone(F_1 \to F_2), \ldots, 
Cone(F_{n-1} \to F_n)), \, F_n
\end{eqnarray*}
has a relative proper Calabi-Yau structure, coming from an orientation of dimension 1, so the corresponding map of moduli spaces of objects
\[ \FFilt \to \PPerf^{\times (n+1)} \times \PPerf \]
is a Lagrangian mapping to a $2$-shifted symplectic space.
\end{corollary}
\begin{proof}
This follows immediately from Theorem \ref{thm:legendrian}.  Alternatively, the fact that the comb is orientable follows from its being contractible.
\end{proof}

\begin{remark}
Note that following the description in \ref{thm:legendrian}, the decomposition into ``discs" (here intervals) is such that the boundary
of the $i$th interval is given by the endpoints $(+\infty) - (p_i)$, if we pick the positive orientation on the base manifold $\R$. So in the 
Lagrangian map of moduli stacks the first factor $\PPerf^{\times (n+1)}$ is endowed with the opposite 2-shifted symplectic
structure, where the last factor $\PPerf$ has the usual 2-shifted symplectic structure. 
\end{remark}

\begin{figure}[h]
    \centering
    \includegraphics[width=0.7\textwidth]{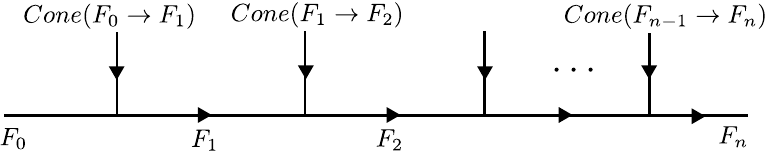}
         \caption{The comb $\X$. We pick a decomposition of $\X$ into overlapping intervals; in this
         case the $i$th interval has endpoints $p_i$ and $+\infty$, and extending the positive orientation on $\R$, its boundary is
         $(+\infty) - (p_i)$.}
    \label{fig:comb}
\end{figure}

Restricting to the open substack $\FFilt_n^\circ$ where all $F_i$ and all $Cone(F_i \to F_{i+1})$ can be represented by vector spaces
in degree zero, we see that the image of the morphism above in each factor
$\PPerf$ lands in some substack $\B GL_{m_i} \subset \PPerf$, and moreover that each map $F_i \to F_{i+1}$ is injective.
Denoting $m_0 = \dim F_0, m_i = \dim Cone(F_i \to F_{i+1})$ and $m = \dim F_n$, this implies that
$m_0 + \dots + m_n = m$. So $\FFilt_n^\circ$ splits into components labeled by
collections of integers $m,\{m_i\}$, and each component has a 2-shifted Lagrangian morphism
\[ \FFilt_n^\circ(m,\{m_i\}) \to (\B GL_{m_0}\times \dots \times \B GL_{m_n}) \times \B GL_m, \]
which can be interpreted as a $2$-shifted Lagrangian correspondence between $\B GL_m$ and $\B L$ for
the Levi subgroup corresponding to the partition $\{m_i\}$. This was originally shown in \cite{Saf2}, where a Lagrangian correspondence
\[\xymatrix{
 & \B P \ar[dl] \ar[dr] & \\
 \B G & & \B L
}\]
is used to define the ``partial group-valued symplectic implosion". Note that once we fix the decomposition $m = m_0 + \dots + m_n$, we can identify
the substack $\FFilt_n^\circ(m,\{m_i\})$ as the classifying space $\B P$ for the parabolic $P$ corresponding to $L$: a map from some other space $\X \to \FFilt_n^\circ(m,\{m_i\})$ determines an invariant filtration of the vector space $k^m$, therefore up to equivalence it is the data of a $P$-bundle over $\X$.

\subsection{Invariant filtrations near punctures on surfaces}
Let $\Sigma$ be an oriented surface with boundary consisting of $n$ circles; draw a collection of $m_i$ concentric circles at each boundary component, 
and choose co-orientations and therefore Legendrian lifts.  Let $\Lambda$ denote the union of these lifts. 
Just as in the previous subsection, the corresponding
category $Sh_\Lambda(\Sigma)$ amounts to the category of sheaves on $\Sigma$ with invariant filtrations at 
the punctures. Note that $\Lambda \cup \del \Sigma$ is just a union of $m = \sum m_i$ circles.  

\begin{corollary} \label{cor:loclagrangian}
The morphism $\SSh_\Lambda(\Sigma) \to \LLoc(\Lambda \cup \del \Sigma) = \LLoc(S^1)^{\times m}$ is $1$-shifted Lagrangian.
\end{corollary} 
\begin{proof}
This follows immediately from Theorem \ref{thm:legendrian}.
\end{proof} 

Again, we can take the substack $\SSh_\Lambda(\Sigma)^\circ$ on which the restriction of the sheaf to $\Sigma$ has cohomology concentrated
in degree zero; the resulting space is an Artin stack is the classical sense, and as in the comb example above, it has components labeled 
by the microlocal ranks along the boundary components; the morphism splits into components
\[ \SSh_\Lambda(\Sigma)^\circ(\{m_i\}) \to \left[\frac{GL(m_1)}{GL(m_1)}\right] \times \dots \times \left[\frac{GL(m_n)}{GL(m_n)}\right], \] 
where the stacky quotient is taken with respect to the adjoint action.

To get a space with a symplectic structure, one can choose another $1$-shifted Lagrangian morphism to the 
moduli space of local systems around the boundary. Performing Lagrangian intersection between the two $1$-shifted Lagrangians gives a $0$-shifted symplectic space. As observed in \cite{Cal}, this intersection of shifted Lagrangians inside of $[\frac{G}{G}]$ recovers the construction of symplectic spaces by quasi-Hamiltonian reduction of \cite{AMM}; see also \cite{Saf1}.

\begin{corollary}
The moduli space of local systems on a surface equipped with invariant filtrations at the punctures, of which 
the conjugacy classes $C_i$ of the associated graded holonomies are pre-specified, carries a $0$-shifted symplectic
structure. 
\end{corollary}
\begin{proof}
This can be deduced from Corollary \ref{cor:loclagrangian} by observing that fixing a conjugacy class $C$ in a reductive group $G$ 
determines a Lagrangian morphism $[\frac{C}{G}] \to [\frac{G}{G}]$. Performing Lagrangian intersection between $\SSh_\Lambda(\Sigma)^\circ(\{m_i\})$ and
$[\frac{C_1}{GL(m_1}] \times \dots \times [\frac{C_n}{GL(m_n)}]$ gives the symplectic structure on the moduli space with prescribed holonomies.
\end{proof}

\begin{figure}[h]
    \centering
    \includegraphics[width=0.5\textwidth]{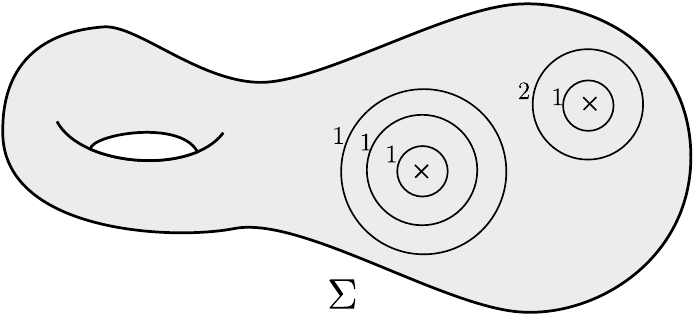}
         \caption{On the surface $\Sigma$ with two punctures, and with fixed microlocal rank along the concentric circles of $\Lambda$.
         If we look at the substack of $Sh_\Lambda(\Sigma)$ of objects with rank 0 at the punctures, the microlocal rank conditions
         mean we have rank 3 local systems equipped with invariant filtrations near each puncture; in this particular case we have two filtrations
         respectively of the form $0\subset k\subset k^2 \subset k^3$ and $0\subset k \subset k^3$}
    \label{fig:invfiltrations}
\end{figure}

\begin{example}
Consider $\Sigma$ an oriented surface with boundary $\del \Sigma = $ union of $n$ circles, without any Legendrians.
The Artin stack $\SSh(\Sigma)^\circ$
has disjoint components $\LLoc_{GL(m)}(\Sigma)$ labeled by the rank $m$ of the local system, and so we have 1-shifted Lagrangian morphisms
\[ \LLoc_{GL(m)}(\Sigma) \to (\LLoc_{GL(m)}(S^1))^{\times n} = \left[\frac{GL(m)}{GL(m)}\right]^{\times n} \]
and picking $n$ conjugacy classes $C_i$ in $G$, we can perform the intersection and get the so-called ``tame" character variety of $\Sigma$.
\end{example}

\begin{example}
Take $\Sigma$ to be the open cylinder with $n$ concentric circles around one of the boundary components, and no circles around the other. 
Fixing the ranks $m = m_1 + \dots + m_n$ at each boundary, we get a 1-shifted Lagrangian morphism
\[ \SSh_\Lambda(\Sigma)^\circ \to \LLoc_{GL_m}(S^1) \times \LLoc_{GL_{m_1}}(S^1) \times \dots \times \LLoc_{GL_{m_n}}(S^1) = 
\left[\frac{G}{G}\right] \times \left[\frac{L}{L}\right], \]
where $L$ is the Levi subgroup corresponding to the partition $\{m_i\}$. This example also appears in \cite{Saf2}, as a 1-shifted Lagrangian correspondence
\[\xymatrix{
 & \left[\frac{P}{P}\right] \ar[dl] \ar[dr] & \\
 \left[\frac{G}{G}\right] & & \left[\frac{L}{L}\right],
}\]
where the identification of our space with $[\frac{P}{P}]$ comes from the observation that an invariant filtration on $\Sigma$ is the same data as a $P$-local system on $S^1$. 

\begin{figure}[h]
    \centering
    \includegraphics[width=0.3\textwidth]{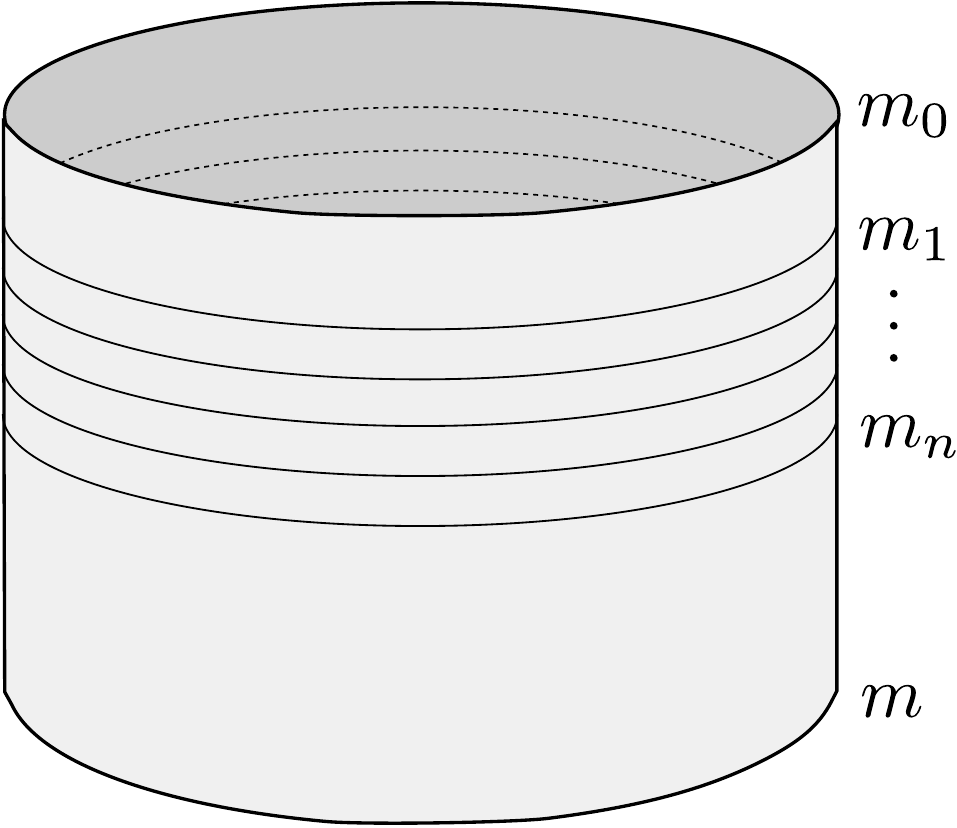}
         \caption{The restriction to the upper boundary components can be assembled into a map to $[\frac{L}{L}]$ where $L$ is a Levi subgroup
         of $G = GL_m$ given by the integers $\{m_i\}$. The restriction to the lower boundary components is a map to $[\frac{G}{G}]$ giving 
         the monodromy of the $G$-local system}
    \label{fig:cylinder}
\end{figure}

Another way of deducing this particular case is by writing $\Sigma$ as the product of a circle and a comb (as in \ref{sec:filt}). Therefore we have an equivalence of derived stacks
\[ \SSh_\Lambda(\Sigma) = \R Map(S^1, \FFilt_n) \]
and the Lagrangian correspondence is obtained from the correspondence in section \ref{sec:filt} by applying the mapping stack functor $\R Map(S^1,-)$, which shifts the degree down to the 1-shifted Lagrangian correspondence above.

\begin{remark} In the rank one case, when $k_i=1$, the correspondence $[\frac{B}{B}] \to [\frac{G}{G}] \times [\frac{T}{T}]$ is a group version of the Grothendieck-Springer correspondence, and we have $[\frac{B}{B}] \simeq [\frac{\tilde G}{G}]$ where $G$ acts on the Springer resolution $\tilde G$ by conjugation.
\end{remark}

\end{example}

\subsection{Stokes filtrations near punctures on surfaces}
Rather than take an invariant filtration around a puncture, 
one can allow the filtration itself to undergo monodromy.  The resulting notion generalizes 
the notion of Stokes structure; we will call it a Stokes filtration.  In most works,
this was presented as suggested above: in 
terms of a sheaf on the boundary circle equipped with a filtration that itself varies.  This notion can
be found e.g. in \cite{Mal}.  Defining what precisely it means for a filtration to vary along a circle 
is nontrivial and somewhat mysterious at the points where the steps in the filtration cross. 

We prefer to turn this notion sideways: rather than a filtered sheaf on $S^1$ with varying filtration, we take
a sheaf on $S^1 \times \R$ with microsupport in a prescribed Legendrian braid closure.  This determines a 
filtration in the $\R$ direction, just as in the previous examples; as it happens, the above notion exactly
captures at the crossings the notion in \cite{Mal} of Stokes filtration.  This idea 
seems to have been known to the experts, but we have not found any systematic exposition of it in the classical 
literature.\footnote{More precisely, we only know the following other occurrences 
of this picture. In \cite{DMR}, there is a letter from Deligne in which the Stokes sheaf is viewed
as a sheaf on an annulus rather than a filtered sheaf on a line. 
 In \cite{KKP}, the idea that a Legendrian knot can be associated to a Stokes
filtration appears as a remark.    Finally, the drawing of at least the projection of a knot already appears in the original work of Stokes \cite{Sto}.}

An attempt at this exposition appears in \cite[Sec. 3.3]{STWZ}.  Here we simply recall that the Deligne-Malgrange
account of the irregular Riemann-Hilbert correspondence on Riemann surfaces can be formulated as follows.  
Suppose we are given a Riemann surface $\Sigma$ with marked points $p_i$, and a specification of a (possibly ramified) irregular type
$\tau_i$, i.e., formal equivalence class of irregular singularity, at each.  Then there is an associated
Legendrian link $\Lambda = \coprod \Lambda(\tau_i)$, a union of links localized near the $p_i$, and an equivalence
of categories between the irregular connections with these singularities, 
and the full subcategory of 
$\Sh_\Lambda(\Sigma)$ on objects which have cohomology concentrated in degree zero, and appropriate rank
stalks and microstalks.  The corresponding component of the moduli space is the
moduli space of Stokes data. Finally, the microlocal restriction morphism
$\Sh_\Lambda(\Sigma) \to \Loc(\Lambda)$ is what would have classically been called ``taking the formal monodromies''. 

\begin{corollary}
The morphism from a moduli space of Stokes data to the moduli space of formal monodromies is 1-shifted-Lagrangian.
\end{corollary} 

\begin{corollary}
A moduli space of Stokes data with formal monodromies taking values in prescribed conjugacy classes 
is 0-shifted-symplectic.  In
particular, any open substack which happens to be a scheme is symplectic in the usual sense. 
\end{corollary}

This recovers and generalizes all constructions of symplectic structures in e.g. \cite{Boa1, Boa2, BY, Mei} for $GL_n$ connections.  

\begin{example}(Wild character variety)
Consider a disc $\Sigma$ punctured at the origin and a trivial rank $n$ vector bundle $E\to\Sigma$, where the origin is marked with the irregular type
\[ Q(z) = \frac{A}{z^r}, \qquad r \in \Z_+, A \in \mf t^{reg} \subset \mf{gl}_n \]
i.e., $A$ has all distinct eigenvalues. A meromorphic connection $\nabla$ on $E$ has this irregular type if it can be brought by a local analytic gauge transformation to the connection defined by the connection one-form $dQ$.

By the irregular Riemann-Hilbert correspondence, the category of meromorphic connections on the trivial vector bundle with this irregular type is equivalent to the category of Stokes data. From the Stokes data, one can recover the monodromy of $\nabla$ and the formal monodromy; this latter is an element of the centralizer $Z_G(A)$ (in this case, the maximal torus $T$) up to conjugation. 

The moduli of Stokes data with fixed monodromies and formal monodromies is commonly known as the \emph{wild character variety}.
Upon fixing conjugacy classes $C_G,C_T$ in $G$ and $T$, the wild character variety can be described as a
quasi-Hamiltonian quotient \cite{Boa2,Boa3}
\[ (G \times T \times (U_+ \times U_-)^r) \sslash_{C_G,C_T} (G\times T) \]
where $U_\pm$ are the unipotent subgroups corresponding to the maximal torus $T$. The moment map to $G$ is taking the monodromy around the singularity, and the map to $T$ is taking the formal monodromy.

In our description, this category of Stokes data becomes a full subcategory of the category $Sh_\Lambda(\Sigma)$ of microlocal sheaves, for a 
corresponding Legendrian link $\Lambda \subset T^\infty$ around the singularity. The monodromy and formal monodromy then become literal
monodromies of the local systems one gets by restriction to the boundary.
To explicitly construct $\Lambda$, one can follow the prescriptions in \cite[Sec. 3.3]{STWZ}. This can be heuristically stated in terms of the
asymptotics of flat sections, i.e. the growth behavior of the solutions to
\[ \frac{df}{dz} = \frac{dQ}{dz} f(z). \]
In this case, the solutions are spanned by $n$ different solutions $f_i \sim \exp(\lambda_i z^{-r})$, where $\lambda_i$ are the eigenvalues of $A$, and we only keep the exponential part of the asymptotics. The Stokes phenomenon refers to the fact that these correspond to asymptotics of solutions in different sectors; as we go from one sector to the other, the growth of these solutions changes. On each sector, we draw concentric strands for the $f_i$, ordering them by growth: the faster-growing ones further from the origin. Whenever we cross a Stokes ray, where the solutions $f_i, f_j$ switch growth asymptotics, 
we introduce a crossing between the $i$ and $j$ strands.

\begin{figure}[h]
    \centering
    \includegraphics[width=0.3\textwidth]{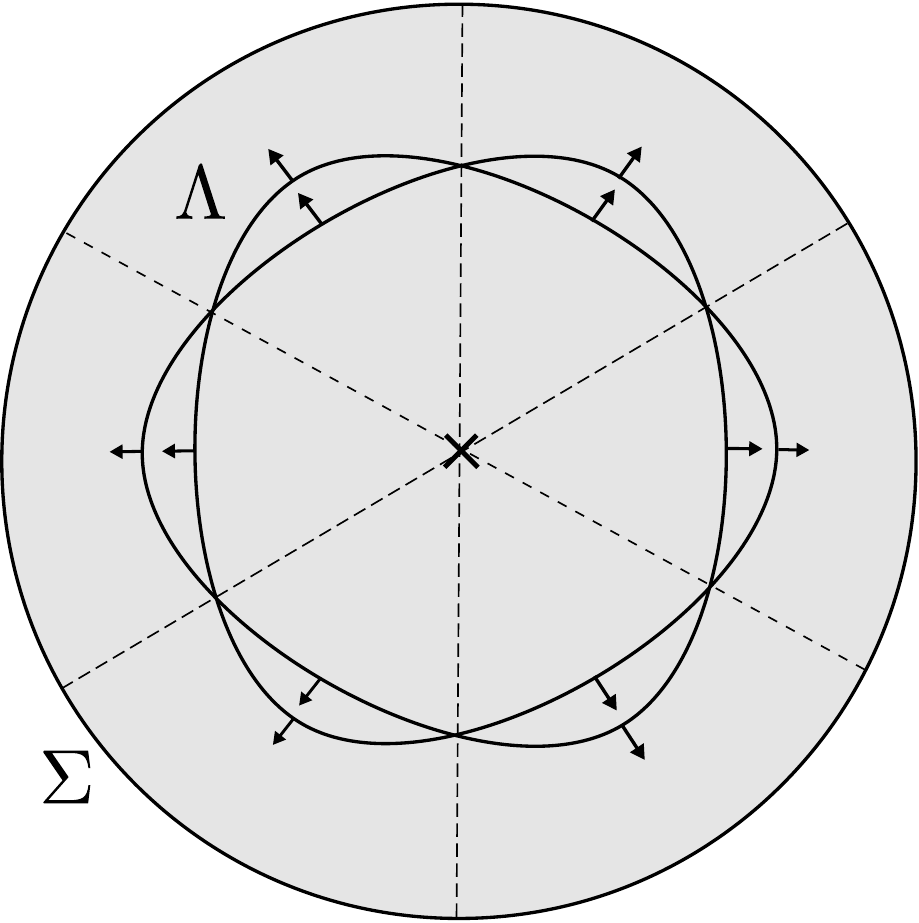}
         \caption{The Legendrian link for the irregular type \usebox{\smlmat} $\frac{1}{z^3}$.
          The Legendrian link $\Lambda \subset T^\infty \Sigma$ is obtained by 
         lifting the projection using the outward coorientation. The dashed lines are the Stokes rays, where the asymptotics of formal solution
         changes. Note that each component of the link is unknotted with itself: this is true of all the examples of this form.}
    \label{fig:wildcharvar}
\end{figure}

In the case we described above ($A \in \mf t^{reg}$), the corresponding Legendrian link $\Lambda$ is the closure of a $(n,2r)$ braid , cooriented outward, where we enforce the condition that the rank of the stalk inside $\Lambda$ is zero, and the microlocal ranks on each component of $\Lambda$ is one. What we call the moduli of Stokes data $\cM_\Lambda$ is the moduli of objects in the full subcategory of $Sh_\Lambda(\Sigma)^\circ$ with those rank conditions. The maps given by restriction to the boundary components can be assembled into a 1-shifted Lagrangian map
\[ \cM_\Lambda \to \left[\frac{G}{G}\right]\times \left[\frac{T}{T}\right] \]
where $[\frac{G}{G}]$ is $G$-local systems on the boundary of the disc, and $[\frac{T}{T}]$ is rank one local systems on $\Lambda$.
In this description, taking quasi-Hamiltonian quotient corresponds to taking intersection with another Lagrangian $[\frac{C_G}{G}] \times [\frac{C_T}{T}]$. The explicit description of the moduli space $\cM_\Lambda$ can be obtained by following the prescriptions in \cite{STZ}; one can check that this stack can in fact be expressed as the quotient $[(G \times T \times (U_+ \times U_-)^r) / (G\times T)]$, agreeing with the previously existing description.
\end{example}

\begin{example}
With the same notation of the previous example, consider the irregular type
\[ Q = \frac{A}{z^{r/2}}, \qquad A \in \mf t^{reg} \subset \mf{gl}_n \]
where $r$ is some odd number. Following the discussion in the last example, we get $n$ solutions
$f_i \sim \exp(\lambda_i z^{-r}) $
The strands $i$ and $j$ will cross whenever $f_i$ and $f_j$ ``switch" growth asymptotics. Suppose for instance that $\lambda_i,\lambda_j \in \R$. Writing $z = R e^{i\theta}$ the asymptotics will switch whenever $Re(z^{-r/2}) = 0$, i.e. on the rays
\[ \theta = \frac{\pi}{r} + \frac{2n\pi}{r}. \]
There are $r$ such rays between any pair $i,j$ even if $\lambda_i,\lambda_j \notin \R$: the expression for the rays is more complicated by the number of rays doesn't change. Therefore $\Lambda$ is the closure of a $(n,r)$ braid.

\begin{figure}[h]
    \centering
    \includegraphics[width=0.3\textwidth]{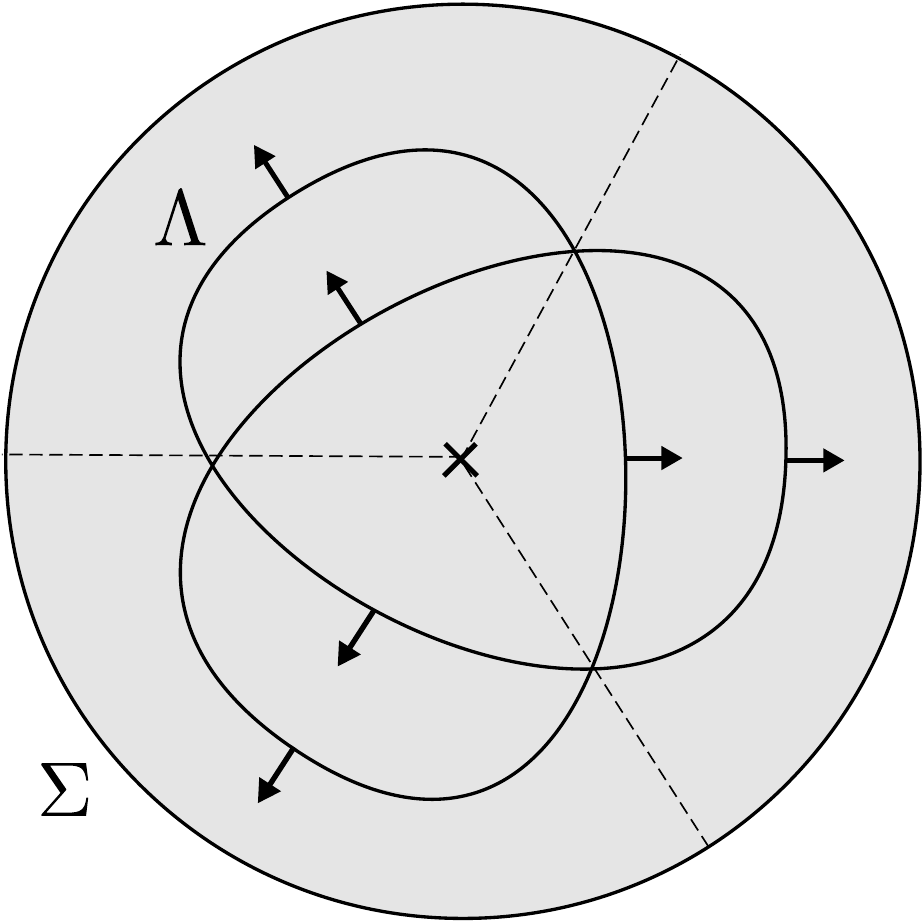}
         \caption{The Legendrian link for the irregular type \usebox{\smlmat} $\frac{1}{z^{3/2}}$.
         The microlocal rank on the link $\Lambda$ is 1.}
         \label{fig:twistedwildcharvar}
\end{figure}

The case where $n=2, r=3$ gets us a trefoil and appears in Stokes' discussion of the Airy equation \cite{Sto}. Note that the trefoil has only one component, so the formal monodromy map lands in $Loc_{k^\times}(S^1) \simeq [\frac{k^\times}{k^\times}]$ which doesn't stand for $\frac{T}{T}$ for any maximal torus $T \subset GL_2$. This explains why this more general case cannot be described just in terms of moment maps into subgroups of $G$.
\end{example}

\begin{remark}
One can play many variations on the theme of Stokes filtrations and irregular singularities. Considering connections with matrices $A$ that are not in the regular locus of $\mf g$, it becomes necessary to look at further less singular terms in the expression for $dQ$. To find the corresponding Legendrians one can still follow the prescriptions in \cite{STWZ}. One obtains cablings of torus knots by other torus knots and cablings of torus knots by such cablings and so on.

But there are many other knots that one can consider: pick any positive braid and close it around the origin into a Legendrian $\Lambda$. Picking rank conditions and monodromies around the components of $\Lambda$, this gives a symplectic space $\cM_\Lambda$ that doesn't necessarily come from an irregular meromorphic singularity. We can expect these spaces to carry some of the same structures as the tame and wild character varieties; whether this is true remains a topic of future research.
\end{remark}

\begin{remark}
For moduli of Stokes data for connections on 
higher dimensional complex varieties, 
the corresponding Riemann-Hilbert theorem recently been proven \cite{AK}.  We expect that it can be reformulated into an analogous ``microsupport in certain
smooth Legendrians'' version, from which we would be able to immediately deduce the existence of
the shifted symplectic structure. 

In terms of the corresponding symplectic/Poisson geometry, the moduli spaces of flat connections/local systems on higher-dimensional algebraic varieties has been recently studied in \cite{PT1,PT2}. We think it would be interesting to investigate the precise relation between our formalism and those results.
\end{remark}

\subsection{Positroid varieties, multiplicative Nakajima varieties, and other cluster structures}
The combinatorics of cluster algebras arising from surfaces was originally organized around data given
variously as a graph on a surface, a triangulation on a surface, etc \cite{Pos, FG, GSV, GSV2, FST}.   
One presentation of this data is in terms of the so-called ``alternating strand diagram'', the manipulation
of which by combinatorial topology \cite{Thu} underlies various theorems of the cluster algebra.  

In \cite{STWZ}, we took the perspective that the alternating strand diagram should be viewed as
a Legendrian knot, that triangulations of the surface give rise to Lagrangian fillings of it, and that 
all the corresponding cluster algebraic formulas are computing the Floer homology between such fillings. 
In particular,  the corresponding cluster $X$-variety was identified as a moduli space of ``rank one'' objects in 
$Sh_\Lambda(\Sigma)$, where $\Lambda$ is the Legendrian lift of the alternating strand diagram, and $\Sigma$
is the base curve.  

In \cite{STW}, there is a slightly different perspective: rather than working from $\Sigma, \Lambda$, the authors start
with a Legendrian $L$ --- one could view it as one of the above-mentioned fillings of $\Lambda$ --- and attach
Weinstein handles to its cotangent bundle along Legendrians which project to simple closed curves.  This perspective
is yet more general than the previous.  

It includes as a special case the multiplicative Nakajima quiver varieties 
of \cite{CS, Yam, BezKap}; this being the case where the attaching circles are contractible.  Indeed, this case is very close to the 
presentation in \cite{BezKap}.  In that reference, rather than locally arboreal singularities, they consider the spaces which are locally
either a smooth surface or modeled on the Lagrangian singularity given by the union of the zero section and the conormal 
to point.  However, this local model admits a noncharacteristic deformation to the union of the zero section and the positive conormal
to a circle.  The deformation is just given by the contact isotopy induced by the Reeb flow. Thus, the present work recovers all constructions of symplectic and Poisson structures on such spaces. For some further discussion of multiplicative quiver varieties, and their relation to quasi-Hamiltonian reduction, see \cite{vdB}.

\subsection{The augmentation variety of knot contact homology} 

Consider a knot or link $K \subset S^3$.  Naturally associated to this is the category of sheaves constructible with respect
to the stratification $S^3 = K \cup S^3 \setminus K$.  This study of this category e.g. allows one to prove that 
the Legendrian
isotopy type of the conormal torus to a knot determines the knot \cite{She, ENS}.    In an appropriate sense, it is equivalent to the 
a category of augmentations of knot contact homology. 

The union of the conormal to the knot with the zero section is not arboreal, but as the above discussion of \cite{BezKap}, this
can be remedied by perturbing the conormal torus by the Reeb flow, resulting in a skeleton given by the union of the zero section and 
the positive conormal to the inward (or outward) co-oriented boundary of a tubular neighborhood of the knot. 

In any case, we can study the space of objects in this category.  It has a map to the category of local systems on $T^2$, which by the results
here, becomes a $0$-shifted Lagrangian morphism on moduli spaces.  We note that the study of this moduli space
was also suggested in \cite{BezKap}.

To select a connected component (indeed, the connected component
corresponding to what is usually called the augmentation variety), we can pick those sheaves whose microsupport on the conormal 
torus is rank one in degree zero, i.e., what are called {\em simple sheaves} in \cite{KS}.  We restrict further to the open locus on which
objects which have no global sections, the main point of which is to eliminate constant summands. 
Let us write $A_1(K)$ for this component. 

Restriction to the microlocal boundary gives a map $A_1(K) \to \Loc_{k^\times}(T^2) = (k^\times)^2$. 
There is an analogous map in knot contact homology, described in \cite{Ng, EENS, AENV}.  Note that objects in $A_1(K)$ are
easy to understand: they are a local system in the complement of $K$, which is extended by a codimension one subspace of
meridian invariants along $K$; or possibly a nontrivial rank one local system supported on $K$.  

From this point of view 
it is clear both why $A_1(K)$ contains the classical $A$-polynomial curve, and also what are the other components: 
the $A$ polynomial curve has to do with $SL_2$ representations of the fundamental group; any such becomes, after rescaling by an eigenvalue of the meridian, a $GL_2$ representation with a meridian invariant subspace.  Similarly it is clear what the other components of $A_1(K)$ are.
(Compare \cite{Cor} for direct proofs of the analogous statements about the augmentation variety, from which it can be deduced that $A_1(K)$ as 
described here is the augmentation variety of knot contact homology.)  

Thus we have shown  that the morphism $A_1(K) \to k^\times \times k^\times$ is (0-shifted) Lagrangian.  This morphism can be quantized using higher genus holomorphic curves as described in
\cite{AENV, EN, ES}; it should be interesting to understand that quantization in our formalism of orientations.

\appendix

\section{The sheaf of morphisms}\label{sec:appendixSheafOfHoms}
In this Appendix we provide a construction of the (co)sheaves of Homs used in this article, that we mentioned from Section \ref{sec:sheaf of Homs}. As explained in the Introduction, we could not find in the literature a clear reference for their construction in the specific setting that we need, namely, for (co)sheaves valued in the $\infty$-category of $k$-linear stable $\infty$-categories.

It turns out that it is more natural to explain this construction using a different model for the same $\infty$-category, namely using the definition of `categorical algebras' from \cite{GepHau}. Given a monoidal $\infty$-category $\cV$, there is an $\infty$-category $\mathrm{Alg}_\mathrm{cat}(\cV)$ of categorical algebras in $\cV$; informally speaking, an object $\cC$ of this category is a double $\infty$-category, the data of which includes a space of objects $S$, objects of $\cV$ for any set of points in $S$, with composition/higher coherence maps parametrized by simplices in $S$. The $\infty$-category of `small (complete) $\cV\mh\infty$-categories' $\mathrm{Cat}^\cV_\infty$ is a certain reflective subcategory of $\mathrm{Alg}_\mathrm{cat}(\cV)$, spanned by objects satisfying a Segal completeness property.

Though \emph{a priori} it may seem like that an object of $\mathrm{Cat}^\cV_\infty$ contains more information an $\infty$-category which is tensored over $\cV$, in fact these approaches are equivalent. For $\cV = \cS$ it is proven in \cite{GepHau} that $\mathrm{Cat}^\cS_\infty$ is equivalent to the $\infty$-category of complete Segal spaces (known to present $\infty$-categories), and for other choices of $\cV$, the recent work \cite{Hei} established several variants of equivalence between $\cV$-tensored categories and categorical algebras. We paraphrase the results in \cite{Hei}:
\begin{theorem}
	There is an equivalence of $\infty$-categories between $\mathrm{Cat}^\mathrm{perf}(\cE)$ and a subcategory of $\mathrm{Cat}^{\Ind(\cE)}_\infty$, spanned by the objects whose underlying $\infty$-category are stable.
\end{theorem}
Using the equivalences above, we will sometimes regard diagrams in $\mathrm{Cat}^\mathrm{perf}(\cE)$, that is, of small stable $\infty$-categories tensored over $\cE$, as valued instead in categorical algebras in $\cE$.

\begin{lemma}\label{lem:homSheafLemma1}
	Let $\cV$ be a monoidal $\infty$-category and $\cF: I^\triangleleft \to \mathrm{Alg_{cat}(\cV)}$ be a limit diagram in categorical algebras in $\cV$. Then, for any two objects $x,y \in \cF(*)$, there is a limit diagram $\cF|^{x,y}:  I^\triangleleft \to \cV$, such that for any $i\in I$ there is an equivalence $\cF|^{x,y}(i) \simeq \cF(i)(\rho x,\rho y)$, where $\rho: S_{\cF(*)} \to S_{\cF(i)}$ is the image of the morphism $\cF(*) \to \cF(i)$ in $\cS$.
\end{lemma}
\begin{proof}
	We will first pull back out diagram of categorical algebras to the space of two points $\{0,1\}$ in the following way. For any fixed space $S \in \cS$, we can define the $\infty$-category $\mathrm{Alg_{cat}}(\cV) \times_\cS \cS_{S/}$ whose objects are a categorical algebra $\cC$, with space of objects $S_\cC$, together with a map $S \to S_\cC$. There is a canonical functor
	\[ \mathrm{Alg}_{\Delta^{op}_S}(\cV) \to \mathrm{Alg_{cat}}(\cV) \times_\cS \cS_{S/} \]
	induced by inclusion into the fiber of $\mathrm{Alg_{cat}}(\cV) \to \cS$ over $S$; this functor has a right adjoint $R$ which sends $(\cC,f:S \to S_\cC)$ to the pullback of $\cC$ along $f$.

	We now take $S = \{0,1\}$ mapping to $S_{\cF(*)}$ at $(x,y)$; this gives a limit diagram valued in $\mathrm{Alg_{cat}}(\cV) \times_\cS \cS_{\{0,1\}/}$ to which we can apply $R$ to get a diagram valued in $\mathrm{Alg}_{\Delta^{op}_{\{0,1\}}}(\cV)$. The functor from this category to $\cV$ sending $\cC \to \cC(0,1)$  preserves limits; one can construct explicitly its left adjoint $\Sigma: \cV \to \mathrm{Alg}_{\Delta^{op}_{\{0,1\}}}(\cV)$, sending an object $V \in \cV$ to the free categorical algebra which has $\Sigma V(0,1)=V$. Composing $R$ with this evaluation at $\{0,1\}$ gives the desired limit diagram $\cF|^{x,y}:  I^\triangleleft \to \cV$.
\end{proof}

The setting of categorical algebras also allows us promptly to define composition maps from the tensor product of $\cV$-valued sheaves:
\begin{lemma}\label{lem:homSheafLemma2}
	For any three objects $x,y,z$ of $\cF(*)$, there is a canonical \emph{composition} map $\cF|^{x,y} \otimes \cF|^{y,z} \to \cF|^{x,z}$ in $\Fun(I^\triangleleft,\cV)$.
\end{lemma}
\begin{proof}
	In the proof of the previous proposition, we can take $S = \{0,1,2\}$ mapping to $S_{\cF(*)}$ at $(x,y,z)$ to get a diagram $\mathrm{Alg}_{\Delta^{op}_{\{0,1,2\}}}(\cV)$; this is the category of $\mathbf{O}_{\{0,1,2\}}$-algebras in $\cV$. The desired composition map then comes from the natural transformation between the functors to $\cV$, given the image of the morphism $((0,1),(1,2)) \to (0,2)$ in $\mathbf{O}_{\{0,1,2\}}$.
\end{proof}

We combine the propositions above with the equivalence between stable $\infty$-categories tensored over $\cE$ and $\Ind(\cE)\mh\infty$-categories for our construction of the hom sheaves.
\begin{proof} (of Proposition \ref{prop:morphismFunctor})
	There is an equivalence between $\mathrm{Cat}^\mathrm{perf}(\cE)$ and a subcategory of $\mathrm{Alg_{cat}}(\Ind(\cE))$, whose objects are $\Ind(\cE)\mh\infty$-categories which are complete, stable and $\cE$-tensored. Recall that limits of stable $\infty$-categories are preserved by the inclusion into $\infty$-categories; together with the correspondence with categorical algebras this implies that $\cF$ can be seen as a limit diagram in $\mathrm{Alg_{cat}}(\Ind(\cE))$. The results then follow from Lemmas \ref{lem:homSheafLemma1} and \ref{lem:homSheafLemma2}.
\end{proof}

If $\cF$ be a sheaf of dg categories on $X$, applying the result above with $\cE = \Perf$ gives a way to define the sheaf of Homs from Definition \ref{def:sheafOfHoms}.

\section{The diagonal bimodule cosheaf}
In this Appendix we give a different description of the cosheaves of diagonal bimodules. To recall, in Section \ref{sec:orientationsCCC}, starting from a $\dgcat$-valued cosheaf $\cF$, for any $U\in \Opens(X)$ using bar complexes we produced a certain $\cF(U)$-bimodule valued $\Compacts$-sheaf, which we denoted $\phi\DDelta_U\cF$ and called the `Verdier dual diagonal $\Compacts$-sheaf'. As its name suggests, this is a representative for the $\Compacts$-sheaf Verdier dual of some \emph{cosheaf} of diagonal bimodules.

We will now explain how to understand these objects more conceptually; for that, it will be necessary to work with cosheaves valued in some $\infty$-category whose objects include the data of dg categories and also bimodules over them; the right setting turns out to be the notion of \emph{endomorphism categories} of \cite{HSS}

\subsection[Endomorphisms in (infinity,2)-categories]{Endomorphisms in $(\infty,2)$-categories}\label{sec:endos}

In this subsection we review ideas of \cite{HSS, JFS} regarding the construction of endomorphism categories and traces.  
Using similar ideas we construct certain variants we require in the sequel.  

\subsubsection[Modeling (infinity,n)-categories]{Modeling $(\infty,n)$-categories}
Recall that an $(\infty, n)$ category is a higher category where all morphisms of dimension above $n$ are invertible; 
the usual `$\infty$-categories' being the $(\infty, 1)$-categories.  In the literature there are many approaches to describe
such structures; as we will use constructions from \cite{HSS,JFS}, we follow them in modeling 
$(\infty, n)$-categories using Barwick's complete $n$-fold Segal spaces  \cite{Barw}.  Comparison theorems to other models 
can be found in  \cite{BSP}. 

Writing $\Delta$ for the simplicial indexing category, a complete $n$-fold Segal space $\cC$ 
can be described by a collection of spaces $\cC_{\vv k}$ indexed by vectors $\vv k \in \Delta^n$, with appropriate morphisms and coherence data between them; for any $0 \le p \le n$, one should regard the space $\cC_{(p)}$, where $(p)$ is the vector $(1,1,\dots,1,0,0,\dots)$ with $p$ entries `1' followed by $(n-p)$ entries `0', as the space of $p$-morphisms of the $(\infty,n)$-category $\cC$. For $\vv k \in \Delta^p$ we can define a complete $p$-fold Segal space $\iota_p \cC$ given by
\[ (\iota_p \cC)_{\vv k} := \cC_{\vv k,0,\dots,0} \]
modeling the maximal $(\infty,p)$-subcategory of $\cC$.

Barwick constructs the category of complete $n$-fold Segal spaces as a certain localization of 
the category of presheaves of spaces over $\Delta^n$; we denote this category as $\mathrm{CSS}(\Delta^n)$. 
This category gives a model for the homotopy theory of $(\infty,n)$-categories \cite[Corollary 14.3 and Theorem 14.6]{Barw}, carrying both projective and injective model structures \cite[App.A]{JFS}. We denote by $\maps^h$ the corresponding notion of derived mapping space, which we will choose to compute in the projective model structure.

In particular, \emph{strict} $n$-categories give objects of $\mathrm{CSS}(\Delta^n)$. To encode the data of such an object, we will use the notion of \emph{computad}, which appears in \cite[Def.3.1]{JFS} following \cite{Str}. An $n$-computad is a particular type presentation of a strict $n$-category; given by explicit sets of $k$-morphisms for $0 \le k \le n$.

Such an $n$-computad $\Theta$ generates a strict $n$-category, of which one can take the nerve, giving a complete $n$-fold Segal space, that is, an object of $\mathrm{CSS}(\Delta^n)$ \cite[Rem.3.2]{JFS}; all these objects carry the same data and as in \emph{op.cit} we will denote them equally by $\Theta$.

Let us recall some computads that will play an important role later.  
\begin{definition} { The walking $i$-morphism, or $i$-globe, $\Theta^{(i)}$} \cite{BSP, JFS}. 
Inductively 
define $\Theta^{(i)}$ as an $i$-computad, by setting $\Theta^{(0)}$ to be the singleton set, and $\Theta^{(i)}$ to be the category with two objects $\{0,1\}$ and with morphisms given by
\[ \Hom_{\Theta^{(i)}}(x,y) = \begin{cases} \Theta^{(0)}, & \text{if } x=y \\
\Theta^{(i-1)}, & \text{if } x=0, y=1 \\
\emptyset, & \text{if } x=1, y=0,
\end{cases} \]
This computad has the property that for any $(\infty, n)$-category $\cC$ there is a canonical homotopy equivalence
\[ \cC_{(i)} \simeq \maps^h(\Theta^{(i)}, \cC) \]
expressing the fact that $\Theta^{(i)}$ represents $i$-morphisms. For small values of $i$ we draw these categories as
\[
\Theta^{(0)} = \bullet, \qquad \Theta^{(1)} =
\begin{tikzpicture}[baseline={([yshift=-.5ex]current bounding box.center)}]
\node (a) at (0,0) {$\bullet$};
\node (b) at (2,0) {$\bullet$};
\draw [->] (a) to (b);
\end{tikzpicture}, \qquad \Theta^{(2)} = 
\begin{tikzpicture}[baseline={([yshift=-.5ex]current bounding box.center)}]
\node (a) at (0,0) {$\bullet$};
\node (b) at (2,0) {$\bullet$};
\draw [->,bend left=45] (a) to node (top) {} (b);
\draw [->,bend right=45] (a) to node (bot) {} (b); 
\draw [darrow] (top) to (bot);
\end{tikzpicture}
\]
where we denote a 2-morphism by a double arrow and so on.
\end{definition}
We note that a $1$-computad is simply given by the data of a quiver, and every $k$-computad is automatically a $n$-computad for every $n \ge k$, by adding empty sets of higher morphisms.

\begin{remark}
As a warning, the computads $\Theta^{i}$ for $i \ge 2$, as well as the most of the other ones we will use here, are not projective cofibrant; therefore one must find a cofibrant replacement when computing the derived space functor $\maps^h$. Nevertheless, when taking a target that is `essentially constant' (\cite[Def.2.7]{JFS}) up to equivalence one can often use a naive mapping space instead, this is the case for $\Theta^{(i)}$; see the discussion in Remarks 3.2 and 3.4 of \emph{op.cit.}
\end{remark}

We will also need the \emph{walking $i$-by-$j$ morphisms} of \cite[Section 3]{JFS}. These $\Theta^{(i);(j)}$ are defined inductively and are used by the authors of that paper to define the notion of (op)lax natural transformations. When $i$ or $j$ is zero we obtain the globe objects above: $\Theta^{(0);(i)} = \Theta^{(i);(0)} = \Theta^{(i)}$.   The only space $\Theta^{(i); (j)}$ that we will need to use explicitly are drawn below.

\begin{definition} 
\[
\Theta^{(1);(1)} = 
\begin{tikzpicture}[baseline={([yshift=-.5ex]current bounding box.center)}]
\node (a) at (0,2) {$\bullet$};
\node (b) at (2,2) {$\bullet$};
\node (c) at (0,0) {$\bullet$};
\node (d) at (2,0) {$\bullet$};
\draw [->] (a) to (b);
\draw [->] (a) to (c);
\draw [->] (c) to (d);
\draw [->] (b) to (d);
\draw [darrow,shorten <=10pt, shorten >=10pt] (c) to (b);
\end{tikzpicture}
\] 
\end{definition}

\begin{definition}\label{def:objs}
The 1-computad $\S$ is given by directed circle
\[ \S =
\begin{tikzpicture}[baseline={([yshift=-.5ex]current bounding box.center)}]
\node (a) at (0,0) {$\bullet$};
\draw [->] (a.-80) arc (-160:170:8mm) node [pos=0.6] (side) {} (a);
\end{tikzpicture}
\]
which generates the strict $1$-category $B\N$, that is, the category with one object and morphism set given by the monoid of natural numbers.
\end{definition}
For any $n \ge 1$ we will equally denote by $\S$ the corresponding strict $n$-category or complete $n$-fold Segal space.

\subsubsection{Oplax natural transformations}
We will follow the exposition in \cite[Section 2.2]{HSS}, mostly borrowing notation from it. Given a complete $n$-fold Segal space $\cC$ and a tuple $\vv k \in \Delta^n$ of natural numbers, there is a complete $n$-fold Segal space $ \Lax_{\vv k}(\cC)$, defined using the walking morphisms \cite[Definition 5.14]{JFS}:
\[ \left(  \Lax_{\vv k}(\cC)\right)_{\vv l} := \maps^h(\Theta^{\vv k;\vv l}, \cC). \]
For a fixed $\vv k$, these spaces for all $\vv l \in \Delta^n$ assemble into the complete $n$-fold Segal space $ \Lax_{\vv k}(\cC)$ \cite[Theorem 5.11]{JFS}. In particular, for $\vv k = (p), \vv l = (0)$, we have an isomorphism
\[ \left( \Lax_{(p)}(\cC)\right)_{(0)} \simeq \maps^h(\Theta^{(p)},\cC) = \cC_{(p)}, \]
exhibiting the space of objects of $ \Lax_{(p)}(\cC)$ as the space of $p$-morphisms of $\cC$. The corresponding morphisms can be understood as (higher) oplax natural transformations in $\cC$. 

Given two complete $n$-fold Segal spaces $\cB,\cC$, \cite[Corollary 5.19]{JFS} defines another complete $n$-fold Segal space $\Fun^\oplax(\cB,\cC)$ (of `oplax transfors') whose $\vv k$ space is given by
\[ \Fun^\oplax(\cB,\cC)_{\vv k} = \maps^h(\cB, \Lax_{\vv k}(\cC)), \]
where $\maps^h$ denotes a strictly functorial model for the derived mapping space between presheaves of spaces. This complete $n$-fold Segal space $\Fun^\oplax(\cB,\cC)$ is functorial on the choice of $\cB,\cC$.

These constructions can also be extended to incorporate symmetric monoidal structures.\footnote{In \cite{JFS} the symmetric monoidal structure is captured by 
strict functors from the category of pointed finite sets to $\mathrm{CSS}(\Delta^n)$, satisfying the Segal condition. For simplicity we will denote a symmetric monoidal category 
merely by the underlying complete $n$-fold Segal space given by the value of the corresponding functor on the finite pointed set $\{*,1\}$.}  
If $\cC$ is a symmetric monoidal complete $n$-fold Segal space then $ \Lax_{\vv k}(\cC)$ also carries a natural symmetric monoidal structure.
If $\cB,\cC$ are symmetric monoidal complete $n$-fold Segal spaces, then there is also a  complete $n$-fold Segal space $\Fun^\oplax_\otimes(\cB,\cC)$ satisfying 
\[ \Fun^\oplax_\otimes(\cB,\cC)_{\vv k} = \maps^h_\otimes(\cB, \Lax_{\vv k}(\cC)), \]
where the subscript $\otimes$ means to consider symmetric monoidal maps
\cite[Corollary 6.9]{JFS}. 
Again, this construction is functorial on the choice of $\cB,\cC$.

\subsubsection{Endomorphism categories}\label{sec:endo}
Let $\cC$ be a symmetric monoidal complete $n$-fold Segal space, for some $n \ge 2$. We will denote by $\cC^\mathrm{rig}$ the underlying $(\infty,n)$-category spanned by the dualizable objects in $\cC$. The functor $(-)^\mathrm{rig}$ preserves limits \cite[Remark 4.6.1.11]{LurHA} and admits a left adjoint \cite[Sec.2.2]{HSS}
\[ \Fr^{\rig}: \Cat_{(\infty,n)} \longleftrightarrow \Cat^\otimes_{(\infty,n)}: (-)^\mathrm{rig}, \]
which from any $(\infty,n)$-category $\cC$ gives a `free rigid symmetric monoidal $(\infty,n)$-category' $\Fr^\rig(\cC)$.

\begin{definition}\label{def:rr}
Consider the following $(\infty, n-1)$-categories: \footnote{Recall $\iota_{n-1}$ means to discard all non-invertible morphisms above $n-1$, and it is written because a priori if $\cA,\cB$ are symmetric monoidal $(\infty,n)$-categories, the object defined by $\Fun^\oplax(\cA,\cB)$ will also be an $(\infty,n)$-category. However, as argued in \cite[Remark 2.3]{HSS}, the categories defined by $\Fun^\oplax_\otimes(\Fr^\rig(\cA),\cB)$ are in fact equivalent to $(\infty,n-1)$-categories, since the rigidity of the images of $\Fr^\rig$ forces the $n$-morphisms to be invertible.  Thus
the $\iota_{n-1}$ serves only as a reminder that we have $(\infty, n-1)$-categories, and we henceforth omit it.
} 
\begin{align*}
\cC^\RR &= \iota_{n-1}
\Fun^\oplax_\otimes(\Fr^\rig(\pt),\cC) \quad \text{``right-rigid subcategory''} \\
\End(\cC) &= \iota_{n-1}
\Fun^\oplax_\otimes(\Fr^\rig(\S),\cC) \quad \text{``endomorphisms''} \\
\end{align*}
\end{definition}

\begin{remark}
The definition above of the category $\End(\cC)$ is from \cite{HSS}, where $\S$ is denoted $B\N$ because 
the category generated by $\S$ models the classifying category of the monoid of natural numbers. Further properties of $\End(\cC)$ can be found in \cite{HSS}. 
\end{remark}

Let us now describe these categories more explicitly.
\begin{proposition} \label{explicit} 
When $\cC$ is a symmetric monoidal $(\infty,2)$-category, we can describe its spaces of objects and morphisms, up to equivalence:
\begin{itemize}
\item The category $\cC^\RR$ has as objects the dualizable objects $A$ of $\cC$, and as morphisms between $A$ and $B$ the right-dualizable morphisms between $A$ and $B$ in $\cC$.
\item \cite{HSS} The category $\End(\cC)$ has as objects pairs $(A,f)$, where $A$ is a dualizable object of $\cC$ and $f$ is an endomorphism of $A$. The morphisms between $(A,f)$ and $(B,g)$ are pairs $(\phi,\alpha)$, such that $\phi$ is a right-dualizable morphism $A \to B$ and $\alpha: \phi \circ f \Rightarrow g \circ \phi$ is a 2-morphism in $\cC$.
\end{itemize}
\end{proposition}

\begin{proof}
We will start by describing the category $\cC^\RR$. The space of objects is given by
\[
\left( \cC^\RR \right)_{(0)} = \Fun^\oplax_\otimes(\Fr^\rig(\pt),\cC)_{(0)} = \maps_\otimes^h (\Fr^\rig(\pt), \Lax_{(0)}(\cC)) \simeq \maps_\otimes^h(\Fr^\rig(\pt), \cC),
\]
using the canonical equivalence $\Lax_{(0)}(\cC) \simeq \cC$. By the adjunction between $(-)^\mathrm{rig}$ and $\Fr^\rig$, this is canonically equivalent to $\maps^h(\pt, \cC^\rig)$, i.e., objects of $\cC^\RR$ are given by the dualizable objects in $\cC$. The space of 1-morphisms is given by a similar calculation:
\[
\left( \cC^\RR \right)_{(1)} = \Fun^\oplax_\otimes(\Fr^\rig(\pt),\cC)_{(1)} = \maps_\otimes^h (\Fr^\rig(\pt), \Lax_{(1)}(\cC)) \simeq \maps^h(\pt, \Lax_{(1)}(\cC)^\rig).
\]
By \cite[Lemma 2.4]{HSS}, the dualizable objects of $\Lax_{(1)}(\cC)$ are given by right-dualizable morphisms $\phi:A\to B$ between dualizable objects in $\cC$. Therefore, the category $\cC^\RR$ can be identified with the 
$(\infty,1)$-subcategory of $\cC$ containing all dualizable objects and right-dualizable 1-morphisms.

The description of $\End(\cC)$ appears in \cite{HSS}, where one calculates:
\begin{align*} \End(\cC)_{(0)} &= \Fun^\oplax_\otimes(\Fr^\rig(\S),\cC)_{(0)} \simeq \maps^h(\S, (\Lax_{(0)}(\cC))^\rig) = \maps^h(\S, \cC^\rig)
\end{align*}
A point in this space is an endomorphism of a dualizable object in $\cC^\rig$, that is, a pair $(A,f)$ where $A$ is a dualizable object of $\cC$ and $f$ is an endomorphism of $A$. A similar calculation shows that a morphism $(A,f) \to (B,g)$ is the data of a pair $(\phi,\alpha)$, such that $\phi$ is a right-dualizable morphism $A \to B$ and $\alpha: \phi \circ f \Rightarrow g \circ \phi$ is a 2-morphism.
\end{proof}

\begin{example}
If $\cC$ is the symmetric monoidal $(\infty,2)$-category $\DGCat^2$, then we can identify $\cC^{\RR}$ with the wide $(\infty,1)$-subcategory of $\DGCat$ containing all objects, but only the right-dualizable morphisms. In other words, $(\DGCat^2)^{\RR}$ is equivalent to the essential image of $\Ind: \dgcat \to \DGCat$.
\end{example}

\subsubsection{The trace of an endomorphism}\label{sec:trace}
Let $\cC$ be an $(\infty,n)$-category. For any two objects $X,Y$ of $\cC$, there is an $(\infty,n-1)$-category of morphisms $\cC(X,Y)$ from $X$ to $Y$.  

When $\cC$ is a symmetric monoidal $(\infty,n)$-category with monoidal unit $1_\cC$, we write 
\[ \Omega\cC := \cC(1_\cC,1_\cC) \]
This can be shown to carry a natural symmetric monoidal structure.

There is \cite[Definition 2.9]{HSS} a \emph{trace functor} $\Tr:\End(\cC) \to \Omega\cC$.  That is, this functor carries
an endomorphism of an object of $\cC$ to an endomorphism of the monoidal unit. In the case where $\cC$ is an $(\infty,2)$-category, the category $\Omega\cC$ has as objects endomorphisms of the unit object $1_\cC$, with morphisms given by 2-morphisms between them in $\cC$. The functor $\Tr$ sends an object $(A,f)$ to the endomorphism of $1_\cC$ given by
\[ 1_\cC \xrightarrow{\coev_A} A \otimes A^\vee \xrightarrow{f \otimes \id_A} A \otimes A^\vee \xrightarrow{\tau} A^\vee \otimes A \xrightarrow{\ev_A} 1_\cC, \]
where $\ev_A,\coev_A$ are the canonical evaluation and coevaluation maps of the dualizable object $A$ and $\tau$ is the involution switching the factors. The trace map on morphisms can also be described explicitly as in \cite[Sec.4.1]{BD2}.

\subsection{The diagonal map}
We will use the category of endomorphisms to define an $\infty$-functor that picks the identity endomorphisms in a coherent way. As above, let $\cC$ be a complete $n$-fold Segal space, with $n \ge 2$.

\begin{definition}\label{def:diag}
The \emph{diagonal map to endomorphisms} is the $(\infty,n-1)$-functor $\Delta: \cC^\RR \to \End(\cC)$ induced by the terminal map $p_\S:\S \to \pt$.
\end{definition}

This map has a very simple description at level of objects and 1-morphisms: $\Delta$ sends a dualizable object $X$ to the object $(X,\Id_X)$ of $\End(\cC)$, and a right-dualizable morphism $\phi:X \to Y$ to the morphism in $\End(\cC)$ given by $(\phi,\id_\phi)$.

\begin{remark}
For clarity we will denote $\Id_X$ (capitalized) when we refer to the identity 1-morphism of an object $X$ (which will later be given by a diagonal bimodule), as opposed to $\id_f$ (lowercase) when referring to the identity 2-morphism on a 1-morphism $f$ (later, a map of bimodules).
\end{remark}

From now on, we specialize to the case where $n=2$, i.e. where $\cC$ is a symmetric monoidal $(\infty,2)$-category. In this case, $\cC^\RR$ and $\End(\cC)$ will be $(\infty,1)$-categories. The remainder of this subsection is devoted to proving that under appropriate conditions on $\cC$, the diagonal map preserves colimits. This will prove to be quite subtle, depending on the internal structure of $\cC$ as a $(\infty,2)$-category. To illustrate the issue at hand, let us make a digression to discuss the problem in ordinary categories.

\subsubsection{Digression: the diagonal map in strict 2-categories}
Let $\cC$ be a strict 2-category, i.e., a category enriched over the category $\mathrm{Cat}$ of small categories. Analogously to the endomorphism categories defined above, we can define its endomorphism category $\End(\cC)$ as the category whose objects are pairs $(A,f)$ of an object $A$ of $\cC$ and an endomorphism $f$ of $A$, and whose 1-morphisms $(A,f) \to (B,g)$ are pairs $(\phi,\alpha)$ forming a square
\[
\begin{tikzpicture}[baseline={([yshift=-.5ex]current bounding box.center)}]
\node (a) at (0,1.5) {$A$};
\node (b) at (1.5,1.5) {$B$};
\node (c) at (0,0) {$A$};
\node (d) at (1.5,0) {$B$};
\draw [->] (a) to node [auto] {$\phi$} (b);
\draw [->] (a) to node [auto, swap] {$f$} (c);
\draw [->] (c) to node [auto, swap] {$\phi$} (d);
\draw [->] (b) to node [auto] {$g$} (d);
\draw [darrow,shorten <=1pt, shorten >=1pt] (c) to node [auto, swap] {$\alpha$} (b);
\end{tikzpicture}
\]

Let us write $\iota_1 \cC$ to denote the underlying category of $\cC$, i.e. whose morphisms are the object sets of the categories $\Hom_\cC(X,Y)$. We then have a diagonal map $\Delta: \iota_1 \cC \to \End(\cC)$ given on objects by $A \mapsto (A, \id_A)$ and on morphisms by $\phi \mapsto (\phi, \id_\phi)$.

We now describe a condition for the existence of a right adjoint to $\Delta$. Consider a fixed endomorphism $A \xrightarrow{f} A$, and let us define a functor
\[ \mathrm{tri}_f: (\iota_1\cC)^{op} \to \mathrm{Sets}, \]
which maps an object $X$ to the set of diagrams of the form 
\[\begin{tikzpicture}[baseline={([yshift=-.5ex]current bounding box.center)}]
\node (left) at (0,1) {$X$};
\node (top) at (2,2) {$A$};
\node (bottom) at (2,0) {$A$};
\draw [->] (left) to node [auto] {$\psi$} (top);
\draw [->] (left) to node (mid) [auto,swap] {$\psi$} (bottom);
\draw [->] (top) to node [auto] {$f$} (bottom);
\draw [darrow,shorten <=10pt, shorten >=10pt] (mid) to node [auto,swap] {$\mu$} (top);
\end{tikzpicture}
\]
in the 2-category $\cC$, and which maps a morphism $g:X\to X'$ to the map of sets of triangles given by composing on the left with $g$.

\begin{definition}
We will say that the 2-category $\cC$ has \emph{right-sided inserters} if for any endomorphism $A \xrightarrow{f} A$, the functor $\mathrm{tri}_f$ is representable by an object in $\iota_1\cC$. In that case we denote the representing object by the notation $\Ins(\Id_A,f)$.
\end{definition}
Unpacking this definition, that is to say that the inserter object $\Ins(A,f)$ comes with a canonical triangle
\[\begin{tikzpicture}[baseline={([yshift=-.5ex]current bounding box.center)}]
\node (left) at (0,1) {$\Ins(\Id_A,f)$};
\node (top) at (2.4,2) {$A$};
\node (bottom) at (2.4,0) {$A$};
\draw [->] (left) to node [auto] {$\psi_f$} (top);
\draw [->] (left) to node (mid) [auto,swap] {$\psi_f$} (bottom);
\draw [->] (top) to node [auto] {$f$} (bottom);
\draw [darrow,shorten <=10pt, shorten >=10pt] (mid) to node [auto,swap] {$\mu_f$} (top);
\end{tikzpicture}
\]
such that any other triangle in the image $\mathrm{tri}_f(X)$ is obtained from the one above by composing with a unique map $X \to \Ins(\Id_A,f)$.

\begin{remark}
The terminology and notation come from the notion of inserters in a 2-category: the inserter of a pair $g,f:A \rightrightarrows B$ is an object with a universal 1-morphism $h: \Ins(g,f) \to A$ and a universal 2-morphism $g \circ h \Rightarrow f \circ h$. This is a type of weighted 2-limit; the right-sided inserter is the particular case where $A=B$ and $g = \Id_A$.
\end{remark}

\begin{example}
The strict 2-category $\Cat^2$ has right-sided inserters (also all inserters): the object $\Ins(\Id_A,f)$ is given by the category whose objects are pairs $(a \in A, m:a \mapsto f(a))$, and whose morphisms $(a,m) \to (a',m')$ are given by morphisms $n: a \to a'$ in $A$ such that $m' \circ n = f(n) \circ m$.
\end{example}

Let us now use inserters to produce an adjoint to the diagonal functor. Consider two objects $(A,f)$ and $(B,g)$ of $\End(\cC)$ and a morphism $(\phi,\alpha): (A,f) \to (B,g)$. We can compose the canonical triangle corresponding to $\Ins(\Id_A,f)$ with the square corresponding to $(\phi,\alpha)$:
\[\begin{tikzpicture}[baseline={([yshift=-.5ex]current bounding box.center)}]
\node (left) at (0,1) {$\Ins(\Id_A,f)$};
\node (top) at (2,2) {$A$};
\node (bottom) at (2,0) {$A$};
\node (topright) at (4,2) {$B$};
\node (botright) at (4,0) {$B$};
\draw [->] (left) to (top);
\draw [->] (left) to node (mid) [auto,swap] {} (bottom);
\draw [->] (top) to node [auto] {$f$} (bottom);
\draw [darrow,shorten <=10pt, shorten >=10pt] (mid) to (top);
\draw [->] (top) to node [auto] {$\phi$} (topright);
\draw [->] (bottom) to node [auto, swap] {$\phi$} (botright);
\draw [->] (topright) to node [auto] {$g$} (botright);
\draw [darrow,shorten <=1pt, shorten >=1pt] (bottom) to node [auto, swap] {$\alpha$} (topright);
\end{tikzpicture}
\qquad \to \qquad
\begin{tikzpicture}[baseline={([yshift=-.5ex]current bounding box.center)}]
\node (left) at (0,1) {$\Ins(\Id_A,f)$};
\node (top) at (3,2) {$A$};
\node (bottom) at (3,0) {$A$};
\draw [->] (left) to node [auto] {$b\circ \psi_f$} (top);
\draw [->] (left) to node (mid) [auto,swap] {$b\circ \psi_f$} (bottom);
\draw [->] (top) to node [auto] {$f$} (bottom);
\draw [darrow,shorten <=10pt, shorten >=10pt] (mid) to (top);
\end{tikzpicture}
\]
which by the universal property of the inserter gives a map $G_{(\phi,\alpha)}: \Ins(\Id_A,f) \to \Ins(\Id_B,g)$. 

One can check that this defines a functor, which we call `right-sided insertion'
\[ \mathrm{RIns}: \End(\cC) \to \iota_1\cC, \]
which on objects maps $(A,f) \mapsto \Ins(\Id_A,f)$, and on morphisms maps $(\phi,\alpha) \mapsto g_{(\phi,\alpha)}$. We can rephrase the universal property of inserters as the following proposition:
\begin{proposition}
$\mathrm{RIns}$ is the right adjoint to the diagonal map $\Delta$, therefore $\Delta$ preserves all colimits existing in $\iota_1 \cC$.
\end{proposition}

\subsubsection{Upgrading to symmetric monoidal $(\infty,2)$-categories}
Now we will replicate the constructions in the digression above to our setting, where $\cC$ is a symmetric monoidal $(\infty,2)$-category. The main new difficulty is that we cannot define the constructions `by hand', as we have done in the strict 2-category case. Moreover, in order to state the needed condition on $\cC$, namely the existence of inserters, we will need to use a notion of representability appropriate to the setting of $\infty$-categories.

We will overcome both of these difficulties by using Lurie's theory of universal fibrations, explained in \cite[Sec.3.3.2 and 4.4.4]{LurHTT}. Given the right conditions, this framework allows one to go back and forth between functors from an $\infty$-category valued in spaces, and fibrations of simplicial sets over it. This will allow us to define the functors we need by specifying the corresponding fibrations, and the representability condition can be then checked at the level of these fibrations.

The relation between fibrations and functors valued in spaces is given by the following proposition. Here $\cS$ denotes the $\infty$-category of spaces, and by $\cH = \uph\cS$ the homotopy category of spaces \cite[Def.1.1.5.14]{LurHTT}. For any $\infty$-category $\cD$, its homotopy category $\uph\cD$ is naturally an $\cH$-enriched category.

\begin{proposition}\cite[Prop.3.3.2.5]{LurHTT}
Let $p: X \to Y$ be a Cartesian fibration of simplicial sets. Then $p$ is a right fibration if and only if $p$ is classified by a map $F: Y^{op} \to \cS$.
\end{proposition}
The notions of Cartesian and right fibrations are explained in \cite[Ch.2]{LurHTT}, and the notion of a map classifying a fibration is defined in \cite[Def.3.3.2.2]{LurHTT}; in particular it implies that for any vertex $y \in Y$ there is an equivalence 
\[ X \times_Y \{y\} \simeq F(y) \]
in the homotopy category $\cH$.

One can translate the condition of representability into the language of fibrations. The appropriate notion of representability is given at the level of the homotopy category:
\begin{definition}\cite[Def.4.4.4.1]{LurHTT}
A functor $F: \cD^{op} \to \cS$ from an $\infty$-category $\cD$ is represented by an object $D \in \cD$ if the corresponding functor
\[ \uph F: \uph\cD \to \cH \]
is represented by $D$ as an $\cH$-enriched functor, in the sense that there is a map of simplicial sets $* \to F(D)$ such that for any object $X \in \cD$ the induced map
\[ \Map_\cC(X,D) \to \Map_\cC(X,D) \times F(D) \to F(X) \]
is an isomorphism in $\cH$. When such an object exists we will say that the functor $F$ is representable.
\end{definition}

Recall the diagonal map $\Delta: \cC^\RR \to \End(\cC)$ which we constructed as a morphisms of complete (1-fold) Segal spaces. In order to regard them in a setting where we can use the formalism of \cite{LurHTT}, we switch to modeling $(\infty,1)$-categories as weak Kan complexes instead. For that, we can use an explicit map \cite{JT}: to each complete (1-fold) Segal space $\cD$, we can assign the simplicial set $d \mapsto (\cD_{d})_0$, that is, the set of vertices of the space $\cD_{d}$, which will be a Kan complex \footnote{Here the notation may be confusing, so let us reiterate: to whatever complete $n$-fold Segal space $\cD$, $\cD_d$ is the space associated to the vector $(d,0,0,\dots)$, and represents the space of a sequence of $d$ composable 1-morphisms, together with their compositions. Meanwhile, the notation $\cD_{(p)}$ with parenthesis, that appeared before, denotes the space associated to the vector $(1,1\dots,1,0,0,\dots)$ with $p$ entries equal to one; this makes sense only for $p \le n$ and represents the space of $p$-morphisms.}. Equally, for any explicit morphism of complete (1-fold) Segal spaces, such as $\Delta$, we get an explicit map of simplicial sets.

From now on let us regard $\cC^{\RR}, \End(\cC)$ as simplicial sets. For a fixed object $(A,f) \in \End(\cC)$ we consider the overcategory $\End(\cC)_{/(A,f)}$ with its canonical map to $\End(\cC)$ and take the pullback of simplicial sets
\[\xymatrix{
	\End(\cC)_{/(A,f)} \times_{\End(\cC)} \cC^\RR \ar[r] \ar[d]_{p_{(A,f)}} & \End(\cC)_{/(A,f)} \ar[d]\\
	\cC^\RR \ar[r]^{\Delta} & \End(\cC)
}\]
The fiber of $p_{(A,f)}$ over an object $Y$ of $\cC^\RR$ is equivalent to the space of  morphisms $(Y,\Id_Y) \to (A,f)$ in $\End(\cC)$.
\begin{lemma}
The map $p_{(A,f)}$ is a right fibration of simplicial sets.
\end{lemma}
\begin{proof}
Follows from the fact that the canonical map from the overcategory is a right fibration \cite[Cor.2.1.2.2]{LurHTT} and from the fact that the collection of right fibrations is stable under pullbacks \cite[\href{https://kerodon.net/tag/014P}{Tag 014P}]{kerodon}
\end{proof}

Now we are ready to state the desired condition on $\cC$.
\begin{definition}
We will say that the symmetric monoidal $(\infty,2)$-category $\cC$ \emph{has right-sided inserters} if, for any $(A,f) \in \End(\cC)$, the functor $F_{(A,f)}: \cC^\RR \to \cS$ classifying $p_{(A,f)}$ is representable.
\end{definition}
Unraveling this definition, $\cC$ has right-sided inserters when for any object $(A,f)$ of $\End(\cC)$, there is an object $\Ins(\Id_A,f)$ of $\cC^\RR$ (i.e., a dualizable object of $\cC$), such that there is a homotopy equivalence $\Map_{\cC^\RR}(X,\Ins(\Id_A,f)) \simeq \Map_{\End(\cC)}((X,\Id_x), (A,f))$.

\begin{theorem}\label{thm:prescolim}
If $\cC$ has right-sided inserters, then the diagonal map $\Delta: \cC^\RR \to \End(\cC)$ has a right adjoint and therefore preserves all colimits existing in $\cC^\RR$.
\end{theorem}
\begin{proof}
This follows from the fact that the existence of adjoints can be checked `pointwise', as explained in \cite[A.2.25]{AMG} following the results in \cite[Sec.5.2.2]{LurHTT}. More specifically, any functor between $\infty$-categories $f: \cA \to \cB$ defines a functor $\tilde F = \cB \to \cP(\cA)$ to the category of presheaves on $\cA$, which sends $b \mapsto \Hom_\cB(f(-), b)$. If it exists, the right adjoint to $f$ is a lift of $\tilde F$ along the Yoneda embedding $\cA \hookrightarrow \cP(\cA)$. Since this is a fully faithful embedding, the existence of this lift can be checked pointwise, i.e. by the existence of a representing object for each fixed object $b$ of $\cB$. For the case of the diagonal map $\Delta:\cC^\RR \to \End(\cC)$, this is exactly the assumption of having right-sided inserters.
\end{proof}

\subsubsection{Inserters in dg categories}
This condition is satisfied for dg categories, so the diagonal map will preserve colimits in that setting. More precisely:
\begin{lemma}\label{lem:dgInserters}
	The $(\infty,2)$-category $\DGCat^2$ has right-sided inserters.
\end{lemma}
\begin{proof}
	This follows from the discussion of comma object in the homotopy 2-category of $\infty$-categories in \cite{RieVer}, together with some facts about dg categories. Consider an object $(\cA,f)$ of $\End(\DGCat^2)$, that is, a small dg category $\cA$ and an endomorphism $f \in \End(\cA)$, i.e. an $\Ind(\cA)$-bimodule. We define $\cA^{\Delta^1}$ to be the functor $\infty$-category $\Fun(\Delta^1,\cA)$; this is by construction small and by \cite[Prop.1.1.3.1]{LurHA} this is stable so with the $k$-linear structure inherited from $\cA$ we consider $\cA^{\Delta^1}$ as an object of $\dgcat$. Let us now define the object $\Ins(\Id_\cA,f)$ to be the pullback
	\[\xymatrix{
		\Ins(\Id_\cA,f) \ar[r] \ar[d]	\pullbackcorner		& \cA^{\Delta^1} \ar[d] \\
		\cA \ar[r]_-{\Id_\cA \times f}		& \cA \times \cA
	}\]
	in $\dgcat$, where the right vertical arrow is given by restriction along the boundary inclusion. The underlying stable $\infty$-category then coincides with the pullback taken as $\infty$-categories. 
	
	In \cite[Def.3.3.15]{RieVer}, the comma object of a cospan $\cB \xrightarrow{f} \cA \xleftarrow{g} \cC$ of $\infty$-categories is defined as the pullback
	\[\xymatrix{
		f\narrowdown g \ar[r] \ar[d]	\pullbackcorner		& \cA^{\Delta^1} \ar[d] \\
		\cB \times \cC \ar[r]_-{f \times g}					& \cA \times \cA
	}\]
	and by \cite[Prop.3.3.18]{RieVer}, $f \narrowdown g$ is a weak comma object in the 2-category of $\infty$-categories, i.e. for any $\infty$-category $\cX$ there is a homotopy equivalence between $\Map(\cX, f \narrowdown g)$ and the space of squares of the form
	\[
	\begin{tikzpicture}[baseline={([yshift=-.5ex]current bounding box.center)}]
		\node (a) at (0,1.5) {$\cX$};
		\node (b) at (1.5,1.5) {$\cC$};
		\node (c) at (0,0) {$\cB$};
		\node (d) at (1.5,0) {$\cA$};
		\draw [->] (a) to (b);
		\draw [->] (a) to (c);
		\draw [->] (c) to node [auto, swap] {$g$} (d);
		\draw [->] (b) to node [auto] {$f$} (d);
		\draw [darrow,shorten <=1pt, shorten >=1pt] (c) to node [auto, swap] {$\alpha$} (b);
	\end{tikzpicture}
	\]
	We see then that our object $\Ins(\Id_\cA,f)$ then is also the pullback of $\Id_\cA \narrowdown f$ along the diagonal inclusion
	\[\xymatrix{
		\Ins(\Id_\cA,f) \ar[r] \ar[d]	\pullbackcorner		& \Id_\cA \narrowdown f \ar[d] \\
		\cA \ar[r]_-{\Delta}							& \cA \times \cA
	}\]
	so the equivalence above restricts to an equivalence between $\Map(\cX,\Ins(\cA,f))$ and the space of squares as above where $\cC = \cA$ and $g = \Id_\cA$, which by definition is the universal property required from a right-sided inserter.
\end{proof}

\subsection{Colimits in the endomorphism category}\label{sec:push}
In this section we will discuss a feature of colimits in $\End(\cC)$, which is the fact that a colimit in this category gives colimits in the categories of endomorphisms of single objects in $\cC^\RR$. Though this statement will be phrased in purely $(\infty,1)$-categorical terms, for its proof we will need one last thing about the $(\infty,2)$-categorical structure of $\cC$.
\begin{proposition}\label{prop:cocartesian}
	The $\infty$-functor $p:\End(\cC) \to \cC^\RR$, given by pullback under the inclusion $\pt \hookrightarrow \S$, is a cocartesian fibration of simplicial sets up to equivalence \cite[Def.2.4.2.1]{LurHTT}.
\end{proposition}
\begin{proof}
	We must prove that the map $p$, which is the map forgetting the endomorphism, possibly composed with equivalences of the source and target simplicial sets, is an inner fibration of simplicial sets and that it satisfies the correct lifting property. The former follows trivially from the fact that every functor is an inner fibration up to equivalence; for the latter, it suffices to produce
\end{proof}

\subsubsection{The push functor}
Let $(B,g)$ be an object of $\End(\cC)$, and as usual let us write $\End(\cC)_{/(B,g)}$ for its overcategory. We denote
\[ \eend(B) = \End(\cC) \times_{\cC^\RR} \pt, \quad \eend(B)_{/g} = \End(\cC)_{/(B,g)} \times_{\cC^\RR} \pt \]
where the point maps to the object $B$ of $\cC^\RR$ in the first expression, and to the final object $(B,\id_B)$ in the second. We can identify the objects of $\eend(B)$ with the 1-endomorphisms of the object $B$ in $\cC$, and then, as the notation suggests, $\eend(B)_{/g}$ is naturally identified with the overcategory of $g$.

\begin{lemma}\label{lem:push}
The inclusion $\eend(B)_{/g} \subseteq \End(\cC)_{/(B,g)}$ has a left adjoint, which we denote 
\[ \Push_{(B,g)}: \End(\cC)_{/(B,g)} \to \eend(B)_{/g} \]
such that the image of an object $\left( (A,f) \xrightarrow{(\phi,\alpha)} (B,g) \right)$ under $\Push_{(B,g)}$ is isomorphic to
\[ \left( \phi \circ f \circ \phi^r \xrightarrow{\alpha \circ ((\phi \circ f) \bullet \epsilon)} g \right) \]
in the homotopy category $h\eend(B)_{/g}$, where $\phi^r$ is the right dual of $\phi$ and $\epsilon$ is the counit of that adjunction between $\phi$ and $\phi^r$.
\end{lemma}
\begin{proof}
The first claim, existence of the left adjoint, follows from Proposition \ref{prop:cocartesian} together with \cite[Ex.5.2.7.10 and Cor.5.2.7.11]{LurHTT}, which imply that $\eend(B)$ is a reflexive subcategory of $\End(\cC)_{/(B,g)}$. As for the second claim, it suffices to produce, for each object $\left( (A,f) \xrightarrow{(\phi,\alpha)} (B,g) \right)$, a morphism
\[ f:\left( (A,f) \xrightarrow{(\phi,\alpha)} (B,g) \right) \to \left( (B, \phi \circ f \circ \phi^r) \xrightarrow{(\id_B,\alpha \circ ((\phi \circ f) \bullet \epsilon))} (B,g) \right)\]
which exhibits its target as the localization of its source relative to $\eend(B)_{/g}$. But since $(B,\id_B)$ is a final object of $\cC^\RR_{/B}$, for any object $(A,\phi)$, its localization relative to $\pt$ is given exactly by the canonical morphism $\overline f:(A,\phi)\to (B,\id_B)$. Therefore, by the proof of above-cited corollary, it is enough to pick an $f$ which is a $p$-coCartesian lift of $\overline{f}$, which is exactly what we did in the proof of Proposition \ref{prop:cocartesian}.
\end{proof}

\begin{definition} \label{def:totalpush}
Let $I$ be a small category with a terminal object $\star \in I$.  Consider a diagram $F: I \to \End(\cC)$, 
with $F(\star) = (B, g)$.  
We define 
$$\Pi F: I \to \End(B)$$ 
by composing
$ \Push_{(B,g)} \circ F_{/(B,g)}: I \to \End(B)_{/g}$
with the forgetful map $\End(B)_{/g} \to \End(B)$. 
\end{definition}

\begin{lemma} \label{lem:colimpush} 
If $\colim_{I \setminus *} F = (B, g)$ then $\colim_{I \setminus *} \Pi F = g$. 
\end{lemma}
\begin{proof}
The forgetful map preserves colimits \cite[Prop.1.2.13.8]{LurHTT}, and 
the $\Push$ preserves colimits by Lemma \ref{lem:push}. 
\end{proof} 

More explicitly, Lemma \ref{lem:colimpush} asserts that if we have a diagram $K \xrightarrow{(A,f)} \End(\cC)$, with a colimit $(B,g) = \colim_{i \in K} (A_i,f_i)$, then in $\End(B)$ we have
$\colim_{i \in K} (\phi_i \circ f_i \circ \phi_i^r) \simeq g$, where $\phi_i: A_i \to B$ is the the map in $\cC^\RR$ coming from the maps to the colimit in $\End(\cC)$.

\subsection{The diagonal cosheaf}
\begin{definition}\label{def:diagcosheaf}
	Let $X$ be a topological space, let $\cC$ be an $(\infty,2)$-category, and $\cF$ a $\cC^{\RR}$-valued precosheaf on $X$.  
	The composition $\Delta \cF$ is an $\End(\cC)$ valued precosheaf on $X$, which we term the {\em diagonal precosheaf}. If $\cF$ is a cosheaf and $\cC$ has right-sided inserters, satisfying the conditions of Theorem \ref{thm:prescolim}, then by that theorem $\Delta \cF$ is a cosheaf and we term it the {\em diagonal cosheaf}. 
\end{definition}

To an open set $U \subset X$  the diagonal (pre)cosheaf assigns the pair $(\Delta \cF)(U) = (\cF(U),\Id_{\cF(U)})$, and to an open inclusion $U \subset V$, the map $(\cF(U),\Id_{\cF(U)}) \to (\cF(V),\Id_{\cF(V)})$ given by $(\phi, \id_\phi)$ where $\phi:\cF(U) \to \cF(V)$ is the corestriction.  
\begin{definition} \label{def:pushdiag} 
	In the situation of Definition \ref{def:diagcosheaf} we abbreviate $\DeltaDelta_X \cF := \Pi \Delta \cF$.  It is a precosheaf valued in the $\infty$-category $\End(\cF(X))$. If $\Delta \cF$ is a cosheaf, so is $\DeltaDelta \cF$, by the proof of Lemma \ref{lem:colimpush}.  
\end{definition}

Let us discuss the diagonal map for the case where $\cC$ is the symmetric monoidal $(\infty,2)$-category $\DGCat^2$, as in Section \ref{sec:dgcat}. Recall that ind-completion gives a embedding $\dgcat \hookrightarrow \DGCat$ of small dg categories as a wide subcategory of $\DGCat$, on all isomorphism classes of objects and all right-dualizable morphisms. Therefore for $\cC = \DGCat^2$ we will identify the $(\infty,1))$-category $\cC^\RR$ with $\dgcat$. For any object $\cA$, the endomorphism $(\infty,1)$-category $\eend(\cA)$ is identified with the category of $\cA$-bimodules. Moreover, given a morphism $\phi: A \to B$ in $\dgcat$ and $M \in \eend(A)$, under the identification with bimodules, the endomorphism $\phi \circ M \circ \phi^r$ of $A$ can be identified with the `pushforward bimodule' $\Phi_! M$ in the notation of \cite{BD1}, where $\Phi = \phi^{op}\otimes \phi$. Since $\DGCat^2$ has right-sided inserters (Lemma \ref{lem:dgInserters}), in this context, Lemma \ref{lem:colimpush} means that if we have a diagram $K \xrightarrow{(A,M)} \End(\cC)$ with colimit $(B,N)$, we have an isomorphism $N \cong \colim_{i\in K} ((\Phi_i)_! M)$.

The following statement is well-known; we give a proof using the diagonal cosheaf of bimodules.
\begin{lemma}\label{lem:propsm}
	A finite colimit of smooth dg categories is smooth.
\end{lemma}
\begin{proof}
	Let $\cB$ be the colimit of a finite diagram $B:I \to \dgcat$, where $B: i \mapsto \cB_i$ with $\cB_i$ smooth. By Lemma \ref{lem:colimpush}, there is an isomorphism
	\[ \colim ((\Phi_i)_! \Id_{\cB_i} \simeq \Id_{\cB}, \]
	and since the pushforward functors $(\Phi_i)_!$ preserves compact objects, as a consequence of Lemma 4.2 of \cite{BD1}, $\Id_{\cB}$ is a finite colimit of compact objects, and therefore compact in the category of bimodules.
\end{proof}

Summarizing, a $\dgcat$-valued cosheaf $\cF$ on $X$ gives an $\End(\DGCat)$-valued cosheaf $\Delta \cF$ on $X$. Applying the push construction, we get an $\End(\cF(X))$-valued cosheaf $\DeltaDelta \cF$, which assigns $U \mapsto \Phi_! \Id_{\cF(U)}$, where $\Phi = \phi^{op}\otimes \phi$, with $\phi:\cF(U) \to \cF(X)$ the corestriction map. We believe that the object constructed explicitly in Definition \ref{def:VerdierDualDiagonal} is related to this object by the $\Compacts$-sheaf version of covariant Verdier duality, in other words:
\begin{conjecture}
	For any $(X,\cF)$ as above, there is an equivalence in $\Sh_{\Compacts}(X,\cF(X)^e\mh\Mod)$ between $D_X (\DeltaDelta\cF) \simeq \varphi\DDelta_X\cF$.
\end{conjecture}

\printbibliography

@article{AENV,
  title={Topological strings, {D}-model, and knot contact homology},
  author={Vafa, Cumrun and Aganagic, Mina and Ekholm, Tobias and Ng, Lenhard},
  journal={Advances in Theoretical and Mathematical Physics},
  volume={18},
  number={4},
  pages={827--956},
  year={2014},
  publisher={International Press of Boston}
}

@article{AK,
  title={{Riemann}-{Hilbert} correspondence for holonomic {D}-modules},
  author={D'Agnolo, Andrea and Kashiwara, Masaki},
  journal={Publications math{\'e}matiques de l'IH{\'E}S},
  volume={123},
  pages={69--197},
  year={2016}
}

@article{AMM,
  title={{Lie} group valued moment maps},
  author={Alekseev, Anton and Malkin, Anton and Meinrenken, Eckhard},
  journal={Journal of Differential Geometry},
  volume={48},
  number={3},
  pages={445--495},
  year={1998},
  publisher={Lehigh University}
}

@article{BGP,
  title={{Coxeter} functors and {Gabriel}'s theorem},
  author={Bernstein, IN and Gel'fand, IM and Ponomarev, VA},
  journal={Russian mathematical surveys},
  volume={28},
  number={2},
  pages={17},
  year={1973},
  publisher={Citeseer}
}

@inproceedings{vdB,
  title={Non-commutative quasi-{Hamiltonian} spaces},
  author={Van Den Bergh, Michel},
  booktitle={Proceedings of Poisson Geometry in Mathematics and Physics},
  pages={273--299},
  year={2008},
  publisher={American Mathematical Society}
}

@article{BezKap,
  title={Microlocal sheaves and quiver varieties},
  author={Bezrukavnikov, Roman and Kapranov, Mikhail},
  journal={Annales de la Facult{\'e} des sciences de Toulouse: Math{\'e}matiques},
  volume={25},
  number={2-3},
  pages={473--516},
  year={2016}
}

@article{BGT,
  title={A universal characterization of higher algebraic {K}-theory},
  author={Blumberg, Andrew and Gepner, David and Tabuada, Gon{\c{c}}alo},
  journal={Geometry \& Topology},
  volume={17},
  number={2},
  pages={733--838},
  year={2013},
  publisher={Mathematical Sciences Publishers}
}

@article{Boa1,
  title={Quasi-{Hamiltonian} geometry of meromorphic connections},
  author={Boalch, Philip},
  journal={Duke Math. J.},
  volume={136},
  number={1},
  pages={369--405},
  year={2007}
}

@article{Boa2,
  title={Geometry and braiding of {Stokes} data; fission and wild character varieties},
  author={Boalch, Philip},
  journal={Annals of mathematics},
  pages={301--365},
  year={2014},
  publisher={JSTOR}
}

@article{Boa3,
  title={Global {Weyl} groups and a new theory of multiplicative quiver varieties},
  author={Boalch, Philip},
  journal={Geometry \& Topology},
  volume={19},
  number={6},
  pages={3467--3536},
  year={2016},
  publisher={Mathematical Sciences Publishers}
}

@article{BY,
      title={Twisted wild character varieties}, 
      author={Philip Boalch and Daisuke Yamakawa},
      year={2015},
      eprint={1512.08091},
      archivePrefix={arXiv},
      primaryClass={math.AG}
}

@article{BD1,
  title={Relative {Calabi}-{Yau} structures},
  author={Brav, Christopher and Dyckerhoff, Tobias},
  journal={Compositio Mathematica},
  volume={155},
  number={2},
  pages={372--412},
  year={2019},
  publisher={London Mathematical Society}
}

@article{BD2,
  title={Relative {Calabi}-{Yau} structures II: shifted {Lagrangians} in the moduli of objects},
  author={Brav, Christopher and Dyckerhoff, Tobias},
  journal={Selecta Mathematica},
  volume={27},
  number={4},
  pages={63},
  year={2021},
  publisher={Springer}
}

@inbook{Cal,
  title={Lagrangian structures on mapping stacks and semi-classical {TFTs}},
  author={Calaque, Damien},
  booktitle={Stacks and categories in geometry, topology, and algebra},
  volume={643},
  pages={1--23},
  year={2015},
  publisher={American Mathematical Society}
}

@article{CPTVV,
  title={Shifted {Poisson} structures and deformation quantization},
  author={Calaque, Damien and Pantev, Tony and To{\"e}n, Bertrand and Vaqui{\'e}, Michel and Vezzosi, Gabriele},
  journal={Journal of topology},
  volume={10},
  number={2},
  pages={483--584},
  year={2017},
  publisher={Wiley Online Library}
}

@inproceedings{Cib,
  title={{Hochschild} homology of an algebra whose quiver has no oriented cycles},
  author={Cibils, Claude},
  booktitle={Proceedings of the Fourth International Conference on Representations of Algebras},
  pages={55--59},
  year={1986},
  organization={Springer}
}

@article{Cohn,
  title={Differential Graded Categories are k-linear Stable Infinity Categories}, 
  author={Lee Cohn},
  year={2016},
  eprint={1308.2587},
  archivePrefix={arXiv},
  primaryClass={math.AT}
}

@article{Cor,
  title={{KCH} representations, augmentations, and {A}-polynomials},
  author={Cornwell, Christopher R},
  journal={Journal of Symplectic Geometry},
  volume={15},
  number={4},
  pages={983--1017},
  year={2017},
  publisher={International Press of Boston}
}

@article{Cos1,
  title={Topological conformal field theories and {Calabi}-{Yau} categories},
  author={Costello, Kevin},
  journal={Advances in Mathematics},
  volume={210},
  number={1},
  pages={165--214},
  year={2007},
  publisher={Elsevier}
}

@article{Cos2,
  title={The partition function of a topological field theory},
  author={Costello, Kevin},
  journal={Journal of Topology},
  volume={2},
  number={4},
  pages={779--822},
  year={2009},
  publisher={Wiley Online Library}
}

@article{CS,
  title={Multiplicative preprojective algebras, middle convolution and the {Deligne}-{Simpson} problem},
  author={Crawley-Boevey, William and Shaw, Peter},
  journal={Advances in Mathematics},
  volume={201},
  number={1},
  pages={180--208},
  year={2006},
  publisher={Elsevier}
}

@book{DMR,
  title={Singularit{\'e}s irr{\'e}guli{\`e}res: correspondance et documents},
  author={Deligne, P. and Malgrange, B. and Ramis, J.P.},
  series={Documents math{\'e}matiques},
  year={2007},
  publisher={Soci{\'e}t{\'e} math{\'e}matique de France}
}

@article{Drin,
  title={{DG} quotients of {DG} categories},
  author={Drinfeld, Vladimir},
  journal={Journal of Algebra},
  volume={272},
  number={2},
  pages={643--691},
  year={2004},
  publisher={Elsevier}
}

@article{Efi,
  title={Homotopy finiteness of some {DG} categories from algebraic geometry},
  author={Efimov, Alexander I},
  journal={Journal of the European Mathematical Society},
  volume={22},
  number={9},
  pages={2879--2942},
  year={2020}
}

@article{EENS,
  title={Knot contact homology},
  author={Ekholm, Tobias and Etnyre, John B and Ng, Lenhard and Sullivan, Michael G},
  journal={Geometry \& Topology},
  volume={17},
  number={2},
  pages={975--1112},
  year={2013},
  publisher={Mathematical Sciences Publishers}
}

@article{ENS,
  title={A complete knot invariant from contact homology},
  author={Ekholm, Tobias and Ng, Lenhard and Shende, Vivek},
  journal={Inventiones mathematicae},
  volume={211},
  pages={1149--1200},
  year={2018},
  publisher={Springer}
}

@article{EN,
  title={Higher genus knot contact homology and recursion for colored {HOMFLY}-{PT} polynomials},
  author={Ekholm, Tobias and Ng, Lenhard},
  journal={Advances in Theoretical and Mathematical Physics},
  volume={24},
  number={8},
  year={2020}
}

@article{ES,
  title={Skeins on Branes}, 
  author={Tobias Ekholm and Vivek Shende},
  year={2021},
  eprint={1901.08027},
  archivePrefix={arXiv},
  primaryClass={math.SG}
}

@article{FG,
  title={Moduli spaces of local systems and higher {Teichm{\"u}ller} theory},
  author={Fock, Vladimir and Goncharov, Alexander},
  journal={Publications Math{\'e}matiques de l'IH{\'E}S},
  volume={103},
  pages={1--211},
  year={2006}
}

@article{FST,
  title={Cluster algebras and triangulated surfaces, part {I}: {Cluster} complexes},
  author={Fomin, Sergey and Shapiro, Michael and Thurston, Dylan},
  journal={Acta Mathematica},
  volume={201},
  number={1},
  pages={83--146},
  year={2008},
  publisher={Springer}
}

@book{GR,
  title={A study in derived algebraic geometry, volume {I}: correspondences and duality},
  author={Gaitsgory, Dennis and Rozenblyum, Nick},
  volume={221},
  year={2019},
  publisher={American Mathematical Society}
}

@article{GaSh,
  title={Mirror symmetry for very affine hypersurfaces},
  author={Gammage, Benjamin and Shende, Vivek},
  journal={Acta Mathematica},
  volume={229},
  number={2},
  pages={287--346},
  year={2022},
  publisher={International Press of Boston}
}

@article{Ganatra,
  title={Cyclic homology, $S^1$-equivariant {Floer} cohomology and {Calabi}-{Yau} structures},
  author={Ganatra, Sheel},
  journal={Geometry \& Topology},
  volume={27},
  number={9},
  pages={3461--3584},
  year={2023},
  publisher={Mathematical Sciences Publishers}
}

@article{GPS1,
  title={Covariantly functorial wrapped {Floer} theory on {Liouville} sectors},
  author={Ganatra, Sheel and Pardon, John and Shende, Vivek},
  journal={Publications math{\'e}matiques de l'IH{\'E}S},
  volume={131},
  number={1},
  pages={73--200},
  year={2020},
  publisher={Springer}
}

@article{GPS2,
  title={Sectorial descent for wrapped {Fukaya} categories},
  author={Ganatra, Sheel and Pardon, John and Shende, Vivek},
  journal={Journal of the American Mathematical Society},
  volume={37},
  number={2},
  pages={499--635},
  year={2024}
}

@article{GPS3,
  title={Microlocal {Morse} theory of wrapped {Fukaya} categories},
  author={Ganatra, Sheel and Pardon, John and Shende, Vivek},
  journal={to appear in Annals of Mathematics},
  year={2024}
}

@article{GPSher,
  title={Mirror symmetry: from categories to curve counts}, 
  author={Sheel Ganatra and Timothy Perutz and Nick Sheridan},
  year={2015},
  eprint={1510.03839},
  archivePrefix={arXiv},
  primaryClass={math.SG}
}

@book{GSV,
  title={Cluster algebras and {Poisson} geometry},
  author={Gekhtman, Michael and Shapiro, Michael and Vainshtein, Alek},
  number={167},
  year={2010},
  publisher={American Mathematical Soc.}
}

@article{GSV2,
  title={Cluster algebras and {Weil}-{Petersson} forms},
  author={Gekhtman, Michael and Shapiro, Michael and Vainshtein, Alek},
  journal={Duke Math. J.},
  volume={126},
  number={1},
  pages={291--311},
  year={2005}
}

@article{Ginz,
      title={{Calabi}-{Yau} algebras}, 
      author={Victor Ginzburg},
      year={2007},
      eprint={math/0612139},
      archivePrefix={arXiv},
      primaryClass={math.AG}
}

@article{GM,
  title={Intersection homology {II}},
  author={Goresky, Mark and MacPherson, Robert},
  journal={Inventiones mathematicae},
  volume={72},
  number={1},
  pages={77--129},
  year={1983},
  publisher={Springer}
}

@article{Hap,
  title={On the derived category of a finite-dimensional algebra},
  author={Happel, Dieter},
  journal={Commentarii Mathematici Helvetici},
  volume={62},
  number={1},
  pages={339--389},
  year={1987},
  publisher={Springer}
}

@article{Hau,
  title={The higher {Morita} category of $E_n$-algebras},
  author={Haugseng, Rune},
  journal={Geometry \& Topology},
  volume={21},
  number={3},
  pages={1631--1730},
  year={2017},
  publisher={Mathematical Sciences Publishers}
}

@article{Hau2,
  title={On (co) ends in $\infty$-categories},
  author={Haugseng, Rune},
  journal={Journal of Pure and Applied Algebra},
  volume={226},
  number={2},
  pages={106819},
  year={2022},
  publisher={Elsevier}
}

@article{Hoy,
  title={The homotopy fixed points of the circle action on {Hochschild} homology}, 
  author={Marc Hoyois},
  year={2018},
  eprint={1506.07123},
  archivePrefix={arXiv},
  primaryClass={math.KT}
}

@article{HSS,
  title={Higher traces, noncommutative motives, and the categorified {Chern} character},
  author={Hoyois, Marc and Scherotzke, Sarah and Sibilla, Nicolo},
  journal={Advances in Mathematics},
  volume={309},
  pages={97--154},
  year={2017},
  publisher={Elsevier}
}

@article{Kas,
  title = {Cyclic homology, comodules, and mixed complexes},
  journal = {Journal of Algebra},
  volume = {107},
  number = {1},
  pages = {195-216},
  year = {1987},
  author = {Christian Kassel}
}

@inproceedings{KKP,
  title={{Hodge} theoretic aspects of mirror symmetry},
  author={Katzarkov, L and Kontsevich, M and Pantev, T},
  booktitle={Proceedings of Symposia in Pure Mathematics},
  volume={78},
  pages={87--174},
  year={2008},
  organization={American Mathematical Society}
}

@article{Kel1,
  title={Invariance and localization for cyclic homology of {DG} algebras},
  author={Keller, Bernhard},
  journal={Journal of pure and applied Algebra},
  volume={123},
  number={1-3},
  pages={223--273},
  year={1998},
  publisher={Elsevier}
}

@inproceedings{Kel2,
  title={On differential graded categories},
  author={Keller, Bernhard},
  booktitle={Proceedings of the International Congress of Mathematicians},
  pages={151--190},
  publisher={European Mathematical Society}
}

@inproceedings{Kon,
  title={Homological algebra of mirror symmetry},
  author={Kontsevich, Maxim},
  booktitle={Proceedings of the International Congress of Mathematicians},
  pages={120--139},
  year={1995},
  publisher={Springer}
}

@article{KoSo-alg,
	title={Notes on {A}-infinity algebras, {A}-infinity categories and non-commutative geometry. I},
	author={Maxim Kontsevich and Yan Soibelman},
	year={2006},
	eprint={math/0606241},
	archivePrefix={arXiv},
	primaryClass={math.RA}
}

@article{KoSo,
  title={Stability structures, motivic {Donaldson}-{Thomas} invariants and cluster transformations}, 
  author={Maxim Kontsevich and Yan Soibelman},
  year={2008},
  eprint={0811.2435},
  archivePrefix={arXiv},
  primaryClass={math.AG}
}

@article{KoS,
  title={Cohomological {Hall} algebra, exponential {Hodge} structures and motivic {Donaldson}-{Thomas} invariants},
  author={Kontsevich, Maxim and Soibelman, Yan},
  journal={Communications in Number Theory and Physics},
  volume={5},
  number={2},
  pages={231--252},
  year={2011},
  publisher={International Press of Boston}
}

@book{KS,
  title={Sheaves on Manifolds},
  author={Kashiwara, M. and Schapira, P.},
  series={Grundlehren der mathematischen Wissenschaften},
  year={1990},
  publisher={Springer}
}

@article{KS-hochschild,
  title={Microlocal {Euler} classes and {Hochschild} homology},
  author={Masaki, Kashiwara and Schapira, Pierre},
  journal={Journal of the Institute of Mathematics of Jussieu},
  volume={13},
  number={3},
  year={2014}
}

@article{Ku,
  title={The nonequivariant coherent-constructible correspondence for toric stacks},
  author={Kuwagaki, Tatsuki},
  journal={Duke Mathematical Journal},
  volume={169},
  number={11},
  pages={2125},
  year={2020},
  publisher={Duke University Press}
}

@book{Lod,
  title={Cyclic Homology},
  author={Loday, J.L.},
  series={Grundlehren der mathematischen Wissenschaften},
  year={2013},
  publisher={Springer}
}

@article{Lur-tft,
  title={On the classification of topological field theories},
  author={Lurie, Jacob},
  journal={Current developments in mathematics},
  volume={2008},
  number={1},
  pages={129--280},
  year={2008},
  publisher={International Press of Boston}
}

@misc{LurHA,
  title={Higher Algebra},
  author={Jacob Lurie},
  howpublished={\url{http://people.math.harvard.edu/~lurie/papers/HA.pdf}}
}

@book{LurHTT,
  title={Higher Topos Theory},
  author={Lurie, Jacob},
  series={Annals of Mathematics Studies},
  year={2009},
  publisher={Princeton University Press}
}

@article{Mal,
  title={La classification des connexions irr{\'e}gulieres {\`a} une variable},
  author={Malgrange, Bernard},
  journal={S{\'e}minaire de l'ENS, Math{\'e}matique et Physique},
  volume={1982},
  pages={353--379},
  year={1979}
}

@article{Mei,
  title={Convexity for twisted conjugation},
  author={Meinrenken, E},
  journal={Mathematical Research Letters},
  volume={24},
  number={6},
  pages={1797--1818},
  year={2017},
  publisher={International Press of Boston}
}

@article{N1,
  title={Microlocal branes are constructible sheaves},
  author={Nadler, David},
  journal={Selecta Mathematica},
  volume={15},
  number={4},
  pages={563--619},
  year={2009},
  publisher={Springer}
}

@article{N3,
  title={Arboreal singularities},
  author={Nadler, David},
  journal={Geometry \& Topology},
  volume={21},
  number={2},
  pages={1231--1274},
  year={2017},
  publisher={Mathematical Sciences Publishers}
}

@article{N4,
  title={Non-characteristic expansions of {Legendrian} singularities}, 
  author={David Nadler},
  year={2016},
  eprint={1507.01513},
  archivePrefix={arXiv},
  primaryClass={math.SG}
}

@article{N5,
  title={Wrapped microlocal sheaves on pairs of pants}, 
  author={David Nadler},
  year={2016},
  eprint={1604.00114},
  archivePrefix={arXiv},
  primaryClass={math.SG}
}

@article{N6,
  title={Mirror symmetry for the {Landau}-{Ginzburg} $A$-model $M = \mathbb{C}^n, W = z_1 \dots z_n$},
  author={Nadler, David},
  journal={Duke Mathematical Journal},
  volume={168},
  number={1},
  year={2019}
}

@article{NS,
  title={Sheaf quantization in {Weinstein} symplectic manifolds}, 
  author={David Nadler and Vivek Shende},
  year={2022},
  eprint={2007.10154},
  archivePrefix={arXiv},
  primaryClass={math.SG}
}

@article{NZ,
  title={Constructible sheaves and the {Fukaya} category},
  author={Nadler, David and Zaslow, Eric},
  journal={Journal of the American Mathematical Society},
  volume={22},
  number={1},
  pages={233--286},
  year={2009}
}

@article{Ng,
  title={Framed knot contact homology},
  author={Ng, Lenhard},
  journal={Duke Math. J.},
  volume={141},
  number={1},
  pages={365--406},
  year={2008}
}

@article{PT1,
  title={Moduli of flat connections on smooth varieties},
  author={Pantev, Tony and To{\"e}n, Bertrand},
  journal={Algebraic Geometry},
  volume={9},
  number={3},
  year={2022}
}

@article{PT2,
  title={{Poisson} geometry of the moduli of local systems on smooth varieties},
  author={Pantev, Tony and To{\"e}n, Bertrand},
  journal={Publications of the Research Institute for Mathematical Sciences},
  volume={57},
  number={3},
  pages={959--991},
  year={2021}
}

@article{PTVV,
  title={Shifted symplectic structures},
  author={Pantev, Tony and To{\"e}n, Bertrand and Vaqui{\'e}, Michel and Vezzosi, Gabriele},
  journal={Publications math{\'e}matiques de l'IH{\'E}S},
  volume={117},
  pages={271--328},
  year={2013}
}

@inproceedings{Par,
  title={{t}-structures in the derived category of representations of quivers},
  author={Parthasarathy, R},
  booktitle={Proceedings of the Indian Academy of Sciences-Mathematical Sciences},
  volume={98},
  number={2},
  pages={187--214},
  year={1988},
  organization={Springer}
}

@article{Pos,
  title={Total positivity, {Grassmannians}, and networks}, 
  author={Alexander Postnikov},
  year={2006},
  eprint={math/0609764},
  archivePrefix={arXiv},
  primaryClass={math.CO}
}

@article{Saf1,
  title={Quasi-{Hamiltonian} reduction via classical {Chern}-{Simons} theory},
  author={Safronov, Pavel},
  journal={Advances in Mathematics},
  volume={287},
  pages={733--773},
  year={2016},
  publisher={Elsevier}
}

@article{Saf2,
  title={Symplectic implosion and the {Grothendieck}-{Springer} resolution},
  author={Safronov, Pavel},
  journal={Transformation Groups},
  volume={22},
  pages={767--792},
  year={2017},
  publisher={Springer}
}

@article{FLTZ,
  title={A categorification of {Morelli}’s theorem},
  author={Fang, Bohan and Liu, Chiu-Chu Melissa and Treumann, David and Zaslow, Eric},
  journal={Inventiones mathematicae},
  volume={186},
  number={1},
  pages={79--114},
  year={2011},
  publisher={Springer}
}

@article{She,
  title={The conormal torus is a complete knot invariant},
  author={Shende, Vivek},
  journal={Forum of Mathematics, Pi},
  volume={7},
  year={2019},
  publisher={Cambridge University Press}
}

@article{Sh,
  title={Microlocal category for {Weinstein} manifolds via h-principle}, 
  author={Vivek Shende},
  year={2017},
  eprint={1707.07663},
  archivePrefix={arXiv},
  primaryClass={math.SG}
}

@article{STW,
  title={On the combinatorics of exact {Lagrangian} surfaces}, 
  author={Vivek Shende and David Treumann and Harold Williams},
  year={2016},
  eprint={1603.07449},
  archivePrefix={arXiv},
  primaryClass={math.SG}
}

@article{STWZ,
  title={Cluster varieties from {Legendrian} knots},
  author={Shende, Vivek and Treumann, David and Williams, Harold and Zaslow, Eric},
  journal={Duke Mathematical Journal},
  volume={168},
  number={15},
  pages={2801--2871},
  year={2019},
  publisher={Duke University Press}
}

@article{STZ,
  title={Legendrian knots and constructible sheaves},
  author={Shende, Vivek and Treumann, David and Zaslow, Eric},
  journal={Inventiones mathematicae},
  volume={207},
  pages={1031--1133},
  year={2017},
  publisher={Springer}
}

@article{Sher,
  title={Formulae in noncommutative Hodge theory},
  author={Sheridan, Nick},
  journal={Journal of Homotopy and Related Structures},
  volume={15},
  number={1},
  pages={249--299},
  year={2020},
  publisher={Springer}
}

@article{Shk,
  title={{Hirzebruch}-{Riemann}-{Roch}-type formula for {DG} algebras},
  author={Shklyarov, Dmytro},
  journal={Proceedings of the London Mathematical Society},
  volume={106},
  number={1},
  pages={1--32},
  year={2013},
  publisher={Oxford University Press}
}

@article{Sim,
  title={Higgs bundles and local systems},
  author={Simpson, Carlos},
  journal={Publications Math{\'e}matiques de l'IH{\'E}S},
  volume={75},
  pages={5--95},
  year={1992}
}

@article{Sta,
  title={Arboreal singularities in Weinstein skeleta},
  author={Starkston, Laura},
  journal={Selecta Mathematica},
  volume={24},
  pages={4105--4140},
  year={2018},
  publisher={Springer}
}

@article{Sto,
  title={On the discontinuity of arbitrary constants which appear in divergent developments},
  author={Stokes, George},
  journal={Transactions of the Cambridge Philosophical Society},
  volume={10},
  pages={105},
  year={1864}
}

@article{Thu,
  title={From Dominoes to Hexagons}, 
  author={Dylan P. Thurston},
  year={2016},
  eprint={math/0405482},
  archivePrefix={arXiv},
  primaryClass={math.CO}
}

@article{Toe1,
  title={The homotopy theory of dg-categories and derived Morita theory},
  author={To{\"e}n, Bertrand},
  journal={Inventiones mathematicae},
  volume={167},
  number={3},
  pages={615--667},
  year={2007},
  publisher={Springer}
}

@article{Toe2,
  title={Derived algebraic geometry},
  author={To{\"e}n, Bertrand},
  journal={EMS Surveys in Mathematical Sciences},
  volume={1},
  number={2},
  pages={153--240},
  year={2014}
}

@article{Toe3,
  title={Structures symplectiques et de Poisson sur les champs en cat{\'e}gories}, 
  author={Bertrand To{\"e}n},
  year={2018},
  eprint={1804.10444},
  archivePrefix={arXiv},
  primaryClass={math.AG}
}

@inproceedings{TVa,
  title={Moduli of objects in dg-categories},
  author={To{\"e}n, Bertrand and Vaqui{\'e}, Michel},
  booktitle={Annales scientifiques de l'Ecole normale sup{\'e}rieure},
  volume={40},
  number={3},
  pages={387--444},
  year={2007}
}

@article{Vol,
  title={The six operations in topology}, 
  author={Marco Volpe},
  year={2023},
  eprint={2110.10212},
  archivePrefix={arXiv},
  primaryClass={math.AT}
}

@article{Vol2,
  title={Verdier duality on conically smooth stratified spaces}, 
  author={Marco Volpe},
  year={2023},
  eprint={2206.02728},
  archivePrefix={arXiv},
  primaryClass={math.AT}
}

@article{Yam,
  title={Geometry of multiplicative preprojective algebra},
  author={Yamakawa, Daisuke},
  journal={International Mathematics Research Papers},
  year={2008},
  publisher={Oxford University Press}
}

@article{Yeu,
  title={Pre-Calabi-Yau structures and moduli of representations}, 
  author={Wai-Kit Yeung},
  year={2022},
  eprint={1802.05398},
  archivePrefix={arXiv},
  primaryClass={math.AG}
}

@article{GepHau,
  title={Enriched $\infty$-categories via non-symmetric $\infty$-operads},
  author={Gepner, David and Haugseng, Rune},
  journal={Advances in mathematics},
  volume={279},
  pages={575--716},
  year={2015},
  publisher={Elsevier}
}

@article{Hei,
  title={An equivalence between enriched $\infty$-categories and $\infty$-categories with weak action},
  author={Heine, Hadrian},
  journal={Advances in Mathematics},
  volume={417},
  pages={108941},
  year={2023},
  publisher={Elsevier}
}

@article{JFS,
  title={(Op) lax natural transformations, twisted quantum field theories, and {``even higher''} Morita categories},
  author={Johnson-Freyd, Theo and Scheimbauer, Claudia},
  journal={Advances in Mathematics},
  volume={307},
  pages={147--223},
  year={2017},
  publisher={Elsevier}
}

@phdthesis{Barw,
  title={$(\infty,n)$-Cat as a closed model category},
  author={Barwick, Clark},
  year={2005},
  school={University of Pennsylvania}
}

@article{BSP,
  title={On the unicity of the theory of higher categories},
  author={Barwick, Clark and Schommer-Pries, Christopher},
  journal={Journal of the American Mathematical Society},
  volume={34},
  number={4},
  pages={1011--1058},
  year={2021}
}

@article{Str,
  title={Limits indexed by category-valued 2-functors},
  author={Street, Ross},
  journal={Journal of Pure and Applied Algebra},
  volume={8},
  number={2},
  pages={149--181},
  year={1976},
  publisher={Elsevier}
}

@article{JT,
  title={Quasi-categories vs Segal spaces},
  author={Joyal, Andr{\'e} and Tierney, Myles},
  journal={Contemporary Mathematics},
  volume={431},
  number={277-326},
  pages={10},
  year={2007},
  publisher={Providence, RI: American Mathematical Society}
}

@article{AMG,
      title={Model $\infty$-categories I: some pleasant properties of the $\infty$-category of simplicial spaces}, 
      author={Aaron Mazel-Gee},
      year={2015},
      eprint={1412.8411},
      archivePrefix={arXiv},
      primaryClass={math.AT}
}

@article{RieVer,
  title={The 2-category theory of quasi-categories},
  author={Riehl, Emily and Verity, Dominic},
  journal={Advances in Mathematics},
  volume={280},
  pages={549--642},
  year={2015},
  publisher={Elsevier}
}

@misc{kerodon,
    title        = {Kerodon},
    author       = {Jacob Lurie},
    howpublished = {\url{https://kerodon.net}},
    year         = {2024},
  }

\end{document}